\documentclass[a4paper,12pt]{amsart}
\usepackage{amssymb,amsmath,amsthm} 
\usepackage[all]{xy}
\usepackage{pstricks}
\usepackage{supertabular}
\usepackage{array}\newcolumntype{R}{>{$}r<{$}}
  \newtheorem{assumption}[equation]{Assumption}
  \newtheorem{axiom}[equation]{Axiom}
  \newtheorem{axioms}[equation]{Axioms}
  \newtheorem{conjecture}[equation]{Conjecture}
  \newtheorem{corollary}[equation]{Corollary}
  \newtheorem{definition}[equation]{Definition}
  \newtheorem{lemma}[equation]{Lemma}
  \newtheorem{proposition}[equation]{Proposition}
  \newtheorem{theocon}[equation]{Theorem--Conjecture}
  \newtheorem{theorem}[equation]{Theorem}
\theoremstyle{remark}
  \newtheorem{remark}[equation]{Remark}
  \newtheorem{remarks}[equation]{Remarks}
  \newtheorem{comment}[equation]{Comment}
  \newtheorem{notation}[equation]{Notation}
\newcommand\ornc{{|\BG_\nc|}}
\newcommand\orc{{|\BG_\oc|}}
\numberwithin{equation}{section}
\swapnumbers
  \def\BC{{ \mathbb C}}
  \def\BF{{ \mathbb F}}
  \def\BG{{ \mathbb G}}
  \def\BL{{ \mathbb L}}
  \def\BN{{ \mathbb N}}
  \def\BQ{{ \mathbb Q}}
  \def\BR{{ \mathbb R}}
  \def\BS{{ \mathbb S}}
  \def\BT{{ \mathbb T}}
  \def\BZ{{ \mathbb Z}}
  \def\ba{{ \mathbf a}}
  \def\bB{{ \mathbf B}}
  \def\bC{{ \mathbf C}}
  \def\bg{{ \mathbf g}}
  \def\bG{{ \mathbf G}}
  \def\bL{{ \mathbf L}}
  \def\bm{{ \mathbf m}}
  \def\bP{{ \mathbf P}}
  \def\bs{{ \mathbf s}}
  \def\bT{{ \mathbf T}}
  \def\bu{{ \mathbf u}}
  \def\bv{{ \mathbf v}}
  \def\bw{{ \mathbf w}}
  \def\bW{{ \mathbf W}}
  \def\bX{{ \mathbf X}}
  \def\bz{{ \mathbf z}}
  \def\fS{{\mathfrak S}}
  \def\CA{{ \mathcal A}}
  \def\CB{{ \mathcal B}}
  \def\CC{{ \mathcal C}}
  \def\CF{{ \mathcal F}}
  \def\CH{{ \mathcal H}}
  \def\CO{{ \mathcal O}}
  \def\CT{{ \mathcal T}}
  \def\al{{\alpha}}
  \def\be{\beta}
  \def\de{\delta}
  \def\De{\Delta}
  \def\e{\varepsilon}
  \def\ga{{\gamma}}
  \def\la{{\lambda}}
  \def\vp{\varphi}
  \def\P{{\Phi}}
  \def\om{\omega}
  \def\rh{\rho}
  \def\si{\sigma}
  \def\th{\theta}
  \def\z{\zeta}
  \def\bmu{{ \boldsymbol\mu}}
  \def\bpi{{ \boldsymbol\pi}}
  \def\bvp{{ \boldsymbol\varphi}}
  \def\brh{{ \boldsymbol\rho}}

  \def\ra{\rightarrow}
  \def\iso{{\,\overset{\sim}{\longrightarrow}\,}\,}
  \def\lexp#1#2{\kern\scriptspace\vphantom{#2}^{#1}\kern-\scriptspace#2}
  \def\inv{^{-1}}
  \def\conj#1{{#1^*}}
  \def\scal#1#2{{\langle{#1},{#2}\rangle}}
  \def\genby#1{{\mathopen\langle#1\mathclose\rangle}}
  \def\po#1{{\overset{#1}{\rightarrow}}}
  \def\ov#1{\overline#1}
  \def\BQl{{\BQ_\ell}}
  \def\Ad{{\operatorname{Ad}}}
  \def\ch{{\operatorname{ch}}}
  \def\Deg{{\operatorname{Deg}}}
  \def\det{{\operatorname{det}}}
  \def\dim{{\operatorname{dim}}}
  \def\Disc{{\operatorname{Disc}}}
  \def\End{{\operatorname{End}}}
  \def\Fam{{\operatorname{Fam}}}
  \def\Feg{{\operatorname{Feg}}}
  \def\fr{{\operatorname{fr}}} 
  \def\Fr{{\operatorname{{Fr}}}}
  \def\Gal{{\operatorname{Gal}}}
  \def\GL{{\operatorname{GL}}}
  \def\grchar{{\operatorname{grchar}}}
  \def\gr{{\operatorname{gr}}}
  \def\Hom{{\operatorname{Hom}}}
  \def\id{{\operatorname{Id}}}
  \def\Id{{\operatorname{Id}}}
  \def\Ind{{\operatorname{Ind}}}
  \def\Infl{{\operatorname{Infl}}}
  \def\Irr{{\operatorname{Irr}}}
  \def\lcm{{\operatorname{lcm}}}
  \def\max{{\operatorname{max}}}
  \def\nc{{\operatorname{nc}}}
  \def\OCFuf{{\operatorname{CF}_{\operatorname{uf}}}}
  \def\oc{{\operatorname{c}}}
  \def\op{{\operatorname{op}}}
  \def\Out{{\operatorname{Out}}}
  \def\reg{{\operatorname{reg}}}
  \def\Reg{{\operatorname{Reg}}}
  \def\Res{{\operatorname{Res}}}
  \def\Trace{{\operatorname{Trace}}}
  \def\tr{{\operatorname{tr}}}
  \def\Uch{{\operatorname{Uch}}}
  \def\Ucsh{{\operatorname{Ucsh}}}
  \def\Un{{\operatorname{Un}}}
  \def\val{{\operatorname{val}}}
  \def\Qlbar{{\bar\BQ_\ell}}
  \def\bGF{{{\bG}^F}}
  \def\bTF{{{\bT}^F}}
  \def\Nh{N^{\operatorname{hyp}}}
  \def\Nr{N^{\operatorname{ref}}}
  \def\ResLG{\Res_\BL^\BG}
  \def\IndLG{\Ind_\BL^\BG}

\def\CHEVIE{{\tt CHEVIE}}
\def\GAP{{\tt GAP}}
  \def\bul{$\bullet\,\,\,$}
  \def\aposteriori{{\it a posteriori\/}}
  \def\eg{{\it e.g.\/}}
  \def\ie{{\it i.e.\/},}
  \def\cf{{cf.\/}}
  \def\via{{\it via\/}}
  \def\viceversa{{\it vice versa\/}}
  \def\viz{{\it viz.\/}}
\def\today{\ifcase\month\or January\or February\or
 March\or April\or May\or June\or July\or August\or
 September\or October\or November\or December\fi
 \space \number\day,\space\number\year}
\title[Split Spetses for primitive reflection groups]
   {Split Spetses for primitive reflection groups}
\author{Michel Brou\'e, Gunter Malle, Jean Michel}
\date{\today} 
\address{Michel Brou\'e, Institut Universitaire de France et
  Universit\'e Paris-Diderot Paris VII,
  CNRS--Institut de Math\'ematiques de Jussieu, 
  175 rue du Chevaleret, F--75013 Paris, France}
  \email{broue@math.jussieu.fr}
  \urladdr{http://people.math.jussieu.fr/~broue/}
\address{Gunter Malle, FB Mathematik, TU Kaiserslautern, Postfach 3049
   D--67653 Kaiserslautern, Germany}
  \email{malle@mathematik.uni-kl.de}
  \urladdr{http://www.mathematik.uni-kl.de/~malle/en/}
\address{Jean Michel, CNRS--Universit\'e Paris VII, Institut de Math\'ematiques de Jussieu, 
  175 rue du Chevaleret, F--75013 Paris, France}
  \email{jmichel@math.jussieu.fr}
  \urladdr{http://people.math.jussieu.fr/~jmichel/}
  \thanks{We thank the CIRM at Luminy and the MFO at Oberwolfach
  where we stayed and worked several times on this paper under very good conditions}
\makeindex
\begin{document}
\begin{abstract}
  Let  $(V,W)$ be an exceptional spetsial irreducible
  reflection  group $W$ on a complex vector space $V$, \ie\ a group $G_n$
  for
  $n \in \{ 4, 6, 8, 14, 23, 24, 25,  26, 27, 28, 29, 30, 32, 33, 34, 35, 36,  37\}$
  in   the  Shephard-Todd notation. We describe how to determine some data
  associated to the corresponding  (split) ``spets'',  given  complete  knowledge
  the same data for all proper subspetses (the method is thus inductive).

  The data determined  here is the set $\Uch(\BG)$ of
  ``unipotent characters'' of $\BG$ and the associated set
  of Frobenius eigenvalues, and its repartition
  into  families. The determination of  the  Fourier matrices linking unipotent
  characters and ``unipotent character sheaves'' will be given in another
  paper.

  The approach works for  all split reflection cosets for
  primitive irreducible reflection groups. The result is that all the above data
  exist and are unique (note that the cuspidal unipotent degrees
  are only determined up to sign).

  We  keep  track of the  complete list of  axioms used. In order to do that,
  we expose in detail some general axioms of ``spetses'', 
  generalizing (and sometimes correcting)
  \cite{sp1} along the way.

  Note that
  to make the induction work, we must consider a class of reflection
  cosets slightly more  general than the  split irreducibles ones:  the 
  reflection cosets with split
  semi-simple  part, \ie\ cosets  $(V,W\vp)$ such  that 
  $V=V_1\oplus V_2$  
  with  $W\subset\GL(V_1)$  and  $\vp|_{V_1} = \id$.  We need also to
  consider  some non-exceptional cosets,  those associated to imprimitive
  complex  reflection groups  which appear  as parabolic  subgroups of the
  exceptional ones.
\end{abstract}
\maketitle
\tableofcontents

{\red \section{\red From Weyl groups to complex reflection groups}\hfill
\smallskip
}

  Let $\bG$ be a connected reductive algebraic group over an algebraic closure
  of a finite field $\BF_q$ and $F:\bG\rightarrow\bG$ an isogeny such that 
  $F^\de$ (where $\de$ is a natural integer)
  defines an $\BF_{q^\de}$-rational structure on $\bG$. The group of fixed points
  $G:=\bG^F$ is a finite group of Lie type, also called finite reductive group. Lusztig has given a classification
  of the irreducible complex characters of such groups. In particular he has
  constructed the important subset $\Un(G)$\index{UchG@$\Un(G)$} of unipotent characters of $G$.
  In a certain sense, which is made precise by Lusztig's Jordan decomposition
  of characters, the unipotent characters of $G$ and of various Levi subgroups of
  $G$ determine all irreducible characters of $G$. 
  
  The unipotent characters are constructed as constituents of representations
  of $G$ on certain $\ell$-adic cohomology groups, on which $F^\de$ also acts.
  Lusztig shows that for a given unipotent character $\gamma\in\Un(G)$, there
  exists a root of unity or a root of unity times the square root of $q^\de$,
  that we denote $\Fr(\gamma)$, such that the eigenvalue of $F^\de$ on any
  $\gamma$-isotypic part of such $\ell$-adic cohomology groups is given by
  $\Fr(\gamma)$ times an integral power of $q^\de$. 
  
  The unipotent characters are naturally partitioned into so-called Harish-Chandra
  series, as follows. If $\bL$ is an $F$-stable Levi subgroup of some $F$-stable
  parabolic subgroup $\bP$ of $\bG$, then {\em Harish-Chandra induction}
  \index{RLG@$R_L^G$}
  $$
    R_L^G:=\Ind_P^G\circ\Infl_L^P:\BZ\Irr(L)\longrightarrow\BZ\Irr(G)
  $$
  where $L:=\bL^F$ and $P:=\bP^F$ defines a homomorphism of character groups
  independent of the choice of $\bP$. A unipotent character of $G$ is called
  {\em cuspidal} if it does not occur in $R_L^G(\lambda)$ for any
  proper Levi subgroup $L<G$ and any $\lambda\in\Un(L)$. The set of constituents
  $$\index{UchGL@$\Un(G,(L,\lambda))$}
    \Un(G,(L,\lambda)):=
      \{\gamma\in\Un(G)\mid
                         \langle\gamma,R_L^G(\lambda)\rangle\ne0\}
  $$
  where $\lambda\in\Un(L)$ is cuspidal, is called the
  {\em Harish-Chandra series} above $(L,\lambda)$. It can be shown that the
  Harish-Chandra series form a partition of $\Un(G)$, if $(L,\lambda)$ runs
  over a system of representatives of the $G$-conjugacy classes of such pairs.
  Thus, given $\gamma\in\Un(G)$ there is a unique pair $(L,\lambda)$ up to
  conjugation such that $L$ is a Levi subgroup of $G$, $\lambda\in\Un(L)$ is
  cuspidal and $\gamma$ occurs as a constituent in $R_L^G(\lambda)$.
  Furthermore, if $\gamma\in\Un(G,(L,\lambda))$ then $\Fr(\gamma)=\Fr(\lambda)$.

  Now let $W_G(L,\lambda):=N_G(\bL,\lambda)/L$\index{WGL@$W_G(L,\lambda)$}, 
  the {\em relative Weyl group} of
  $(L,\lambda)$. This is always a finite Coxeter group. Then
  $\End_{\BC G}(R_L^G(\lambda))$ is an Iwahori-Hecke algebra 
  $\CH(W_G(L,\lambda))$\index{HWGL@$\CH(W_G(L,\lambda))$}
  for $W_G(L,\lambda)$ for a suitable choice of parameters. This gives a
  natural parametrization of $\Un(G,(L,\lambda))$ by characters of
  $\CH(W_G(L,\lambda))$, and thus, after a choice of a suitable specialization
  for the corresponding generic Hecke algebra,
  a parametrization
  $$
    \Irr(W_G(L,\lambda))\longrightarrow\Un(G,(L,\lambda))
    \,,\quad
    \chi\mapsto\gamma_\chi
    \,,
  $$
  of the Harish-Chandra series above $(L,\lambda)$
  by $\Irr(W_G(L,\lambda))$. In particular, the
  characters in the {\em principal series} $\Un(G,(T,1))$, where $T$ denotes a
  maximally split torus, are indexed by $\Irr(W^F)$, the irreducible
  characters of the $F$-fixed points of the Weyl group $W$. 
  
  More generally, if $d\geq 1$ is an integer and if $\bT$ is an $F$-stable subtorus of $\bG$
  such that
  \begin{itemize}
    \item
      $\bT$ splits completely over $\BF_{q^d}$
    \item
      but no subtorus of $\bT$ splits over any smaller field, 
  \end{itemize}
  then
  its centralizer $\bL:=C_\bG(\bT)$ is an $F$-stable {\em $d$-split} Levi
  subgroup (not necessarily lying in an $F$-stable parabolic subgroup). 
  We assume here and in the rest of the introduction that $F$ is a Frobenius
  endomorphism to simplify the exposition; for the ``very twisted'' Ree and 
  Suzuki groups one has to replace $d$ by a cyclotomic polynomial over an
  extension of the rationals as is done in \ref{phigroup}.

  Here,
  again using $\ell$-adic cohomology of suitable varieties Lusztig induction
  defines a linear map
  $$
    R_L^G:\BZ\Irr(L)\longrightarrow\BZ\Irr(G)
    \,,
  $$
  where again $L:=\bL^F$. As before we say that $\gamma\in\Un(G)$
  is {\em $d$-cuspidal} if it does not occur in $R_L^G(\lambda)$ for any
  proper $d$-split Levi subgroup $L<G$ and any $\lambda\in\Un(L)$, and we write
  $\Un(G,(L,\lambda))$ for the set of constituents of $R_L^G(\lambda)$,
  when $\lambda\in\Un(L)$ is $d$-cuspidal. By \cite[3.2(1)]{bmm} these
  \emph{$d$-Harish-Chandra series}, for any fixed $d$, again form a partition
  of $\Un(G)$. 
  The relative
  Weyl groups $W_G(L,\lambda):=N_G(\bL,\lambda)/L$ are now in general complex
  reflection groups. It is shown (see \cite[3.2(2)]{bmm}) that
  again there exists a parametrization of $\Un(G,(L,\lambda))$ by the
  irreducible characters of some cyclotomic Hecke algebra $\CH(W_G(L,\lambda))$
  of $W_G(L,\lambda)$ and hence, after a choice of a suitable specialization
  for the corresponding generic Hecke algebra, a parame\-trization
  $$
    \Irr(W_G(L,\lambda))\longrightarrow\Un(G,(L,\lambda))
    \,,\quad
       \chi\mapsto\gamma_\chi
       \,,
  $$
  of the $d$-Harish-Chandra series above $(L,\lambda)$
  by $\Irr(W_G(L,\lambda))$. Furthermore, there exist signs $\epsilon_\chi$ such
  that the degrees of characters belonging to $\Un(G,(L,\lambda))$ are given by
  $$
    \gamma_\chi(1) = \epsilon_\chi\,\lambda(1)/S_\chi
    \,,
  $$
  where $S_\chi$ denotes the Schur element of $\chi$ with respect to the
  canonical trace form on $\CH(W_G(L,\lambda))$ (see \cite[\S 7]{maG} for
  references).

  Attached to $(\bG,F)$ is the set $\Ucsh(G)$ of characteristic functions of
  $F$-stable unipotent character sheaves of $\bG$. Lusztig showed that these
  are linearly independent and span the same subspace of $\BC\Irr(G)$ as
  $\Un(G)$. The base change matrix $S$ from $\Un(G)$ to $\Ucsh(G)$ is called
  the Fourier matrix of $G$. Define an equivalence relation on $\Un(G)$ as
  the transitive closure of the following relation:
  $$
    \gamma\sim\gamma' \Longleftrightarrow
    \text{ there exists }
    A\in\Ucsh(G)
    \text{ with }
    \langle\gamma,A\rangle\ne0\ne\langle\gamma',A\rangle
    \,.
  $$
  The equivalence classes of this relation partition $\Un(G)$ (and also
  $\Ucsh(G)$) into so-called {\em families}. Lusztig shows that the intersection of any family
  with the principal series $\Un(G,(T,1))$, 
  is a two-sided cell in $\Irr(W^F)$
  (after identification of $\Irr(W^F)$ with the principal series $\Un(G,(T,1))$ as above). 

  All of the above data are \emph{generic} in the
  following sense. Let $\BG$ denote the complete root datum of $(\bG,F)$, that
  is, the root datum of $\bG$ together with the action of $q^{-1}F$ on it.
  Then there is a set $\Uch(\BG)$, together with maps
  $$
  \begin{aligned}
        \Deg:\Uch(\BG)\longrightarrow\BQ[x]\,,\,
        \gamma \mapsto \Deg(\gamma)\,,\,\\
        \lambda : \Uch(\BG)\longrightarrow\BC^\times[x^{1/2}] \,,\,
        \gamma \mapsto \Fr(\gamma)\,,\,
   \end{aligned}
  $$
  such that for all groups $(\bG',F')$ with the same complete root datum $\BG$
  (where ${F'}^\de$ defines a $\BF_{{q'}^\de}$-rational structure)
  there are bijections $\psi_{G'}:\Uch(\BG)\longrightarrow\Un({\bG'}^{F'})$
  satisfying
  $$
    \psi_{G'}(\gamma)(1)=\Deg(\gamma)(q')\quad\text{and}\quad
    \Fr(\psi_{G'}(\gamma))=\Fr(\gamma)(q^{\prime\de})
    \,.
  $$
  Furthermore, by results of Lusztig and Shoji, Lusztig induction $R_L^G$
  of unipotent characters is
  generic, that is, for any complete Levi root subdatum $\BL$ of $\BG$ with
  corresponding Levi subgroup $L$ of $G$ there is a linear map
  $$
    R_\BL^\BG : \BZ\Uch(\BL)\longrightarrow\BZ\Uch(\BG)
  $$\index{RLBGB@$R_\BL^\BG$}
  satisfying
  $$
    R_L^G\circ\psi_L=\psi_G\circ R_\BL^\BG
  $$
 (see \cite[1.33]{bmm}).

  The following has been observed on the data: for $W$ irreducible
  and any scalar $\xi\in Z(W)$ there is
  a permutation with signs $E_\xi$ of $\Uch(\BG)$ such that
  $$
    \Deg({E_\xi(\gamma)})(x) = \Deg(\gamma)(\xi^{-1}x)
    \,.
  $$
  We call this the \emph{Ennola-transform},
  by analogy with what Ennola first observed on the relation between
  characters of general linear and unitary groups.
  In the case considered here, $Z(W)$ has
  order at most~2. Such a permutation $E_\xi$ turns out to be of order
  the square of the order of $\xi$
  if $W$ is of type $E_7$ or $E_8$, and of the same order as $\xi$
  otherwise.
  
  Thus, to any pair consisting of a finite Weyl group $W$
  and the automorphism induced by $F$ on its reflection representation,
  is associated a complete root
  datum $\BG$, and to this is associated a set
  $\Uch(\BG)$ with maps $\Deg$, $\Fr$, $E_\xi$ ($\xi\in Z(W)$)
  and linear maps $R_\BL^\BG$ for any
  Levi  subdatum $\BL$ satisfying a long list of properties.
\medskip

  Our aim is to try and treat a complex reflection group as a Weyl group
  of some yet unknown object.
  Given $W$ a finite subgroup generated by (pseudo)-reflections of a finite
  dimensional complex vector space $V$, and a finite order automorphism $\vp$
  of $V$ which normalizes $W$, we first define the corresponding
  \emph{reflection coset} by $\BG := (V,W\vp)$. Then we try to
  build ``unipotent characters'' of $\BG$, or at least to build their degrees (polynomials
  in $x$), Frobenius eigenvalues (roots of unity times a power (modulo 1) of
  $x$); in a coming paper we shall
  build their Fourier matrices.

  Lusztig (see \cite{app} and \cite{exo}) knew already a solution for Coxeter groups
  which are not   Weyl groups,
  except the Fourier matrix for $H_4$ which was determined by Malle in 1994
  (see \cite{maFou}).

  Malle gave a solution for imprimitive spetsial complex reflection groups in
  1995 (see \cite{maIG}) and proposed (unpublished) data for many primitive spetsial groups.

  Stating now a long series of precise axioms --- many of a technical nature ---
  we can now show that \emph{there is a unique solution for all primitive 
  spetsial complex reflection groups}, \ie\ groups $G_n$
  for
  $n \in \{ 4, 6, 8, 14, 23, 24, 25,  26, 27, 28, 29, 30, 32, 33, 34, 35, 36,  37\}$
  in   the  Shep\-hard--Todd notation,
  and the symmetric groups.
\smallskip
 
  Let us introduce our basic objects and some notation.

\begin{itemize}
  \item
    A complex vector space $V$ of dimension $r$, 
    a finite reflection subgroup $W$ of $\GL(V)$,
    a finite order element $\vp \in N_{\GL(V)}(W)$.
  \item
    $\CA(W) :=$ the reflecting hyperplanes arrangement of $W$,
    and for $H\in \CA(W)$, 
    \begin{itemize}
\smallskip
      \item
        ${W_H} :=$ the fixator of $H$ in $W$, a cyclic group of order ${e_H}$,
\smallskip
      \item
        ${j_H }:=$ an eigenvector for reflections fixing $H$.
    \end{itemize}
\smallskip
    \item
          ${\Nh_W} := |\CA(W)|$ the number of reflecting hyperplanes.
\end{itemize}

  The action of $N_{\GL(V)}(W)$ on the monomial of degree $\Nh_W$
  $$
    \prod_{H\in\CA(W)} j_H \in SV
  $$
  defines a linear character of $N_{\GL(V)}(W)$,
  which coincides with $\det_V$ on restriction to $W$,
  hence (by quotient with $\det_V$) defines a character
  $$
    \th : N_{\GL(V)}(W) \longrightarrow N_{\GL(V)}(W)/W \longrightarrow \BC^\times
    \,.
  $$

  We set
  $$
      \BG = (V,W\vp)
  $$
  and we define the ``polynomial order'' of $\BG$ by the formula
  $$
    |\BG| := 
       {(-1)^r \th(\vp)}\, 
       {x^{\Nh_W}}  
       \dfrac{1}{\dfrac{1}{|W|}\sum_{w\in W}\dfrac{1}{\det_V(1-w\vp x)^*}}
       \in \BC[x]
   $$
   (where $z^*$ denotes the complex conjugate of the complex number $z$).

  Notice that when $W$ is a true Weyl group  and $\vp$ is a graph automorphism, 
  then the polynomial $|\BG|$ is the
  order polynomial discovered by Steinberg for the corresponding family of
  finite reductive groups.
\bigskip

\noindent
  {\textsc A particular case.}
\smallskip

  Let us quickly state our results for the cyclic group of order 3, the smallest complex
  reflection group which is not a Coxeter group.
  
  For the purposes of that short exposition,
  we give some {\textsl ad hoc} definitions of the main notions (Hecke algebras,
  Schur elements, unipotent characters, $\xi$-series, etc.)
  which will be given in a more general and more systematic context in the paper below.

  Let $\z := \exp(\frac{2\pi i}{3})$.
  We have
  $$
    V := \BC\,,\, W := \genby{\z}\,,\,\vp := 1\,,\, \BG := (\BC,W)\,,\quad
    \Nh = 1\,,\quad
  $$
  $$
    |\BG| = x(x^3-1) 
    \,.
  $$
\bigskip

\textsl{Generic Hecke algebra $\CH(W,(a,b,c))$}
\smallskip

  For indeterminates $a,b,c$, we define an algebra over
  $\BZ[a^{\pm 1},b^{\pm 1},c^{\pm 1}]$ by
 $$
   \CH(W,(a,b,c)) := 
   \genby{\,\bs \,\mid\, (\bs-a)(\bs-b)(\bs-c) = 0\,}
   \,.
 $$
 The algebra $\CH(W,(a,b,c))$ has three linear characters
 $\chi_a,\chi_b,\chi_c$
 defined by $\chi_t(\bs) = t$ for $t \in \{a,b,c\}$.
\medskip
   
\textsl{Canonical trace}
\smallskip

 The algebra $\CH(W,(a,b,c))$ is endowed with the symmetrizing form defined by
 $$
   \tau(\bs^n) :=
   \left\{
   \aligned
     & \sum_{\al, \be, \ga >0\atop \al+\be+\ga=n}
             a^\al b^\be c^\ga \quad\text{for } n>0\,, \\
     & \sum_{\al, \be, \ga \leq 0\atop \al+\be+\ga=n} 
             a^\al b^\be c^\ga  \quad\text{for } n\leq 0\,.
   \endaligned
   \right.
 $$
\smallskip

\textsl{Schur elements of $\CH(W,(a,b,c))$}
\smallskip

  We define three elements of $\BZ[a^{\pm 1},b^{\pm 1},c^{\pm 1}]$ which we
  call Schur elements by
 $$
   S_a =\dfrac{(b-a)(c-a)}{bc}
   \,\,,\,\,
   S_b =\dfrac{(c-b)(a-b)}{ca}
     \,\,,\,\,
   S_c = \dfrac{(a-c)(b-c)}{ab}\,,
 $$
 so that
 $$
   \tau = \dfrac{\chi_a}{S_a}+\dfrac{\chi_b}{S_b}+\dfrac{\chi_c}{S_c}
   \,.
 $$ 
\smallskip
 
\textsl{Spetsial Hecke algebra for the principal series} 
\smallskip

  This is the specialization of the generic Hecke algebra defined by
  $$
   \CH(W,(x,\z,{\z^2})) := 
   \genby{\, \bs \,\mid\, (\bs-x)(1+\bs{+\bs^2}) = 0\,}
   \,.
 $$
 Note that the specialization to the group algebra factorizes through this.
\smallskip

\textsl{Unipotent characters}
\smallskip

 There are 4 unipotent characters of $\BG$, denoted $\rho_0,\rho_\z,\rho_\z^*,\rho$.
 Their degrees and Frobenius eigenvalues, are given by the following table:

$$
\begin{array}{|c|c|c|}\hline
   \ga        & \Deg(\ga)               & \Fr(\ga) \\
 \hline
   \rho_0     & 1                                        & 1        \\
   \rho_{\z}   & \dfrac{1}{1-\z^2}x(x-\z^2) & 1        \\
   \rho_{\z}^*   & \dfrac{1}{1-\z} x(x-\z)     & 1        \\
   \rho          & \dfrac{\z}{1-\z^2} x(x-1)  & \z^2  \\
 \hline
\end{array}
$$
\smallskip

  We set
  $
    \Uch(\BG) := \{ \rho_0,\rho_\z,\rho_\z^*,\rho \}
    \,.
  $
\medskip
  
  \textsl{Families}
\smallskip
   
  $\Uch(\BG)$ splits into two families:
    $
      \{\rho_0\}\,,\,\{\rho_\z,\rho_\z^*,\rho\}
      \,.
    $
\medskip

   \textsl{Principal $\xi$-series for $\xi$ taking values $1, \z,\z^2$}
\smallskip
   
  \begin{enumerate}
    \item
      We define the principal $\xi$-series by
      $$
        \Uch(\BG,\xi) := \{ \ga\,\mid\, \Deg(\ga)(\xi) \neq 0\}\,,
      $$
      and we say that a character $\ga$ is $\xi$-cuspidal if
      $$
        \left(\dfrac{|\BG|(x)}{\Deg(\ga)(x)}\right)|_{x=\xi} \neq 0
        \,.
      $$
    \item
      $
        \Uch(\BG) = \Uch(\BG,\xi) \sqcup \{\ga_\xi\}
      $
      where $\ga_\xi$ is $\xi$-cuspidal.
\smallskip
      
    \item
      Let $\CH(W,\xi) := \CH(W,(\xi\inv x,\z,\z^2))$ be the specialization of the
      generic Hecke algebra at $a = \xi\inv x\,, b=\z\,, c=\z^2$.
      There is a natural bijection
      $$
        \Irr\left( \CH(W,\xi) \right) \iso \Uch(\BG,\xi)
        \quad,\quad
        \chi_t \mapsto \ga_t \quad\text{for } t = a,b,c
      $$
      such that:
      \begin{enumerate}
        \item
          We have
          $$
            \Deg({\ga_t})(x) = \pm \left( \dfrac{x^3-1}{1-\xi x} \right) \dfrac{1}{S_t(x)}
            \,,
          $$
          where $S_t(x)$ denotes the corresponding specialized Schur element.
        \item
          The intersections of the families with the set $\Uch(\BG,\xi)$
          correspond to the Rouquier blocks of $\CH(W,\xi)$.
      \end{enumerate}
\end{enumerate}
\smallskip

  \textsl{Fourier matrix}
\smallskip

  The Fourier matrix for the 3-element family is
  $$
    \dfrac{\z}{1-\z^2}
    \begin{pmatrix}
      \z^2 & -\z  & -1 \\
      -\z  & \z^2 & 1  \\
      -1   & 1    & 1
    \end{pmatrix}
    \,.
  $$
\newpage

{\red \section{\red Reflection groups, braid groups, Hecke algebras}\hfill
\smallskip
}

  The following notation will be in force throughout the paper.
  
  We denote by $\BN$ the set of nonnegative integers.
  
  We denote by $\bmu$\index{mub@$\bmu$} the group of all roots
  of unity in $\BC^\times$.  
  For $n\geq 1$, we denote
  by $\bmu_n$\index{mubn@$\bmu_n$} the subgroup of $n$-th roots
  of unity, and we set $\z_n := \exp(2\pi i/n) \in \bmu_n$\index{zn@$\z_n$}.
  
  If $K$ is a number field, a subfield of $\BC$, we denote by $\BZ_K$ the ring
  of algebraic integers of $K$.
  We denote by $\bmu(K)$\index{mubK@$\bmu(K)$} the group of roots of unity in $K$,
  and we set $m_K := |\bmu(K)|$\index{mK@$m_K$}.
  We denote by $\ov K$\index{Kbar@$\ov K$} the algebraic closure of $K$ in $\BC$.
   
  We denote by $z\mapsto z^*$ the complex conjugation on $\BC$.
  For a Laurent polynomial $P(x) \in\BC[x,x\inv]$, we set
  $
    P(x)^\vee := P(1/x)^*
    \,.
  $
\bigskip

\subsection{Complex reflection groups and reflection cosets}\hfill
\smallskip

\subsubsection{Some notation}\label{notationforcrg}\hfill
\smallskip

  Let $V$ be a finite dimensional complex vector space,
  and let $W$ be a finite subgroup of $\GL(V)$ generated by reflections 
  (a \emph{finite complex reflection group}). 
   
   We denote by $\CA(W)$\index{AW@$\CA(W)$} 
   (or simply by $\CA$\index{A@$\CA$} when there is no ambiguity)
   the set of reflecting hyperplanes of reflections in $W$.      
   If $H \in \CA(W)$, we denote by $e_H$\index{eH@$e_H$}
   the order of the fixator $W_H$\index{WH@$W_H$} of  $H$ in $W$, a cyclic group consisting of
   1 and all reflections around $H$. Finally, we call 
   \emph{distinguished reflection}\index{distinguished reflection} around $H$
   the reflection with reflecting hyperplane $H$ and non trivial eigenvalue
   $\exp(2\pi i/e_H)$.
   
   An element of $V$ is called \emph{regular} if it belongs to none of the reflecting
   hyperplanes. We denote by $V^\reg$\index{Vreg@$V^\reg$} the set of regular elements, that is
   $$
     V^\reg = V - \bigcup_{H\in\CA(W)} H
     \,.
   $$
   
   We set
   $$
     \Nr_W := |\{w\in W\mid\text{ $w$ is a reflection}\}|\index{Nr@$\Nr_W$}
     \,\,\text{ and }\,\,
     \Nh_W := |\CA(W)|\index{Nh@$\Nh_W$}
   $$
   so that
   $
     \Nr_W = \sum_{H\in \CA(W)} (e_H-1)
     \,\,\text{ and }\,\,
     \Nh_W = \sum_{H\in \CA(W)}  1
     \,.
   $   
   We set
   \begin{align}\label{e=N+N}
     e_W := \sum_{H\in \CA(W)} e_H = \Nr_W + \Nh_W\index{eW@$e_W$}
     \,, \text{ so that }
     e_H = e_{W_H}
     \,.
   \end{align}
   
   The \emph{parabolic subgroups}\index{parabolic subgroups} of $W$ are 
   by definition the fixators of subspaces of $V$:
   for $I\subseteq V$, we denote by $W_I$ \index{WI@$W_I$}
   the fixator of $I$ in $W$. Then the map
   $$
     I \mapsto W_I
   $$
   is an order reversing bijection from the set of intersections of elements of $\CA(W)$
   to the set of parabolic subgroups of $W$.
\smallskip

\subsubsection{Some linear characters}\label{somelinearcharacters}\hfill
\smallskip

  Let $W$ be a reflection group on $V$.
  Let $SV$\index{SV@$SV$} be the symmetric algebra of $V$, 
  and let $SV^W$\index{SVW@$SV^W$} be the subalgebra
  of elements fixed by $W$.

  For $H\in\CA(W)$, let us denote by $j_H\in V$ an eigenvector for the group $W_H$
  which does not lie in $H$, and let us set\index{JW@$J_W$}
  $$
    J_W := \prod_{H\in\CA(W)} j_H \in SV
    \,,
  $$
  an element of the symmetric algebra of $V$ well defined up to multiplication by a nonzero
  scalar, homogeneous of degree $\Nh_W$.
  
  For $w\in W$, we have (see e.g. \cite[4.3.2]{berkeley})
  $
    w.J_W = \det_V(w) J_W
  $
  and more generally, there is a linear character on $N_{\GL(V)}(W)$
  extending ${\det_V}{|_W}$ and
  denoted by $\widetilde\det_V^{(W)}$, defined as follows:
  $$
    \nu.J_W = \widetilde\det_V^{(W)} \!(\nu) J_W
    \quad\text{for all } \nu \in N_{\GL(V)}(W)
    \,.
  $$\index{detti@$\widetilde\det_V^{(W)}$}
  
\begin{remark}
  The character $\widetilde\det_V^{(W)}$ is in general different from $\det_V$, as can
  easily be seen by considering its values on the center of $\GL(V)$. But
  by what we said above
  it coincides with $\det_V$ on restriction to $W$.
  
  It induces a linear character
  $$
    \det'_V : N_{\GL(V)}(W)/W \ra \BC^\times
  $$
  defined as follows: for $\ov\vp \in N_{\GL(V)}(W)/W$ with preimage
  $\vp \in N_{\GL(V)}(W)$, we set
  $$
     \det'_V( \ov\vp) :=  \widetilde\det_V^{(W)}(\vp)\det_V(\vp)^*
    \,.
  $$\index{det'@$\det'_V$}
\end{remark}

  Similarly, the element (of degree $\Nr_W$) of $SV$ defined by
  $$\index{JWvee@$J_W^\vee$}
    J_W^\vee := \prod_{H\in\CA(W)} j_H^{e_H-1}
  $$
  defines a linear character
  $\widetilde\det_V^{(W)\vee}$\index{dettivee@$\widetilde\det_V^{(W)\vee}$}
  on $N_{\GL(V)}(W)$,
  which coincides with $\det_V\inv$ on restriction
  to $W$, hence a character
  $$
     {\det'_V}^\vee : N_{\GL(V)}(W)/W \ra \BC^\times
     \quad,\quad
     \ov\vp \mapsto \widetilde\det_V^{(W)\vee}\!\!(\vp)\det_V(\vp)   
    \,.
  $$\index{det'vee@${\det'_V}^\vee$}
\smallskip

  The \emph{discriminant}\index{discriminant}, element of degree $\Nh_W+\Nr_W$ of $SV^W$
  defined by
  $$\index{DiscW@$\Disc_W$}
    \Disc_W := J_WJ_W^\vee = \prod_{H\in\CA(W)} j_H^{e_H}\,,
  $$
  defines a character
  $$
    \De_W := \det'_V {\det'_V}^\vee : N_{\GL(V)}(W)/W \ra \BC^\times 
    \,.
  $$\index{DeW@$\De_W$}

  Let $\vp \in N_{\GL(V)}(W)$ be an element of finite order. Let $\z$ be a root
  of unity.
  We recall
  (see \cite{springer} or \cite{brmi}) that an element $w\vp \in W\vp$ is called
  \emph{$\z$-regular} if there exists an eigenvector
   for $w\vp$ in $V^\reg$ with eigenvalue $\z$.

\begin{lemma}\label{detofregular}
  Assume that $w\vp$ is $\z$-regular. Then
  \begin{enumerate}
    \item
      $
        \widetilde\det_V^{(W)}(w\vp) = \z^{\Nh_W}
        \quad\text{and}\quad
        \widetilde\det_V^{(W)\vee}(w\vp) = \z^{\Nr_W}
        \,.
      $
    \smallskip
    \item
      $
        \det'_V(\ov\vp) = \z^{\Nh_W} \det_V(w\vp)\inv
        \quad\text{and}\ \ 
        {\det'_V}^\vee(\ov\vp) = \z^{\Nr_W}\det_V(w\vp)
        \,.
      $
    \smallskip
    \item
      $ \De_W(\ov\vp) = \z^{e_W} \,.$
  \end{enumerate}
\end{lemma}

\begin{proof}\hfill

  Let $V^*$ be the dual of $V$.  We denote by $\scal{\cdot}{\cdot}$ the natural 
  pairing $V^*\times V \ra K$, which extends naturally to a pairing
  $V^*\times SV\ra K$  ``evaluation of functions on $V^*$''.
  
  We denote by ${V^*}^\reg$ the set of elements of $V^*$ fixed by none of the reflections of $W$ (acting on the
  right by transposition): this is the set of regular elements of $V^*$ for the complex reflection group $W$
  acting through the contragredient representation.
  
  Let $\al \in {V^*}^\reg$ be
  such that
  $
    \al w\vp = \zeta\al
    \,.
  $
  Since $\al$ is regular, we have 
  $
    \scal\al{J_W} \neq 0
    \,.
  $
  But
  $$
     \left\{
     \begin{aligned}
       &\scal{\al w\vp}{J_W} = \z^{\Nh_W} \scal\al{J_W} \\
       &\scal\al{w\vp(J_W)} = \widetilde\det_V^{(W)}(w\vp) 
         \scal\al{J_W}
      \end{aligned}
    \right.
  $$
  which shows that $\widetilde\det_V^{(W)}(w\vp)  = \z^{\Nh_W}$. A similar proof
  using $J_W^\vee$
  shows that $\widetilde\det_V^{(W)\vee}(w\vp) = \z^{\Nr_W}$.
  Assertions (2) and (3) are then immediate.
\end{proof}

\begin{remark}\label{zetatoe}
  As a consequence of the preceding lemma, we see that if $w\in W$ is a $\z$-regular
  element, then 
  $$
    \det_V(w) = \z^{\Nh_W}\,\,,\,\,\det_V(w)\inv = \z^{\Nr_W}
    \,\,\text{ and thus }\,\,
    \z^{e_W} = 1
    \,.
  $$
\end{remark}

\subsubsection{Field of definition}\hfill
\smallskip

  The following theorem has been proved through a case by case analysis
  \cite{benard}  (see also \cite{field}).

\begin{theorem}\label{bessisrationalfield}\hfill

  Let $W$ be a finite reflection group on $V$. Then the field
  $$
    \BQ_W := \BQ(\tr_V(w)\,\mid\, w\in W)
  $$\index{QBW@$\BQ_W$}
  is a splitting field for all complex representations of $W$.
\end{theorem}

  The ring of integers of $\BQ_W$ will be denoted by $\BZ_W$.\index{ZBW@$\BZ_W$}  
  If $L$ is any number field, we set
  $$\index{LW@$L_W$}
    L_W := L((\tr_V(w))_{w\in W})
    \,,
  $$
  the composite of $L$ with $\BQ_W$.
\smallskip

\subsubsection{Reflection cosets}\hfill
\smallskip

  Following \cite{sp1}, we set the following definition. 
  
\begin{definition}
  A \emph{reflection coset}\index{reflection coset}
  on a characteristic zero field $K$ is a pair $\BG = (V,W\vp)$\index{GB@$\BG$}
  where 
  \begin{itemize}
    \item
      $V$ is a finite dimensional $K$-vector space,
    \item
      $W$ is a finite subgroup of $\GL(V)$ generated by reflections,
    \item
      $\vp$ is an element of finite order of $N_{\GL(V)}(W)$.
  \end{itemize}
\end{definition}

  We then denote by
  \begin{itemize}
    \item
      $\ov\vp$\index{pvbar@$\ov\vp$} the image of $\vp$ in $N_{\GL(V)}(W)/W$, so that the reflection
      coset may also be written $\BG = (V,W,\ov\vp)$,
      and we denote by $\de_\BG$\index{deGB@$\de_\BG$} 
      the order of $\ov\vp$,
    \item
      $\Ad(\vp)$ the automorphism of $W$ defined by $\vp$~; it is the image of $\vp$ in
      $N_{\GL(V)}(W)/C_{\GL(V)}(W)$,
    \item
      $\Out(\vp)$ (or $\Out(\ov\vp)$) the image of $\vp$
      in the outer automorphism
      group of $W$, \ie\ the image of $\vp$ in $N_{\GL(V)}(W)/WC_{\GL(V)}(W)$
      (note that $\Out(\vp)$ is an image of both $\ov\vp$ and $\Ad(\vp)$).
  \end{itemize}
  
  The reflection coset $\BG = (V,W,\ov\vp)$ is said to be
      \emph{split} if $\,\ov\vp = 1$ (\ie\ if $\de_\BG = 1$).

\begin{definition}\hfill
  \begin{enumerate}
    \item
      If $K = \BQ$ (so that $W$ is a Weyl group), we say that $\BG$ is \emph{rational}.
      
      A ``generic finite reductive group'' $(X,R,Y,R^\vee,W\vp)$
      as defined in \cite{bmm} defines a rational reflection coset $\BG = (\BQ\otimes_\BZ Y,W\vp)$.
      We then say that $(X,R,Y,R^\vee,W\vp)$ \emph{is associated with} $\BG$.
    \item
      There are also \emph{very twisted rational} reflection cosets $\BG$ 
      defined over $K = \BQ(\sqrt{2})$ (resp.
      $K = \BQ(\sqrt{3})$) by very twisted generic finite reductive groups associated with systems
      $\lexp2B_2$ and $\lexp2F_4$ (resp. $\lexp2G_2$). 
      Again, such very twisted generic finite reductive groups
      are said to be \emph{associated with} $\BG$. Note that, despite of the notation, 
      very twisted rational
      reflection cosets are not defined over $\BQ$~: $W$ is rational on $\BQ$ but not
      $W\vp$.
    \item
      If $K\subset\BR$ (so that $W$ is a Coxeter group), we say that $\BG$ is \emph{real}.
  \end{enumerate}
\end{definition}

  For details about what follows, the reader may refer to \cite{brma1} and \cite{bmm}.
  \begin{itemize}
    \item
      In the case where $\BG$ is rational, given a prime power $q$, any choice of an
      associated  generic  finite reductive group determines a
      connected reductive algebraic group $\bG$ defined over $\ov\BF_q$
      and endowed with a Frobenius endomorphism $F$ defined by $\vp$
      (\ie\ $F$ acts as $q\vp$ on $X(\bT)$ where $\bT$ is an $F$-stable
       maximal torus of $\bG$).
      Such groups are called the \emph{reductive groups associated with $\BG$}.
    \item  
      In the case where $K = \BQ(\sqrt{2})$ (resp. $K = \BQ(\sqrt{3}$)), given
      $\BG$ very twisted rational and
      $q$ an odd power of
      $\sqrt 2$ (resp. an odd power of $\sqrt 3$), any choice of an
      associated very twisted generic  finite reductive group determines a
      connected reductive algebraic group $\bG$ defined over $\ov\BF_{q^2}$
      and endowed with an isogeny acting as $q\vp$ on $X(\bT)$. Again, this
      group is called a \emph{reductive group associated with $\BG$}.
  \end{itemize}
\smallskip

  Theorem \ref{bessisrationalfield} has been
  generalized in \cite[Thm 2.16]{maEg} to the following result.

\begin{theorem}\label{charfields}\hfill
  
  Let $\BG = (V,W\vp)$ be a reflection coset. 
  Let\index{QBGB@$\BQ_{\BG}$}
  $$
    \BQ_\BG :=\BQ(\tr_V(w\vp)\mid w\in W)
  $$
  be the character field of the subgroup $\langle W\vp\rangle$ 
  of $\GL(V)$ generated by $W\vp$.
  Then every
  $\vp$-stable complex irreducible character of $W$ has an extension to
  $\langle W\vp\rangle$ afforded by a representation defined over $\BQ_\BG$.
\end{theorem} 
\smallskip

\subsubsection{Generalized invariant degrees}\label{generalizeddegrees}\hfill
\smallskip

  In what follows, $K$ denotes a number field which is stable under
  complex conjugation, and 
  $\BG = (V,W\vp)$ is a reflection coset over $K$.

  Let $r$ denote the dimension of $V$.
  
  One defines the family $((d_1,\z_1),(d_2,\z_2),\dots,(d_r,\z_r))$ of 
  \emph{generalized invariant degrees}
  of $\BG$ (see for example \cite[4.2.2]{berkeley}):
  there exists a family 
  $(f_1,f_2,\dots,f_r)$ of $r$ homogeneous algebraically independent elements 
  of $SV^W$ and a family $(\z_1,\z_2,\dots,\z_r)$ of elements of $\bmu$
  such that
  \begin{itemize}
    \item
      $SV^W = K[f_1,f_2,\dots,f_r]\,,$
    \item
      for $i = 1,2,\dots,r$, we have $\deg(f_i) = d_i$ and $\vp.f_i = \z_if_i$.  
  \end{itemize}
\smallskip

\begin{remark}
  Let $\Disc_W=\sum_\bm a_\bm f^\bm$ be the expression of 
  $\Disc_W$ as a polynomial in the fundamental invariants $f_1,\ldots,f_r$,
  where the sum runs over the monomials $f^\bm=f_1^{m_1}\cdots f_r^{m_r}$.
  Then for every $\bm$ with $a_\bm \neq 0$ we have
  $\Delta_W(\ov\vp)=\zeta_1^{m_1}\cdots\zeta_r^{m_r}$.

  In the particular case where $w\in W$ is $\zeta$-regular and the order
  of $\zeta$ is one of the invariant degrees $d_i$, we recover \ref{detofregular}(2)
  by using the result of Bessis \cite[1.6]{bessiszariski} that in that case
  $f_i^{e_W/d_i}$ is one of the monomials occurring in $\Disc_W$.
\end{remark}
\smallskip
  
  The character
  $
    \det'_V : N_{\GL(V)}(W)/W \ra K^\times
  $
  (see \ref{somelinearcharacters} above)
  defines the root of unity
  $$
    \det_\BG := \det_V'(\ov\vp)
    \,.
  $$\index{detGB@$\det_{\BG}$}
  Similarly, the character ${\det_V'}^\vee$
  attached to $J_W^\vee$ (see \ref{somelinearcharacters}) defines a root of unity
  $
    \det_\BG^\vee := {\det_V'}^\vee(\ov\vp)
  $\index{detGBvee@$\det_{\BG}^\vee$}
  attached to $\BG$.
  
  The character
  $ \De_W$
  defined by the discriminant of $W$
  defines in turn a root of unity
  by
  $$\index{DeGB@$\De_{\BG}$}
    \De_\BG := \det_\BG .\det_\BG^\vee = \De_W(\ov\vp)
    \,.
  $$  
  The following lemma collects a number of conditions under which
  $\De_\BG = 1$.
    
  \begin{lemma}\label{discriminanttrivial}\hfill
  
    We have $\De_\BG = 1$ if (at least) one of the following conditions is satisfied.
    \begin{enumerate}
      \item
        If the reflection coset $\BG$ is split. Moreover in that case we have
        $
          \det_\BG = \det_\BG^\vee = 1
          \,.
        $  
      \item
        If $W\vp$ contains a 1-regular element.
      \item
        If $\BG$ is real (\ie\ if $K \subset \BR$).
    \end{enumerate}
  \end{lemma}
  
  \begin{proof}\hfill
 \smallskip
  
    (1) is trivial.
 \smallskip
     
     (2) We have $\De_\BG = \De_W(\ov\vp) = 1^{e_W} =1$ by Lemma \ref{detofregular}(2).
 \smallskip
 
    (3) Consider the element 
    $
      J_W = \prod_{H\in\CA(W)} j_H
    $     
    introduced above. Since $\vp \in N_{\GL(V)}(W)$, $\vp$ acts on $\CA(W)$, hence
    $J_W$ is an eigenvector of $\vp$. If $\vp$ has finite order, the corresponding eigenvalue
    is an element of $\bmu(K)$, hence is $\pm 1$ if $K$ is real. 
    
    Moreover, all reflections in $W$ are ``true reflections'', that is $e_H = 2$ for all $H\in\CA(W)$.
    It follows that
    $
      \Disc_W = J_W^2
    $
    and so that $\Disc_W$ is fixed by $\vp$. 
  \end{proof}
\smallskip 

\subsection{Uniform class functions on a reflection coset}\hfill
\medskip

  The next paragraph is extracted from \cite{sp1}. It is reproduced for the convenience of the reader since
  it fixes conventions and notation.
\smallskip
  
\subsubsection{Generalities, induction and restriction}\hfill
\smallskip

  Let $\BG=(V,W\vp)$ be a reflection coset over $K$.

  We denote by ${\OCFuf(\BG)}$\index{CFuf@${\OCFuf(\BG)}$}
  the $\BZ_K$-module of all
  $W$-invariant functions
  on the coset $W\vp$ (for the natural action of $W$ on $W\vp$ by conjugation)
  with values in $\BZ_K$, called
  {\it uniform class functions on $\BG$}.
  For $\al\in\OCFuf(\BG)$,
  we denote by $\al^*$ its complex conjugate.

  For $\al,\al' \in \OCFuf(\BG)$, we set
  \index{<,>@$\scal{}{}_\BG$}
  $
    \scal{\al}{\al'}_\BG:= \dfrac{1}{|W|}
    \sum_{w\in W} \al(w\vp) {\al'(w\vp)^*}
    \, .
  $\index{scalal@$\scal{\al}{\al'}_{\BG}$}
\medskip

{\sl
  Notation
}
\smallskip

\bul
  If $\BZ_K \ra \CO$ is a ring morphism, we denote by
  $\OCFuf(\BG,\CO)$\index{CFufO@$\OCFuf(\BG,\CO)$}
  the $\CO$-module of $W$-invariant functions
  on $W\vp$ with values in $\CO$, which we call the module
  of {\it uniform class functions on $\BG$} with values
  in $\CO$.
  We have
  $\OCFuf(\BG,\CO) = \CO \otimes_{\BZ_K} \OCFuf(\BG)$.

\bul
  For $w\vp\in W\vp$, we denote by $\ch_{w\vp}^\BG$\index{chwp@$\ch_{w\vp}^{\BG}$}
  (or simply $\ch_{w\vp}$) the characteristic
  function of the orbit of $w\vp$ under $W$.
  The family $\left(\ch_{w\vp}^\BG\right)$
  (where $w\vp$ runs over a complete set of representatives
  of the orbits of $W$ on $W\vp$) is a basis of $\OCFuf(\BG)$.

\bul
  For $w\vp\in W\vp$, we set\index{Rwp@${R_{w\vp}^{\BG}}$}
  $$
    {R_{w\vp}^\BG}:= |C_W(w\vp)| \ch_{w\vp}^\BG
  $$
(or simply $R_{w\vp}$).

{\narrower
\smallskip
{\small
\begin{remark}\hfill

  In the case of reductive groups, we may choose $K=\BQ$.
  For $(\bG,F)$ associated to $\BG$,
  let $\Uch(\bG^F)$ be the set of unipotent characters of
  $\bG^F$: then the map which
  associates to $R^\BG_{w\vp}$ the Deligne-Lusztig character
  $R^\bG_{\bT_{wF}}(\id)$ defines an isometric embedding (for the
  scalar products $\scal\alpha{\alpha'}_\BG$ and
  $\scal\alpha{\alpha'}_{\bG^F}$) from $\OCFuf(\BG)$ onto the
  $\BZ$-submodule of $\BQ\Uch(\bG^F)$ consisting of the $\BQ$-linear
  combinations of Deligne-Lusztig characters (i.e., ``unipotent
  uniform functions'') having integral scalar product with
  all Deligne-Lusztig characters.
\end{remark}
}
\smallskip
}
\smallskip

\bul
  Let $\genby{W\vp}$ be the subgroup of $\GL(V)$
  generated by $W\vp$. We recall that we denote by
  $\ov{\vp}$ the image of $\vp$ in $\genby{W\vp}/W$
  --- thus $\genby{W\vp}/W$ is cyclic and
  generated by $\ov{\vp}$.

  For $\psi\in\Irr(\genby{W\vp})$, we denote by 
  $R_\psi^\BG$\index{Rpsi@$R_\psi^{\BG}$}
  (or simply $R_\psi$) the restriction of $\psi$ to the
  coset $W\vp$. We have
$
    {R_\psi^\BG} =
      \dfrac{1}{|W|}\sum_{w\in W} \psi(w\vp) R_{w\vp}^\BG,
$
  and we call such a function an
  {\it almost character\/}\index{almost character}
  of $\BG$.

  Let $\Irr(W)^{\ov\vp}$ denote the set of $\ov\vp$-stable
  irreducible characters of $W$.
  For $\th\in\Irr(W)^{\ov\vp}$, we denote by 
  $E_\BG(\th)$\index{EGB@$E_{\BG}(\th)$}
  (or simply $E(\th)$) the set of restrictions to $W\vp$
  of the extensions of $\th$ to characters of $\genby{W\vp}$.
\smallskip

  The next result is well-known (see \eg\ \cite[\S II.2.c]{dm0}), and easy to prove.

\begin{proposition}\label{orthonormalbasis}\hfill
  \begin{enumerate}
    \item
      Each element $\al$ of $E_\BG(\th)$ has norm $1$ (\ie\ $\scal{\al}{\al}_\BG = 1$),
    \item
      the sets $E_\BG(\th)$ for $\th\in\Irr(W)^{\ov\vp}$
      are mutually orthogonal, 
    \item
      $
        \OCFuf(\BG,K) =
        \bigoplus^\perp_{\th\in\Irr(W)^{\ov\vp}} KE_\BG(\th),
      $
      where we set
      $
        KE_\BG(\th) :=  KR_\psi^\BG
      $
      for some (any) $\psi\in E_\BG(\th)$.
  \end{enumerate}
\end{proposition}
\medskip

{\sl
  Induction and restriction
}
\smallskip

  Let $\BL = (V,W_\BL w\vp)$ be a subcoset of maximal
  rank of $\BG$ \cite[\S 3.A]{sp1},
  and let $\al \in \OCFuf(\BG)$ and $\be \in \OCFuf(\BL)$.
  We denote
\smallskip

$\bullet$
  by ${\ResLG}\al$\index{ResLG@${\ResLG}\al$} the restriction of $\al$
  to the coset $W_\BL w\vp$,
\smallskip

$\bullet$
  by ${\IndLG}\be$\index{IndLG@${\IndLG}\be$} the uniform class function on $\BG$ defined by
  \begin{align}{\label{induction}}
    \IndLG \be(u\vp) := \dfrac{1}{|W_\BL |}
    \sum_{v\in W}\tilde\be(vu\vp v^{-1})
    \quad \text{for } u\vp \in W\vp,
  \end{align}
  where
  $
    \tilde\be(x\vp) = \be (x\vp)
    \text{ if } x \in W_\BL w,
    \text{ and }
    \tilde\be(x\vp) = 0
    \text{ if } x \notin W_\BL w
    \, .
  $
  In other words, we have
   \begin{align}\label{induction2}
      \IndLG \be (u\vp) =  \hskip -8mm
      \sum_{{v \in W/W_\BL },
      {\lexp v(u\vp) \in W_\BL w\vp}}  \hskip -8mm
      \be (\lexp  v(u\vp))
      \, .
   \end{align}

  We denote by $1^\BG$\index{1GB@$1^{\BG}$} the constant
  function on $W\vp$ with value~1.
  For $w\in W$, let us denote by $\BT_{w\vp}$\index{TBwp@$\BT_{w\vp}$}
  the maximal torus of $\BG$ defined by
  $
    \BT_{w\vp} := (V,w\vp)
    .
  $
  It follows from the definitions that
   \begin{align}{\label{Rwisind}}
    R_{w\vp}^\BG = \Ind_{\BT_{w\vp}}^\BG 1^{\BT_{w\vp}}
    .
 \end{align}
\smallskip

  For $\alpha\in\OCFuf(\BG)$, $\beta\in\OCFuf(\BL)$ we have the {\it Frobenius
  reciprocity\/}:
  \begin{align}{\label{frobenius}}
    {\scal\al{\IndLG\be}}_\BG =
    {\scal{\ResLG\al}\be}_\BL \, .
 \end{align}

{\narrower
\smallskip
{\small
\begin{remark}\hfill

  In the case of reductive groups, assume that $\BL$ is a
  Levi subcoset of $\BG$ attached to the Levi subgroup $\bL$.
  Then $\Ind_\BL^\BG$ corresponds to Lusztig
  induction from $\bL$ to $\bG$ (this results from definition
  \ref{induction} applied to a Deligne-Lusztig character which,
  using the transitivity of Lusztig induction, agrees with Lusztig
  induction). Similarly, the Lusztig restriction of a uniform
  function is uniform by  \cite[Thm.7]{delu}, so by
  (\ref{frobenius}) $\Res^\BG_\BL$ corresponds to Lusztig
  restriction.

\end{remark}
}
\smallskip
}

  For further details, like a Mackey formula for induction and restriction, 
  the reader may refer to \cite{sp1}.
\smallskip

  We shall now introduce notions which extend or sometimes differ from
  those introduced in \cite{sp1}: here we introduce two polynomial orders
  $|\BG^\nc|$ and $|\BG^\oc|$ which both differ slightly 
  (for certain twisted reflection cosets)
  from the definition of
  polynomial order given in \cite{sp1}.
\medskip

\subsubsection{Order and Poincar\'e polynomial}\hfill
\smallskip

{\sl Poincar\'e polynomial}\index{Poincar\'e polynomial}
\smallskip
  
  We recall that we denote
  by $SV$ the symmetric algebra of $V$ and by $SV^W$ the subalgebra of fixed points
  under $W$. 
  
  The group $N_{\GL(V)}(W)/W$ acts on the graded vector space 
  $SV^W = \bigoplus _{n=0}^{\infty} SV^W_n$. For any $\ov\vp\in N_{\GL(V)}(W)/W$,
  define its graded character by
  $$
    \grchar(\ov\vp ; SV^W) := \sum _{n=0}^{\infty} \tr(\ov\vp ; SV^W_n) x^n \in \BZ_K[[x]]
    \,.
  $$\index{grchar@$\grchar(\ov\vp ; SV^W)$}

  Let $\BG = (V,W,\ov\vp)$ with $\dim\,V = r$.

  Let us denote by $((d_1,\z_1),\dots,(d_r,\z_r))$ the family of generalized invariant
  degrees  (see \ref{generalizeddegrees} above).
  We have (see \eg\ \cite[3.5]{bmm})
  $$
    \grchar(\ov\vp ; SV^W) = \dfrac{1}{|W|}\sum_{w\in W}\dfrac{1}{\det_V(1-w\vp x)}
      =  \dfrac{1}{ \prod_{i=1}^{i=r} (1-\z_i x^{d_i})}
     \,.
  $$
  The \emph{Poincar\'e polynomial} $P_\BG(x) \in \BZ_K[x]$\index{PGB@$P_\BG(x)$}
   of $\BG$ is defined by
  \begin{equation}\label{poincarepoly}
  \begin{aligned}
      P_\BG(x)
      &= 
        \dfrac{1}{{\grchar(\ov\vp ; SV^W)}} 
         =
        \dfrac{1}{\dfrac{1}{|W|}\sum_{w\in W}\dfrac{1}{\det_V(1-w\vp x)}} \\
      &=
        \prod_{i=1}^{i=r} (1-\z_i x^{d_i})
        \,.
   \end{aligned}
   \end{equation} 
  The Poincar\'e polynomial is semi-palindromic (see \cite[\S 6.B]{sp1}), that is,
  \begin{equation}\label{poincaresemipalin}
    P_\BG(1/x) = (-1)^r\z_1\z_2\cdots\z_r x^{-(\Nr_W +r)} P_\BG(x)^*
    \,.
  \end{equation}
\smallskip

{\sl Graded regular representation}
\smallskip

  Let us denote by $SV^W_+$ the maximal graded ideal of $SV^W$ (generated by
  $f_1,f_2,\dots,f_r$). We call the finite
  dimensional graded vector space\index{KWgr@$KW^\gr$}
  $$
    KW^\gr := SV/SV^W_+SV
  $$
  the \emph{graded regular representation}.
  
  This has the following properties (\cf\ \eg\ \cite[chap.~V, \S 5, th.~2]{bou}).
  
\begin{proposition}\label{gradedregular}\hfill 
  \begin{enumerate}
        \item
      $KW^\gr$ has a natural $N_{\GL(V)}(W)$-action, and we have an
      isomorphism of graded $KN_{\GL(V)}(W)$-modules
      $$
         SV \simeq KW^\gr \otimes_K SV^W
         \,.
      $$
    \item
      As a $KW$-module, forgetting
      the graduation, $KW^\gr$ is isomorphic to the regular representation of $W$.
    \item
      Denoting by $KW^{(n)}$ the subspace of $KW^\gr$ generated by the elements of degree $n$,
      we have
      \begin{enumerate}
        \item
          $KW^\gr = \bigoplus_{n=0}^{\Nr_W} KW^{(n)}\,,$
        \item
          the $\det_V$-isotypic component of $KW^\gr$ is the one-dimensio\-nal subspace
          of $KW^{(\Nh_W)}$ generated by $J_W$,
        \item
          $KW^{(\Nr_W)}$ is the one-dimensional subspace generated by $J_W^\vee$ and is
          the $\det_V^*$-isotypic component of $KW^\gr$.
      \end{enumerate}
  \end{enumerate}
\end{proposition}
\smallskip

{\sl Fake degrees of uniform functions}
\smallskip

$\bullet$
  We denote by ${\tr_{KW^\gr}} \in \OCFuf(\BG,\BZ_K[x])$\index{trKWgr@${\tr_{KW^\gr}}$}
  the uniform class function on $\BG$
  (with values in the polynomial ring $\BZ_K[x]$)
  defined by the character of the graded regular
  representation $KW^\gr$. Thus the value
  of the function $\tr_{KW^\gr}$ on $w\vp$ is
  $$
    \tr_{KW^\gr}(w\vp) :=
    \sum_{n=0}^{\Nr_W} \tr(w\vp;KW^{(n)})x^n.
  $$
  We call $\tr_{KW^\gr}$ {\it the graded regular character\/}.\index{graded regular character}
\medskip

$\bullet$
  We define the {\it fake degree\/}, a linear function
$
    {\Feg_\BG}  :  \OCFuf(\BG) \ra K[x],
$\index{FegGB@${\Feg_\BG}$}
  as follows:
  for $\alpha \in \OCFuf(\BG)$, we set
  \begin{align}\label{fakedegree}
     \Feg_\BG(\!\alpha\!) \!:=\! \scal{\al}{\tr_{KW^\gr}}_\BG =
     \hskip -2mm \sum_{n=0}^{\Nr_W}
\hskip -1mm
     \left( \hskip -0.8mm
       \dfrac{1}{|W|}
\hskip -1mm
       \sum_{w\in W}
\hskip -1mm
\al (w\vp) \tr(w\vp ; KW^{(n)})^*
\hskip -0.8mm
     \right)
     x^n
     \, .
  \end{align}
  We shall often omit the subscript $\BG$ (writing then
  $\Feg(\al)$) when the context allows it.
Notice that
  \begin{align} \label{degofRw}
    \Feg(R_{w\vp}^\BG) = \tr_{KW^\gr}(w\vp)^*
    ,
  \end{align}
  and so in particular that
  \begin{align}\label{degofRwinz}
    \Feg(R_{w\vp}^\BG) \in \BZ_K[x]
    .
    \end{align}

\begin{lemma}\label{regandrw}\hfill

  We have
  $$
    \tr_{KW^\gr} =
      \dfrac{1}{|W|} \sum_{w\in W} \Feg_\BG(R_{w\vp}^\BG)^*R_{w\vp}^\BG
      \,.
  $$
\end{lemma}

\begin{proof}[Proof of \ref{regandrw}]\hfill

  It is an immediate consequence of the
  definition of $R_{w\vp}^\BG$ and of (\ref{degofRw}).
\end{proof}
\medskip

{\sl
  Fake degrees of almost characters
}
\smallskip

  Let $E$ be a $K\genby{W\vp}$-module. Its character $\th$
  is a class function on $\genby{W\vp}$.
  Its restriction $R_{\th}$ to $W\vp$ is a uniform class function on $\BG$.
  Then the fake degree of $R_{\th}$ is~:
  \begin{align}{\label{moduleandfake}}
    \Feg_\BG(R_{\th}) = \tr(\vp\,;\,\Hom_{KW}(KW^\gr,E))
    .
  \end{align}
  Notice that
  \begin{align}\label{moduleandfakeinz}
    \Feg_\BG(R_{\th}) \in \BZ[\exp{2i\pi/\de_\BG}][x]
  \end{align}
  (we recall that $\de_\BG$ is the order of the
  twist $\ov{\vp}$ of $\BG$).
  
  The polynomial $\Feg_\BG(R_{\th})$ is called 
  \emph{fake degree of $\th$}.
  \index{fake degree}
\medskip

  Let $\th \in \Irr(W)^{\ov{\vp}}$. Whenever
  $\psi\in\Irr(\genby{W\vp})$ is an extension
  of $\th$ to $\genby{W\vp}$, then
  $\Reg_\th^\BG := {\Feg}_\BG(R_\psi)^*\cdot R_\psi$
  depends only on $\th$ and is the orthogonal projection of
  $\tr_{KW^\gr}$ onto $K[x]E_\BG(\th)$,
  so that in other words, we have
  \begin{align}\label{RGandfake}
    \tr_{KW^\gr} = \sum_{\th\in\Irr(W)^{\ov\vp}}
                   \Reg_\th^\BG \,.
  \end{align}

{\sl
  Polynomial order and fake degrees
}
\smallskip

  From the isomorphism
  $
    SV \cong KW^\gr \otimes_K (SV)^W
  $
  of $\genby{W\vp}$-modules  (see \ref{gradedregular})
  we deduce for $w\in W$ that
  $$
    \tr(w\vp;SV) = \tr(w\vp;KW^\gr)\tr(w\vp ; SV^W )
    \,,
  $$
  hence
  $$
    \tr(w\vp;SV) = \tr(w\vp;KW^\gr) \dfrac{1}{|W|}\sum_{v\in W} \dfrac{1}{\det_V(1-xv\vp)}
    \,.
  $$
  Computing the scalar product with a class function $\al$ on $W\vp$ gives
  \begin{align}\label{factorizationfake}
    \dfrac{1}{|W|}\sum_{w\in W}
    \dfrac{\al(w\vp)}{\det_V(1-xw\vp)^*} =
    \Feg_\BG(\alpha)
    \dfrac{1}{|W|}\sum_{w\in W}
    \dfrac{1}{\det_V(1-xw\vp)^*},
  \end{align}
  or, in other words
  \begin{align}\label{factorization}
    {\scal{\al}{\tr_{SV}}}_\BG = {\scal{\al}{\tr_{KW^\gr}}}_\BG
    {\scal{1^\BG}{\tr_{SV}}}_\BG
    .
  \end{align}

  Let us set\index{SGB@${S_\BG}(\al)$}
  \begin{align}\label{defofs}
    {S_\BG}(\al) := {\scal{\al}{\tr_{SV}}}_\BG
    \,.
  \end{align}
  Then (\ref{factorization}) becomes
  \begin{align}\label{factorizationbis}
    S_\BG(\al) = \Feg_\BG(\al) S_\BG(1^\BG)
    \,.
  \end{align}
  By definition of the Poincar\'e polynomial we have
  $
    S_\BG(1^\BG) = 1/P_\BG(x)^*,
  $
  hence
  \begin{align}\label{defofs2}
    {S_\BG}(\al) := \dfrac{ \Feg_\BG(\al)}{P_\BG(x)^*}
    \,.
  \end{align}
\smallskip

  For a subcoset\ $\BL = (V, W_\BL w\vp)$ of maximal rank of $\BG$,
  by the Frobenius reciprocity (\ref{frobenius}) we have
 \begin{align}\label{defoffeg}
    \Feg_\BG(\IndLG 1^\BL)= {\scal{1^\BL}{\ResLG\tr_{KW^\gr}}}_\BL=
    \sum_{n=0}^{\Nr_W} \tr(w\vp ; (KW^{(n)})^{W_\BL})^* x^n \,,
 \end{align}
  where $(KW^{(n)})^{W_\BL}$ are the $W_\BL$-invariants in $KW^{(n)}$.
  
  Let us recall that
  every element $w\vp \in W\vp$ defines a maximal torus
  (a minimal Levi subcoset)
  $
    \BT_{w\vp} := (V,w\vp)
  $
  of $\BG$.  

\begin{lemma}\label{indexoflevi}\hfill
  \begin{enumerate}
    \item
      We have
      $$
        \dfrac{P_\BG(x)^*} {P_\BL(x)^*} = \Feg_\BG(\IndLG 1^\BL),
      $$
    \item
      $P_\BL(x)$ divides $P_\BG(x)$ {\rm(}in $\BZ_K[x]${\rm\,)},
    \item
      for $w\vp\in W\vp$, we have
      $$
        \dfrac{P_\BG(x)}{P_{\BT_{w\vp}}(x)} = \tr_{KW^\gr}(w\vp) = \Feg_\BG(R_{w\vp}^\BG)^*
        \,.
      $$
  \end{enumerate}
\end{lemma}

\begin{proof}[Proof of \ref{indexoflevi}]\hfill

  (1)
  By (\ref{defofs2}), we have
  $
    \Feg(\IndLG 1^\BL) = S_\BG(\IndLG 1^\BL) P_\BG(x)^*
    \,.
  $
  By Frobenius reciprocity,
  for any class function $\al$ on $\BL$ we have
  $$
    S_\BG(\IndLG \al) = \scal{\IndLG \al}{\tr_{SV}}_\BG
                                     = \scal{\al}{\Res^\BG_\BL \tr_{SV}}_\BL
                                     = S_\BL(\al)
    \,,
  $$
  and so
  $
    S_\BG(\IndLG 1^\BL) = S_\BL(1^\BL) = \dfrac{1}{P_\BL(x)^*}
    \,.
  $
  
  (2)
  is an immediate consequence of (1).
  
  (3)
  follows from (1) and from formulae (\ref{Rwisind}) and (\ref{degofRw}).
\end{proof}

  Let us now consider a Levi subcoset
  $\BL = (V,W_\BL w\vp)$.
  By \ref{indexoflevi}, for $vw\vp \in W_\BL w\vp$, we have
  $$
    \tr_{KW^\gr}(vw\vp) = \dfrac{P_\BG(x)}{P_{\BL}(x)} \dfrac{P_\BL(x)}{P_{\BT_{vw\vp}}(x)}
                                       = \dfrac{P_\BG(x)}{P_{\BL}(x)}  \tr_{KW_\BL^\gr}(vw\vp)
    \,,
  $$
  and by \ref{indexoflevi}, (1)
  \begin{equation}\label{resoftr}
    \Res^\BG_\BL \tr_{KW^\gr} 
                                                    =  \Feg(\IndLG 1^\BL)^*  \tr_{KW_\BL^\gr}
    \,.
  \end{equation}

\begin{lemma}\label{degreeinduced}\hfill

  For $\be \in \OCFuf(\BL)$ we have
  $$
    \Feg_\BG(\IndLG \be) =  \dfrac{P_\BG(x)^*} {P_\BL(x)^*} \Feg_\BL(\be) \, .
  $$
\end{lemma}

  Indeed, by Frobenius reciprocity, (\ref{resoftr}), and Lemma \ref{indexoflevi},
  \begin{align*}
   \Feg_\BG(\IndLG \be) &= {\scal{\IndLG \be}{\tr_{KW^\gr}}}_\BG
                                              = {\scal{\be}{\ResLG\tr_{KW^\gr}}}_\BL \\
                                           &=  \Feg_\BG(\IndLG 1^\BL){\scal{\be}{\tr_{KW_\BL^\gr}}}_\BL 
                                              = \Feg_\BG(\IndLG 1^\BL) \Feg_\BL(\be)  \\
                                           &= \dfrac{P_\BG(x)^*} {P_\BL(x)^*} \Feg_\BL(\be) 
   \,.                            
  \end{align*}

{\narrower
\smallskip
{\small
\begin{remark}\hfill

  In the case of reductive groups, it follows from (\ref{degofRw})
  and (\ref{indexoflevi} (3)) that $\Feg(R_{w\vp})(q)$ is the degree of
  the Deligne-Lusztig character $R_{\bT_{w\vp}}^{\bG}$. Since the regular
  representation of $\bG^F$ is uniform, it follows that $\tr_{KW^\reg}$
  corresponds to a (graded by $x$) version of the unipotent part of the
  regular representation of $\bG^F$, and that $\Feg$
  corresponds indeed to the (generic) degree for unipotent uniform
  functions on $\bG^F$.
\end{remark}
}
\smallskip
}
\smallskip

{\sl
  Changing $x$ to $1/x$
}
\smallskip

  As a particular uniform class function on $\BG$,
  we can consider the
  function $\det_V$ restricted to $W\vp$, which we
  still denote by $\det_V$. Notice that this restriction
  might also be denoted by $R_{\det_V}^\BG$, since it is
  the almost character associated to the character of
  $\genby{W\vp}$ defined by $\det_V$.

\begin{lemma}\label{xto1/x}\hfill

  Let $\al$ be a uniform class function on $\BG$.
  We have
  $$
    S_\BG(\al\det_V^*)(x)
      = (-1)^r x^{-r} {S_\BG(\al^*)}(1/x)^*
      \,.
  $$
\end{lemma}

\begin{proof}\hfill

  Since
  $
    S_\BG(\al) =  \dfrac{1}{|W|}\sum_{w\in W}
    \dfrac{\al(w\vp)}{\det_V(1-xw\vp)^*}
    \,,
  $
  we see that
  \begin{align*}
    S_\BG(\al\det_V^*)(x)^* 
      &= \dfrac{1}{|W|}\sum_{w\in W}\dfrac{\al(w\vp)^*\det_V(w\vp)}{\det_V(1-xw\vp)} \\
      &= \dfrac{1}{|W|}\sum_{w\in W}\dfrac{\al(w\vp)^*}{\det_V((w\vp)\inv-x)} \\
      &= (-1)^r x^{-r} \dfrac{1}{|W|}\sum_{w\in W}\dfrac{\al(w\vp)^*}{\det_V(1-w\vp/x)^*} \\
      &= (-1)^r x^{-r} {S_\BG(\al^*)}(1/x)
      \,.
  \end{align*}
  
\end{proof}

\begin{corollary}\label{degreeofsteinberg}\hfill

  We have
  $$
    \Feg_\BG(\det_V^*) = \z_1^*\z_2^*\cdots\z_r^* x^{\Nr_W}
    \,.
  $$
\end{corollary}

\begin{proof}[Proof of \ref{degreeofsteinberg}]\hfill

  Applying Lemma \ref{xto1/x} for $\al = 1^\BG$ gives
  $$
    S_\BG(\det_V^*)(x)
      = (-1)^r x^{-r} {S_\BG(1^\BG)}(1/x)^*  
      =  (-1)^r x^{-r} \dfrac{1}{P_\BG(1/x)}
      \,,
  $$
  hence by (\ref{poincaresemipalin})
  $$
    S_\BG(\det_V^*)(x) = \z_1^*\z_2^*\cdots\z_r^* x^{\Nr_W} \dfrac{1}{P_\BG(x)^*} \,,
  $$
  and the desired formula follows from (\ref{defofs2}).
\end{proof}    
\medskip

{\sl
  Fake degree of $\det_V$ and some computations
}
\smallskip

\begin{proposition}\label{degofdet}\hfill

  We have
  $$
  \left\{
  \begin{aligned}
    &\Feg_\BG(\det_V)(x) = \det_V'(\ov\vp) x^{\Nh_W} \\
    &\Feg_\BG(\det_V^*)(x) ={\det_V'}^\vee(\ov\vp)^* x^{\Nr_W} \\
  \end{aligned}
  \right.
  $$
\end{proposition}

\begin{corollary}\label{compuofdetprime}\hfill

  We have
  $$
  \left\{
  \begin{aligned}
    &{\det_V'}^\vee(\ov\vp) = \z_1\z_2\cdots\z_r \\
    &\det_V'(\ov\vp) = \De_W(\ov\vp) \z_1^*\z_2^*\cdots\z_r^* \\
  \end{aligned}
  \right.
  $$
\end{corollary}

\begin{proof}[Proof of \ref{degofdet} and \ref{compuofdetprime}]\hfill

  By Propositions \ref{orthonormalbasis} and \ref{gradedregular}(3),
  we see that
  $$
    \begin{aligned}
    \Feg_\BG(\det_V)(x)
      &= 
        \left(
          \dfrac{1}{|W|}\sum_{w\in W} \det_V(w\vp)\tr(w\vp ; KW^{(\Nh_W)})^*
        \right)
        x^{\Nh_W} \\
      &= \dfrac{\det_V(\vp)}{\tr(\vp ; KW^{(\Nh_W)})}  x^{\Nh_W} = \det'_V(\ov\vp)  x^{\Nh_W}
        \,.
      \end{aligned}
    $$
    A similar proof holds for $\Feg_\BG(\det_V^*)(x)$.
    
    The corollary follows then from \ref{degreeofsteinberg} and from
    $\det_V'{\det_V'}^\vee = \De_W$.
\end{proof}

\begin{corollary}\label{compuofdetprimeregular}\hfill

  Assume that $w\vp$ is a $\z$-regular element of $W\vp$. Then
  $$
  \left\{
  \begin{aligned}
    &\Feg_\BG(\det_V)(x) = \det_V(w\vp) (\z\inv x)^{\Nh_W} \\
    &\Feg_\BG(\det_V^*)(x) = \det_V^*(w\vp) (\z\inv x)^{\Nr_W} \,.
  \end{aligned}
  \right.
  $$
  
  In particular, we have
  $$
  \left\{
  \begin{aligned}
    &\Feg_\BG(\det_V)(\z) = \det_V(w\vp) \\
    &\Feg_\BG(\det_V^*)(\z) = \det_V^*(w\vp)  \,.
  \end{aligned}
  \right.
  $$
\end{corollary}

  Note that the last assertion of the above lemma will be generalized in
  \ref{fakeonregular}.

\begin{proof}\hfill

  By Proposition \ref{degofdet}, and since
  $
    \det_V'(\ov\vp) =  \widetilde\det_V(w\vp) \det_V(w\vp)^*
    \,,
  $
  we have
  $$
    \Feg_\BG(\det_V)(x) = \det_V'(\ov\vp)^* x^{\Nh_W}
                                         = \widetilde\det_V(w\vp)^*\det_V(w\vp) x^{\Nh_W}
    \,.
  $$
  Now by Lemma \ref{detofregular}, we know that
  $
    \widetilde\det_V(w\vp) = \z^{\Nh_W}
    \,,
  $
  which implies that
  $
    \Feg_\BG(\det_V)(x) = \det_V(w\vp) (\z\inv x)^{\Nh_W}
    \,.
  $
  
  The proof of the second equality goes the same.
\end{proof}

\begin{lemma}\label{degrwphi1overx}\hfill

  \begin{enumerate}
    \item
      For all $w \in W$ we have
      $$
        \Feg_\BG(R_{w\vp})(1/x) 
          = \det_V(w\vp)(\prod_{i=1}^{i=r}\z_i^*) x^{-\Nr_W} \Feg_\BG(R_{w\vp})(x)^*
        \,.
      $$
    \item
      If $w\vp$ is a $\z$-regular element, we have
      $$
        \Feg_\BG(R_{w\vp})(x)^\vee = (\z\inv x)^{-\Nr_W} \Feg_\BG(R_{w\vp})(x)
        \,.
      $$
  \end{enumerate}
\end{lemma}

\begin{proof}\hfill

  (1)
  By (\ref{indexoflevi}, (3)), we have
  $
    \Feg_\BG(R_{w\vp}^\BG) = \dfrac{P_\BG(x)^*}{P_{\BT_{w\vp}}(x)^*}
    \,.
  $  
  By (\ref{poincarepoly}), we have
  $
    P_\BG(x) = \prod_{i=1}^{i=r} (1-\z_i x^{d_i})
    \,.
  $
  Moreover, $P_{\BT_{w\vp}}(x) = \det_V(1-w\vp x)$.
  It follows that
  $$
    \Feg_\BG(R_{w\vp}^\BG) = \dfrac{\prod_{i=1}^{i=r} (1-\z_i^* x^{d_i})}{\det_V(1-w\vp x)^*}
    \,.
  $$
  The stated formula follows from the equality
  $
    \sum_{i=1}^{i=r} (d_i-1) = \Nr_W
    \,,
  $
  which is well known (see \eg\ \cite[Thm. 4.1.(2)(b)]{berkeley}).
  
  (2)
  By Corollary \ref{compuofdetprime}, we know that
  $
    \prod_{i=1}^{i=r} \z_i^*  = {\det'_V}^{\vee}(\ov\vp)^*\,,
  $
  hence (by Lemma \ref{detofregular}(2))
  $$
    \prod_{i=1}^{i=r} \z_i^* = \z^{-\Nr_W} \det_V(w\vp)^*
    \,.
  $$
  It follows from (1) that
  $$
    \Feg_\BG(R_{w\vp})(1/x)^*
          = (\z\inv x)^{-\Nr_W} \Feg_\BG(R_{w\vp})(x)
    \,.
  $$
\end{proof}
\smallskip

\subsubsection{The polynomial orders of a reflection coset}\hfill
\smallskip

   We define two ``order polynomials'' of $\BG$, which are both elements of
   $\BZ_K[x]$, and which coincide when $\BG$ is real. This differs from \cite{sp1}.

\begin{definition}\label{polynomialorders}\hfill

   \begin{itemize}
     \item[(nc)]
       The noncompact order polynomial of $\BG$ is the element of $\BZ_K[x]$ defined
       by
       \begin{align*}
         |\BG|_\nc :&= (-1)^r \Feg_\BG(\det_V^*)(x) P_\BG(x)^* \\
                           &=  (-1)^r {\det_V'}^\vee(\ov\vp)^* x^{\Nr_W} \prod_{i=1}^{i=r} (1-\z_i^* x^{d_i}) \\
                           &= (\z_1^*\z_2^*\cdots\z_r^*)^2 x^{\Nr_W} \prod_{i=1}^{i=r}(x^{d_i}-\z_i)
                           \,.
        \end{align*}\index{GOnc@$\ornc$}                 
     \item[(c)]
       The compact order polynomial of $\BG$ is the element of $\BZ_K[x]$ defined
       by
       \begin{align*}
         |\BG|_\oc :&= (-1)^r \Feg_\BG(\det_V)(x) P_\BG(x)^* \\
                           &=  (-1)^r {\det_V'}(\ov\vp)^* x^{\Nh_W} \prod_{i=1}^{i=r} (1-\z_i^* x^{d_i}) \\
                           &= \De_W(\ov\vp)^* x^{\Nh_W} \prod_{i=1}^{i=r}(x^{d_i}-\z_i)
         \,.
        \end{align*}\index{GOc@$\orc$}                     
   \end{itemize}
\end{definition}
 \smallskip
 
\begin{remark}\hfill

  1.  In the particular case of a maximal torus
   $
     \BT_{w\vp} := (V,w\vp)
   $
   of $\BG$, 
   it is readily seen that
   $$
     \left\{
       \begin{aligned}
         &|\BT_{w\vp}|_\nc = (-1)^r \det_V(w\vp)^* \det(1-w\vp x)^* \\
         &|\BT_{w\vp}|_\oc = (-1)^r \det_V(w\vp) \det(1-w\vp x)^* =  \det_V(x-w\vp) \,.
       \end{aligned}
     \right.
   $$

  2.  In general the order polynomials are not monic. 
  Nevertheless, if $\BG$ is real
  they are equal and monic  (see Lemma \ref{discriminanttrivial} for the case of $|\BG|_\oc$).   
  If $\BG$ is real, we set $|\BG| := |\BG|_\nc = |\BG|_\oc$.
\end{remark}

\begin{remark}
  If $\BG$ is rational or very twisted rational, and if $\bG$ (together with the endomorphism $F$)
  is any of the associated reductive groups attached to the choice of a suitable prime power $q$,
  then (see \eg\ \cite[11.16]{steinberg})
  $$
   |\BG|_{x=q} =  |\bG^F|
   \,.
 $$
\end{remark}

\begin{proposition}\label{GoverT}\hfill

\begin{enumerate}
  \item
    We have
    $$
      \left\{
        \begin{aligned}
          &\dfrac{|\BG|_\oc}{|\BT_{w\vp}|_\oc} 
             = \widetilde\det_V(w\vp)^* x^{\Nh_W}  \Feg_\BG(R_{w\vp})(x) \,, \\
         &\dfrac{|\BG|_\nc}{|\BT_{w\vp}|_\nc} 
            = \widetilde\det_V^\vee(w\vp)^* x^{\Nr_W}  \Feg_\BG(R_{w\vp})(x) \,.
        \end{aligned}
      \right.
    $$
  \item
    If moreover $w\vp$ is a $\z$-regular element of $W\vp$, we have
    $$
      \left\{
        \begin{aligned}
          &\dfrac{|\BG|_\oc}{|\BT_{w\vp}|_\oc} = (\z\inv x)^{\Nh_W} \Feg_\BG(R_{w\vp})(x) \,, \\
          &\dfrac{|\BG|_\nc}{|\BT_{w\vp}|_\nc} = (\z\inv x)^{\Nr_W} \Feg_\BG(R_{w\vp})(x) \,.
        \end{aligned}
      \right.
    $$
\end{enumerate}
\end{proposition}

\begin{proof}\hfill

  (1)
  By Definition \ref{polynomialorders} and the above remark, 1, and by
  Lemma \ref{indexoflevi},
  we have
  $$
    \begin{aligned}
      \dfrac{|\BG|_\oc}{|\BT_{w\vp}|_\oc} 
           &= {\det'_V(\ov\vp)}^* \det_V(w\vp)^* x^{\Nh_W}\dfrac{P_\BG(x)^*}{P_{\BT_{w\vp}}(x)^*} \\
           &= \widetilde\det_V(w\vp)^* x^{\Nh_W}  \Feg_\BG(R_{w\vp})(x) \,,
    \end{aligned}
  $$
  proving (1) in the compact type. The proof for the noncompact type is similar.
  
  (2) follows from Lemma \ref{detofregular}.
 \end{proof}
\smallskip

\subsection{$\P$-Sylow theory and $\P$-split Levi subcosets}\hfill
\medskip

\subsubsection{The Sylow theorems}\hfill\label{correctspets1}
\smallskip

{\small
  Here we correct a proof given in \cite{sp1}, \viz\ in Th. \ref{sylow} (4) below.
}

\begin{definition}\label{phigroup}\hfill
  \begin{itemize}
    \item
      We call \emph{$K$-cyclotomic polynomial} (or \emph{cyclotomic
      polynomial in $K[x]$}) 
      a monic irreducible polynomial of degree at least $1$ in $K[x]$ which divides $x^n-1$ 
      for some integer $n\geq 1$. 
    \item
      Let $\P \in K[x]$ be a cyclotomic polynomial.
      A \emph{$\P$-reflection coset\ }\index{Psp@$\P$-reflection coset}
      is a torus whose polynomial order is a power of $\P$.
  \end{itemize}
\end{definition}

\begin{remark}\hfill

  If $G$ is an associated finite reductive group, then a $\P_d$-reflection
  coset is the reflection datum of a torus which splits over $\BF_{q^d}$ but 
  no subtorus splits over any proper subfield.
\end{remark}

\begin{theorem}\label{sylow}\hfill

  Let $\BG$ be a reflection coset over $K$ and let $\P$ be a $K$-cyclotomic polynomial.
  \begin{enumerate}
    \item If $\P$ divides $P_\BG(x)$, there exist nontrivial
      $\P$-subcosets of $\BG$.
    \item Let $\BS$ be a maximal $\P$-subcoset of $\BG$.
      Then 
      \begin{enumerate}
        \item
          there is $w\in W$ such that
          $
            \BS = (\ker \P(w\vp),(w\vp)|_{\ker \P(w\vp)})
            \,,
          $
        \item
          $|\BS| = \P^{a(\P)}$, the full contribution of
          $\P$ to $P_\BG(x)$.
     \end{enumerate}
    \item Any two maximal $\P$-subcosets of $\BG$
      are conjugate under $W$.
    \item Let $\BS$ be a maximal $\P$-subcoset of $\BG$. We set
      $\BL := C_\BG(\BS)$ and
      $W_\BG(\BL) := N_{W}(\BL)/W_\BL .$
      Then
      $$
        \dfrac{|\BG|_\nc}{|W_{\BG}(\BL)||\BL|_\nc}
        \equiv
        \dfrac{|\BG|_\oc}{|W_{\BG}(\BL)||\BL|_\oc}
        \equiv 1 \!\mod{\P}.
      $$
    \item
      With the above notation, we have 
      $$
        \BL = (V,W_\BL w\vp)
        \,\text{ with }\,
        W_\BL = C_W(\ker \P(w\vp))
        \,.
      $$
      We set $V(\BL,\P) := \ker\P(w\vp)$
      viewed as a vector space over the field $K[x]/(\P(x))$ through
      its natural structure of $K[w\phi]$-module.
      Then the pair $(V(\BL,\P),W_\BG(\BL))$ is a reflection
      group.  
  \end{enumerate}
\end{theorem}

  The maximal $\P$-subcosets of $\BG$ are called the 
  \emph{Sylow $\P$-subcosets}.\index{Sylow $\P$-subcosets}
\smallskip

\begin{proof}[Proof of \ref{sylow}]\hfill

  As we shall see, assertions (1) to (3) are
  consequences of the main results of Springer
  in \cite{springer} (see also Theorem 3.4 in \cite{brma1}).
  
  Assertion (5) is nothing but a reformulation of a
   result of Lehrer and Springer (see for example
  \cite[Thm.5.6]{berkeley})
\smallskip

  For each $K$-cyclotomic polynomial $\P$ and $w \in W$,
  we denote by $V(w\vp,\P)$ the kernel of the endomorphism
  $\Phi(w\vp)$ of $V$ (\ie\ $V(\BL,\P)$ viewed as a $K$-vector space).
  Thus 
  $$
    \BS(w\vp,\P) := (V(w\vp,\P),(w\vp){|_{V(w\vp,\P)}})
  $$
  is a torus of $\BG$.

  Let us denote by $K'$ a Galois extension of
  $K$ which splits $\P$, and set $V':=K'\otimes_K V$.
  For every root $\zeta$ of $\P$ in $K'$, we set
  $V' (w\vp,\zeta) := V' (w\vp,(x-\zeta))$. It is clear
  that
  $
    V' (w\vp,\P) = \bigoplus_\zeta V' (w\vp,\zeta)
  $
  where $\zeta$ runs over the set of roots of $\P$,
  and thus
  $$
    \dim V' (w\vp,\P) = \deg(\P) \cdot \dim V' (w\vp,\zeta)
    \,.
  $$
  It follows from \cite{springer}, 3.4 and 6.2, that for all such $\z$
    \begin{enumerate}
      \item[(S1)]
        $\max_{(w\in W)}\dim V' (w\vp,\zeta) = a(\P)$,
      \item[(S2)]
        for all $w\in W$, there exists $w'\in W$ such that
        $\dim V'(w'\vp,\zeta) = a(\P)$ and
        $V' (w\vp,\zeta) \subset V' (w'\vp,\zeta)$,
      \item[(S3)]
        if $w_1,w_2 \in W$ are such that
        $\dim V' (w_1\vp,\zeta)
          = \dim V' (w_2\vp,\zeta) = a(\P)$,
        there exists
        $w \in W$
        such that $w \cdot V' (w_1\vp,\zeta) =
        V'(w_2\vp,\zeta)$.
    \end{enumerate}

  Now, (S1) shows that there exists
  $w \in W$ such that the rank of $\BS(w\vp,\P)$ is
  $a(\P)\deg(\P)$, which implies the first assertion
  of Theorem \ref{sylow}.

  If $V' (w\vp,\zeta) \subseteq V'(w'\vp,\zeta)$ then we have
  $V'(w\vp,\sigma(\zeta)) \subseteq V'(w'\vp,\sigma(\zeta))$
  for all $\sigma\in\Gal(K'/K)$, hence
  \begin{enumerate}
    \item[$\bullet$]
      $V'(w\vp,\P) \subset V'(w'\vp,\P)$,
    \item[$\bullet$]
      ${w'\vp}_{|_{V(w\vp,\P)}} = {w\vp}_{|_{V(w\vp,\P)}}$.
  \end{enumerate}
  So (S2) implies that for all $w\in W$,
  there exists
  $w'\in W$ such that the rank of $\BS(w'\vp,\P)$ is
  $a(\P)\deg(\P)$
  and $\BS(w\vp,\P)$ is contained in $\BS(w'\vp,\P)$,
  which proves assertion (2) of Theorem \ref{sylow}.

  For the same reason, (S3) shows that if
  $w_1$ and $w_2$ are two elements of $W$ such that
  both
  $\BS(w_1\vp,\P)$ and $\BS(w_2\vp,\P)$ have rank
  $a(\P)\deg(\P)$,
  there exists $w\in W$ such that
  $\BS(ww_1\vp w\inv,\P) = \BS(w_2\vp,\P)$,
  which proves assertion (3) of Theorem \ref{sylow}.
\smallskip

  The proof of the fourth assertion
  requires several steps.

\begin{lemma}\label{congruence}\hfill

  For $\BS$ a Sylow $\P$-subcoset of $\BG$ let
  $\BL := C_{\BG}(\BS)$. Then for any class function $\al$ on $\BG$,
  we have
  $$
     \Feg_\BG(\al)(x) \equiv
       \dfrac{1}{|W_\BG(\BL)|}
       \dfrac{P_\BG(x)^*}{P_\BL(x)^*}
       \Feg_\BL(\ResLG\al)(x)
    \mod \P(x)
    \,.
  $$
\end{lemma}

  \begin{proof}[Proof of \ref{congruence}]\hfill
  
  We have
  $$
      P_\BG(x)^* S_\BG(\al) = \Feg_\BG(\al)(x) = 
      \dfrac{1}{|W|} \sum_{w\in W}
      \al(w\vp)\dfrac{ P_\BG(x)^*}{\det(1\!-\!xw\vp)^*}
      \,,
  $$
  and it follows from the first two assertions of
  Theorem \ref{sylow} that
  $$
       \Feg_\BG(\al)(x) \equiv
      \dfrac{P_\BG(x)^*}{|W|} \sum_w\dfrac{\al(w\vp)}
      {\det(1-xw\vp)^*} \!\mod \P(x),
  $$
    where $w$ runs over those elements of $W$ such that
    $V(w\vp,\P)$ is of maximal dimension.
    These subspaces are permuted transitively by $W$ (\cf\ (S3)
    above).
    Let $V(w_0\vp,\P)$ be one of them, and let
    $\BS$ be the Sylow $\P$-subcosets defined by
  $$
      \BS =
      (V(w_0\vp,\P),(w_0\vp)_{|_{V(w_0\vp,\P)}}).
  $$
    Let us recall that we set
    $\BL = C_{\BG}(\BS) = (V,W_\BL w_0\vp)$.
    The group $N_W(\BS)$ consists of all $w \in W$
    such that
    $$
      w.V(w_0\vp,\P) \!=\! V(w_0\vp,\P)
      \quad\text{and}\quad
      (ww_0\vp w^{-1})_{|_{V(w_0\vp,\P)}} =
      (w_0\vp)_{|_{V(w_0\vp,\P)}}
      \,.
    $$
    Since in this case every $w \in W$
    such that $w.V(w_0\vp,\P)=V(w_0\vp,\P)$
    belongs to $N_W(\BS)$, we have
    $$
      \Feg(\al)(x) \equiv
      |W:N_W(\BS)| \frac{P_\BG(x)^*}{|W|}
      \sum_{w \sim w_0}\dfrac{\al(w\vp)}{\det(1-xw\vp)^*}
      \mod{\P(x)} \,,
    $$
    where ``$w \sim w_0$'' means that
    $V(w\vp,\P) = V(w_0\vp,\P)$.

        Following (S3) above, the elements $w \sim w_0$ are
        those such that
        $ww_0^{-1}$ acts trivially on
        $V(w_0\vp, \P)$, \ie\
        are the elements of the coset
        $W_\BL w_0$.
        Thus
         $$
          \sum_{w \sim w_0}\dfrac{\al(w\vp)}{\det(1-xw\vp)^*} =
          \sum_{w \in W_\BL}
          \dfrac{\al(ww_0\vp)}{\det(1-xww_0\vp)^*} =
          |W_\BL| S_\BL(\ResLG(\al)).
         $$

        We know that $N_W(\BS) \subset N_W(\BL)$.
        Let us check the reverse inclusion.
        The group $N_W(\BL)/W_\BL$ acts on $V^{W_\BL}$
        and centralizes $(w_0\vp)_{|_{V^{W_\BL}}}$.
        Hence it stabilizes the characteristic subspaces of
        $(w_0\vp)_{|_{V^{W_\BL}}}$, among them
        $V(w_0\vp,\P)$. This shows that $N_W(\BL) = N_W(\BS)$.
        In particular $|N_W(\BS)|=|W_\BL||W_{\BG}(\BL)|$.
\end{proof}
\smallskip

  Applying Lemma \ref{congruence} to the case where
  $\al = 1^\BG$ gives 
  
\begin{proposition}\label{degrescongrus}\hfill

  Let $\BS$ be a Sylow $\P$-subcoset of $\BG$, and set
  $\BL = C_{\BG}(\BS)$. 
  \begin{enumerate}
    \item
      $
         \dfrac{1}{|W(\BL)|} \dfrac{P_\BG(x)^*}{P_\BL(x)^*}
          \equiv 1 \mod \P(x) \,.
      $
   \item
     Whenever $\al$ is a class function on $\BG$, we have
     $$
       \Feg_\BG(\al)(x) \equiv \Feg_\BL(\Res^\BG_\BL(\al))(x) \mod \P(x)
       \,.
     $$ 
  \end{enumerate}
\end{proposition}

  Now applying Lemma \ref{congruence} and Proposition
  \ref{degrescongrus}
  to the cases where $\al$ is $\det_V$ and $\det_V^*$
  gives the desired congruences in Theorem \ref{sylow}(4)
  (thanks to Definition \ref{polynomialorders})
       $$
        |\BG|_\nc/(|W_{\BG}(\BL)||\BL|_\nc)
        \equiv
        |\BG|_\oc/(|W_{\BG}(\BL)||\BL|_\oc)
        \equiv 1 \!\mod \P(x)
        \,.
      $$
\end{proof}
\medskip

{\sl
  Fake degrees and regular elements
}
\smallskip

  Let $\al$ be a class function on $\BG$. Let $w\vp$ be
  $\P$-regular, so that the maximal torus
  $\BT_{w\vp}$ of $\BG$ is a Sylow $\P$-subcoset. 
  We get from Proposition \ref{degrescongrus}(2) that
  $$
    \Feg_\BG(\al)(x) \equiv \al(w\vp) \mod{\P(x)},
  $$
  which can be reformulated into the following proposition.

\begin{proposition}\label{fakeonregular}\hfill

  Let $w\vp$ be $\zeta$-regular for some root of unity $\zeta$.
  \begin{enumerate}
    \item
      For any $\al\in\OCFuf(\BG)$ we have
      $ \Feg_\BG(\al)(\zeta) = \al(w\vp). $
    \item
      In particular, we have
      $
        \Feg_\BG(R_{w\vp})(\zeta) = |C_W(w\vp)|
        \,.
      $
  \end{enumerate}
\end{proposition}
\smallskip
 
\subsection{The associated braid group}\label{thebraidgroup}\hfill
\smallskip

\subsubsection{Definition}\hfill
\smallskip

  Here we let $V$ be a complex vector space of finite dimension $r$, and 
  $W \subset \GL(V)$ be a complex reflection group on $V$.

  We recall some notions and results from \cite{bmr}.
  
   Choosing a base point $x_0 \in V^\reg =  V-\bigcup_{H\in \CA(W)} H$, 
   we denote by $\bB_W := \pi_1(V^\reg/W,x_0)$\index{BbW@$\bB_W$}
   the corresponding braid group, and we set
   $\bP_W := \pi_1(V^\reg,x_0)$\index{PbW@$\bP_W$}
   
   Since the covering 
   $
     V^\reg \twoheadrightarrow V^\reg/W
   $
   is Galois by Steinberg's theorem (see \eg\ \cite[4.2.3]{berkeley}),
   we have the corresponding short exact sequence
   $$
     1 \ra \bP_W \ra \bB_W \ra W \ra 1
     \,.
   $$
   
   A \emph{braid reflection}\index{braid reflection} $\bs$ in $\bB_W$ 
   is a generator of the monodromy around the image in $V/W$ of a reflecting 
   hyperplane $H\in \CA(W)$. 
   We then say that $\bs$ is a \emph{braid reflection around $H$} (or \emph{around the orbit of $H$ under $W$}).
\smallskip

\subsubsection{Lengths}\hfill
\smallskip

  For each $H\in \CA(W)$, there is a linear character\index{lH@$l_H$}
  $$\index{lH@$l_H$}
    l_H : \bB_W \ra \BZ
  $$
  such that, whenever $\bs$ is a braid reflection in $\bB_W$,
  $$
    l_H(\bs) =
    \left\{
      \begin{aligned}
        &1\,\,{\text{ if $\bs$ is a braid reflection around the orbit of }H},  \\
        &0 \,\,\text{ if not} \,.
      \end{aligned}
    \right.
  $$
  We have $l_H = l_{H'}$ if and only if $H$ and $H'$ are in the same $W$-orbit.
  We set
  $$
    l := \sum_{H\in \CA(W)/W} l_H
    \,.
  $$
\smallskip

\subsubsection{The element $\bpi_W$}\label{theelementbpi}\hfill
\smallskip
   
   We denote by $\bpi_W$\index{pibW@$\bpi_W$} 
   (or simply by $\bpi$\index{pib@$\bpi$} if there is no ambiguity) the element
   of $\bP_W$ defined by the loop
   $$
     \bpi_W : [0,1] \ra V^\reg
     \,\,,\,\,
     t \mapsto \exp(2\pi i t) x_0
     \,.
   $$
   We have 
   $
     \bpi_W \in Z\bB_W
     \,,
   $
   and if $W$ acts irreducibly on $V$, we know by 
   \cite[Thm. 1.2]{dmm} (see also
   \cite[12.4]{bessisarxiv} and \cite[2.24]{bmr})
   that
   \emph{$\bpi_W$ is a generator of $Z\bP_W$} called 
   \emph{the positive generator of  $Z\bP_W$}.
   
   We have
   $$
     l_H(\bpi_W) = |\text{orbit of }H\text{ under }W| e_H
   $$
   and in particular (by formula \ref{e=N+N})
   $$
     l(\bpi_W) = \Nr_W + \Nh_W
     \,.
   $$
 \smallskip
 
\subsubsection{Lifting regular automorphisms}\label{liftingregular}\hfill
\medskip

  For this section one may refer to \cite[\S 18]{berkeley} (see also \cite[\S 3]{dm}).
\smallskip

\noindent
\textsl{A. Lifting a $\z$-regular element $w\vp$}
\smallskip

\begin{itemize}
  \item  
    We fix a root of unity $\z$ and a $\z$-regular element $w\vp$ in $W\vp$, and we
    let $\de$ denote the order of $w\vp$ modulo $W$. Notice that since 
    $(w\vp)^\de$ is a $\z^\de$-regular element of $W$, then (by \ref{zetatoe})
    $\z^{\de e_W}=1$.
    
    Let us also choose $a,d \in \BN$ such that $\z = \z^a_d$ 
    ($a/d$ is well defined in $\BQ/\BZ$). 
    By what precedes we know that
    $d \mid e_Wa\de$, or, in other words,  $e_W\de a/d \in \BZ$.
\smallskip

  \item
    Let us denote by $x_1$ a $\z$-eigenvector of $w\vp$, and let us choose
    a path $\ga$ from $x_0$ to $x_1$ in $V^\reg$.
\smallskip

  \item
    We denote by $\pi_{x_1,a/d}$ the path in $V^\reg$ from $x_1$ to $\z x_1$
    defined by
    $$
      \pi_{x_1,a/d} : t \mapsto \exp(2\pi i at/d). x_1
      \,.
     $$\index{pix1@$\pi_{x_1,a/d}$}
    Note that $\pi_{x_1,a/d}$ does depend on the choice of $a/d \in \BQ$
    and not only on $\z$.
\end{itemize}
\smallskip
  
  Following \cite[5.3.2.]{berkeley}, we have 
\smallskip
  
  \begin{itemize}
    \item
      a path $[w\vp]_{\ga,a/d}$\index{wpga@$[w\vp]_{\ga,a/d}$}
      (sometimes abbreviated $[w\vp]$) in $V^\reg$
      from $x_0$ to $(w\vp)(x_0)$, defined as follows:
      $$
        \xymatrix{
          [w\vp]_{\ga,a/d} :
          &x_0\ar@{~>}[rr]^{\ga} &&x_1\ar@/^2pc/[rr]^{\pi_{x_1,a/d}} &&\z x_1
          \ar@{~>}[rr]^{(w\vp)(\ga\inv)\qquad} &&(w\vp)(x_0)
        }
      $$
    \item
      an automorphism $\ba(w\vp)_{\ga,a/d}$\index{abwp@$\ba(w\vp)_{\ga,a/d}$}
      (sometimes abbreviated
      $\ba(w\vp)$) of $\bB_W$, defined as follows: 
      for $g\in W$ and $\bg$ a path in $V^\reg$ from $x_0$ to $gx_0$,
      the path $\ba(w\vp)_{\ga,a/d}(\bg)$ from $x_0$ to $\Ad(w\vp)(g)x_0$
      is
      \begin{align}\label{defofaofphi}
        \xymatrixrowsep{0.6pc} 
        \xymatrix{
         &\ba(w\vp)_{\ga,a/d}(\bg)\, :&&&&  \\
         &x_0\ar@{~>}[r]^{[w\vp]} &(w\vp)(x_0)\ar@{~>}[r]^{\quad(w\vp)(\bg)\quad} 
	  &(w\vp g)(x_0) \ar@{~>}[rr]^{\qquad\Ad(w\vp)(g)([w\vp]\inv)\qquad\quad}
          &&\Ad(w\vp)(g)(x_0)
          }
      \end{align}
  \end{itemize}
  with the following properties.
\smallskip

\begin{lemma}\label{propertyofa}\hfill
  
  \begin{enumerate}
    \item
      The automorphism $\ba(w\vp)_{\ga,a/d}$ has finite order, equal to the order of $\Ad(w\vp)$
      acting on $W$.
\smallskip

    \item
      The path
      $$
        \brh_{\ga,a/d} := [w\vp]_{\ga,a/d}\cdot
              \ba(w\vp)([w\vp]_{\ga,a/d})\cdots\ba(w\vp)^{\de-1}([w\vp]_{\ga,a/d})
      $$
      (often abbreviated $\brh$)
      defines an element of $\bB_W$ which satisfies
      $$
        \brh_{\ga,a/d}^d  = \bpi^{\de a}
        \,.
      $$     
  \end{enumerate}
\end{lemma}

\begin{remark}\hfill
  \begin{enumerate}
    \item  
      Notice that $a/d \in \BQ$ is unique up to addition of an integer, so that $\brh$
      is defined by $w\vp$ up to multiplication by a power of $\bpi^\de$.
    \item
      Let us consider another path $\ga'$ from $x_0$ to $x'_1$, another eigenvector 
      of $w\vp$ with eigenvalue $\z$. Then 
      \begin{enumerate}
        \item
          the element $\brh_{\ga',a/d}$ is conjugate to $\brh_{\ga,a/d}$ by an element 
          of $\bP_W$, and
        \item
          the element $\ba(w\vp)_{\ga',a/d}$ is conjugate to $\ba(w\vp)_{\ga,a/d}$ 
          by an element of $\bP_W$.
      \end{enumerate}
  \end{enumerate}
\end{remark}
\smallskip

\noindent
\textsl{B. When $\vp$ is $1$-regular}
\smallskip

  Now assume moreover that $\vp$ is $1$-regular, and choose for base point $x_0$ an
  element fixed by $\vp$. Let us write $1 = \exp(2\pi i n)$ for some $n \in \BZ$ (which plays here
  the role played by $a/d$ above).
\smallskip
%

\begin{lemma}\label{whenphireguylar}\hfill

  \begin{enumerate}
    \item
      The corresponding loop $[\vp]$ defines $\bpi^n$.
    \item
      The path $[w\vp]_{\ga,n}$ defines a lift $\bw_{\ga,n}$ 
      (abbreviated to $\bw$) of $w$ in $\bB_W$.
    \item
      We have
      $
        \ba(w\vp)_{\ga,n} = \Ad(\bw_{\ga,n})\cdot \ba(\vp)
        \,.
      $
  \end{enumerate}
\end{lemma}

\begin{proof}\hfill

  (1) is obvious. (2) results from the fact that the path $[w\vp]_{\ga,a/d}$ starts at $x_0$
  and ends at $w\vp x_0 = wx_0$.
  
  Since $\vp(x_0) = x_0$, Definition \ref{defofaofphi} becomes simply
  $\ba(\vp)(\bg) = \vp(\bg)$, a lift of $\Ad(\vp)(g)$ to $\bB_W$. To prove
  (3), we notice that by (2) we have
  $$
    \xymatrixrowsep{0.6pc}
        \xymatrix{
         \ba(w\vp)_{\ga,a/d}(\bg)\, = &&&&  \\
         \!\!\!\!\!x_0\ar@{~>}[r]^{\bw} &wx_0\ar@{~>}[r]^{\,(w\vp)(\bg)\quad} 
	  &(w\vp g)(x_0) \ar@{~>}[rr]^{\qquad\Ad(w\vp)(g)(\bw\inv)\qquad\quad}
          &&\Ad(w\vp)(g)(x_0) .\!\!\!
    }
  $$
  Since
  $$
    \xymatrixrowsep{0.6pc} 
        \xymatrix{
         &\Ad(\bw)(\bg)\, = &&&&  \\
         &x_0\ar@{~>}[r]^{\bw} &wx_0\ar@{~>}[r]^{\,w(\bg)\quad} 
	  &wgx_0 \ar@{~>}[rr]^{\qquad\Ad(w)(g)(\bw\inv)\qquad\quad}
          &&\Ad(w)(g)(x_0)
    }
  $$
  we see that
  $$
    \xymatrixrowsep{0.6pc} 
        \xymatrix{
         &\!\!\Ad(\bw)(\vp(\bg))\, = &&&&  \\
         &x_0\ar@{~>}[r]^{\bw} &wx_0\ar@{~>}[r]^{\,(w\vp)(\bg)\quad} 
	  &wgx_0 \ar@{~>}[rr]^{\qquad\Ad(\vp w)(g)(\bw\inv)\qquad\quad}
          &&\Ad(w)(g)(x_0) \,,
    }
  $$
  thus showing that $\Ad(\bw)(\vp(\bg)) = \ba(w\vp)(\bg)\,$.
\end{proof}
\smallskip

  In that case, we can consider the semidirect product $\bB_W\rtimes \genby{\ba(\vp)}$,
  in which we set $\bvp := \ba(\vp)$. Then assertion (2) of Lemma \ref{propertyofa}
  becomes
  $$
    \brh = (\bw\bvp)^\de
    \quad\text{and}\quad
    (\bw\bvp)^{\de d} = \bpi^{\de a}
     \,.
  $$

\noindent
\textsl{C. Centralizer in $W$ and centralizer in $\bB_W$}
\smallskip

  Now we return to the general situation (we are no more assuming that
  $\vp$ is 1-regular).
  
  By (5) in Theorem \ref{sylow}, 
  we know that the centralizer of $w\vp$ in $W$, denoted by $W(w\vp)$\index{Wwp@$W(w\vp)$},
  is a complex reflection group on the $\z$-eigenspace $V(w\vp)$\index{Vwp@$V(w\vp)$} of $w\vp$.
\smallskip

  Assume from now on that the base point $x_0$ is chosen in  $V(w\vp)^\reg$. 
  Let us denote by $\bB_W(w\vp)$\index{BbWwp@$\bB_W(w\vp)$} the braid group (at $x_0$)
  of $W(w\vp)$ on $V(w\vp)$, 
  and by $\bP_W(w\vp)$\index{PbWwp@$\bP_W(w\vp)$} its pure braid group.
\smallskip

  Since the reflecting hyperplanes of $W(w\vp)$ are the intersections with $V(w\vp)$
  of the reflecting hyperplanes of $W$ (see for example \cite[18.6]{berkeley}), the 
  inclusion of $V(w\vp)$ in $V$ induces an inclusion
  $$
    V(w\vp)^\reg \hookrightarrow V^\reg
  $$
   which in turn induces a natural morphism (see again \cite[18.6]{berkeley})
  $$
    \bB_W(w\vp) \ra C_{\bB_W}(\ba(w\vp))
    \,.
  $$
  
  The next statement has been proved in all cases if $\ov\vp = 1$
  \cite[12.5,(3)]{bessisarxiv}. 
    
\begin{theocon}\label{th}\hfil

  The following assertion is true if $\ov\vp = 1$, and it is a conjecture
  in the general case:
  
  The natural morphism
   $$
     \bB_W(w\vp) \ra C_{\bB_W}(\ba(w\vp))
   $$
   is an isomorphism.
\end{theocon}

\begin{remark}\label{samepi}
  It results from the above Theorem--Conjecture that the positive generator $\bpi_W(w\vp)$ of
  the center $Z\bP_W(w\vp)$ of $\bP_W(w\vp)$ is identified with the positive generator $\bpi$ of $\bP_W$.
\end{remark}
    
\subsection{The generic Hecke algebra}\hfill
\smallskip

\subsubsection{Definition}\label{generaldefgenerichecke}\hfill
\smallskip
   
   \emph{The generic Hecke algebra $\CH(W)$ of $W$} is defined as follows. 
   Let us choose a $W$-equivariant
   set of indeterminates 
   $$
     \bu := (u_{H,i})_{(H\in\CA(W))(i=0,\dots,e_H-1)}
     \,.
   $$
   The algebra $\CH(W)$ is the quotient of the group algebra of $\bB_W$ over the
   ring of Laurent polynomials 
   $
     \BZ[\bu,\bu\inv] := \BZ[(u^{\pm 1}_{H,i})_{H,i}]
   $
   by the ideal generated by the elements
   $\prod_{i=0}^{e_H-1} (\bs-u_{H,i})$ for $H\in\CA(W)$
   and $\bs$ running over the set of braid reflections around $H$.
\smallskip

  The linear characters of the generic Hecke algebra $\CH(W)$
  are described as follows.
  
  Let $\chi : \CH(W) \ra \BZ[\bu,\bu\inv]$ be an algebra morphism.
  Then there is a $W$-equivariant family of integers $(j_H^\chi)_{H\in\CA(W)},
  j_H^\chi \in\{0,\dots,e_H-1\}$, such that, whenever $\bs$ is a braid
  reflection around $H$, we have
  $
    \chi(\bs) = u_{H,j_H^\chi}
    \,.
  $
  
\subsubsection{Parabolic subalgebras}\hfill
\smallskip

  Let $I$ be an intersection of reflecting hyperplanes of $W$, and let
  $\bB_{W_I}$ be the braid group of the parabolic subgroup $W_I$
  of $W$.
  
  If $\bu = (u_{H,i})_{(H\in\CA(W))(i=0,\dots,e_H-1)}$ is a $W$-equivariant
  family of indeterminates as above, then the family
  $$
    \bu_I := (u_{H,i})_{(H\in\CA(W_I))(i=0,\dots,e_H-1)}
  $$
  is a $W_I$-equivariant family of indeterminates.
  
  We denote by $\CH(W_I,W)$ the quotient of the group algebra of $\bB_{W_I}$
  over $\BZ[\bu_I,\bu_I\inv]$ by the ideal generated by the elements
   $\prod_{i=0}^{e_H-1} (\bs-u_{H,i})$ for $H\in\CA(W_I)$
   and $\bs$ a braid reflection of $\bB_{W_I}$ around $H$.
   
   The algebra $\CH(W_I,W)$ is a specialization of the generic Hecke
   algebra of $W_I$, called the \emph{parabolic subalgebra\index{parabolic subalgebra}
   of $\CH(W)$ associated
   with $I$}.
   
   The natural embeddings of $\bB_{W_I}$ into $\bB_W$ (see \eg\ \cite[\S 4]{bmr})
   are permuted
   transitively by $\bP_W$. The choice of such an embedding defines a
   morphism of $\CH(W_I,W)$ onto a subalgebra of $\CH(W)$ (\cite[\S 4]{bmr}).

\subsubsection{The main Theorem--Conjecture}\label{2.5.3}\hfill
\smallskip

\noindent\textsl{Notation.} 
\begin{itemize}
  \item 
    An element $P(\bu) \in \BZ[\bu,\bu\inv]$ is called
    \emph{multi-homogeneous} if, for each $H \in \CA(W)$,
    it is homogeneous as a Laurent polynomial in the
    indeterminates $\{u_{H,i}\,\mid\,i = 0,\ldots,e_H-1\}\,.$
  \item
    The group
    $
      \fS_W := \prod_{H\in\CA(W)/W} \fS_{e_H}
    $\index{SW@$\fS_W$}
    acts naturally on the set of indeterminates $\bu$.
  \item
    We denote by $\xi\mapsto \xi^\vee$\index{xivee@$\xi^\vee$}
     the involutive automorphism of
    $\BZ[\bu,\bu\inv]$ which sends $u_{H,i}$ to $u_{H,i}\inv$
    (for all $H$ and $i$).
\end{itemize}

  The following assertion is conjectured to be true
  for all finite reflection groups. It has been proved for almost
  all irreducible complex reflection groups
  (see \cite{sp1} and \cite{mami} for more details).

  \begin{theocon}\label{proprHecke}\hfill
  \begin{enumerate}
    \item
      The algebra $\CH(W)$ is free of rank $|W|$ over $\BZ[\bu,\bu\inv]$,
      and by extension of scalars to the field $\BQ(\bu)$ it becomes
      semisimple.
    \item
      There exists a unique symmetrizing form
      $$
         \tau_W  : \CH(W) \ra \BZ[\bu,\bu\inv] 
      $$\index{tau@$\tau$}
      (usually denoted simply $\tau$)
      with the following properties.
      \begin{enumerate}
        \item
          Through the specialization
           $
             u_{H,j} \mapsto \exp\left(2i\pi j/e_H\right)
             ,
            $
            the form $\tau$ becomes the canonical symmetrizing form
            on the group algebra of $W$.
        \item
            For all $b\in B$, we have
            $
               \tau(b\inv)^\vee =
               \dfrac{\tau(b\bpi)}{\tau(\bpi)}
               \,,
            $
            where $\bpi := \bpi_W$ (see \ref{theelementbpi}).
       \end{enumerate}
    \item
            If $I$ is an intersection of reflecting hyperplanes of $W$,
            the restriction of $\tau_W$ to a naturally embedded parabolic
            subalgebra $\CH(W_I,W)$ is the corresponding 
            specialization of the form $\tau_{W_I}$.
    \item
       The form $\tau$ satisfies the following conditions.
       \begin{enumerate}
          \item
            For $b\in B$, $\tau(b)$ is
            invariant under the action of $\fS_W$.
          \item
            As an element of $\BZ[\bu,\bu\inv]$, $\tau(b)$ is
            multi-homogeneous of degree $l_H(b)$ in the
            indeterminates $\{u_{H,i}\,\mid\,i = 0,\ldots,e_H-1\}$
            for all $H\in\CA(W)$.
            In particular, we have
            $$
              \tau(1) = 1
              \text{\quad and\quad}
              \tau(\bpi_W) =
               (-1)^{\Nr_W}
                \!\!
                \prod_{H\in\CA(W) \atop 0\leqslant i\leqslant e_H-1} u_{H,i}
                \,.
            $$
      \end{enumerate}
  \end{enumerate}
\end{theocon}
\smallskip

\subsubsection{Splitting field}\hfill
\smallskip

  An irreducible complex reflection group in $\GL(V)$ is said to be
  \emph{well-generated} \index{well-generated}
  if it can be generated by $\dim(V)$ reflections (see \eg\ \cite[\S 4.4.2]{berkeley}
  for more details).
\smallskip

  The following theorem has been proved in \cite{maR}.
  
\begin{theorem}\label{rationalityhecke}\hfill

    Assume assertion (1) of Theorem--Conjecture \ref{proprHecke} holds.
  
    Let $W$ be an irreducible complex reflection group, and let\index{mW@$m_W$} 
    $$
      m_W := 
      \left\{
        \begin{aligned}
          & |ZW| \quad\text{if $W$ is well-generated,} \\
          & |\bmu(\BQ_W)| \quad\text{else.}
        \end{aligned}
      \right.
    $$\index{mW@$m_W$}
    Let us choose a $W$-equivariant set of indeterminates $\bv := (v_{H,j})$ 
    subject to the conditions
    $$
      v_{H,j}^{m_W} = \z_{e_H}^{-j} u_{H,j}
      \,.
    $$
    Then the field
   $
      \BQ_W(\bv)
    $
    is a splitting field for $\CH(W)$.
\end{theorem}

  We denote by $\Irr (\CH(W))$\index{IrrHW@$\Irr (\CH(W))$} the set of irreducible characters of 
  $$
    \BQ_W(\bv)\CH(W) := \BQ_W(\bv)\otimes_{\BZ[\bu,\bu\inv]} \CH(W)
    \,,
  $$
  which is also the set of irreducible characters of
  $$
    \ov\BQ(\bv)\CH(W) := \ov\BQ(\bv)\otimes_{\BZ[\bu,\bu\inv]} \CH(W)
    \,.
  $$
  
  Following \cite[\S 2D]{maG}, we see that the action of the group $\fS_W$ on 
  $\BZ[\bu,\bu\inv]$ by permutations of the indeterminates $u_{H,j}$ extends to an action
  on $\ov\BQ[\bv,\bv\inv]$. Indeed, we let $\fS_W$ act trivially on $\ov\BQ$ and for all
  $\si\in\fS_W$ we set
  $$
    \si(v_{H,j}) := \exp
                                \left(
                                  2\pi i(\si(j)-j)/e_Hm_W
                                \right)
                                v_{H,j}
   \,.
  $$
  That action of $\fS_W$ induces an action on $\CH(W)$, and then an action on 
  $\Irr(\CH(W))$ by
  \begin{equation}\label{actionofS}
    \si(\chi)(h) := \si(\chi(\si\inv(h)))
    \quad\text{for }\, \si\in\fS_W\,,\,h\in\CH(W)\,,\,\chi\in\Irr(\CH(W))
    \,.
  \end{equation}
\smallskip

\subsubsection{Schur elements}\hfill
\smallskip

  The next statement follows from Theorem \ref{rationalityhecke} by a general argument
  which goes back to Geck \cite{geck}.

\begin{proposition}\label{schurelements}\hfill

  Assume Theorem--Conjecture \ref{proprHecke} holds.

  For each $\chi \in \Irr(\CH(W))$ there is a non-zero
  $S_\chi \in\BZ_W[\bv,\bv\inv]$ such that
  $$
     \tau = \sum_{\chi\in\Irr(\CH(W))} \dfrac{1}{S_\chi} \chi
     \,. 
  $$\index{Schi@$S_\chi$}
\end{proposition} 

  The Laurent polynomials $S_\chi$ are called the 
  \emph{generic Schur elements}\index{Schur element} of $\CH(W)$ (or of $W$).
\smallskip

  Let us denote by $S \mapsto S^\vee$\index{Svee2@$S \mapsto S^\vee$}
  the involution of $ \BQ_W(\bv)$ consisting in
  \begin{itemize}
    \item
      $v_{H,j}^\vee := v_{H,j}\inv$ for all $H\in\CA(W)$, $j = 0,\dots,e-1$,
    \item
      complex conjugating the scalar coefficients.
  \end{itemize}
  Note that this extends the previous operation $\la\mapsto\la^\vee$ on $\BZ[\bu,\bu\inv]$
  defined above in \ref{2.5.3}.

  The following property of the Schur elements 
  (see \cite[2.8]{sp1}) is an immediate consequence of 
  the characteristic property (see Theorem \ref{proprHecke}(2)(b)) of the canonical trace $\tau$.
    
\begin{lemma}\label{schurpalin}\hfill

  Assume Theorem--Conjecture \ref{proprHecke} holds.

  Whenever $\chi \in \Irr(\CH(W))$, we have
    $$
      S_\chi(\bv)^\vee = \dfrac{\tau(\bpi)}{\om_\chi(\bpi)} S_\chi(\bv)
      \,,
    $$
    where $\om_\chi$\index{ox@$\om_\chi$} denotes the central character
    corresponding to $\chi$.
\end{lemma}
\smallskip

\subsubsection{About Theorem--Conjecture \ref{proprHecke}}\hfill
\smallskip  

  We make some remarks about assertions (4) and (3) of \ref{proprHecke}.

  Note   that  the equality $\tau(1)=1$   results  from   the  formula  
  $\tau(b^{-1})^\vee = \tau(b\bpi)/\tau(\bpi)$ (condition (2)(b))
  applied with $b=1$. The same formula applied with $b=\bpi^{-1}$ shows that $\tau(\bpi)$
  is  an invertible element of the Laurent polynomial ring, thus a monomial.
  Multi-homogeneity  and invariance  by $\fS_W$ will then imply the claimed
  equality in (4)(b) up  to a  constant; that constant can be checked by
  specialization. 
  
  Thus in order to prove (4)(b) (assuming (2)), we just have to prove
  multi-homogeneity.
  It is stated in \cite[p.179]{sp1} that (a) and (b) are implied by
   \cite[Ass. 4]{sp1} (which is the same as conjecture 2.6 of 
  \cite{mami}), assumption that we repeat below (Assumption \ref{assumptioniontau}).
  
  In \cite{brema} and \cite{gim}, it is shown that \ref{assumptioniontau} holds
  for all imprimitive irreducible complex reflection groups.
  
\begin{assumption}\label{assumptioniontau}\hfill

  There is a section 
  $$
    W\hookrightarrow \bB_W\,,\, w\mapsto\bw
  $$
  with image $\bW$,
  such that
  $1\in \bW$, and for $\bw\in\bW-\{1\}$ we have $\tau(\bw)=0$. 
\end{assumption}

  Let us spell out a proof of that implication.

\begin{lemma}\label{jeanontau}\hfill

  Under Assumption \ref{assumptioniontau}, properties (4)(a) and (4)(b) of Theorem 
  \ref{proprHecke} hold.
\end{lemma}

\begin{proof}\hfill

  By the homogeneity property of the character
  values (see \eg\ \cite[Prop. 7.1, (2)]{sp1}),
  applied to the grading given by
  each function $l_H$, we see that for $\chi\in\Irr(\CH(W))$ and $b\in\bB_W$, the
  value $\chi(b)$ is multi-homogeneous of degree $l_H(b)$.

  From  the definition of the Schur elements  $S_\chi$, it follows that
  $\tau(b)$  is multi-homogeneous of degree $l(b)$ if and only if
  the Schur elements $S_\chi$ are multi-homogeneous of degree $0$.

  Let $M$ be the matrix $\{\chi(\bw)\}_{\chi\in\Irr(\CH(W)),\bw\in\bC}$
  where $\bC$ is a subset of $\bW$ which consists of the lift of one representative
  of each conjugacy class of $W$. Then the equation for the Schur elements reads
  $S=X\cdot M^{-1}$ where $S$ is the vector $(1/S_\chi)_{\chi\in\Irr(\CH(W))}$
  and $X$ is the vector $(1,0,\ldots,0)$ (assuming that $\bC$ starts with $1$).
  From this
  equation, it results that the inverses of the Schur elements are the cofactors
  of the first column of $M$ divided by the determinant of $M$, which are
  multi-homogeneous of degree 0. 
  From the same equation, since $X$ is invariant by $\fS_W$, it
  results that for $\sigma\in\fS_W$, we have (see \ref{actionofS})
  $\sigma(S_\chi)=S_{\sigma(\chi)}$, which implies that $\tau$ is invariant by
  $\fS_W$.
\end{proof}

  Note that, for the above proof, we just need the existence of
  $\bC$ and not of $\bW$. 
\smallskip

  We  now turn to assertion (3) of Theorem--Conjecture \ref{proprHecke}. 
  The proof of unicity of
  $\tau$  given in \cite{sp1} using part (2) of Theorem-Assumption 1.56 
  applies to any parabolic subalgebra of the generic algebra. Hence
  assertion (3) would follow from the next assumption.
  
\begin{assumption}\label{assumptionparabolic}\hfill

  Let  $\bB_I$  be  a  parabolic  subgroup  of  $\bB_W$  corresponding  to the
  intersection  of  hyperplanes  $I$,  and  let $\bpi_I$ be the corresponding
  element  of  the  center  of  $\bB_I$.  Then  for  any $b\in\bB_I$, we have
  $\tau(b^{-1})^\vee=\tau(b\bpi_I)/\tau(\bpi_I)$. 
\end{assumption}
\smallskip

  From now on we shall assume that Theorem--Conjecture \ref{proprHecke} holds.
\smallskip

\subsubsection{The cyclic case}\hfill
\smallskip

 For what follows we refer to \cite[\S 2]{brma2}.
  
  Assume that $W= \genby{s} \subset \GL_1(\BC)$ is cyclic of order $e$. 
  Denote by $\bs$ the corresponding braid reflection in $\bB_W$.
  Let
  $\bu =  (u_i)_{i=0,\dots,e-1}$ be  a set of indeterminates.
  
  Then clearly there exists a unique symmetrizing form
  $\tau$ on the generic Hecke algebra $\CH(W)$ of $W$ 
  (an algebra over $\BZ[(u_i^{\pm 1})_{i=0,\dots,e-1}]$)
  such that
  $$
    \tau(1) = 1 \,\,\text{ and } 
    \tau(\bs^i) = 0 \text{  for } i=1,\dots,e-1
    \,.
  $$
  This is the form from \ref{proprHecke}.
  
  For $0\leq i\leq e-1$,
  let us denote by $\chi_i : \CH(W) \ra \BQ(\bu)$ the character defined
  by $\chi_i(\bs) = u_i$. We set
  $
    S_i(\bu) := S_{\chi_i}(\bu)
    \,.
  $

\begin{lemma}\label{brmazy}\hfill

  The Schur elements $S_i(\bu)$ of $W  = \genby s$ have the form
  $$
    S_i(\bu) = \prod_{j\neq i} \dfrac{u_j-u_i}{u_j} = 
    \dfrac{1}{P(0,\bu)}
    \left(
      t \dfrac{d}{dt}P(t,\bu)
    \right) |_{t=u_i}
    \,,
  $$
  where
  $
    P(t,\bu) := (t-u_0)\cdots (t-u_{e-1})
    \,.
  $
\end{lemma}

\begin{proof}\hfill
   
   The first formula is in \cite[Bem. 2.4]{brma2}.
    For the second, notice that we have
    \begin{equation*}\label{yes,P}
      \dfrac{d}{dt}P(t,\bu) =
      P(0,\bu) 
      \sum_i  \dfrac{1}{u_i} \prod_{j\neq i} \dfrac{t-u_i}{u_i-u_j} S_i(\bu)
    \,.
  \end{equation*}
\end{proof}
  
%
%
%
\smallskip

  The following Lemma will be used later. Its proof is a straightforward computation
  (see Lemma \ref{schurpalin}).
  
\begin{lemma}\label{schurtcheque}\hfill

  With the above notation, we have
  $$
    S_i(\bu)^\vee = (-1)^{e-1} u_i^{-e} ( \prod_{j=0}^{e-1} u_j )\,S_i(\bu)
         = -u_i^{-e} P(0,\bu) S_i(\bu)
    \,.
  $$
\end{lemma}
\smallskip

\subsection{$\P$-cyclotomic Hecke algebras, Rouquier blocks}\hfill
\smallskip

  We now consider specialized cyclotomic Hecke algebras involving only
  a single indeterminate, $x$.
\smallskip

  Let $K$ be a number field, stable by complex conjugation $\la\mapsto \la^*$.  
  Let $W$ be
  a finite reflection group on a $K$-vector space $V$ of dimension $r$.
  
  Let $\P(x)$  be a $K$-cyclotomic 
  polynomial --- see Definition \ref{phigroup}. 
  We assume that the roots of $\P(x)$ have order $d$, and we
  denote by $\z$ one of these roots.
  
\subsubsection{$\P$-cyclotomic Hecke algebras}\hfill
\smallskip

  We recall that we set $m_K = |\bmu(K)|$. We choose an indeterminate $v$ such that
  $v^{m_K} = \z\inv x$. 
 
\begin{definition}\label{abstractcyclo}\hfill

  \begin{enumerate}
    \item
      A cyclotomic specialization is a morphism 
      $
        \si : \BZ[(u^{\pm 1}_{H,i})_{H,i}] \ra K[v^{\pm 1}]
      $\index{cyclotomic specialization}
      defined as follows:
  
      There are 
      \begin{itemize}
        \item
          a $W$-equivariant family $(\z_{H,i})_{(H\in\CA(W))(i=0,\dots,e_H-1)}$
          of roots of unity in $K$, 
        \item
          and a $W$-equivariant family
          $(m_{H,i})_{(H\in\CA(W))(i=0,\dots,e_H-1)}$ of rational numbers, 
      \end{itemize}
      such that  
      \begin{enumerate}
        \item
          $m_K m_{H,i} \in \BZ$ for all $H$ and $i$,
        \item
          the specialization $\si$ is of the type
          $
            \si : u_{H,i} \mapsto \z_{H,i} v^{m_Km_{H,i}}
            \,.
          $
       \end{enumerate}
     \item
       A $\P$-cyclotomic Hecke algebra of $W$\index{Pcyclo@$\P$-cyclotomic Hecke algebra}
       is the algebra 
       $$
         \CH_\si := K[v^{\pm 1}]\otimes_\si \CH(W)
       $$ 
       defined by applying a cyclotomic specialization 
       $
        \si : \BZ[(u^{\pm 1}_{H,i})_{H,i}] \ra K[v^{\pm 1}]
      $
      to the base ring of the generic Hecke algebra of $W$, 
      which satisfies the following conditions:
      
       For each $H\in \CA(W)$, the polynomial
       $$ 
         P_H(\bu)(t) = \prod_{i=0}^{e_H-1} (t-u_{H,i}) \in \BZ[\bu,\bu\inv][t]
       $$
       specializes under $\si$ to a polynomial
       $P_H(t,x)$ such that
       \begin{enumerate}
         \item[(CA1)]
           $P_H(t,x) \in K[t,x]\,,$
         \item[(CA2)]
           $P_H(t,x) \equiv  t^{e_H}-1 \pmod {\Phi(x)}$.
       \end{enumerate}
   \end{enumerate}
\end{definition}

\begin{remark}\hfill

\begin{enumerate}
  \item
    It follows from Theorem \ref{rationalityhecke} that the field $K(v)$ is a splitting field for $\CH_\si$.
  \item
    Property (2),(CA2) of the preceding definition shows that 
    \begin{enumerate}
      \item
        a $\P$-cyclotomic Hecke algebra $\CH_\si$ as above specializes to the group algebra $K W$
        through the assignment $v \mapsto 1$ (which implies $x \mapsto \z$);
      \item
         the specialization of 
         $K[\bu,\bu\inv]\otimes_{\BZ[\bu,\bu\inv]}\CH(W)$
         to the group algebra $KW$ (given by $u_{H,i} \mapsto \z_{e_H}^i$ for $0\leq i\leq e_H-1$) factorizes through its
         specialization to any $\P$-cyclotomic algebra.
    \end{enumerate}
\end{enumerate}
\end{remark}
\smallskip

\subsubsection{The Rouquier ring ${R}_K(v)$}\hfill
\smallskip

\begin{definition}\label{Rouquier ring}\hfill

  \begin{enumerate}
    \item
      We call \emph{Rouquier ring}\index{Rouquier ring} of $K$ and denote by ${R}_K(v)$ the
      $\mathbb{Z}_K$-subalgebra of $K(v)$
      $$
        {R}_K(v):=\mathbb{Z}_K[v,v^{-1},(v^n-1)^{-1}_{n\geq 1}]
        \,.
      $$\index{Rouquier ring}
    \item
      Let $\si: u_{H,j} \mapsto \z_{H,j} v^{n_{H,j}}$ be a
      cyclotomic specialization defining a $\P$-cyclotomic Hecke algebra
      $\mathcal{H}_\si$.      
      The \emph{Rouquier blocks}\index{Rouquier block} of
      $\mathcal{H}_\si$ are the blocks of the algebra
      ${R}_K(v)\mathcal{H}_\si$.
  \end{enumerate}
\end{definition}

\begin{remark}\label{rouquierCRAS}\hfill  

  It has been shown by Rouquier
  (cf. \cite[Th.1]{rou}), that if $W$ is a Weyl group and $\mathcal{H}_\si$ is
  its ordinary Iwahori-Hecke algebra, then the Rouquier blocks of  $\mathcal{H}_\si$ coincide with the
  families of characters defined by Lusztig. In this sense, the Rouquier
  blocks generalize the notion of ``families of characters'' to 
  the $\P$-cyclotomic Hecke algebras of all complex reflection groups. 
\end{remark}

  Observe that the Rouquier ring $R_K(v)$ is a Dedekind domain (see \cite[\S 2.B]{brki}).

%
%
\smallskip

\subsubsection{The Schur elements of a cyclotomic Hecke algebra}\hfill
\smallskip

  In this section we assume that Conjecture \ref{proprHecke} holds.

  Let us recall the form of the Schur elements of the cyclotomic
  Hecke algebra $\mathcal{H}_\si$  \cite[Prop. 2.5]{brki}
  (see also \cite[Prop.4.3.5]{maria}).

\begin{proposition}\label{formofschur}\hfill
  
  If $\chi$ is an irreducible character of
  $K(v)\mathcal{H}_\si$, then its Schur element $S_{\chi}$ is
  of the form
  $$
    S_{\chi} = m_{\chi} v^{a_{\chi}} \prod_\Psi \Psi(v)^{n_{\chi,\Psi}}
  $$
  where $m_{\chi} \in\mathbb{Z}_K$, $a_{\chi} \in \mathbb{Z}$, 
  $\Psi$ runs over the $K$-cyclotomic polynomials
  and $(n_{\chi,\Psi})$ is a family of almost all zero elements
  of $\BN$.
\end{proposition}

  The bad prime ideals of a cyclotomic Hecke algebra
  have been defined in \cite[Def. 2.6]{brki}
  (see also \cite{maro}, and \cite[Def. 4.4.3]{maria}).

\begin{definition}\label{bad}\hfill

  A prime ideal $\mathfrak{p}$ of $\mathbb{Z}_K$ lying over a prime
  number $p$ is $\si$-bad \index{sbad@$\si$-bad prime ideal} for $W$, if there exists 
  $\chi \in \textrm{\emph{Irr}}(K(v)\mathcal{H}_\si)$ with 
  $m_{\chi} \in \mathfrak{p}$. 
  In this case, $p$ is called a $\si$-bad prime 
  number\index{sbadn@$\si$-bad prime number} for $W$.
\end{definition}

\begin{remark}\hfill

  In the case of the principal series of a split finite reductive group, that is, if $W$ is a Weyl group and $\CH_\si$
  is the usual Hecke algebra of $W$ --- the algebra which will be called below (see \ref{characspetsial})
  the {\sl ``1-cyclotomic special algebra of compact type''\/} ---, it is well kown (this goes back, at least implicitly,
  to \cite{lu1} and \cite{lu2}) that the corresponding bad prime ideals are
  the ideals generated by the bad prime numbers (in the usual  sense) for $W$.
\end{remark}
\smallskip

\subsubsection{
  Rouquier blocks, central morphisms, and the functions $a$ and $A$
}\label{rouquierblocksaA}\hfill
\smallskip

  The next two assertions have been proved in \cite[Prop. 2.8 \& 2.9]{brki}
  (see also \cite[\S 4.4.1]{maria}).

\begin{proposition}\label{Rouquier blocks and central characters}\hfill

  Let $\chi,\psi \in \Irr(K(v)\CH_\si)$. The characters $\chi$ and
  $\psi$ are in the same Rouquier block of $\mathcal{H}_\si$ if
  and only if there exist 
  \begin{itemize}
    \item
      a finite sequence
      $\chi_0,\chi_1,\ldots,\chi_n \in \Irr(K(v)\CH_\si)\,,$ 
    \item
      and a finite sequence $\mathfrak{p}_1,\ldots,\mathfrak{p}_n$ 
      of $\si$-bad prime ideals for $W$
  \end{itemize}
   such that
  \begin{enumerate}
    \item 
      $\chi_0=\chi$ and $\chi_n=\psi$,
    \item 
      for all $j$ $(1\leq j \leq n)$,\,\,
        $
          \omega_{\chi_{j-1}} \equiv \omega_{\chi_j}
          \emph{ mod } \mathfrak{p}_jR_K(v)
          \,.
        $
  \end{enumerate}
\end{proposition}

  Following the notations of \cite[\S 6B]{sp1}, for every element $P(v) \in \mathbb{C}(v)$, we call
  \begin{itemize}
    \item 
      valuation of $P(v)$ at $v$ and denote by $\val_v(P)$ the order of $P(v)$ at 0,
    \item 
      degree of $P(v)$ at $v$ and denote by $\deg_v(P)$ the negative of the
      valuation of $P(1/v)$.
  \end{itemize}
  Moreover, as $x = \z v^{m_K}$, we set
  $$
    \val_x(P(v)):=\frac{\val_v(P)}{|\bmu(K)|} 
    \quad\textrm{ and}\quad
    \deg_x(P(v)):=\frac{\deg_v(P)}{|\bmu(K)|}
    \,.
  $$ 
  For $\chi\in\Irr(K(v)\CH_\si) $, we define
    \begin{equation}\label{defaA}
      a_{\chi}:=\val_x(\dfrac{S_1(v)}{S_{\chi}(v)})
       \quad\textrm{ and }\quad
      A_{\chi}:=\deg_x(\dfrac{S_1(v)}{S_{\chi}(v)})
      \,.
   \end{equation}\index{achi@$a_\chi$}\index{Achi@$A_\chi$}

\begin{remark}\hfill

  For $W$ a Weyl group, the integers $a_\chi$ and $A_\chi$ are
  just those for the generic character corresponding to $\chi$ (see
  Notation \ref{aArho} below).
\end{remark}

  The following result is proven in \cite[Prop. 2.9]{brki}.

\begin{proposition}\label{aA}\hfill

  Let $\chi,\psi \in \Irr(K(v)\CH_\si)$. If the characters $\chi$ and
  $\psi$ are in the same Rouquier block of $\mathcal{H}_\si$, then
   $$
     a_{\chi}+A_{\chi}=a_{\psi}+A_{\psi}
     \,.
   $$
\end{proposition}

  For all Coxeter groups, Lusztig has proved 
  that if $\chi$ and $\psi$ belong to the same family, hence 
  (by Remark \ref{rouquierCRAS})
  to the same Rouquier block of the Hecke algebra, then
  $a_{\chi}=a_{\psi}$ and $A_{\chi}=A_{\psi}$. This assertion has also been 
  generalized by a case by case analysis 
  (see \cite[Prop.4.5]{brki}, \cite[Th.5.1]{maro}, \cite[Th.6.1]{maria2},
  and \cite[\S 4.4]{maria} for detailed references) to the general case.

\begin{theorem}\label{aAconstant}\hfill

  Let $W$ be a complex reflection group, and let $\CH_\si$ be a cyclotomic
  Hecke algebra associated with $W$.
  
  Whenever $\chi,\psi \in \Irr (K(v)\CH_\si)$ belong to the same Rouquier block of $\CH_\si$,
  we have
  $$
    a_{\chi}=a_{\psi} \quad\text{and}\quad A_{\chi}=A_{\psi}
    \,.
  $$
\end{theorem}  
\medskip

\newpage
{\red \section{\red Complements on finite reductive groups}\hfill
\smallskip
}

\subsection{Notation and hypothesis for finite reductive groups}
\label{bmm3}
\hfill
\smallskip

  Before proceeding our development for complex reflection groups, we collect some facts
  from the theory of finite reductive groups associated to Weyl groups. More precisely,
  we state a number of results and conjectures about Deligne-Lusztig varieties
  associated with regular elements, Frobenius eigenvalues of unipotent characters attached
  to such varieties, actions of some braid groups on these varieties, in connection with the so-called
  abelian defect group conjectures and their specific formulation in the case of finite reductive groups
  (see \cite[\S 6]{abdefconj}, and also \cite{brmi}).
  
  These results and conjectures will justify and guide most of the definitions and properties given in the
  following paragraphs about the more general situation where finite reductive groups are
  replaced by spetsial reflection cosets.
  
  Let $\bG$ be a quasi-simple connected reductive group over the algebraic 
  closure of a prime field $\BF_p$, endowed with an isogeny $F$  
  such that $G :=\bGF$ is finite.
\smallskip

  Our notation is standard:
  \begin{itemize}
      \item
        $\delta$ is the
        smallest  power  such  that  $F^\delta$  is a {\sl split Frobenius\/}
        (this exists since $\bG$ is quasi-simple). We denote by $x\mapsto x.F$
        the action of $F$ on elements or subsets of $\bG$.
    \item
      The real number $q$ ($\sqrt p$ raised to an integral power) is defined by the following 
       condition:
       $F^\delta$ defines a rational structure on $\BF_{q^\delta}$.  
    \item
      $\bT_1$ is a maximal torus of $\bG$ which is stable under $F$ and
      contained in an $F$-stable Borel subgroup of $\bG$,
      $W$ is its Weyl group.
      The action on $V := \BC\otimes_\BZ X(\bT_1)$ induced
      by  $F$ is of the form $q\vp$, where $\vp$ is an element of finite order
      of $N_{\GL(V)}(W)$ which is 1-regular. Thus $\de$  is the order  of the 
      element of 
      $N_{\GL(V)}(W)/W$ defined by $\vp$. For $w\in W$ we shall also sometimes 
      note
      $
        \vp(w) := \vp w\vp\inv = \lexp\vp w
        \,.
      $
    \end{itemize}

\smallskip
  
  We also use notation and results from previous work about the braid group of
  $W$ and the Deligne--Lusztig varieties associated to $\bG$
  (see \cite{brmi}, \cite{dm2}).
\smallskip
  
  We use freely definitions and notation introduced in \S \ref{thebraidgroup} above.
  Recall that for $a,d \in \BN$,
  $\z_d^a := \exp(2\pi i a/d)$, and let $w\vp$ be a $\z_d^a$-regular element for $W$.
  
  It is possible to choose a base point $x_0$ fixed by $\vp$ 
  in one of the fundamental chambers of $W$, which we do.
  Indeed, $\vp$ stabilizes the positive roots 
  corresponding to the $F$-stable  Borel subgroup containing $\bT_1$
  and consequently fixes their sum.
  
  Since  $\bG$  is  quasi-simple, $W$ acts irreducibly on $V$, and so
  (\cite{deligne1} or \cite{brsa})
  the  center  of the pure braid group $\bP_W$ is cyclic. We denote by
  $\bpi$ its positive generator.
\smallskip

  Following \cite[\S 3]{dm}, we choose all our paths satisfying the
  conditions of \cite[Prop. 3.5]{dm}.
  This allows us to define (\cite[Def. 3.7]{dm}) the lift 
  $\bW\subset\bB_W$\index{W@$\bW$}
  of $W$, and the monoid $\bB_W^+$\index{BW+@$\bB_W^+$} generated by $\bW$.
  We denote $w\mapsto\bw$ (resp. $\bw\mapsto w$) this lift (resp. the
  projection $\bW\ra W$).
\smallskip

  In particular, we choose a regular eigenvector $x_1$ of $w\vp$ 
  associated with the eigenvalue $\z^a_d$, 
  and a path $\ga$ in $V^\reg$ from $x_0$ to $x_1$ satisfying those conditions.

  Then, following \S \ref{liftingregular} above:

  \begin{itemize}
    \item
      As in Lemma \ref{propertyofa} and Lemma \ref{whenphireguylar},
      we lift $w\vp$ to a path $\bw_{\ga,a/d}$ 
      (abbreviated $\bw$)
      in $V^\reg$ from the base point $x_0$ to $(w\vp)(x_0) = wx_0$,
      we denote by $\bvp$ the automorphism of $\bB_W$ defined by $\vp$,
      and we have
      $
        \ba(w\vp)_{\ga,a/d} = \Ad(\bw)\cdot \ba(\vp)
        \,.
      $
    \item        
      If
      $$
         \brh := \bw\cdot \bvp(\bw)\cdot\cdots\cdot\bvp^{\de-1}(\bw)
         \,,
      $$
      we have $\brh \in \bB_W$ and $\brh^d = \bpi^{\de a}$.
   \item
       Both $\bw$ and $\brh$ belong to $\bB_W^+$.
  \end{itemize}

  We will work in the semi-direct product $\bB_W^+\rtimes\langle\bvp\rangle$
  where we have $(\bw\bvp)^\delta=\brh$ and $(\bw\bvp)^d=\bpi^a\bvp^d$. We
  denote by $\bB_W(\bw\bvp)$ \index{BbWbwp@$\bB_W(\bw\bvp)$}
  (resp. $\bB_W^+(\bw\bvp)$) the centralizer of
  $\bw\bvp$ in $\bB_W$ (resp. $\bB_W^+$).

  We denote by $\CB$ the variety of Borel subgroups of $\bG$ and by $\CT$ the
  variety of maximal tori. The orbits of 
  $\bG$ on $\CB\times\CB$ are in canonical bijection with $W$, and we denote 
  $\bB\po w\bB'$ the  fact that the pair $(\bB,\bB')$ belongs to the orbit
  parametrized by $w\in W$.

  The Deligne--Lusztig variety
  $\bX_{\bw\bvp}$\index{Xwp@$\bX_{\bw\bvp}$} is defined as in 
  \cite[D\'ef.~1.6 and \S 6]{brmi}
  (following \cite{deligne3}). It is irreducible; indeed, since $\bw\bvp$ is a root of
  $\bpi$, the decomposition of $\bw$ contains at least one reflection of each orbit
  of $\bvp$, and the irreducibility follows from \cite[8.4]{dm},
  (it is also the principal result of \cite{boro}).
  
  Note that if $\bw\in\bW$ then the Deligne--Lusztig variety $\bX_{\bw\bvp}$
  associated to the ``braid element'' $\bw\bvp$ is nothing but the classical
  Deligne--Lusztig variety
  $$
    \bX_{w\vp}=\{\bB\in\CB\mid \bB\po w\bB.F\}
  $$
  associated to the ``finite group element'' $w\vp$.
  When $a$ is prime to $\delta$ and $2a\le d$,
  by choosing for $\bw\bvp$ the $a$-th power of the lift of a
  Springer element (see \cite[3.10 and 6.5]{brmi}) 
  we may ensure that $\bw = \bw_{\gamma,a/d} \in\bW$.
\smallskip

\subsection{Deligne--Lusztig varieties attached to regular elements}
\hfill

  We generalize to regular elements a formula of Lusztig
  giving the product of Frobenius eigenvalues on
  the cohomology of a Coxeter variety. 
  
  In what follows we consider the usual Deligne--Lusztig variety $\bX_{w\vp}$
  (attached to $w\vp \in \GL(V)$).

\begin{proposition}\label{XwFd}\hfill

  Let $w\vp$ be a $\z_d$-regular element of $W\vp$.
  If $\bT_{w\vp}$ is an $F$-stable maximal torus of $\bG$
  of type $w\vp$, we have
  $$
    |\bX_{w\vp}^{F^d}|
    = \frac{|\bGF|}{|\bT^F_{w\vp}|}
    \,.
  $$
\end{proposition}

  In order to proceed, we reformulate that proposition in terms which can be generalized to
  complex reflection groups; the proof of that reformulation is then 
  immediate by Proposition \ref{GoverT}(1).

\begin{corollary}\label{XwFd2}\hfill

  Let $w\vp$ be a $\z_d$-regular element of $W\vp$.
  We have
  $$
    |\bX_{w\vp}^{F^d}|
    = \widetilde\det_V(w\vp)\inv q^{\Nh_W} \Deg R^\bGF_{w\vp}
    \,.
  $$
\end{corollary}

\begin{proof}[Proof of \ref{XwFd}]\hfill

  The proposition generalizes \cite[3.3(ii)]{cox} and our proof is
  inspired by the argument given there.
  
  We shall establish a bijection
  \begin{equation}\label{bijection}
  \begin{gathered}
    \bX_{w\vp}^{F^d} = 
        \left\{
          \bB\in\CB \,\mid \, \bB \po{w} \bB.F \,,\, \bB.F^d = \bB
        \right\} \\
    \tilde\longleftrightarrow \\
    \left\{
       (\bT,\bB)\in \CT\times\CB \,\mid \, \bT.F = \bT \,,\, \bT\subset \bB\,,\, \bB \po{w} \bB.F
     \right\} \,,
  \end{gathered}
  \end{equation}
  and the result will follow from counting the above set of pairs.
  
  Let us check first that last implication.
  Notice that any maximal $F$-stable torus $\bT$ which is contained in a
  Borel subgroup $\bB$ such that $\bB \po{w} \bB.F$ is a torus of type
  $w\vp$ (see \cite[\S 1]{delu}). Since all these tori of type $w\vp$ are $\bGF$-conjugate
  to $\bT$, their number is
  $$
    \dfrac{|\bGF|}{|N_{\bGF}(\bT)|}=\dfrac{1}{|W(w\vp)|}\dfrac{ |\bGF|}{|\bTF|}
    \,.
  $$
  To find the number of pairs $(\bT,\bB)$ as above, we must multiply the above
  number by the number of Borel subgroups $\bB'$ such that $\bT\subset \bB'$ and
  $\bB'\po w \bB'.F$. Given $\bB$ such that $\bT\subset\bB$ and $\bB \po{w} \bB.F$,
  any other $\bB'$ containing $\bT$ is $\bB' = \lexp{w'} \bB$ for some 
  $w' \in W = N_{\bG}(\bT)/\bT$, 
  and $\bB \po{w} \bB.F$ if and only if $w'\in W(w\vp)$.
  Thus the number of pairs $(\bT,\bB)$ as above is $\frac{|\bGF|}{|\bT^F_{w\vp}|}$.
\smallskip

  Now we establish the bijection (\ref{bijection}).

  1. We first show that, whenever $\bB \in \bX_{w\vp}^{F^d}$, there exists a
  unique $F$-stable maximal torus $\bT$ such that, for all $i\geq 0$,
  $\bT \subset \bB.F^i$. 
\smallskip

  Let $\bB\in \bX_{w\vp}$, so that $\bB\xrightarrow w \bB.F)$. 
  The sequence
  $$
    (\bB, \bB.F,\ldots,\bB.F^d)
  $$ 
  defines an element of the variety
  $\bX_{\bpi\bvp^d}$ (associated with the ``braid element $\bpi\bvp^d$'', and
  relative to $F^d$) (see \cite[\S 1]{brmi}).
  Let $\bB_1$ be the unique Borel subgroup
  such that 
  $$
    \bB\xrightarrow{w_0} \bB_1\xrightarrow{w_0}\bB.F^d
  $$
  where $w_0$ is the longest element of $W$,
  and such
  that $(\bB,\bB_1,\bB.F^d)$ defines the same element of $\bX_{\bpi\bvp^d}$ as
  that sequence. Since $\bB$ and $\bB_1$ are opposed, $\bT :=\bB\cap\bB_1$
  is a maximal torus.
  
  Let us prove that, for all $i\geq 0$, $\bT \subset \bB.F^i$.

\begin{itemize}
  \item 
    Assume first that  $i\le\lfloor d/2\rfloor$.
    
    If $v := wF(w)\ldots F^{i-1}(w)$, then
    the Borel subgroup $\bB.F^i$ is the unique
    Borel subgroup $\bB'$ such that
    $$
      \bB\xrightarrow{v}\bB'\xrightarrow {v\inv w_0}\bB_1
       \,.
    $$ 
    Since such a Borel
    subgroup  can be found  among those containing  $\bT$, 
    \ie\ a Borel subgroup $\lexp{w'}\bB$ for some $w'\in W = N_\bG(\bT)/\bT$,
    we have
    $\bB.F^i \supset\bT$.
  \item
    Now assume $i\ge\lceil d/2\rceil$. Since $\bB=\bB.F^d$, if we set now
    $v := F^i(w)F^{i+1}(w)\ldots  F^{d-1}(w)$,
    we get similarly that the
    Borel subgroup $\bB.F^i$ is the unique Borel subgroup $\bB'$ such that
    $$
      \bB_1  \xrightarrow  {w_0  v\inv}\bB'\xrightarrow{v}\bB
    $$
     and by  the  same  reasoning we get that
    $\bB.F^i\supset\bT$.
\end{itemize}
\smallskip

  Since $\bT = \bigcap_i \bB.F^i$, it is clear that $\bT$ is $F$-stable.
\smallskip

  2.  Conversely, if $(\bT,\bB)$ is a pair such that
  $\bT$ is an  $F$-stable maximal torus  of  type $w\vp$  and  
  $\bB\supset\bT$, then
  $\bB\xrightarrow w \bB.F$ and $\bB\in \bX_{w\vp}^{F^d}$. Indeed, let as
  above $(\bB,\bB_1,\bB.F^d)$ define the same element of $\bX_{\bpi\bvp^d}$
  as $(\bB, \bB.F,\ldots,\bB.F^d)$. Then each $\bB.F^i$ contains $\bT$
  and by a reasoning similar as above, as either $\bB_1=\bB.F^{d/2}$ if $d$
  is even, or $\bB_1$ is defined by its relative position to its neighbours,
  in each case $\bB_1$ contains $\bT$. As both $\bB$ and $\bB.F^d$ are the
  unique Borel subgroup which intersect $\bB_1$ in $\bT$, they coincide.
\end{proof}
\smallskip

\subsection{On eigenvalues of Frobenius}\hfill

\subsubsection{The Poincar\'e duality}\index{Poincar\'e duality}\hfill
\smallskip

  Let us briefly recall a useful consequence of Poincar\'e duality 
  (see for example \cite{deligne2}) in our context.

  Here $q$ is a power of a prime number $p$, and $\ell$ is a prime number
  different from $p$.
  
  Let $X$ be a smooth separated irreducible variety of dimension $d$, defined over the field
  $\BF_q$ with $q$ elements, with corresponding Frobenius endomorphism $F$.
   Its \'etale cohomology groups $H^i(X,\BQ_\ell)$ and
  $H^i_c(X,\BQ_\ell)$ are naturally endowed with an action of $F$.
  
  The \emph{Poincar\'e duality} has the following consequence.
 
\begin{theorem}\label{poincare}\hfill

    For $0\leq i\leq 2d$, there is a natural non degenerate pairing of
    $\genby{F}$-modules
    $$
      H^i_c(X,\BQ_\ell) \times H^{2d-i}(X,\BQ_\ell) \longrightarrow H^{2d}_c(X,\BQ_\ell)
      \,.
   $$
\end{theorem}

%
  For $n \in \BZ$ we denote
  by $\BQ_\ell(n)$ the $\BQ_\ell$-vector space of dimension 1 where we let
  $F$ act by multiplication by $q^{-n}$.

  Since, as a $\BQ_\ell\genby F$-module, we have
  $
    H^{2d}_c(X,\BQ_\ell) \cong \BQ_\ell(-d)
    \,,
  $
  the Poin\-car\'e duality may be reformulated as follows:
  
  For $0\leq i\leq 2d$, there is a natural non degenerate pairing of
    $\genby{F}$-modules
    $$
      H^i_c(X,\BQ_\ell) \times H^{2d-i}(X,\BQ_\ell) \longrightarrow \BQ_\ell(-d)
      \,.
   $$ 

\subsubsection{Unipotent characters in position ${\bw\bvp}$}\hfill
\smallskip

  In this subsection and the following until \ref{wrootofpi}, $\bw$ will be
  any element of $\bB_W^+$ such that the variety $\bX_{\bw\vp}$ is irreducible.

  We will denote by $\Un(\bG^F)$ the set of unipotent characters of $\bG^F$
  and by $\Un(\bGF,{\bw\bvp})$ those appearing in any of the
  (compact support) cohomology spaces $H^n_c(\bX_{{\bw\bvp}},\Qlbar)$.
  
  We will denote by $\Id$ the trivial character of $\bG^F$, and by
  $\conj\ga$ the complex conjugate of a character $\ga$.
\smallskip
  
  Since the dimension of $\bX_{\bw\bvp}$ is equal to $l(\bw\bvp)$,
  the Poincar\'e duality as stated above (\ref{poincare}) has the following consequences.

\begin{proposition}\label{consepoinca}\hfill

  \begin{enumerate}
    \item
      The set of unipotent  characters  appearing  in some
      (noncompact support) cohomology spaces $H^n(\bX_{{\bw\bvp}},\Qlbar)$
      is $\conj{\Un(\bGF,{\bw\bvp})}$, the set of all complex conjugates of elements
      of $\Un(\bGF,{\bw\bvp})\,.$
   \item   
       For  $\gamma\in\Un(\bGF,{\bw\bvp})$, to any eigenvalue
       $\lambda$ of  $F^\delta$ on the $\gamma$-isotypic component   of
       $H^i_c(\bX_{{\bw\bvp}},\Qlbar)$ is associated the eigenvalue 
       $q^{\delta l(\bw\bvp)}/\lambda$   on  the  
       $\conj\gamma$-isotypic  component  of
       $H^{2l(\bw\bvp)-i}(\bX,\Qlbar)$.
  \end{enumerate}
\end{proposition}
 
\begin{remark}\label{forgottenproof}\hfill

  Note that $1$ is the unique eigenvalue of minimal module of $F^\de$ on
  $H^\bullet(\bX_{{\bw\bvp}},\Qlbar)$, corresponding to the case where $\gamma$ 
  is the trivial character in $H^0(\bX_{\bw\bvp},\Qlbar)$. 
  
  Indeed, by Poincar\'e duality it suffices to check that there is a unique
  eigenvalue of maximal module, equal to $q^{l(w\vp)}$, in the compact support cohomology
  $H^\bullet_c(\bX_{\bw\bvp},\Qlbar)$. This follows from the fact that only the identity
  occurs in $H_c^{2l(w\vp)}(\bX_{\bw\bvp},\Qlbar)$ since $\bX_{\bw\bvp}$ is irreducible, 
  and that the modules of the eigenvalues in $H^n_c(\bX_{\bw\bvp},\Qlbar)$ are
  less or equal to $q^{n/2}$ (see for example \cite[3.3.10(i)]{dmr}).
\end{remark}
  
  The next properties are consequences of results of Lusztig and of Digne--Michel 
  (for the next one, see \eg\ \cite[3.3.4]{dmr}).

\begin{proposition}\label{eigenoffrob}\hfill

  Let $\gamma\in\Un(\bGF,{\bw\bvp})$.
  Let $\la$ be the eigenvalue of $F^\delta$ on the $\gamma$-isotypic component   of
  $H^i_c(\bX_{\bw\bvp},\Qlbar)$. Then 
  $
     \la = q^{f\delta}\lambda_\gamma
     \,,
  $
  where
  \begin{itemize}
    \item
      $\lambda_\gamma$  is a  root of  unity independent  of $i$  and of ${\bw\bvp}$
    \item
      $f = \frac n 2$ for some $n \in \BN$, and the image of $f$
      in $\BQ/\BZ$ is  independent  of $i$  and of ${\bw\bvp}$.
  \end{itemize}
\end{proposition}
\smallskip

\subsection{Computing numbers of rational points}  \hfill
\smallskip

  We denote by $\CH(W,x)$ the ordinary Iwahori-Hecke algebra of $W$ defined over
  $\BC[x,x\inv]$: using our previous notation, $\CH(W,x)$ is the $(x-1)$-cyclotomic Hecke algebra
  such that, for all reflecting hyperplanes $H$ of $W$, we have
  $
    P_H(t,x) = (t-x)(t+1)
    \,.
  $
  Notice that $\CH(W,x)$ is indeed an $(x-1)$-cyclotomic Hecke algebra for $W$ at the
  regular element 1.
   
  We choose once for all a square root $\sqrt{x}$ of the indeterminate $x$.
  Since the algebra $\CH(W,x)$ is split over $\BC(\sqrt x)$, the specialisations
  $$
    \sqrt x \mapsto 1
    \,\,\text{ and }\,\,
    \sqrt x \mapsto (q^{\de/2})^{m/\de}
  $$
  define bijections between (absolutely) irreducible characters:
  $$
    \Irr(\CH(W,x)) \longleftrightarrow \Irr(W)
    \,\,\text{ and }\,\,
    \Irr(\CH(W,x)) \longleftrightarrow \Irr (\CH(W,q^m))
  $$
  for all $m$ multiple of $\de$. As a consequence, we get well-defined bijections
  $$
    \left\{
      \begin{aligned}
        &\Irr(W) \longleftrightarrow  \Irr (\CH(W,q^m))\,, \\
        &\chi \mapsto \chi_{q^m}\,.
      \end{aligned}
     \right.
  $$
  
  The automorphism $\bvp$ of $\bB_W$ defined by $\vp$ induces an automorphism
  $\vp_x$ of the generic Hecke algebra, and the field $\BC(\sqrt x)$ splits the
  semidirect product algebra $\CH(W,x)\rtimes \genby{\vp_x}$ (see \cite{maR}). 
  Hence the above
  bijections extend to bijections 
  $$
    \psi\mapsto \psi_{q^m}
  $$
  between
  \begin{itemize}
    \item[--]
      extensions to $W\rtimes \genby{\vp}$ of $\vp$-stable
      characters of $W$
   \item[--]
      and
      extensions to $\CH(W,q^m)\rtimes \genby{\vp_{q^m}}$ of $\vp_{q^m}$-stable
      characters of $\CH(W,q^m)$.
  \end{itemize}
  
   Each character $\chi$ of the Hecke algebra $\CH(W,x)$ defines, by composition with
   the natural morphism $\bB_W \ra \CH(W,x)^\times$, a character of $\bB_W$.
   If $\tilde\chi$ is a character of $\CH(W,x)\rtimes \genby{\vp_x}$, it also defines a
   character of $\bB_W\rtimes\genby\bvp$~; in particular this gives a meaning
   to the expression $\tilde\chi(\bw\bvp)$ for $\bw\in\bB_W$.
  
  For $\chi$ a $\vp$-stable character of $W$ we choose an extension 
  to $W\rtimes \genby{\vp}$ denoted $\tilde\chi$.
  We define
  $$
    R_{\tilde\chi}:=|W|^{-1}\sum_{v\in W}\tilde\chi(v\vp)R_{v\vp}
    \,,
  $$
  where, for $g\in \bGF$ and $v\in W$,
  $$
    R_{v\vp}(g) := \sum_i(-1)^i\Trace(g\mid H^i_c(\bX_{v\vp},\Qlbar))
    \,.
  $$
  Notice then that, for $\ga\in\Un(\bGF,{\bw\bvp})$ the expression
  $$
    \langle \gamma,R_{\tilde\chi} \rangle_\bGF \tilde\chi_{q^m}(\bw\bvp)
  $$
  depends only on $\chi$ and on $\ga$, which gives sense to the next result
  (\cite[3.3.7]{dmr}).
    
\begin{proposition}\label{asymptofixedpoints}\hfill

  For any $m$ multiple of $\delta$ and $g\in \bGF$, we have
  $$
    |\bX_{\bw\bvp}^{gF^m}|
    =
      \sum_{\gamma\in\Un(\bGF,{\bw\bvp})}
          \lambda_\gamma^{m/\delta}\gamma(g)
      \sum_{\chi\in\Irr(W)^\vp}
         \langle \gamma,R_{\tilde\chi} \rangle_\bGF \tilde\chi_{q^m}(\bw\bvp)
     \,.
  $$
\end{proposition}

  Let us draw some consequences of the last proposition when $\bw$ satisfies
  the assumptions of Section \ref{bmm3} (so that $\bw = \bw_{\ga,a/d}$).

  Since by assumption we have $(\bw\bvp)^d = \bpi^a\vp^d$, it follows that
   \begin{equation}\label{wrootofpi}
    \bw^{\lcm(d,\de)} = \bpi^{a.\lcm(d,\de)/d}
    \,.
  \end{equation}
  By \cite[6.7]{sp1}, we know that
  \begin{equation}\label{chidepi}
    \chi_{q^m}(\bpi) = q^{m(l(\bpi)- (a_\chi+A_\chi))}
      = q^{ml(\bpi)(1- (a_\chi+A_\chi)/l(\bpi))}
  \end{equation}
  where, as usual, $a_\chi$ and $A_\chi$ are the
  valuation and the degree of the generic degree of $\chi$
  (see Section \ref{rouquierblocksaA}).
  It follows that 
  \begin{equation}\label{chidewphi}
    \tilde\chi_{q^m}(\bw\bvp) 
    =
    \tilde\chi(w\vp) q^{m l({\bw})(1- (a_\chi+A_\chi)/l(\bpi))}
  \end{equation}
  (the power of $q$ is given by the above equation and the
  constant in front by specialization).
  
  For all $\chi$ such that
  $
    \langle \gamma, R_{\tilde\chi}\rangle_\bGF\ne 0
    \,
  $
  since the functions $a$ and $A$ are constant on families, we have
  $a_\chi=a_\gamma$ and $A_\chi=A_\gamma$. 
  So 
  \begin{equation}\label{eq}
     |\bX_{\bw\bvp}^{gF^m}| 
     =
     \sum_{\gamma\in\Un(\bGF,{\bw\bvp})}
       \langle\gamma, R_{\bw\bvp}\rangle_\bGF
       \lambda_\gamma^{m/\delta}\gamma(g) q^{m l(\bw\bvp)(1- (a_\chi+A_\chi)/l(\bpi))}.
  \end{equation}

\subsection{Some consequences of abelian defect group conjectures}\hfill
\smallskip
 
  The next conjectures are special cases of 
  abelian defect group conjectures\index{Abelian defect group conjectures}
  for finite reductive groups (see for example \cite{boston}).

\begin{conjecture}\label{disjointness}\hfill

  \begin{enumerate}
    \item 
      $H_c^{\mathrm{odd}} := \bigoplus_{i\,\mathrm{odd}} \, H^i_c(\bX_{\bw\bvp},\Qlbar)$ and
      $\, H_c^{\mathrm{even}} := \bigoplus_{i\,\mathrm{even}} \, H^i_c(\bX_{\bw\bvp},\Qlbar)\!\!$ 
      are disjoint as $\bGF$-modules.
    \item 
      $F^{\delta}$ is semi-simple on $H^i_c(\bX_{\bw\bvp},\Qlbar)\,$
      for all $i\geq 0\,.$
  \end{enumerate}
\end{conjecture}

  Let us then set
  $$
    \CH_c(\bX_{\bw\bvp}) := \End_\bGF \left(
                                                  \bigoplus_i H^i_c(\bX_{\bw\bvp},\Qlbar)
                                                                      \right)
    \,.
  $$
  Comparing  the Lefschetz formula 
  $$
    |\bX_{\bw\bvp}^{gF^m}| 
    =
    \sum_i(-1)^i\Trace(gF^m\mid H^i_c(\bX_{\bw\bvp},\Qlbar))
  $$  
  with (\ref{eq}), we see that the preceding Conjecture \ref{disjointness} implies
  \begin{enumerate}
    \item
      there is a single   eigenvalue   
      $
        \Fr_\gamma^c := \lambda_\gamma q^{\delta l(\bw\bvp)(1-(a_\ga+A_\ga)/l(\bpi))}
      $ 
      of  $F^\delta$  on  the $\gamma$-isotypic  part  of  
      $\bigoplus_i  H^i_c(\bX_{\bw\bvp},\Qlbar)\,,$
    \item
      $F^\delta$ is central in  $ \CH_c(\bX_{\bw\bvp}) \,.$
  \end{enumerate}

  Since $|\bX_{\bw\bvp}^{gF^m}|\in\BQ$, the conjecture implies also:
  $$
    \forall \sigma \in \Gal(\bar\BQ/\BQ)\,,\,
    \left\{
      \begin{aligned}
        &\sigma(\Un(\bGF,{\bw\bvp}))=\Un(\bGF,{\bw\bvp}) \\
        &\Fr^c_{\sigma(\gamma)}=\sigma(\Fr^c_\gamma)
        \,.
      \end{aligned}
    \right.
  $$

  By Poincar\'e duality, there is similarly a single eigenvalue $\Fr_\gamma$ 
  of $F^\delta$ attached to $\gamma$ on  $H^\bullet(\bX_{\bw\bvp},\Qlbar)$
  and since by Proposition \ref{consepoinca}(2)
  $
    \Fr_\gamma \Fr_{\conj\gamma}^c=q^{\delta l(\bw\bvp)}
  $
  we get 
  \begin{equation}\label{FrGamma}
    \Fr_\gamma=\lambda_\gamma q^{\delta l(\bw\bvp)(a_\ga+A_\ga)/l(\bpi)}.
  \end{equation}
  We get similarly that 
  $$
    F^\delta 
    \,\text{ is central in }\, 
    \CH(\bX_{\bw\bvp}) := \End_\bGF \left ( \bigoplus_i H^i(\bX_{\bw\bvp},\Qlbar) \right)
    \,.
  $$

  The next conjecture may be found in \cite{brma2} (see also \cite{brmi}, \cite{boston}).
  
  As in \S 2.4 above, we denote by $w\vp$ a $\z_d^a$-regular element for $W$,
  which we lift as in \S 3.1 above to an element $\bw\bvp$ such that $\bw$ and
  $\brh := (\bw\bvp)^\de$ belong to $\bB^+_W$, and $\brh^d = \bpi^{a\de}$.
  We denote by $\P$ the $d$-th cyclotomic polynomial (thus $\P \in \BZ[x]$).
  
\begin{conjecture}\label{commutingcyclotomic}\hfill

  The algebra $\CH_c(\bX_{\bw\bvp})$
   is the specialization at $x=q$
   of a $\P$-cyclotomic Hecke algebra $\CH_{W}(w\vp)_c(x)$ for $W(w\vp)$ over $\BQ$,
   with the following properties:
   \begin{enumerate}
     \item
       Let us denote by $\tau_q$ the corresponding specialization of the canonical
       trace of $\CH_{W}(w\vp)_c(x)$ to the algebra $\CH_c(\bX_{\bw\bvp})$. Then
       for $\bv \in \bB_{W(w\vp)}$, we have
       $$
         \sum_i (-1)^i \tr(\bv,H^i(\bX_{w\vp},\BQl))
          =  \Deg(R^\bG_{w\vp}) \tau_q(\bv)
          \,.
        $$
     \item
       The central element 
       $\bpi^{a.\lcm(d,\de)/d} = \brh^{\lcm(d,\de)/\de}$ corresponds to the action of $F^{\lcm(d,\de)}$.
  \end{enumerate}
\end{conjecture}

  Let us draw some consequences of Conjecture \ref{commutingcyclotomic}.

\subsubsection{Consequence of \ref{commutingcyclotomic}: Computation of $\tau_q(\bpi)$}\hfill
\smallskip

  The following proposition makes Conjecture \ref{commutingcyclotomic}
  more precise, and justifies Conjecture \ref{conj2} below.

\begin{proposition}\label{A}\hfill

  \begin{enumerate}
    \item
      Assume that \ref{commutingcyclotomic} holds. Assume moreover that $a= 1$ and that $d$ is a
      multiple of $\de$. Then
      $$
         \tau_q(\bpi) = \widetilde\det_V(w\vp)\inv q^{\Nh_W} = (\z\inv q)^{\Nh_W} 
        \,.
      $$
    \item
      Assume that, for $I \in \CA_W(w\vp)$, the minimal polynomial of
      $\bs_I$ on $\CH_c(\bX_{w\vp})$ is $P_I(t)$. Then
      $$
         \prod_{I\in\CA_W(w\vp)}\!\!\!\!\!\!\!P_I(0) 
         = (-1)^{\Nh_{W(w\vp)}} \widetilde\det_V(w\vp)\inv q^{\Nh_W}
         = (-1)^{\Nh_{W(w\vp)}} (\z\inv q)^{\Nh_W} 
         \,.
      $$
  \end{enumerate}
 \end{proposition}

\begin{proof}\hfill

  (1)
  By  \ref{commutingcyclotomic}(2), the element $\bpi$ acts
  as $F^d$ on the algebra $\CH_c(\bX_{w\vp})$.
  Hence by the Lefschetz formula, we have 
  $$
    |\bX^{F^d}_{w\vp}|
    =
    \sum_i (-1)^i \tr(\bpi,H^i(\bX_{w\vp},\BQl))
    \,,
  $$
  and hence by \ref{commutingcyclotomic} we find
  $$
    |\bX^{F^d}_{w\vp}| =  \Deg(R^\bG_{w\vp})  \tau_q(\bpi)
    \,.
  $$
  Now by \ref{XwFd}, we have
  $$
     |\bX^{F^d}_{w\vp}|  = \widetilde\det_V(w\vp)\inv q^{\Nh_W} \Deg(R^\bG_{w\vp}) 
     \,,
  $$
  which implies the formula.
\smallskip

  (2)
  By \cite[2.1(2)(b)]{sp1}, we know that
  $$
    \tau_q(\bpi) = (-1)^{\Nh_{W(w\vp)}}  \prod_{I\in\CA_W(w\vp)} P_I(0)
    \,,
  $$
  which implies the result.
\end{proof}
\smallskip

\subsubsection{Consequence of \ref{commutingcyclotomic}: Computation of Frobenius eigenvalues}
  \label{compufrob}\hfill
\smallskip
  
  Recall that there is an extension $L(v)$ of $\BQ(x)$ which splits the algebra 
  $\CH_W(w\vp)_c(x)$, where $L$ is an abelian extension of $\BQ$ and $v^k = \z\inv x$
  for some integer $k$.
  
  Choose a complex number $(\z\inv q)^{1/k}$.

  Assume Conjecture \ref{commutingcyclotomic}
  holds. 
  Then all unipotent characters in $\Un(\bGF,\bw\bvp)$
  are defined over $L[(\z\inv q)^{1/k}]$,        
  and the specializations
  $$
    v \mapsto 1 \,\,\text{ and }\,\, v \mapsto (\z\inv q)^{1/k}
  $$
  define bijections
  \begin{align}\label{bijectioncharac}
    \left\{
      \begin{aligned}
        &\Irr(W(w\vp)) \longleftrightarrow \Irr(\CH(\bX_{w\vp},\Qlbar))
	\longleftrightarrow  \Un(\bGF,{\bw\bvp}) \\
        &\chi \mapsto \chi_q\mapsto \gamma_\chi
      \end{aligned}
   \right.
  \end{align}

\begin{remark}\hfill

  It is known from the work of Lusztig (see \eg\ \cite{geck2} and the references therein)  
  that the unipotent 
  characters of $\bG^F$ are defined over an extension of
  the form $L(q^{1/2})$ where $L$ is an abelian extension of $\BQ$.
\end{remark}

  Recall that $\Fr_{\gamma_\chi}$ denotes the eigenvalue of $F^\de$ on the
  $\ga_\chi$-isotypic component of $H^{\bullet}(\bX_{\bw\bvp},\Qlbar)$. 
  By \ref{commutingcyclotomic}(2), we have
  $$
    \Fr_{\gamma_\chi}^{\lcm(d,\de)/\de}
    = \omega_{\chi_q}(\brh)^{\lcm(d,\de)/\de}
    \,.
  $$
  Since the algebra $\CH_W(w\vp)_c(x)$ specializes
  \begin{itemize}
    \item
      for $x = \z$, to the group algebra
      of $W(w\vp)$, 
    \item
      and for $x = q$, to the algebra $\CH_c(\bX_{\bw\bvp})$,
  \end{itemize}
  we also have
  $$
    \omega_{\chi_q}(\brh)
    = \omega_\chi(\rho)(\zeta\inv q)^{c_\chi}
  $$ 
  for some $c_\chi\in\BN/2$.
 
  Comparing with formula \ref{FrGamma} we find
  $$
    c_\chi = \delta l(\bw\bvp)(1-(a_{\ga_\chi}+A_{\ga_\chi})/l(\bpi))
                = (e_W -(a_{\ga_\chi}+A_{\ga_\chi}))\de a/d 
    \,,
  $$
  which proves:

\begin{proposition}\label{valueoffrobenius}\hfill

  Whenever $w\vp$ is $\z_d^a$-regular and
  $\chi \in\Irr(W(w\vp))$, if $\ga_\chi$ denotes the corresponding
  element in $\Un(\bGF,\bw\bvp)$ we have
  $$
    \lambda_{\gamma_\chi}^{\lcm(d,\de)/\de} =
      \omega_\chi(\rho)^{\lcm(d,\de)/\de}
      \zeta^{-(e_W -(a_{\ga_\chi}+A_{\ga_\chi})) \lcm(d,\de) a/d}
    \,.
  $$
\end{proposition}
\smallskip
  
\subsection{Actions of some braids}\hfill
\smallskip
 
  Now we turn to the equivalences of \'etale sites defined in \cite{brmi} and studied
  also in \cite{dm2}. For the definition of the operators $D_\bv$ we refer the reader to
  \cite{brmi}.
%

\begin{theorem} \label{Mw}\hfill
  
  There is a morphism 
  $$
    \left\{
    \begin{aligned}
      &\bB_W^+(\bw\bvp)\to\End_\bGF(\bX_{\bw\bvp}) \\
      &\bv\mapsto D_\bv
    \end{aligned}
    \right.
  $$
  such that:
  \begin{enumerate}
    \item 
      The operators $D_\bv$ are equivalences of \'etale sites on
      $\bX_{\bw\bvp}$.
 \smallskip
  
 \textsl{The next assertion has only been proved for the cases where 
 	    $W$ is of type $A, B$ or $D_4$ \cite{dm}. It is conjectural in the 
	    general case.}
\smallskip

    \item  
      The map
      $\bv \mapsto D_\bv$ induces representations
      $$
        \rho_c : \Qlbar \bB_W(\bw\bvp)\to\CH_c(\bX_{\bw\bvp})
         \,\,\text{ and }\,\,
       \rho :  \Qlbar \bB_W(\bw\bvp)\to\CH(\bX_{\bw\bvp})
         \,.
      $$
    \item 
      $D_{({\bw\bvp})^\delta}=F^\delta$.
  \end{enumerate}
\end{theorem}
\smallskip

\subsubsection{More precise conjectures}\hfill
\smallskip

  The next conjecture is also part of abelian defect group conjectures 
  for finite reductive groups (see for example \cite{boston}). It makes
  conjecture \ref{commutingcyclotomic} much more precise.
  
  Results similar to (CS) are proved in \cite{dmr}  and \cite{dm} for several cases.
  
\begin{conjecture}\label{conj2}\hfill
\smallskip
 
  \textsc{Compact support Conjecture (cs)}
  \begin{enumerate}
    \item
      The morphism $\rho_c : \Qlbar \bB_W({\bw\bvp})\to\CH_c(\bX_{\bw\bvp})$
       is surjective.
    \item
      It induces an isomorphism between  $\CH_c(\bX_{\bw\bvp})$ 
      and the specialization at $x=q$ of
      a $\Phi$-cyclotomic Hecke algebra $\CH_{W}(w\vp)_c(x)$ over $\BQ$
      at $w\vp$ for $W(w\vp)$.
  \end{enumerate}
  
  \textsc{Noncompact support Conjecture (ncs)}
  \begin{enumerate}
    \item
      The morphism $\rho : \Qlbar \bB_W({\bw\bvp}) \to \CH(\bX_{\bw\bvp})$
       is surjective.
    \item
      It turns $\CH(\bX_{\bw\bvp})$ into the specialization at $x=q$ of
      a $\Phi$-cycloto\-mic Hecke algebra $\CH_W(w\vp)(x)$ over $\BQ$
      at $w\vp$ for $W(w\vp)$.
  \end{enumerate}
\end{conjecture}
\smallskip

\subsubsection{Consequence of \ref{conj2}: noncompact support and characters in $\Un(\bGF,\bw\bvp)$}\hfill
\smallskip

%
%

  We refer the reader to notation introduced in \ref{compufrob}, in particular to
  bijections \ref{bijectioncharac}.

  We can establish some evidence towards identifying $\CH(\bX_{\bw\bvp})$
  with a cyclotomic Hecke algebra of noncompact type.
  For example, we have the following lemma.
  We denote by $\CA(w\vp)$ the set of reflecting hyperplanes of 
  $W(w\vp)$ in its action on $V(w\vp)$.
  
\begin{lemma}\label{red1}\hfill

  Assume that $\CH(\bX_{\bw\bvp})$
  is the specialization at $x=q$ of a cyclotomic Hecke algebra for the
  group $W(w\vp)$.
  
  Then whenever $I \in \CA(w\vp)$, the corresponding polynomial $P_I(t,x)$
  has only one root of degree 0 in $x$, namely 1. 
  If $\bs_I$ is the corresponding braid reflection
  in $\bB_W(\bw\bvp)$, this root is the eigenvalue of $D_{\bs_I}$ on
  $H^0(\bX_{\bw\bvp},\Qlbar)$.
\end{lemma}

\begin{proof}\hfill

  Let $\chi \in \Irr(W(w\vp))$ be a linear character, thus corresponding \via\  \ref{bijectioncharac}
  to a linear character $\chi_q$ of $\CH(\bX_{\bw\bvp})$. 
  Let $I \in \CA(w\vp)$, and let $\bs_I$ be the corresponding braid reflection
  in $\bB_W(\bw\bvp)$. Let us set
  $
     u_{I,j_I} :=  \chi_q(\bs_I)
     \,
  $
  so that
  $u_{I,j_I} = \xi_{I,j_I} q^{m_{I,j_I}}
  $
  where $\xi_{I,j_I}$ is a root of unity and $m_{I,j_I}$ is a rational number.
  
  The element $\brh = ({\bw\bvp})^\delta$ is central in $\bB_W(\bw\bvp)$.
  Since $W(w\vp)$ is irreducible (see for example \cite[Th.~5.6, 6]{berkeley})
  it follows that there exists some $a \in \BQ$ such that
  $$
    \chi_q(\brh)
    =
    \chi(\rho) \prod_I u_{I,j_I}^{a e_I}
    \,.
  $$
  
  Now $\chi_q(\brh)$
  is the eigenvalue of $F^\delta$ on the $\gamma_{\chi}$-isotypic component of
  $H^\bullet(\bX_{\bw\bvp},\Qlbar)$, and we know (see Remark \ref{forgottenproof}
  above) that there is a unique such eigenvalue of minimum module,
  which is $1$, corresponding to the case where $\gamma$ is the trivial character in
  $H^0(\bX_{\bw\bvp},\Qlbar)$. 
      
  It follows that there is a unique
  linear character $\chi$ of $W(w\vp)$ such that $\chi_q(\bs)$ has minimal
  module for each braid reflection $\bs$. Since $D_\bs$ acts trivially on
  $H^0(\bX_{\bw\bvp},\Qlbar)$ we have $\chi_q(\bs)=1$, so the unique
  minimal $m_{I,j_I}$ is $0$.
\end{proof}
\smallskip
   
\subsection{Is there a stronger Poincar\'e duality~?}\hfill
\smallskip

  We will see now how (CS) and (NCS) are connected, under some
  conjectural extension of Poincar\'e duality. 

\begin{conjecture}\label{conj3}\hfill

  For  any  $\bv\in  \bB_W(\bw\bvp)$  and  any  $n$  large  enough multiple of
  $\delta$, Poincar\'e duality holds for $D_{\bv(\bw\vp)^n}$,  \ie\ we have a 
  perfect pairing of $D_{\bv(\bw\vp)^n}$-modules:
  $$
    H^i_c(\bX_{\bw\bvp},\Qlbar)\times H^{2l(\bw\bvp)-i}(\bX_{\bw\bvp},\Qlbar)
    \to
    H^{2l(\bw\bvp)}_c(\bX_{\bw\bvp},\Qlbar) \,.
  $$
\end{conjecture}

\begin{remarks}\label{fujiwara}\hfill
  
  1. First note that for $n$ large enough $\bv(\bw\vp)^n\in
  \bB_W^+(\bw\vp)$, so there is a well-defined endomorphism
  $D_{\bv(\bw\vp)^n}$.  Indeed, since $(\bw\vp)^n$ is  a power of $\bpi$ for
  $n$  divisible enough, the element $\bv({\bw\bvp})^n$ is positive for $n$
  large enough.

  2.  The  Lefschetz formula which would be implied by \ref{conj3} at least
  holds.  Indeed, Fujiwara's  theorem (see  \cite[2.2.7]{dmr}) states
  that when $D$ is a finite morphism and $F$ a Frobenius, then for $n$
  sufficiently  large $D  F^n$ satisfies  Lefschetz's trace  formula~; this implies
  that, for $n$ large enough and multiple of $\delta$, the endomorphism
  $D_{\bv(\bw\vp)^n}$  satisfies Lefschetz's trace formula, since
  $D_{(\bw\bvp)^\delta}=F^\delta$ is a Frobenius.

  3. Conjecture \ref{conj3} implies that $D_\bv$ acts trivially   on
  $H^0(\bX_{\bw\bvp},\Qlbar)$. 
\end{remarks}

  The following theorem is a consequence of what precedes.
  It refers to the definitions introduced below (see
  Def. \ref{localcyclo}), and the reader is invited to read them before
  reading this theorem.

\begin{theorem}\label{afterstrongpoincare}\hfill

  Assuming conjectures \ref{disjointness}, \ref{conj2}, \ref{conj3},
  \begin{enumerate}
    \item
      the algebra $\CH_c(\bX_{\bw\bvp})$ is a spetsial $\P$-cyclotomic algebra
      of $W$ at $w\vp$ of compact type, and
      the algebra $\CH(\bX_{\bw\bvp})$ is a spetsial $\P$-cyclotomic algebra
      of $W$ at $w\vp$ of  noncompact type,
    \item
      $\CH_c(\bX_{\bw\bvp})$ is the compactification of $\CH(\bX_{\bw\bvp})$,
      and 
      $\CH(\bX_{\bw\bvp})$ is the noncompactification of $\CH_c(\bX_{\bw\bvp})$.
  \end{enumerate}
\end{theorem}

  We can give a small precision about Conjectures \ref{conj2}
  (which will be reflected in Definition \ref{localcyclo} below)
  using now the strong Poincar\'e Conjecture \ref{conj3}.

\begin{lemma}\label{red2}\hfill

  Assume \ref{conj3} and \ref{disjointness}, and assume that $\CH_c(\bX_{\bw\bvp})$
  is the specialization at $x \mapsto q$ of a $\z$-cyclotomic Hecke algebra of
  $W(w\vp)$. 
  
  Let $I$ be a reflecting hyperplane for $W(w\vp)$, and if $\bs_I$ denotes the
  corresponding braid reflection in $\bB_W(\bw\bvp)$, let us denote by $\nu_I$ 
  the eigenvalue of $D_\bs$ on
  $H^{2l(\bw\bvp)}_c(\bX_{\bw\bvp},\Qlbar)$. Assume that $\nu_I = \la_I q^{m_I}$
  where $\la_I$ is a complex number of module 1 and $m_I \in \BQ$ is
  independent of $q$.
  
  Then $\la_I = 1$.
\end{lemma}

\begin{proof}\hfill

  From \ref{conj3} and \ref{red1} we get that $\nu_I$ is the unique eigenvalue 
  of maximal module of $D_{\bs_I}$ on $H^\bullet_c(\bX_{\bw\bvp})$; the
  eigenvalue of $F^\delta$ on $H^{2l(w)}_c(\bX_{\bw\bvp},\Qlbar)$ is also
  the unique eigenvalue of maximal module (equal to $q^{\delta l(w)}$).

  As  remarked in \ref{fujiwara} 2., for sufficiently large $n$ multiple of
  $\delta$ the endomorphism $D_{\bs_I (\bw\bvp)^n}$ satisfies the Lefschetz
  fixed point formula. Its eigenvalue on
  $H^{2l(w)}_c(\bX_{\bw\bvp},\Qlbar)$ is $\la_I q^{m_I} q^{\delta n l(w)}$,
  and this is the dominant term in the Lefschetz formula. Since the formula
  sums to an integer, this term must be a real number, thus $\la_I=1$.
\end{proof}

\begin{remark}\hfill

  Note that the assumption of the previous lemma on the shape of $\nu_I$
  is reasonable since we
  believe that it suffices to prove it in the case where $W(w\vp)$ is cyclic,
  and in the latter case $D_\bs$ is a root of $F$.
  
  Incidently, assume that $W(w\vp)$ is cyclic of order $c$,
  and let $\bs$ be the positive generator of $\bB_W(\bw\bvp)$. 
  Since $({\bw\bvp})^\delta$ is a power of $\bs$, we get (comparing lengths):
  $$
    \bs^{\frac{ac\delta}d}=({\bw\bvp})^\delta
    \,.
  $$
\end{remark}

\newpage
{\red \section{\red Spetsial $\P$-cyclotomic Hecke algebras}\hfill
\smallskip
}

  Leaning on properties and conjectures stated in the previous paragraph,
  we define in this section the special type of cyclotomic Hecke algebras
  which should occur as building blocks of the spetses: these algebras 
  (called ``spetsial cyclotomic Hecke algebras'') satisfy properties which
  generalize properties of algebras occurring as commuting algebras
  of cohomology of Deligne-Lusztig varieties attached to regular elements
  (see \S 3 above).
\smallskip

  Let $K$ be a number field which is stable under complex conjugation,
  denoted by $\la \mapsto \la^*$. 
  Let $\BZ_K$\index{ZBK@$\BZ_K$} be the ring of integers of $K$.
  
  Let $V$ be an $r$-dimensional vector space over $K$.
  
  Let $W$ be a finite reflection subgroup of $\GL(V)$
  and $\vp \in N_{\GL(V)}(W)$ be an element of finite order.
  We set $\BG := (V,W\vp)$.
\smallskip

\subsection{Prerequisites and notation}\hfill
\smallskip

  Throughout, $w\vp\in W\vp$ denotes a regular element.
  If $w\vp$ is $\z$-regular for a root of unity $\z$ with irreducible
  polynomial $\P$ over $K$, we say that $w\vp$ is \emph{$\P$-regular}.
  
  We set the following notation:
  \begin{itemize}
    \item[-]
      $V(w\vp) := \ker \P(w\vp)$\index{Vwp2@$V(w\vp)$} as a $K[x]/(\P)$-vector space,
    \item[-]
      $W(w\vp) := C_W(w\vp)$\index{Wwp2@$W(w\vp)$}, a reflection group on $V(w\vp)$
      (see above Theorem \ref{sylow}(5)),
    \item[-]
      $\CA(w\vp)$\index{Awp@$\CA(w\vp)$} is the set of reflecting hyperplanes of 
      $W(w\vp)$ in its action on $V(w\vp) ,$
    \item[-]
      $e_W(w\vp) := e_{W(w\vp)} \,.$\index{eWwp@$e_W(w\vp)$}
  \end{itemize} 
  Note that $K[x]/(\P)$ contains $\BQ_{W(w\vp)}$.
  
  The next theorem follows from Springer--Lehrer's theory of regular elements 
  (\cf\ \eg\ \cite[Th. 5.6]{berkeley}).
  
\begin{theorem}\label{springerlehrer}\hfill

  Assume that $W$ is irreducible. If $w\vp$ is regular, then $W(w\vp)$ acts
  irreducibly on $V(w\vp)$.
\end{theorem}
\smallskip

\subsection{Reduction to the cyclic case}\hfill
\smallskip
  
  We relate data for $W(w\vp)$ with local data for $W_I(w\vp)$,
  $I\in\CA(w\vp)$.
    
  So let $I \in \CA(w\vp)$. 
  We denote by $W_I$ the fixator of $I$ in $W$, a parabolic subgroup
  of $W$. The element $w\vp$ normalises the group $W_I$ 
  (since it acts by scalar multiplication
  on $I$), and it is also a $\P$-regular element for $W_I$. 
  Moreover, the group $W_I(w\vp)$,
  the fixator of the hyperplane $I$ in $W(w\vp)$, is cyclic.
  $$
     \xymatrixrowsep{1pc} \xymatrixcolsep{2pc}
    
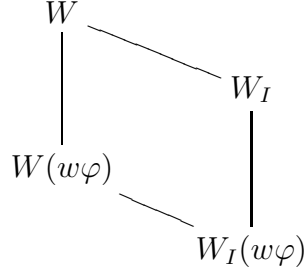
\xymatrix{
      &W\ar@{-}[dr]  \ar@{-}[dd]   &                              \\
      &                                              &W_I  \ar@{-}[dd]  \\
      &W(w\vp) \ar@{-}[dr]              &                              \\
      &                                              &W_I(w\vp)
    }
  $$  
  Thus we have a reflection coset
  $$
    \BG_I := (V,W_Iw\vp)
    \,.
  $$\index{GBI@$\BG_I $}
  
  Note that
  $$
    \CA(W_I) = \{ H\in\CA(W)\,|\,H\cap V(w\vp) = I\,\}
    \,.
  $$\index{AWI@$\CA(W_I)$}
  
  The ``reduction to the cyclic case'' is expressed first in a couple of simple
  formul\ae\ relating ``global data'' for $W$ to the collection of ``local data''
  for the family $(W_I)_{I\in\CA(w\vp)}$, such as:
  \begin{equation}\label{local1}
  \left\{
    \begin{aligned}
      &\CA(W) = \bigsqcup_{I\in\CA(w\vp)} \CA(W_I)
            \,,\text{ from which follow} \\
      &\Nr_W =\! \!\sum_{I\in\CA(w\vp)}\! \Nr_{W_I}
          \,,\,\,
          \Nh_W =\! \! \sum_{I\in\CA(w\vp)} \! \Nh_{W_I}
          \,,\,\,
         e_W =\!\!  \sum_{I\in\CA(w\vp)}\! e_{W_I}
         \,.
     \end{aligned}
   \right.
  \end{equation}
  \begin{equation}\label{local2}
  \left\{
    \begin{aligned}
      &J_W = \prod_{I\in\CA(w\vp)} J_{W_I}
            \,,\text{ from which follows} \\
      & \widetilde\det_V^{(W)}\!(w\vp) = \prod_{I\in\CA(w\vp)}  \widetilde\det_V^{(W_I)}\!(w\vp)
             \,.
     \end{aligned}
    \right.
  \end{equation}
  
  Notice the following consequence of (\ref{local1}), where
  we set the following piece of notation:
  $$
    e_{W_I}(w\vp) := e_{W_I(w\vp)}
    \,,
  $$\index{eWIwp@$e_{W_I}(w\vp)$}
  which we often abbreviate $e_I$.\index{eI@$e_I$}
  
\begin{lemma}\label{edividese}\hfill
  \begin{enumerate}
    \item
      Whenever $I\in\CA(w\vp)$, $e_{W_I}(w\vp)$ divides $e_{W_I}$.
    \item
      Assume that there is a single orbit of reflecting hyperplanes for $W(w\vp)$.
      Then $e_W(w\vp)$ divides $e_W$.     
  \end{enumerate}
\end{lemma}

\begin{proof}\hfill

  (1)
  Consider the discriminant for the contragredient representation of $W$ on
  the symmetric algebra $S(V^*)$ of the dual of $V$ (see \ref{somelinearcharacters} above),
  which we denote by $\Disc_W^*$.
  With obvious notation,
  we have
  $$
    \Disc_W^* = J_W^*{J_W^*}^\vee = \prod_{H\in\CA(W)} (j_H^*)^{e_H}
    \,.
  $$
  By restriction to the subspace $V(w\vp)$, we get
  $$
    {\Disc_W^*}{|_{V(w\vp)}} = \prod_{I\in\CA(w\vp)} (j_I^*)^{e_{W_I}}
    \,.
  $$
  Since ${\Disc_W^*}$ is fixed under $W$, ${\Disc_W^*}{|_{V(w\vp)}}$ is fixed under $W(w\vp)$.
  Since ${\Disc_W^*}$ is a monomial in the $j_I^*$'s,
  it follows from \cite[Prop.3.11,2]{berkeley} that ${\Disc_W^*}{|_{V(w\vp)}}$ must be a power
  of the discriminant of ${W_I}(w\vp)$, which shows that $e_{W_I}(w\vp)$ divides $e_{W_I}$.
%
  
  (2)
  If all reflecting hyperplanes of $W(w\vp)$ are in the same orbit as $I$, the relation
  $
    e_W = \sum_{I\in\CA(w\vp)} e_{W_I}
  $
  may be written in $W$ and in $W(w\vp)$ as
  $$
    e_W = \Nh_{W(w\vp)} e_{W_I}
    \,\,\text{ and }\,\,
    e_W(w\vp) = \Nh_{W(w\vp)} e_{W_I}(w\vp)
    \,,
  $$
  from which it follows that
  $$
    \dfrac{e_W}{e_W(w\vp)} = \dfrac{e_{W_I}}{e_{W_I}(w\vp)} \in \BN
    \,.
  $$
\end{proof}
\smallskip

\begin{remark}\hfill

  The conclusion in (2) of Lemma \ref{edividese} need not be true in general.
  For example, consider the case where $W = G_{25}$ (in Shephard--Todd
  notation), and let $w$ be a 2-regular element of $W$. Then $W(w) = G_5$.
  It follows that $e_W = 36$ and $e_W(w) = 24$, so $e_W(w)$ does not divide
  $e_W$. Note that $G_5$ has two orbits of reflecting hyperplanes.
\end{remark}
\smallskip

  \begin{remark}\hfill

  As a special case,
  assume that we started with a split coset \ie\ $\vp \in W$, and assume
  that $W$ acts irreducibly on $V$. Let us denote by $d$ the order of the
  regular element $w$ of $W$.
    
  Then by Theorem \ref{springerlehrer} $W(w)$ acts irreducibly on $V(w)$, 
  hence its center $ZW(w)$ is cyclic. Let us choose a generator $s$ of that
  center. Then we have
  \begin{enumerate}
    \item
      $s$ is regular (since $s$ acts as scalar multiplication on $V(w)$)
    \item
      $W(s) = W(w)\,.$
  \end{enumerate}
\end{remark}
\smallskip

\subsection{Spetsial $\P$-cyclotomic Hecke algebras at $w\vp$}\hfill
\smallskip

\subsubsection{A long definition}\label{longdefinition}\hfill
\smallskip
  
  We still denote by $\P$ the $K$-cyclotomic polynomial
  such that $\P(\z) = 0$, where $w\vp$ is $\z$-regular and $\z$
  has order $d$.
  
  Notice that, by definition
  \begin{itemize}
    \item
      $w\vp$ acts on $V(w\vp)$ as a scalar whose minimal polynomial over $K$ is $\P$, 
    \item
      $K[x]/(\P) = K_{W(w\vp)}$.
  \end{itemize}
   
\begin{definition}\label{localcyclo}\hfill

  A \emph{spetsial $\Phi$-cyclotomic Hecke algebra for $W$ at $\,w\vp\,$}%
  \index{spetsial $\Phi$-cyclotomic Hecke algebra for $W$ at $\,w\vp\,$} 
  is a
  $\ov K[x,x^{-1}]$-algebra denoted 
  $\CH_W(w\vp)$\index{HWwp@$\CH_W(w\vp)$}, specialization of the generic
  Hecke algebra of $W(w\vp)$ through a morphism $\si$
  and subject to supplementary conditions listed below:
   
  There are 
  \begin{itemize}
    \item[-]
      a $W(w\vp)$-equivariant family 
      $(\z_{I,j})_{I\in\CA(w\vp),j=0,\dots,e_I-1}$
      of elements of $\bmu_{e_I}$ ,
    \item[-]
      a $W(w\vp)$-equivariant family
      $(m_{I,j})_{I\in\CA(w\vp),j=0,\dots,e_I-1}$\index{mIj@$m_{I,j}$} of nonnegative 
      elements of $\dfrac{1}{|ZW|} \BZ$ ,      
      
  \end{itemize}
  such that
  $
    \si : u_{I,j} \mapsto \z_{I,j} v^{|ZW|m_{I,j}}
  $   
  where $v$ is an indeterminate such that $v^{|ZW|} = \z\inv x$,
  with the following properties.
   
   For each $I\in \CA(w\vp)$, the polynomial
   $
     \prod_{j=0}^{e_I-1} (t-u_{I,j})
   $ 
  specialises to a polynomial
  $P_I(t,x)$\index{PI@$P_I(t,x)$}
   satisfying the conditions:
  \begin{enumerate}
    \item[\textsc{(ca1)}]
      $P_I(t,x) \in K_{W(w\vp)}[t,x]\,,$
    \item[\textsc{(ca2)}]
       $P_I(t,x) \equiv  t^{e_I}-1 \pmod {\Phi(x)} \,,$
  \end{enumerate}
  and the following supplementary conditions.
\smallskip

\noindent
  \textsc{Global conditions}
  \begin{itemize}
    \item[\textsc{(ra)}]
      The algebra $\CH_W(w\vp)$ splits over $K_{W(w\vp)}(v)$.
    \item[\textsc{(sc1)}] 
      All Schur elements of irreducible characters of  $\CH_W(w\vp)$
      belong to $\BZ_K[x,x\inv]$.
    \item[\textsc{(sc2)}] 
      There is a unique irreducible character $\chi_0$\index{x0@$\chi_0$} of $\CH_W(w\vp)$ with the
      following property: 
          Whenever $\chi$ is an irreducible character of $\CH_W(w\vp)$
          with Schur element $S_\chi$, we have $S_{\chi_0}/S_\chi \in K[x]$.
      Moreover, $\chi_0$ is linear.
    \item[\textsc{(sc3)}] 
      Whenever $\chi$ is an irreducible character of $\CH_W(w\vp)$
      its Schur element $S_\chi$ divides $\Feg_\BG(R^\BG_{w\vp})$
      in $K[x,x\inv]$.
      
       For $\chi$ an irreducible character of $\CH_W(w\vp)$, we call
       \emph{generic degree}\index{generic degree} of $\chi$ the element of $K[x]$
       defined by
       $$
         \Deg(\chi) := \dfrac{\Feg_\BG(R^\BG_{w\vp})}{S_\chi}
         \,.
       $$\index{Degchi@$\Deg(\chi)(x)$}
  \end{itemize}
   
\noindent
  \textsc{Local conditions}
  
  Whenever $I \in \CA(w\vp)$, let us denote by $\CH_{W_I}(w\vp)$ the
  parabolic subalgebra of $\CH_W(w\vp)$ corresponding to the
  minimal parabolic subgroup $W_I(w\vp)$ of $W(w\vp)$.
  We set $\BG_I := (V,W_I w\vp)$.
  The following conditions concern the collection of parabolic
  subalgebras $\CH_{W_I}(w\vp)$ ($I\in\CA(w\vp)$).
\smallskip

 The algebras $\CH_{W_I}(w\vp)$ have to satisfy all previous conditions
 \textsc{(ca1)}, \textsc{(ca2)}, as well as conditions \textsc{(ra)}, \textsc{(sc1)}, \textsc{(sc2)}, \textsc{(sc3)} that we state
  again now, plus 
  \begin{itemize}
    \item
      for the {\sl noncompact support type\/}, conditions  \textsc{(ncs1)},  \textsc{(ncs2)},  \textsc{(ncs3)} stated below,
    \item
      for the {\sl compact support type\/}, conditions  \textsc{(cs1)},  \textsc{(cs2)},  \textsc{(cs3)} stated below.
  \end{itemize}
\smallskip

\noindent
  \textsc{Common local conditions}
\smallskip

  Notice that the following conditions impose some properties of rationality to the
  local algebra $\CH_{W_I}(w\vp)$ coming from the global datum $\BG = (V,W\vp)$.
  
  \begin{itemize}
    \item[\textsc{(ra$_I$)}]
      The algebra $\CH_{W_I}(w\vp)$ splits over $K_{W(w\vp)}(v)$,
      (where $v$ is an indeterminate such that 
      $v^{|ZW|}= \z\inv x$).
    \item[\textsc{(sc1$_I$)}] 
      All Schur elements of irreducible characters of  $\CH_{W_I}(w\vp)$
      belong to $\BZ_K[x,x\inv]$.
    \item[\textsc{(sc2$_I$)}] 
      There is a unique irreducible character $\chi_0^I$ of $\CH_{W_I}(w\vp)$ with the
      following property: 
          Whenever $\chi$ is an irreducible character of $\CH_{W_I}(w\vp)$
          with Schur element $S_\chi$, we have $S_{\chi_0^I}/S_\chi \in K[x]$.
      Note that since $W_I(w\vp)$ is cyclic, $\chi_0^I$ is linear.
     \item[\textsc{(sc3$_I$)}]
       Whenever $\chi$ is an irreducible character of $\CH_{W_I}(w\vp)$
       its Schur element divides $\Feg_\BG(R^{\BG_I}_{w\vp})$.
  \end{itemize}
\smallskip

  We set $e := e_I = e_{W_I}(w\vp)$.
    
  Let us define $a_1(x),\dots,a_e(x)  \in K_{W(w\vp)}[x]$ (the $a_j(x)$ depend on $I$) by
   $$
            P_I(t,x) = t^e - a_1(x) t^{e-1} + \dots + (-1)^ea_e(x)
            \,.
   $$
    
  \textsc{Noncompact support type (ncs)}
\smallskip

    We say that the algebra is of \emph{noncompact support type}\index{noncompact support type}
     if
    \begin{enumerate}
     \item[\tiny\textsc{(NCS0)}]
       $P_I(t,x) \in K[t,x]$,
     \item[\tiny\textsc{(NCS1)}]
       1 is a root of $P_I(t,x)$ (as a polynomial in $t$) and it is the only
       root which has degree 0 in $x$. In particular $a_1(0) = 1$.
    \item[\tiny\textsc{(NCS2)}] 
      The unique character ${\chi_0^I}$ defined by condition \textsc{(sc2$_I$)}
      above is the restriction of $\chi_0$ to $\CH_{W_I}(w\vp)$, and is
      defined by
      $$
        {\chi_0^I}(\bs_I) = 1
        \,.
      $$
      In other words, $\chi_0$ defines the \emph{trivial character} on $\bB_W(w\vp)$.
     \item[\tiny\textsc{(NCS2')}] 
       We have
       \begin{align*}
         &S_{\chi_0}(x) = (\z\inv x)^{-\Nr_W} \Feg(R^\BG_{w\vp})(x)\,,  \\ 
         &S_{\chi_0^I}(x) = (\z\inv x)^{-\Nr_{W_I}} \Feg(R^{\BG_I}_{w\vp})(x)\,,
       \end{align*}
       and in particular
       $$
         \Deg({\chi_0})(x) = (\z\inv x)^{\Nr_W} 
         \,.
       $$
     \item[\tiny\textsc{(NCS3)}]
       $P_I(0,x) = (-1)^e a_e(x) = - (\z\inv x)^{\Nr_{W_I}} \,.$
  \end{enumerate}
\smallskip

  \textsc{Compact support type (cs)}
\smallskip
    
    We say that the algebra is of \emph{compact support type}\index{compact support type} 
    if
    \begin{enumerate}
    \item[\tiny\textsc{(CS0)}]
      For $j=1,\dots,e$, we have $\z^{jm_I}a_j(x) \in K[x]$.
\smallskip

    \item[\tiny\textsc{(CS1)}]
      There is only one root (as a polynomial in $t$, and in some field extension of $K(x)$) 
      of $P_I(t,x)$ of highest degree in $x$, namely 
      $ (\z\inv x)^{\frac{e_{W_I}}{e_I}}\,.$
    \item[\tiny\textsc{(CS2)}] 
      The unique character $\chi_0^I$ defined by condition \textsc{(sc2$_I$)} above is the restriction
      of $\chi_0$ to $\CH_{W_I}(w\vp)$, and is defined by
      $$
        \chi_0^I(\bs_I) = (\z\inv x)^{\frac{e_{W_I}}{e_I}}
        \,.
      $$
      \item[\tiny\textsc{(CS2')}] 
       We have
       \begin{align*}
         &S_{\chi_0} = \Feg(R^\BG_{w\vp})\,,  \\ 
         &S_{\chi_0^I} =  \Feg(R^{\BG_I}_{w\vp})\,,
       \end{align*}
       and in particular
       $$
         \Deg({\chi_0})(x) = 1 
         \,.
       $$
     \item[\tiny\textsc{(CS3)}]
       $P_I(0,x) = (-1)^e a_e(x) = - (\z\inv x)^{\Nh_{W_I}} \,.$
  \end{enumerate}

\end{definition}

\subsubsection{From compact type to noncompact type and vice versa}\hfill

    Let us first state some elementary facts about polynomials.
   
  Let $P(t,x) = t^e-a_1(x)t^{e-1}+\dots+(-1)^ea_e(x) \in K[t,x]$ such that
  $
     P(t,x) = \prod_{j=0}^{e-1} (t-\la_j)
     \,,
  $
  where the $\la_j$ are nonzero elements in a suitable extension
  of $K(x)$.
  
  Assume that $P$ is ``$\z$-cyclotomic'', \ie\ that
  $
    P(t,\z) = t^e-1
    \,.
  $ 

   Choose an integer $m$ and consider the polynomial
  $$
    P^{[m]}(t,x) := \dfrac{t^e}{P(0,x)} P(x^m t\inv,x)
    \,.
  $$\index{Pm@$P^{[m]}(t,x)$}
  Then
  $$
    P^{[m]}(t,x) = \prod_{j=0}^{e-1} 
       (t-x^m \la_j\inv)
    \,,
  $$
  and
  $$
     P^{[m]}(t,\z) = \z^{me}((\z^{-m}t)^e-1)
     \,.
  $$

  Define
  $$
    P^{[m,\z]} (t,x) :=
     \z^{-me}P^{[m]}(\z^m t,x)
      = \dfrac{t^e}{P(0,x)} P((\z\inv x)^m t\inv,x)
    \,.
  $$\index{Pmz@$P^{[m,\z]} (t,x)$}
  We have
  $$
    P^{[m,\z]} (t,x) = \prod_{j=0}^{e-1} 
       (t-(\z\inv x)^m \la_j\inv)
    \,,
  $$
  and
  $$
     P^{[m,\z]}(t,\z) = t^e-1
     \,.
  $$
  Note that $P(t,x) \mapsto P^{[m,\z]} (t,x)$ is an involution.
  Write
  $
    P^{[m,\z]} (t,x) = t^e-b_1(x)t^{e-1}+\dots+(-1)^eb_e(x)
    \,.
  $
  
\begin{remarks}\hfill

  \begin{enumerate}
    \item
      If the highest degree term of $a_1(x)$ is 
      $(\z\inv x)^m$, 
      then $b_1(0) =1$,
      and if $a_1(0) = 1$ then the highest degree term of $b_1(x)$ 
      is $(\z\inv x)^m$.
    \item       
      If
      $P_I(0,x) = - (\z\inv x)^{\Nr_{W_I}} \,,$
      then
      $P_I^{[e_{W_I}/e_I,\z]}(0,x) = - (\z\inv x)^{\Nh_{W_I}} \,,$
      and \viceversa .
  \end{enumerate}
  
  Let us prove (2). By definition of $P^{[m,\z]} (t,x)$, we have
  $
    P^{[m,\z]} (0,x) = (\z\inv x)^{em} P(0,x),
  $
  whence
  $$
    P^{[e_{W_I}/e_I,\z]} (0,x) = (\z\inv x)^{e_{W_I}} P(0,x)
    \,.
  $$
  Now $e_{W_I} = \Nr_{W_I} + \Nh_{W_I}$, so
  if $P(0,x) =  - (\z\inv x)^{\Nr_{W_I}}$
  (resp. if $P(0,x) =  - (\z\inv x)^{\Nh_{W_I}}$),
  we see that 
  $
    P^{[e_{W_I}/e_I,\z]} (0,x) = - (\z\inv x)^{\Nh_{W_I}}
  $
  (respectively
  $
    P^{[e_{W_I}/e_I,\z]} (0,x) = - (\z\inv x)^{\Nr_{W_I}}
  $).
\end{remarks}
  
  The following lemma is then easy to prove. It is also a definition.
  
 \begin{lemma}\label{csncs}
 
   Assume given a $W(w\vp)$-equivariant family 
   of polynomials $(P_I(t,x))_{I\in\CA(w\vp)}$ in $K[t,x]$.
   
   For $I\in\CA(w\vp)$, set $m_I :=  \dfrac{e_{W_I}}{e_I}$.\index{mI@$m_I$}
  Then:
   \begin{enumerate}
     \item
       If the family $(P_I(t,x))_{I\in\CA(w\vp)}$ defines a
       ``spetsial $\P$-cyclotomic Hecke algebra $\CH_W(w\vp)$ of $W$ at $w\vp$ 
       of compact support type''
       then the family $(P^{[m_I,\z]}_I(t,x))_{I\in\CA(w\vp)}$ defines a
       ``spetsial $\P$-cyclo\-tomic Hecke algebra $\CH^{nc}_W(w\vp)$of $W$ 
       at $w\vp$ of noncompact support type'',
       called the ``noncompactification of $\CH_W(w\vp)$''.
     \item
       If the family $(P_I(t,x))_{I\in\CA(w\vp)}$ defines a
       ``spetsial $\P$-cyclotomic Hecke algebra $\CH_W(w\vp)$ of $W$ at $w\vp$ 
       of noncompact support type''
       then the family $(P_I^{[m_I,\z]}(t,x))_{I\in\CA(w\vp)}$ defines a
       ``spetsial $\P$-cyclotomic Hecke algebra $\CH^c_W(w\vp)$ of $W$ at $w\vp$ 
       of compact support type'', 
       called the ``compactification of $\CH_W(w\vp)$''.
   \end{enumerate}
 \end{lemma}
\smallskip

\subsubsection{A normalization}\hfill
\smallskip

  Let $\CH_W(w\vp)$ be a spetsial $\P$-cyclotomic Hecke algebra of $W$ at $w\vp$
  of noncompact type, defined by a family of polynomials $(P_I(t,x))_{I\in\CA(w\vp)}$.
  We denote by $\CH^c_W(w\vp)$ its compactification, defined by the family
  $(P_I^{[m_I,\z]}(t,x))_{I\in\CA(w\vp)}$.

  In the case where $W$ is a Weyl group, the spetsial
  $\P$-cyclotomic
  algebras $\CH_W(w\vp)$ and $\CH^c_W(w\vp)$ should have the following interpretation
  for every choice of a prime power $q$ (see \S 3 above).
  
   There is an appropriate Deligne--Lusztig variety $\bX_{\bw\bvp}$,
   endowed with an action of the braid group $\bB_W(w\vp)$ as automorphisms
   of \'etale sites, such that   
  \begin{itemize}
    \item
    (\textsl{Noncompact type})
     the element $\bs_I \in \bB_W(w\vp)$
     has minimal polynomial $P_I(t,q)$ when acting on $H^\bullet(\bX_{\bw\bvp},\BQ_\ell)$,
   \item
     (\textsl{Compact type})  the element
      $\z^{-m_I}\bs_I$ has minimal polynomial $P_I^{[m_I,\z]}(t,q)$ when acting on 
      $H^\bullet_c(\bX_{\bw\bvp},\BQ_\ell)$.
 \end{itemize}
 
 \begin{remark}
 It results from (CA2) in definition \ref{localcyclo} that the set
 $\{\zeta_{I,j}\}_j$ is equal to $\bmu_{e_I}$, but we have not yet chosen
 a specific bijection, which is how the data may appear in practice --- see
 the second step of algorithm \ref{an algo}. We now make the
 specific choice that $\zeta_{I,j}=\zeta_{e_I}^j$; such a choice determines
 the indexation of the characters of $\CH_W(w\phi)$ by those of $W(w\phi)$.
 \end{remark}
  With the above choice, we have
  $
    P_I(t,x) = (t-1)\prod_{j=1}^{e_I-1} \left( t- \z_{e_I}^j(\z\inv x)^{m_{I,j}} \right)
  $
  (where $m_{I,j} >0$ for all $j=1,\dots,e_I-1$, see (NCS1)), and
  we see that the minimal polynomial of $\bs_I$ on $H^\bullet_c(\bX_{\bw\bvp},\BQ_\ell)$
  is then
  $$
    \widetilde P_I(t,x) :=
    (t-x^{m_I}) \prod_{j=1}^{e_I-1} \left( t-\z^{m_I} \z_{e_I}^{-j}(\z\inv x)^{m_I-m_{I,j}} \right)
    \,.
  $$
  The polynomial $\widetilde P_I(t,x)$ is cyclotomic (\ie\
  reduces to $t^{e_I}-1$ when $x\mapsto \z$) if and only if $\z^{m_I} \in \bmu_{e_I}$.  
  That last condition is equivalent to
  $$
    \z^{e_{W_I}} = 1
    \quad\text{ \ie\ }\quad
    \De_{W_I}(w\vp) = 1
    \,.
  $$
  
  Let us denote by $\widetilde\CH_W^c(w\vp)$ the specialisation of the generic Hecke algebra
  of $W(w\vp)$ defined by the above polynomials $\widetilde P_I(t,x)$.
  
  The following property results from Lemma \ref{discriminanttrivial}.

\begin{lemma}\label{cyclotomicindeed}\hfill

  If $\BG = (V,W\vp)$ is real, then the algebra $\widetilde\CH_W^c(w\vp)$ is $\P$-cyclotomic.
\end{lemma}
\smallskip

\subsubsection{Rationality questions}\hfill
\smallskip

  A spetsial $\P$-cyclotomic Hecke algebra for $W(w\vp)$ over $K$
  is split over $K(\z)(v)$ where $v$ is an indeterminate such that
  $v^{|ZW|} = \z\inv x$, by Def. \ref{localcyclo} (\textsc{ra}).
  
  Since $\bmu_{|ZW|} \subseteq K$, the extensions
  $K(\z,v)/K(x)$ and  $K(v)/K(\z\inv x)$ are Galois.
  
  For $I \in \CA(w\vp)$, let us set
  $$
    P_I(t,x) = \prod_{j=0}^{e_I-1} (t-\z_{e_I}^j(\z\inv x)^{m_{I,j}})
                 = \prod_{j=0}^{e_I-1} (t-\z_{e_I}^j v^{n_{I,j}})
    \,,
  $$
  with
  $$
    m_{I,j} = \dfrac{n_{I,j}}{{|ZW|}}
    \,\,\text{ where }\, 
    n_{I,j} \in \BZ
    \,.
  $$
  Since $P_I(t,x) \in K(x)[t]$, its roots are permuted by the Galois group
  $\Gal(K(\z,v)/K(\z,x))$.
  Let us denote by $g \in \Gal(K(\z,v)/K(\z,x))$ the element defined by
  $
    g(v) = \z_{|ZW|} v
    \,.
  $
  Since $g$ permutes the roots of $P_I$, 
  there is a permutation $\si$ of $\{0,\dots,e_I-1\}$ such that
  $$
    g(\z_{e_I}^jv^{n_{I,j}}) =
    \z_{e_I}^{j} \z_{|ZW|}^{n_{I,j}} v^{n_{I,j}} =
     \z_{e_I}^{\si(j)}v^{n_{I,\si(j)}}
     \,,
   $$
   and so
   \begin{equation}\label{uniquem}
     n_{I,\si(j)} = n_{I,j}\,\,\text{ and } \,\,
      \z_{e_I}^{\si(j)} 
      =
      \z_{e_I}^{j} \z_{|ZW|}^{n_{I,j}} 
      \,.
   \end{equation}
 
\begin{remark}
  By Equation \ref{uniquem}, we see that if $j$ is such that
  $m_{I,j}\neq m_{I,j'}$ for all $j'\neq j$,
  then $\si(j) = j$, which implies that
  $
    \z_{e_I}^j 
      =
      \z_{e_I}^{j} \z_{|ZW|}^{n_{I,j}}
      \,,
  $
  hence that $n_{I,j}$ is a multiple of ${|ZW|}$, and so $m_{I,j} \in \BZ$.
\end{remark}
  
  By \textsc{(SC1)}, the Schur elements of $\CH_{W_I}(w\vp)$
  (see \ref{brmazy})
  $$
    S_j := \dfrac{1}{P(0,x)}\left( t \dfrac{d}{dt}P(t,x)\right){|_{t=\z_{e_I}^j v^{n_{I,j}}}}
  $$
  belong to $K[x]$, hence are fixed by $\Gal(K(\z,v)/K(\z,x))$, \ie\  we have
  $
    S_{\si(j)} = S_j
    \,,
  $
  or, in other words
  \begin{equation}\label{invarianceofschur}
    \left( t\dfrac{d}{dt}P(t,x) \right) {|_{t=\z_{e_I}^{j} \z_{|ZW|}^{n_{I,j}} v^{n_{I,j}}}}
    =
    \left( t\dfrac{d}{dt}P(t,x) \right) {|_{t=\z_{e_I}^jv^{n_{I,j}}}}
    \,.
  \end{equation}
    
\subsubsection{Ennola twist}\index{Ennola}\hfill
\smallskip

  Let us choose an element in $W\cap Z\GL(V)$, the scalar
  multiplication by $\e \in \bmu(K)$.  
  Then the element $\e w\vp$ is $\P(\e\inv x)$-regular, and we obviously
  have $W(\e w\vp) = W(w\vp)$.
  
  Assume given a $W(w\vp)$-equivariant family $(P_I(t,x))_{I\in\CA(w\vp)}$ 
   of polynomials in $K[t,x]$. 
   For $I\in\CA(w\vp)$, set 
   $$
      (\e.P_I)(t,x) :=  P_I(t,\e\inv x)
      \,.
   $$\index{edotP@$\e.P$}
   Note that the map
   $P \mapsto \e.P$ is an operation of order
   the order of $\e$.
   
   The following lemma is also a definition.   
   Its proof is straightforward. 
     
\begin{lemma}\label{polforennola}\hfill

  \begin{itemize}
     \item[(cs)]
      Assume that the family $(P_I(t,x))_{I\in\CA(w\vp)}$ defines a
      spetsial $\P$-cyclotomic Hecke algebra $\CH_W(w\vp)$ of $W$ at $w\vp$ 
      of compact support type.
      
      The family $((\e.P_I)(t,x))_{I\in\CA(w\vp)}$ defines a
      spetsial $\P(\e\inv x)$-cy\-clo\-tomic Hecke algebra of $W$ at $\e w\vp$ of compact support type,
      denoted $\e.\CH_W(w\vp)$ and called the Ennola $\e$-twist of $\CH_W(w\vp)$.
     \item[(ncs)]
      Assume that the family $(P_I(t,x))_{I\in\CA(w\vp)}$ defines a
      spetsial $\P$-cyclotomic Hecke algebra $\CH_W(w\vp)$ of $W$ at $w\vp$ of 
      noncompact support type
      .
      
      The family $((\e.P_I)(t,x))_{I\in\CA(w\vp)}$ defines a
      spetsial $\P(\e\inv x)$-cy\-clo\-tomic Hecke algebra of $W$ at $\e w\vp$ of noncompact support type,
      denoted by $\e.\CH_W(w\vp)$ and called the Ennola $\e$-twist of $\CH_W(w\vp)$.
  \end{itemize}
\end{lemma}

\begin{remark}\hfill

  Assume that
  $\CH_W(w\vp)$ is split over $K(\z)(v)$ for some $k$ such that 
  $k|m_K$ and $v^k = \z\inv x$. Then we see that
  $\e.\CH_W(w\vp)$ is split over $K(\e\z)(v_\e)$ 
  if $v_\e^k =\e\inv\z\inv x$.
  
  Thus the field $K(\e^{1/k},v)$ splits both $\CH_W(w\vp)$
  and $\e.\CH_W(w\vp)$.
\end{remark}
\smallskip

\subsection{More on spetsial $\P$-cyclotomic Hecke algebras}\hfill
\smallskip

  Let $\CH_W(w\vp)$ be a spetsial $\P$-cyclotomic Hecke algebra
  attached to the $\z$-regular element $w\vp$. 
  
  Note that, unless specified, $\CH_W(w\vp)$ may be of noncompact type
  or of compact type.
\smallskip

\subsubsection{Computation of $\om_\chi(\bpi)$ and applications}\hfill
\smallskip

  Choose a positive integer $h$ and an indeterminate 
  $v$ such that $v^h = \z\inv x$ and such that $\CH_W(w\vp)$ splits over
  $\ov K(v)$. 
  
  Whenever $\chi$ is an (absolutely) irreducible character of $\CH_W(w\vp)$ over $\ov K(v)$, 
  we denote by $\chi_{v=1}$ the irreducible character of $W(w\vp)$ defined by the
  specialization $v\mapsto 1$.

  We denote by $\si_\chi$\index{sichi@$\si_\chi$}
  the sum of the valuation and the degree of the Schur element
  (a Laurent polynomial) $S_\chi(x)$.

    Since $S_\chi(x)$ is semi-palindromic (see \cite[\S 6.B]{sp1}), we have
    $$
      S_\chi(x)^\vee = (\text{Constant}). x^{-\si_\chi} S_\chi(x)
      \,.
    $$
    We have (see Lemma \ref{schurpalin}, assuming \ref{proprHecke})
    $$
      S_\chi(x)^\vee = \dfrac{\tau(\bpi)}{\om_\chi(\bpi)} S_\chi(x)
      \,.
    $$
    From what precedes and by comparing with the specialization $v \mapsto 1$ we get
    \begin{equation}\label{omegapi}
      \om_{\chi}(\bpi) = v^{h\si_\chi} \tau(\bpi) = (\z\inv x)^{\si_\chi}  \tau(\bpi)
      \,.
    \end{equation}

\begin{itemize}
  \item
    Now in the NCS case we have
    $
      \tau(\bpi) = (\z\inv x)^{\Nr_W} = v^{h\Nr_W}
    $
  \item
    while in the CS case we have
    $
      \tau(\bpi) = (\z\inv x)^{\Nh_W} = v^{h\Nh_W}
    $
\end{itemize}
  from which we deduce

\begin{proposition}\label{omegaofpi}\index{omchipi@$\om_{\chi}(\bpi)$}\hfill
  \begin{enumerate}
    \item
      In the NCS case, 
      $\om_{\chi}(\bpi) = (\z\inv x)^{{\Nr_W}+\si_{\chi}} = v^{{h(\Nr_W}+\si_{\chi})}$.
    \item
      In the CS case, 
      $\om_{\chi}(\bpi) = (\z\inv x)^{\Nh_W+\si_{\chi}} = v^{h(\Nh_W+\si_{\chi})}$.
  \end{enumerate}
\end{proposition}
\smallskip
  
  Assume that the algebra $\CH_W(w\vp)$ is defined over $K(x)$ by the family of
  polynomials
  $$
    \left( P_I(t,x) := \prod_{j=0}^{e_I-1}(t-\z_{e_I}^jv^{hm_{I,j}}) \right)_{I\in\CA_W(w\vp)}
  $$
  of
  $K_{W(w\vp)}[t,x]$ and that it splits over $\ov K(v)$.  
  We set
  $$
     N_W := \left\{ \begin{aligned}
                              &\Nr_W \quad \text{ if $\CH_W(w\vp)$ is of noncompact type,}\\
                              &\Nh_W \quad \text{if $\CH_W(w\vp)$ is of compact type.}
                            \end{aligned}
                   \right.
  $$\index{NW@$N_W$}
   
  Any specialization of the type
  $
    v\mapsto \la
  $
  where $\la$ is an $h$-th root of unity defines a bijection
  $$
    \left\{
      \begin{aligned}
        &\Irr\left( \CH_W(w\vp) \right) \ra \Irr \left( W(w\vp) \right) \,, \\
        &\chi \mapsto \chi_{v=\la} \,,
      \end{aligned}
    \right.
  $$\index{xvla@$\chi_{v=\la}$}
  whose inverse is denoted
   $$
    \left\{
      \begin{aligned}
        &\Irr\left( W(w\vp) \right) \ra \Irr \left( \CH_W(w\vp) \right) \,, \\
        &\th \mapsto \th^{(\la\inv v)} \,.
      \end{aligned}
    \right.
  $$\index{thla@$\th^{(\la\inv v)}$}

\begin{lemma}\label{specializinglemma}\hfill

  Let $\chi \in \Irr(\CH_W(w\vp))$.
  \begin{enumerate}
    \item
      Let $\brh \in Z\bB_W(w\vp)$ such that $\brh^n = \bpi^a$
      for some $a,n\in\BN$.
      Then 
      $$
        h(N_W+\si_\chi)a/n \in \BZ
        \,.
      $$     
      If moreover $\chi$ is rational over $\ov K(x)$, we have
      $$
        (N_W+\si_\chi)a/n \in \BZ
        \,.
      $$
    \item
      Whenever $\la$ is an $h$-th root of unity, then
      $$
        \om_\chi(\brh) = \om_{\chi_{v=\la}}(\rh) (\la\inv v)^{h(N_W+\si_\chi)a/n}
        \,.
      $$
      In particular,
      $$
        \om_\chi(\brh) = \om_{\chi_{v=1}}(\rh)  v^{h(N_W+\si_\chi)a/n}
        \,.
      $$
    \item
      We have
      $$
        \om_{\chi_{v=\la}}(\rho) = \la^{h(N_W+\si_\chi)a/n} \om_{\chi_{v=1}}(\rho)
        \,.
      $$
  \end{enumerate}
\end{lemma}

\begin{proof}[Proof of \ref{specializinglemma}]\hfill

  By Proposition \ref{omegaofpi},
  $$
    \om_\chi(\bpi) = v^{h(N_W+\si_\chi)}
    \,.
  $$
  Since $\brh^n = \bpi^a$, it follows that, whenever $\la \in \ov K$,
  we have
  \begin{equation*}\tag{*}
    \om_\chi(\brh) = \kappa (\la\inv v)^{h(N_W+\si_\chi) a/n}
    \quad\text{for some } \kappa\in\ov K
    \,.
  \end{equation*}
  
  Since the character $\chi$ is rational over $\ov K(v)$,
  $ \om_\chi(\brh) \in \ov K(v)$, which implies
  $h(N_W+\si_\chi)a/n \in \BZ$. 
  
  If $\chi$ is rational over $\ov K(x)$ we have 
  $\om_\chi(\brh) \in \ov K(x)$, which implies
  $(N_W+\si_\chi)a/n \in \BZ$. This proves (1).
\smallskip
  
  As $\la^h = 1$, by specializing $v \mapsto \la$ in (*) we find
  $
    \kappa = \om_{\chi_{v=\la}}(\rho)
    \,.
  $
  
  Assertion (3) follows from the equality
  $$
    \om_\chi(\brh) = \om_{\chi_{v=\la}}(\rh) (\la\inv v)^{h(N_W+\si_\chi)a/n} = 
      \om_{\chi_{v=1}}(\rh)  v^{h(N_W+\si_\chi)a/n}
    \,.
  $$
\end{proof}
\smallskip

\subsubsection{Compactification and conventions}\hfill 
\smallskip

  Assume now that $\CH_W(w\vp)$ is a spetsial $\P$-cyclotomic Hecke algebra
  attached to the $\z$-regular element $w\vp$, \emph{of noncompact type}, defined
  by the family of polynomials $(P_I(t,x))_{I\in\CA_W(w\vp)}$ with
  $$
    P_I(t,x) = \prod_{j=0}^{e_I-1} (t-\z_{e_I}^j(\z\inv x)^{m_{I,j}})
    \,.
  $$

\noindent  
{\sl Some notation.}
\smallskip

  For $I \in \CA_W(w\vp)$, we denote by $\bs_I$ the braid reflection around $I$ 
  in $\bB_W(w\vp)$, and
  by $T_I$ the image of $\bs_I$ in $\CH_W(w\vp)$. Thus we have
  $$
    \prod_{j=0}^{e_I-1} (T_I-\z_{e_I}^j(\z\inv x)^{m_{I,j}}) = 0
    \,.
  $$
  The map $\bs_I \mapsto T_I$ extends to a group morphism
  $$
    \bB_W(w\vp) \ra \CH_W(w\vp)^\times
    \,,\quad
    b \mapsto T_b
    \,.
  $$
  
   We denote by $\chi_{0}$ the unique linear character of $\CH_W(w\vp)$
  such that (see \ref{longdefinition})
  $
    S_{\chi_{0}}(x) = (\z\inv x)^{-\Nr_W} \Feg(R^\BG_{w\vp})(x)
    \,.
  $
\smallskip
 
  Let us denote by $\CH_W^\oc(w\vp)$\index{HWc@$\CH_W^\oc(w\vp)$}
  the compactification of  $\CH_W(w\vp)$.
  By definition, $\CH_W^\oc(w\vp)$ is generated by a family of elements 
  $(T_I^\oc)_{I\in\CA_W(w\vp)}$ satisfying
  $$ 
    \prod_{j=0}^{j=e_I-1} (T_I^\oc-\z_{e_I}^{-j}(\z\inv x)^{m_I-m_{I,j}}) = 0
    \,,
  $$
  where $m_I := e_{W_I}/e_I$.
  The map $\bs_I \mapsto T_I^\oc$ extends to a group morphism
  $$
    \bB_W(w\vp) \ra \CH^\oc_W(w\vp)^\times
    \,,\quad
    b \mapsto T_b^\oc
    \,.
  $$

 We denote by $\chi_{0,\oc}$ the unique linear character of $\CH_W^\oc(w\vp)$
  such that (see \ref{longdefinition})
  $
    S_{\chi_{0,\oc}} = \Feg(R_{w\vp})
    \,.
  $
\smallskip
  
  For the definition of the opposite algebra $\CH_W^\oc(w\vp)^\op$,
  the reader may refer to \cite[1.30]{sp1}.
  
  Finally, we recall (see introduction of Section 2) that for $P(x) \in \ov K[x,x\inv]$, we set
  $
    P(x)^\vee := P(x\inv)^*
    \,.
  $
\medskip

\begin{proposition}[Relation between $\CH_W(w\vp)$ and $\CH_W^\oc(w\vp)$]\label{compactopposite}\hfill

  \begin{enumerate}
    \item
      The algebra morphism
      $$
        \CH_W(w\vp) \ra \CH_W^\oc(w\vp)^\op
      $$
      defined by
      $$ 
        T_I \mapsto (\z\inv x)^{m_{I}} (T_I^\oc)\inv
      $$
      is an isomorphism of algebras.
    \item
      There is a bijection
      $$
        \Irr \,\CH_W^\oc(w\vp) \iso \Irr\, \CH_W(w\vp)
        \,,\quad
        \chi \mapsto \chi^\nc
        \,,
      $$\index{xnc@$\chi^\nc$}
      defined by
      $$
        \chi^\nc(T_b) := \chi_{0,\oc}(T_b^\oc )\chi(T^\oc_{b\inv})
        \quad\text{whenever }\, b\in \bB_W(w\vp)
        \,.
      $$
      We have
      $
        {(\chi_{0,\oc})}^\nc = \chi_0
        \,.
      $
    \smallskip
    \item
      By specialisation $v \mapsto 1$, $\chi$ and $\chi^\nc$ become dual characters
      of the group $W(w\vp)$:
      $$
        \chi^\nc|_{v=1} = (\chi|_{v=1})^*
        \,.
      $$
    \smallskip
    \item
      Assuming \ref{proprHecke}, we have
      \begin{enumerate}
        \item
          $
            S_{\chi^\nc}(x) = S_\chi(x)^\vee = (\z\inv x)^{-\si_\chi} S_\chi(x)
            \,,
          $
        \smallskip
        \item
          $\si_{\chi^\nc} + \si_\chi = 0\,.$
        \smallskip
    \end{enumerate}
  \end{enumerate}
\end{proposition}

\begin{proof}\hfill

  (1) and (2) are immediate consequences of the definition of $\CH_W^\oc(w\vp)$. (3) follows
  immediately from the definition of the correspondence, since $\chi_{0,\oc}$
  specialises to the trivial character.
  
  By Theorem--Conjecture \ref{proprHecke}, the generic Schur elements 
  are multihomogeneous of degree 0 (we recall that this is proven, for example, for all imprimitive
  irreducible complex reflection groups --- see \ref{assumptioniontau} and \ref{jeanontau} above),
  so by construction of the compactification we have
  $ 
    S_{\chi^\nc}(x) = S_\chi(x)^\vee
    \,.
  $
  The assertion (4)(a) follows from \ref{omegapi}
  and the assertion (b) is obvious.
\end{proof}

  Let us recall (see \ref{longdefinition}) that
  \begin{align*}
    & \Deg(\chi) := \dfrac{\Feg(R_{w\vp})}{S_\chi} \,, \\
    & \Deg({\chi^\nc}) =  \dfrac{\Feg(R_{w\vp})}{S_{\chi^\nc}} \,.
  \end{align*}
  
  We denote by $\de_\chi$\index{dechi@$\de_\chi$} 
  (resp. $\de_{\chi^\nc}$\index{dechinc@$\de_{\chi^\nc}$}) the sum of the valuation and of the 
  degree of $\Deg(\chi)(x)$ (resp. of $\Deg({\chi^\nc})(x)$). 

\begin{lemma}\label{deltasigma}\hfill

  \begin{enumerate}
    \item
      We have
      \begin{align*}
        &\de_\chi = \Nr_W-\si_\chi 
        &\text{and}\quad 
        &\Deg(\chi)(x)^\vee = (\z\inv x)^{-\de_{\chi}} \Deg(\chi)(x) \,, \\
        &\de_{\chi^\nc} = \Nr_W-\si_{\chi^\nc} 
        &\text{and}\quad 
        &\Deg({\chi^\nc})(x)^\vee = (\z\inv x)^{-\de_{\chi^\nc}} \Deg({\chi^\nc})(x) \,.
      \end{align*}   
    \item
      We have
      $$
        \Deg({\chi^\nc})(x) = (\z\inv x)^{\Nr_W-\de_\chi} \Deg(\chi)(x)
                                         = (\z\inv x)^{\si_\chi} \Deg(\chi)(x)
        \,.
      $$
  \end{enumerate}
\end{lemma}

\begin{proof}\hfill

  (1) is immediate. To prove (2), notice that
  \begin{align*}
    \Deg(\chi)(x)^\vee  &= \dfrac{\Feg(R_{w\vp})(x)^\vee}{S_\chi(x)^\vee} 
                                        = \dfrac{(\z\inv x)^{-\Nr_W}\Feg(R_{w\vp})(x)}{S_{\chi^\nc}(x)} \\
                                      &= (\z\inv x)^{-\Nr_W}\Deg({\chi^\nc})(x)
  \end{align*}
  and (2) results from (1).
\end{proof}
\smallskip

\noindent
{\sl The case $W(w\vp)$ cyclic}\hfill
\smallskip
  
  We assume now  that $W(w\vp)$ is cyclic of order $e$. Let $s$ be its distinguished generator,
  and let $\bs$ be the corresponding braid reflection in $\bB_W(w\vp)$.
    
  Let $\CH_W(w\vp)$ be a spetsial $\P$-cyclotomic Hecke algebra 
  of compact type associated with $w\vp$, defined by the polynomial
  $$
    \prod_{j=0}^{e-1} (t-\z_e^j (\z\inv x)^{m_j})
    \,,
  $$
  where $m_j$ are nonnegative
  rational numbers such that $em_j \in \BN$. We have $m_0 = e_W/e > m_j$ for all $j$.
  
  Let us denote by $v$ an indeterminate such that $v^e = \z\inv x$, so that the
  algebra $\CH_W(w\vp)$ splits over $K(v)$. 
  
  For each $j$, we denote by $\th_j$ the character of $W(w\vp)$ defined by
  $
    \th_j(s) = \z_e^j
    \,.
  $
  
  The specialization $v \mapsto 1$ defines a bijection
  $$
    \left\{
    \begin{aligned}
      &\Irr\left(W(w\vp)\right) \ra \Irr\left(\CH_W(w\vp)\right) ,\\
      &\th_j \mapsto \th_j^{(v)} \,,
    \end{aligned}
    \right.
  $$
  by the condition $\th_j^{(v)}(\bs) := \z_e^j v^{em_j}$.
  We set
  $
    \si_j := \si_{\th_j^{(v)}}
    \,.
  $
  
  Let $\CH_W^\nc(w\vp)$ be the noncompactification of $\CH_W(w\vp)$, defined by
  the polynomial
  $$
    \prod_{j=0}^{e-1} (t-\z_e^{-j}(\z\inv x)^{m-m_j})
    \,.
  $$
  
  We denote by $\th_j^{(v),\nc}$ the character of $\CH_W^\nc(w\vp)$ which specializes
  to $\th_{-j}$ for $v=1$. So 
  $\th_j^{(v),\nc}(\bs) = \z_e^{-j} v^{e(m-m_j)}$. We set
  $
    \si_j^\nc := \si_{\th_j^{(v),\nc}}
    \,.
  $
  
\begin{lemma}\label{alphacyclic}\hfill
  $$
    \left\{
          \begin{aligned}
              &\si_j = em_j-\Nr_W \,,\\
              &\si_j^\nc  = e(m-m_j)-\Nh_W = \Nr_W-em_j \,.
          \end{aligned}
     \right.
  $$
\end{lemma}

\begin{proof}\hfill 

  Let $S_j$ be the Schur element of $\CH_W(w\vp)$ corresponding to $\th_j^{(v)}$. 
  By definition, the integer
  $\si_j$ is defined by an equation
  $$
    S_j(x)^\vee = \la x^{-\si_j}S_j(x)
  $$
  for some complex number $\la$.
  
  On the other hand, it results from \ref{schurtcheque} and from Definition \ref{longdefinition},
  (\textsc{ncs3}) and (\textsc{cs3}),  that
  $$
       \left\{
        \begin{aligned}
           &S_j(x)^\vee = v^{-em_j + \Nr_W}S_j(x) \,,\\
           &S_j^\nc(x)^\vee = v^{-e(m-m_j) + \Nh_W}S_j^\nc(x) \,,
         \end{aligned}
       \right.
  $$
  proving that $\si_j = em_j-\Nr_W$ and $\si_j^\nc  = e(m-m_j)-\Nh_W$. Since $em = \Nh_W + \Nr_W$,
  we deduce that $\si_j^\nc  = \Nr_W - em_j \,.$
\end{proof}
\smallskip

\noindent
{\sl A generalization of the cyclic case}\hfill
\smallskip

  We shall present now a generalization of Lemma \ref{alphacyclic}
  to the general case, where $w\vp$ is a $\z$-regular element for $W$,
  and $W(w\vp)$ not necessarily cyclic.
  
  Let $\CH_W(w\vp)$ we a spetsial $\P$-cyclotomic Hecke algebra (either of compact type
  or of noncompact type) attached to $w\vp$, defined by a family of polynomials
  $$
    \left( P_I(t,x) = \prod_{j=0}^{e_I-1} (t-\z_{e_I}^j(\z\inv x)^{m_{I,j}}) \right)_{I\in \CA_W(w\vp)}
    \,.
  $$ 
  Any linear character $\chi$ of $\CH_W(w\vp)$ is defined by a family $(j_{I,\chi})$
  where $I \in \CA_W(w\vp)$ and $0\leq j_{I,\chi} \leq e_I-1$ such that, if $\bs_I$
  denotes the braid reflection attached to $I$, we have
  $$
    \chi(\bs_I) = \z_{e_I}^{j_{I,\chi}}(\z\inv x)^{m_{I,j_{I,\chi}}}
    \,.
  $$
  
  Whenever $I \in \CA_W(w\vp)$, we denote by $\nu_I$ the cardinality of the orbit of $I$
  under $W(w\vp)$. 
  
  Note that the second assertion of the following Lemma reduces to Lemma \ref{alphacyclic}
  in the case where $W(w\vp)$ is cyclic.
  
\begin{lemma}\label{mandsigma}\hfill

  \begin{enumerate}
    \item
      Whenever $\chi$ is a linear character of $\CH_W(w\vp)$, we have
      $$
        \chi(\bpi) = \prod_{I\in\CA_W(w\vp)} (\z\inv x)^{e_Im_{I,j_{I,\chi}}}
                        = \prod_{I\in\CA_W(w\vp)/W(w\vp)} (\z\inv x)^{\nu_Ie_Im_{I,j_{I,\chi}}}
        \,.
      $$
    \item
      We have
      $$
        N_W + \si_\chi = \sum_{I\in\CA_W(w\vp)} e_Im_{I,j_{I,\chi}}
                                    = \sum_{I\in\CA_W(w\vp)/W(w\vp)}\nu_Ie_Im_{I,j_{I,\chi}}
        \,.
      $$
  \end{enumerate}  
\end{lemma}

\begin{proof}\hfill

  The assertion (2) follows from (1) and from \ref{omegaofpi}. Let us prove (1).
  
  From \cite[2.26]{bmr}, we know that in the abelianized braid group 
  $\bB_W/[\bB_W,\bB_W]$, we have
  $
    \bpi = \prod_{I\in\CA_W(w\vp)} \bs_I^{e_I}
    \,,
  $
  which implies (1).
\end{proof}
\smallskip 

\subsubsection{Ennola action}\index{Ennola as Galois}\hfill
\smallskip

  If $\bGF$ is a finite reductive group,
  with Weyl group $W$ of type $B_n$, $C_n$, $D_{2n}$,
  $E_7$, $E_8$, $F_4$, $G_2$,
  then ``changing $x$ into $-x$'' in the generic degrees formul\ae\ corresponds 
  to a permutation on the set of 
  unipotent characters, which we call \emph{Ennola transform}. The Ennola transform 
  permutes the generalized
  $d$-Harish--Chandra series (see \cite{bmm}), sending the $d$-series (corresponding 
  to the cyclotomic
  polynomial $\P_d(x)$) to the series corresponding to the cyclotomic polynomial $\P_d(-x)$. 
  We shall now
  introduce appropriate tools to generalize the notion of Ennola transform to the
  setting of ``spetses''.
\smallskip

  Throughout this paragraph, we assume that the reflection group
  \emph{$W$ acts irreducibly on $V$}. Its center $ZW$ is cyclic and acts by scalar
  multiplications on $V$. By abuse of notation,
  for $z\in ZW$ we still denote by $z$ the scalar by which $z$ acts on $V$.
  We set $c := |ZW|$.
  
  We define an operation of $Z\bB_W$ on the disjoint union
  $$
    \bigsqcup_{z\in ZW} \Irr \left( \CH_W(zw\vp) \right)
    \,.
  $$

  Let $\bz_0$ be the positive generator of $Z\bB_W$. 
  For $\bz \in Z\bB_W$, we denote by $z$ its image in $ZW$. 
  There is a unique $n$ ($0\leq n\leq c-1$) such that $z = \z_c^n$.
  
  The element $w'\vp := zw\vp$ is then $\z' := z\z$-regular.
  We have $K(\z) = K(\z')$ and
  the algebra $\CH_W(zw\vp)$ is split over $K(\z)(v')$ where $v' := \z_{hc}^{-n} v$.
  We have
  $
    {v'}^h = {\z'}\inv x
    \,,
  $
  and thus $\ov K(v) = \ov K(v')$.
  
  Consider the character $\xi$ defined on $Z\bB_W$ by the condition
  $$
     \xi : \bz_0 \mapsto \z_{hc} = \exp\left(2\pi i/hc\right)
     \,,
  $$\index{xi@$\xi$}
  so that $\xi(\bz) = \z_{hc}^n$. Note that $\xi(\bz^h) = z$.
\smallskip
 
  The group $Z\bB_W$ acts on $\ov K(v)$ as a Galois group, as follows:
  $$
    \e : 
    \left\{
      \begin{aligned}
        &Z\bB_W \ra \Gal\left(\ov K(v)/\ov K(v^{hc})\right)\,, \\
        &\bz \mapsto \left( v \mapsto \xi(\bz)\inv v \right) \,.
      \end{aligned}
    \right.
  $$
\smallskip

  Notice that if $\CH_W(w\vp)$ is a spetsial $\P(x)$-cyclotomic Hecke algebra  
  attached to a regular element $w\vp\in W\vp$, then the algebra $z.\CH_W(w\vp)$
  (see Lemma \ref{polforennola})
  is a spetsial $\P(z\inv x)$-cyclotomic Hecke algebra attached to the regular element $zw\vp\in W\vp$.

\begin{definition}\label{defennola}\hfill

  For $\chi \in \Irr \left( \CH_W(w\vp) \right)$,
  $\bz \in Z\bB_W$ and so $z = \xi(\bz^h)$,
  we denote by $\bz.\chi$ 
  (the Ennola image of $\chi$ under $\bz$)
  the irreducible character of $\CH_W(zw\vp)$
  over $\ov K(v)$ defined by the following condition:
  $$
    (\bz.\chi)_{v=\xi(\bz)} = \chi_{v=1}
    \,.
  $$\index{zdotchi@$\bz.\chi)$}
  In other words, the following diagram is commutative:
  $$
    \xymatrix{
      \Irr(\CH_W(w\vp)) \ar[rr]^{\bz\cdot}\ar[rd]_{v\mapsto 1} & &\Irr(\CH_W(zw\vp))\ar[dl]^{v\mapsto \xi(\bz)} \\
      &\Irr(W) &
    }
  $$    
\end{definition}
    
  In particular, the element $\bpi = \bz_0^c$ defines the following permutation
  $$
    \left\{
      \begin{aligned}
        &\Irr\left(\CH_W(w\vp)\right) \ra \Irr\left(\CH_W(w\vp)\right) \\
        &\chi \mapsto \bpi.\chi
          \quad\text{where}\quad   (\bpi.\chi)_{v=\z_h} = \chi_{v=1}
        \,,
      \end{aligned}
    \right.
  $$
  so it acts on $ \Irr(\CH_W(w\vp))$ as a generator of $\Gal(\ov K(v)/\ov K(x))$.
\smallskip

\begin{lemma}\label{ennolalemma}\hfill

  Let $\brh$ be an element of $Z\bB_W(w\vp)$ 
  (hence of $Z\bB_W(zw\vp)$) such that
  $\brh^n = \bpi^a$ for some $a,n \in \BN$.
  Then for $\bz \in Z\bB_W$ and $\chi \in \Irr\left(\CH_W(w\vp)\right)$ we have
  \begin{enumerate}
    \item
      $
        \om_{\bz\cdot\chi_{v=1}}(\rho) = \xi(\bz)^{-h(\Nh_W+\si_\chi)a/n}\om_{\chi_{v=1}}(\rho)
        \,,
      $
    \item
      $
         \om_{\bz.\chi}(\brh) = \om_\chi(\brh) \xi(\bz)^{-h(\Nh_W+\si_\chi)a/n} \,.
      $
  \end{enumerate}
\end{lemma}

\begin{proof}\hfill

  By Definition~\ref{defennola}, we know that
  $$
    \om_{\chi_{v=1}} =   \om_{\bz\cdot\chi_{v=\xi(\bz)}}
    \,.
  $$
  Thus by Lemma~\ref{specializinglemma}(2), we get
  $$
    \om_{\chi_{v=1}}(\rho) = \om_{\bz\cdot\chi_{v=\xi(\bz)}}(\rho) = 
     \xi(\bz)^{h(\Nh_W+\si_\chi)a/n}\om_{\bz\cdot\chi_{v=1}}(\rho)
  $$
  from which (1) follows.
  
  Now (2) follows from (1) and from Lemma \ref{specializinglemma}(1).
\end{proof}

\begin{proposition}\label{rationalandpiaction}\hfill

  Let $\chi \in \Irr\left(\CH_W(w\vp)\right)$.
  \begin{enumerate}
    \item
      The character $\chi$ takes its values in $\ov K(x)$ if and only if
      $
        \bpi\cdot\chi = \chi
        \,.
      $
    \item
      Assume $\chi$ is such that $\bpi\cdot\chi = \chi$.
      Then for all $\brh \in Z\bB_W(w\vp)$ such that
      $\brh^n = \bpi^a$, we have
      $$
        (\Nh_W + \si_\chi)\frac{a}{n} \in \BZ
        \,.
      $$
  \end{enumerate}
\end{proposition}

\begin{remark}\hfill

  Consider for example the case of the Weyl group of type $E_7$, and choose $K = \BQ$,
  $\z = 1$, $w\vp = 1$. The algebra $\CH_W(1)$ is then the ``usual'' Hecke algebra over
  $\BZ[x,x\inv]$, and we have $h=2$. We set $x=v^2$.
  
  All the irreducible characters of $\CH_W(1)$ are $\ov\BQ(x)$-rational, except for the
  two characters of
  dimension  512 denoted $\phi_{512,11}$ and $\phi_{512,12}$, which take their
  value  in $\BQ[v]$. The Galois action of $\bpi$ is given by $v\mapsto -v$, and
  we have $\bpi\cdot\phi_{512,11}=\phi_{512,12}$.
\end{remark}

\begin{proof}[Proof of \ref{rationalandpiaction}]\hfill

  (1)
  follows as $\bpi$ acts as a generator of $\Gal(\ov K(v)/\ov K(x))$.
  
  (2)
  Applying Lemma \ref{ennolalemma} above to the case $\bz = \bpi$ gives
  $$
     \om_{\bpi.\chi}(\brh) = \om_\chi(\brh) \z_h^{-h(\Nh_W+\si_\chi)a/n}
     \,,
     \quad\text{hence }\,
     \z_h^{-h(\Nh_W+\si_\chi)a/n} = 1\,,
   $$
   from which the claim follows.
\end{proof}
\smallskip

\begin{remark}\label{knowledge}\hfill

  By Lemma \ref{deltasigma}(1), we know that
  $
    \si_\chi = \Nr_W -\de_\chi
    \,,
  $
  which implies that
  $$
    (\Nh_W+\si_\chi)\dfrac{a}{n} = (e_W-\de_\chi)\dfrac{a}{n}
    \,.
  $$
  Since $e_W \dfrac{a}{n}$ is the length of $\brh$, it lies in $\BZ$,
  and so, as elements of $\BQ/\BZ$, we have 
  $$
    (\Nh_W+\si_\chi)\dfrac{a}{n} = -\de_\chi\dfrac{a}{n}
    \,.
  $$
  In other words, the knowledge of the element $(\Nh_W+\si_\chi)\dfrac{a}{n}$
  as an element of $\BQ/\BZ$ is the same as the knowledge of the root of
  unity
  $
    \z_n^{-\de_\chi a}
    \,.
  $
\end{remark}
\smallskip

\subsubsection{Frobenius eigenvalues}\index{Frobenius eigenvalue}
\label{paragraphonFrob}\hfill 
\smallskip

%
  Here we develop tools necessary to generalize to reflection cosets 
  Proposition \ref{valueoffrobenius}, where we compute the Frobenius
  eigenvalues for uni\-potent characters in a principal $w\vp$-series.
\smallskip

  We resume the notation from \S \ref{liftingregular}: 
  \begin{itemize}
    \item
      $w\vp$ is a $\z$-regular element, $\P$ is the minimal polynomial of $\z$
      over $K$, and we have $\z = \exp(2\pi ia/d)$,
    \item
      we have an
      element $\brh_{\ga,a/d}$ of the center of the braid group
      $\bB_W(w\vp)$ which satisfies $\brh_{\ga,a/d}^{ d} = \bpi^{a\de}$.
  \end{itemize}

  We still denote by $\CH_W(w\vp)$ a spetsial $\P$-cyclotomic Hecke algebra
  associated with $w\vp$, either of compact or of noncompact type.
  
\begin{notation}\label{deffrobenius}\hfill

  Whenever $\chi$ is an (absolutely) irreducible character of $\CH_W(w\vp)$
  over  $\ov K(v)$,
  we define a monomial $\Fr_{w\vp}^{(\brh_{\ga,a/d})}(\chi)$ in $v$ by the following
  formulae:
  $$
     \Fr_{w\vp}^{(\brh_{\ga,a/d})}(\chi) := 
      \left\{
        \begin{aligned}
          &\z^{l(\brh_{\ga,a/d})}\om_\chi(\brh_{\ga,a/d})
             \,\,\text{ for $\CH_W(w\vp)$ of compact type,} \\
          &\om_\chi(\brh_{\ga,a/d})
             \hskip1,7cm\text{for $\CH_W(w\vp)$ of noncompact type.\!\!}
         \end{aligned}
        \right.
  $$
\end{notation}

  When there is no ambiguity about the ambient algebra, we shall note 
  $\Fr^{(\brh_{\ga,a/d})}(\chi)$\index{Frrho@$\Fr^{(\brh_{\ga,a/d})}(\chi)$} instead of
  $\Fr_{w\vp}^{(\brh_{\ga,a/d})}(\chi)$.
\smallskip

  Since $\brh_{\ga,a/d}^d = \bpi^{\de a}$, the following lemma is an immediate corollary
  of Lemma \ref{specializinglemma}.
  Recall that we denote by $h$ the integer such that $\z\inv x = v^h$.

\begin{lemma}\label{compuoffrobenius}\hfill

  Whenever $\la$ is an $h$-th root of unity we have
  $$
    \Fr^{(\brh_{\ga,a/d})}(\chi) =
     \left\{
        \begin{aligned}
          &\z^{e_W\frac{\de a}{d}}  \om_{\chi_{v=\la}}(\rh) (\la\inv v)^{h(\Nh_W+\si_\chi)\frac{\de a}{d}}
            \text{ (compact type),} \\
          &\om_{\chi_{v=\la}}(\rh) (\la\inv v)^{h(\Nr_W+\si_\chi)\frac{\de a}{d}}
            \hskip1cm\text{ (noncompact type),}
         \end{aligned}
        \right.
  $$
  and in particular
  $$
    \Fr^{(\brh_{\ga,a/d})}(\chi) =
     \left\{
        \begin{aligned}
          &\z^{e_W\frac{\de a}{d}}  \om_{\chi_{v=1}}(\rh) v^{h(\Nh_W+\si_\chi)\frac{\de a}{d}}
            \text{ (compact type),} \\
          &\om_{\chi_{v=1}}(\rh)v^{h(\Nr_W+\si_\chi)\frac{\de a}{d}}
            \hskip1cm\text{ (noncompact type)}.
         \end{aligned}
        \right.
  $$
\end{lemma}
\smallskip

  The above lemma shows that the value of $\Fr^{(\brh_{\ga,a/d})}(\chi)$ does not depend on
  the choice of $\ga$.
\smallskip
  
  If we change $a/d$ (in other words, if we replace $\brh_{\ga,a/d}$ by $\brh_{\ga,a/d}\bpi^{n\de}$),
  we get (in the compact case):
  \begin{equation}\label{changingrho}
     \Fr^{(\brh_{\ga,a/d}\bpi^{n\de})}(\chi) = \z^{-(\Nh_W+\si_\chi)n\de} \Fr^{(\brh_{\ga,a/d})}(\chi)
  \end{equation}
  
  If we force $\brh_{\ga,a/d}$ to be as short as possible (\ie\ if we assume $0\leq a<d$),
  the monomial $\Fr^{(\brh_{\ga,a/d})}(\chi)$ depends only on $\chi$. In that case we denote
  it by  $\Fr_{w\vp}(\chi) $.

  Notice that if $\chi$ is $\ov K(x)$-rational, or equivalently if $\bpi\cdot\chi = \chi$
  (see \ref{rationalandpiaction}), 
  we have $\om_\chi(\brh_{\ga,a/d}) \in \ov K(x)$,
  and $\Fr^{(\brh_{\ga,a/d})}(\chi)$ is a monomial in $x$.

\begin{definition}\label{defoffrob}\hfill

  Assume $\chi$ is $\ov K(x)$-rational. Then the Frobenius eigenvalue of $\chi$ is
  the root of unity defined by\index{frchi@$\fr(\chi)$}
  $$
    \fr(\chi) := \Fr_{w\vp}(\chi)|_{x=1}
    \,.
  $$
\end{definition}

\begin{proposition}\label{valueoffr}\hfill

  Let $\chi$ be an irreducible character of the algebra $\CH^\oc_W(w\vp)$ of
  compact type. Assume that $\chi$ is $\ov K(x)$-rational.
  \begin{enumerate}
    \item
      We have
      $$
        \fr(\chi) =
           \z^{\frac{\de a}{d}(\Nr_W-\si_\chi)} \om_{\chi_{v=1}}(\rh) = \z^{\frac{\de a}{d}\de_\chi} \om_{\chi_{v=1}}(\rh)
        \,.
      $$
    \item
      For the corresponding character $\chi^\nc$ of the associated algebra of noncompact type, we have
      $$
        \fr(\chi^\nc) = 
          \z^{-\frac{\de a}{d} (\Nr_W+\si_{\chi^\nc})} \om_{\chi^\nc_{v=1}}(\rh) =
          \z^{-\frac{\de a}{d}\de_\chi} \om_{\chi^\nc_{v=1}}(\rh) = \fr(\chi)^*
          \,.
      $$
    \end{enumerate}
\end{proposition}

\begin{proof}\hfill

  It is direct from Lemma \ref{compuoffrobenius}.
\end{proof}
\medskip

\textsl{A definition for general characters}
\smallskip

{\small
\begin{comment}\hfill

  In the case of a character $\chi$ which is not $\ov K(x)$-rational, we can only attach a
  set of roots of unity to the orbit of $\chi$ under $\Gal(\ov K(v)/\ov K(x))$.
  
  Consider for example the case of the Weyl group of type $E_7$, and choose $w\vp := w_0$,
  the longest element. Then the algebra $\CH_W(w_0)$ has two irrational characters, say
  $\chi_1$ and $\chi_2$, which
  correspond to two unipotent cuspidal characters of the associated finite reductive groups
  (these characters belong to the same Lusztig family as the principal series unipotent characters
  $\rho_{\chi_{512,11}}$ and $\rho_{\chi_{512,12}}$).
  
  These two unipotent cuspidal characters can be distinguished by their Frobenius eigenvalues, which
  are $i$ and $-i$.
  
  Here we shall only attach to the Galois orbit $\{\chi_1,\chi_2\}$ the set of two roots of unity
  $\{i,-i\}$.
\end{comment}
}
\smallskip

  From now on, in order to make the exposition simpler, we assume that 
  \emph{$\CH_W(w\vp)$ is of compact type}.

  Let $\chi \in \Irr(\CH_W(w\vp))$. Let $k$ denote the length of the orbit of $\chi$ under $\bpi$.
  It follows from \ref{ennolalemma} that
  $$
    \om_{\bpi.\chi}(\brh) = \om_\chi(\brh)\z^{(\Nh_W+\si_\chi)\de}
    \,,
  $$
  hence $d$ divides $k(\Nh_W+\si_\chi)\de$. Thus $(\Nh_W+\si_\chi)\de a/d$ defines an
  element of $\BQ/\BZ$ of order $k'$ dividing $k$.
  
  We shall attach to the orbit of $\chi$ under $\bpi$ an orbit of roots of unity under the action of the group
  $\bmu_{k'}$, 
  as follows.
  
  Choose a $k$-th root $\z_0$ of $\z$, and set $x_0 := \z_0 v^{h/k}$ so that 
  $v^{h/k} = \z_0\inv x_0$.
 Then we have
  $$
    v^{h(\Nh_W+\si_\chi)\de a/d} = (\z_0\inv x_0)^{k(\Nh_W+\si_\chi)\de a/d}
  $$
  where $k(\Nh_W+\si_\chi)\de a/d \in \BZ$.  
  
  By Lemma \ref{compuoffrobenius}, we see that $\Fr_{w\vp}(\chi)$ is a monomial in $x_0$. 
    
  We recall (see Remark \ref{knowledge} above) that, as an element of $\BQ/\BZ$, we have
  $(\Nh_W+\si_\chi)\frac{\de a}{d} = -\de_\chi \frac{\de a}{d}$, and that it is defined by the
  root of unity $\z^{-\de_\chi \de}$.

\begin{definition}\label{deffrobgeneral}\hfill

  If $\chi$ has an orbit of length $k$ under $\bpi$, we attach to that orbit
      the set of roots of unity defined as
      $$
        \fr(\chi) = \{ \Fr_{w\vp}(\chi)|_{x_0=1}\,\la^{k(\Nh_W+\si_\chi)\frac{\de a}{d}} \,;\, (\la \in \bmu_k)\}
          = \Fr_{w\vp}(\chi)|_{x_0=1}\,\bmu_{k'}
        \,,
      $$
      with $k'$ the order of the element of $\BQ/\BZ$ defined by 
      $$
        (\Nh_W+\si_\chi)\frac{\de a}{d} = -\de_\chi \frac{\de a}{d}      
        \,,
      $$
      (in other words, $k'$ is the order of $\z^{-\de_\chi\de}$).
\end{definition}

\begin{remark}\label{orbitonfrob}\hfill

  The set $\fr(\chi)$ is just the set of all $k'$-th roots of $(\Fr_{w\vp}(\chi)|_{x_0=1})^{k'}$.
\end{remark}

  A computation similar to the computation made above for the rational case gives

\begin{proposition}\label{valueoffrogbeneral}\hfill
  
  Let $\chi$ be an irreducible character of the algebra $\CH^\oc_W(w\vp)$ of
  compact type.
  Assume that $\chi$ has an orbit of length $k$ under $\bpi$, and that $(\Nh_W+\si_\chi)\frac{\de a}{d}$
  has order $k'$ in $\BQ/\BZ$.

      The set of Frobenius eigenvalues attached to that orbit is the set of all $k'$-th roots of
      $$
        \z^{k'\frac{\de a}{d}\de_\chi} (\om_{\chi_{v=1}}(\rh))^{k'} 
        \,.
      $$
\end{proposition}

\begin{definition}\hfill

  For $\chi$ an irreducible character of the compact type algebra $\CH^\oc_W(w\vp)$,
  the set of Frobenius eigenvalues attached to the orbit of the corresponding character $\chi^\nc$
      of the associated noncompact type algebra is 
      $$
        \fr(\chi^\nc) = \fr(\chi)^*
        \,,
      $$
      the set of complex conjugates of elements of $\fr(\chi)$.
\end{definition}
\smallskip

\subsubsection{Ennola action and Frobenius eigenvalues}\hfill
\smallskip
  
  Let us now compute the effect of the Ennola action on Frobenius eigenvalues.
  Recall that we assume $\CH_W(w\vp)$ to be of compact type.
  
  We start by studying the special case of the action of the permutation defined by $\bpi$ on
  $ \Irr \left( \CH_W(w\vp) \right)$.

\begin{lemma}\label{jeanlemma}\hfill

  For $\brh = \brh_{\ga,a/d}$ as above, whenever $\chi \in \Irr\left(\CH_W(w\vp)\right)$,
  we have
  $$
    \Fr^{\brh\bpi^\de}(\chi) = \Fr^\brh(\bpi .\chi)x^{\de(\Nh_W+\si_\chi)}
    \,.
  $$
\end{lemma}

\begin{proof}[Proof of \ref{jeanlemma}]\hfill

  On the one hand,
  by Definitions \ref{defennola} and \ref{deffrobenius} we have
  \begin{align*}
    \Fr^{\brh\bpi^\de}(\chi) 
      &= \z^{l(\brh\bpi^\de)}\om_{\chi}(\brh\bpi^\de)
        = \z^{l(\brh)} \z^{l(\bpi^\de)}
            \om_{\chi}(\brh) \om_{\chi}(\bpi^\de) \\
      &= \left( \z^{l(\brh)}\om_{\chi}(\brh) \right)
            \left( \z^{l(\bpi^\de)} \om_{\chi}(\bpi^\de) \right) \,.
 \end{align*}
   Proposition \ref{omegaofpi} gives that
   $$
     \om_\chi(\bpi^\de)
       = v^{h(\Nh_W+\si_\chi)\de}
     \,.
   $$
   Morever,
   $
     \z^{l(\bpi^\de)} = (\z^\de)^{l(\bpi)} = 1
     \,
   $
   since $(w\vp)^\de$ is a $\z^\de$-regular element of $W$ (see the second remark following
   \ref{detofregular}).
  It follows that
  $$
     \Fr^{\brh\bpi^\de}(\chi) = \left( \z^{l(\brh)}\om_{\chi}(\brh) \right) v^{h(\Nh_W+\si_\chi)\de}
     \,.
  $$
  
  On the other hand, by definition
  $$
    \Fr^\brh(\bpi .\chi) = \z^{l(\brh)} \om_{\bpi.\chi}(\brh)
    \,.
  $$
   By Lemma \ref{ennolalemma}, and since
   $
     \xi(\bpi) = \xi(\bz_0^c) = \exp(2\pi i/h)
     \,,
   $
   this yields
   $$
     \om_{\bpi.\chi} (\brh) = \om_\chi(\brh)\exp(-2\pi i(\Nh_W+\si_\chi)\de a/d)
       = \om_\chi(\brh) \z^{-(\Nh_W+\si_\chi)\de}
     \,,
   $$
   hence
   $$
      \Fr^\brh(\bpi .\chi) = \left( \z^{l(\brh)}\om_{\chi}(\brh) \right) \z^{(\Nh_W+\si_\chi)\de}
      \,.
    $$    
   The lemma follows.
\end{proof}

  Now we know by (\ref{changingrho}) that 
  $$
    \Fr^{\brh\bpi^\de}(\chi) = \z^{-(\Nh_W+\si_\chi)\de} \Fr^{\brh}(\chi)
    \,,
  $$
  which implies by Lemma \ref{jeanlemma}
  $$
    \Fr^{\brh}(\bpi .\chi) = \z^{(\Nh_W+\si_\chi)\de} \Fr^\brh(\chi)x^{-\de(\Nh_W+\si_\chi)}
    \,.
  $$
  The following proposition is now immediate.
  Note that its statement contains a slight abuse of notation, since for $\chi$
  a $\ov K(x)$-rational character, $\fr(\chi)$ is not a set --- it has then to be considered as
  a singleton.

\begin{proposition}\label{piandfrobenius}\hfill
  $$
    \fr(\bpi\cdot\chi) = \{ \z^{(\Nh_W+\si_\chi)\de} \la \,;\, (\la\in \fr(\chi))\,\}
    \,.
  $$
\end{proposition}
\smallskip

  Let us now consider the general case of Ennola action by an element $\bz\in Z\bB_W$.
  By Lemma \ref{ennolalemma}, we have
  \begin{align*}
    \Fr_{zw\vp}^{\bz^\de\brh}(\bz\cdot\chi) 
      &= (z\z)^{l(\bz^\de\brh)}  \om_{\bz\cdot\chi}(\bz^\de\brh) \\
      &= (z\z)^{l(\bz^\de\brh)} \om_{\chi}(\bz^\de\brh) \xi(\bz)^{-h(\Nh_W+\si_\chi)l(\bz^\de\brh)/l(\bpi)} \\
      &= (z\z)^{l(\bz^\de\brh)} (z\z)^{-(\Nh_W+\si_\chi)\de}  \om_{\chi}(\bz^\de\brh) \,,
  \end{align*}
  hence
  \begin{align*}
     &\dfrac{\Fr_{zw\vp}^{\bz^\de\brh}(\bz\cdot\chi)}{\Fr_{w\vp}^{\brh}(\chi)}
       = (z\z)^{l(\bz^\de)}  z^{l(\brh)} 
        (z\z)^{-(\Nh_W+\si_\chi)\de}  \om_{\chi}(\bz^\de) \\
      &= (z\z)^{l(\bz^\de)}  z^{l(\brh)} 
        (z\z)^{-(\Nh_W+\si_\chi)\de}  \om_{\chi_{v=1}}(z^\de) v^{h(\Nh_W+\si_\chi)l(\bz^\de)/l(\bpi)} \,.
  \end{align*}
  Since $\Fr_{zw\vp}^{\bz^\de\brh}(\bz\cdot\chi) = \Fr_{zw\vp}(\bz\cdot\chi)$ up to an integral
  power of $(z\z)^{(\Nh_W+\si_\chi)\de}$ (see (\ref{changingrho})), what precedes proves the
  following proposition.
  
\begin{proposition}\label{ennolaandfrobeniusgeneral}\hfill

  Up to an integral power of $(z\z)^{(\Nh_W+\si_\chi)\de}$, we have
  $$
    \dfrac{\Fr_{zw\vp}^{\bz^\de\brh}(\bz\cdot\chi)}{\Fr_{w\vp}^{\brh}(\chi)} =
       (z\z)^{l(\bz^\de)}  z^{l(\brh)} 
        \om_{\chi_{v=1}}(z^\de) v^{h(\Nh_W+\si_\chi)l(\bz^\de)/l(\bpi)}
        \,.
  $$
\end{proposition}
\smallskip

\subsection{Spetsial data at a regular element, spetsial groups}\hfill
\smallskip

\subsubsection{Spetsial data at a regular element}\hfill
\smallskip

\begin{definition}\label{defspetsial}

  Let $\BG = (V,W,\ov\vp)$ be a reflection coset and let $w\vp$ be a
  $\P$-regular element of $W\vp$.
  
  We say that $\BG = (V,W,\ov\vp)$ is \emph{spetsial at $\P$}\index{spetsial at $\P$}
  (or \emph{spetsial at $w\vp$}\index{spetsial at $w\vp$})
  if there exists a spetsial $\P$-cyclotomic Hecke algebra of $W$ at $w\vp$.
\end{definition}

  We are not able at the moment to classify the irreducible reflection cosets
  which are spetsial at an arbitrary cyclotomic polynomial $\P$. But:
  \begin{itemize}
    \item
      One can classify the irreducible split reflection cosets which are spetsial at $x-1$~:
      they are precisely the spetsial reflection groups (see Proposition \ref{characspetsial}
      below).
    \item
      If $\BG = (V,W)$ is a split spetsial reflection coset on $K$, and if $w$
      is a $\P$-regular element of $W$, then $(V(w),W(w))$ is spetsial at $\P$~:
      for $W$ primitive, this will be a consequence of the construction of the spets data 
      associated with $\BG$ (see \S 5 below).
   \end{itemize}
\smallskip

\subsubsection{The $1$-spetsial algebra $\CH_W$}\hfill
\smallskip

  Let us now consider the special case where $w\vp = \id_V$.
  
\begin{definition}\label{HW}\hfill

  Let $W$ be a reflection group.
  We denote by $\CH_W$ (resp. $\CH^\nc_W$) the algebra defined by the collection of
  polynomials $P_H(t,x)_{H\in\CA(W)}$ where
      $$
        \begin{aligned}
          &P_H(t,x) = (t-x)(t^{e_H-1}+\dots+t+1) \\
            (resp. \quad
             &P_H(t,x) = (t-1)(t^{e_H-1}+\dots+tx^{e_H-2}+x^{e_H-1})
             \,\,)\,.
         \end{aligned}
       $$
\end{definition}

\begin{proposition}\label{characspetsial}\hfill

  Let $\BG = (V,W)$ be a split reflection coset.
  \begin{enumerate}
    \item
      There is at most one $1$-cyclotomic spetsial Hecke algebra $\CH_W(\id_V)$
      of compact type (resp. of
      noncompact type), namely the algebra $\CH_W$ (resp $\CH_W^\nc$).
     \item
      $(V,W)$ is spetsial at $1$ if and only if it is spetsial according to \cite[\S 3.9]{maS}.
  \end{enumerate}
\end{proposition}

\begin{proof}\hfill

  We only consider the compact type case. 
  
  (1)
  Since $w\vp = 1$, we have $W(\vp) = W$, $\CA(w\vp) = \CA$, and for each $H\in\CA$
  we have $e_{W_H}/e_H = 1$.
  
  Hence by Definition \ref{localcyclo}, (\textsc{ca1}), we have
  $P_H(t,x) \in K[t,x]$, and by Definition \ref{localcyclo}, (\textsc{cs1}), we see that
  $P_H(t,x)$ is divisible by $(t-x)$ and
  $$
    P_H(t,x) = (t-x)Q_H(t)
  $$
  for some $Q_H(t) \in K[t]$.
  
  By Definition \ref{localcyclo}, (\textsc{ca2}), we see that 
  $Q_H(t) = (t^{e_H-1}+\dots+t+1)$.
  
  (2)
  If the algebra $\CH_W$ is a $1$-cyclotomic spetsial Hecke algebra,
  it follows from Definition \ref{longdefinition}, Global conditions, (\textsc{sc1}),
  that its Schur elements belong to $K(x)$. This shows that $(V,W)$ is spetsial
  according to \cite[\S 3.9]{maS} by condition (ii) of \cite[Prop. 3.10]{maS}.
  
  Reciprocally, assume that $(V,W)$ is spetsial according to \cite[\S 3.9]{maS}.
  Then we know that all parabolic subgroups of $W$ are still spetsial according
  to \cite[\S 3.9]{maS} (see \eg\ \cite[Prop. 7.2]{maG}). The only properties
  which are not straightforward to check among the list of assertions in
  Definition \ref{longdefinition} are the properties concerning the splitting fields of
  algebras. Since the spetsial groups according to \cite[\S 3.9]{maS}
  are all well--generated, these properties hold by \cite[Cor. 4.2]{maR}.
\end{proof}
\smallskip

\subsubsection{Spetsial reflection groups}\label{spetsialgroups}\hfill
\smallskip

  Let $(V,W)$ be a reflection group on $\BC$. Assume that the corresponding
  reflection representation is defined over a number field $K$ (so that $\BQ_W\subseteq K$).
  
  The  algebra $\CH_W$
  has been defined above (\ref{HW}).
  
  Let $v$ be such that $v^{|\mu(K)|} = x$. Whenever $\chi$ is an absolutely irreducible
  character of the spetsial algebra $\CH_W$,
  \begin{itemize}
    \item
      we denote by $S_\chi$ the corresponding Schur element 
      (so $S_\chi \in K[v,v\inv]$),
    \item
      We denote by $1$ the unique character of $\CH_W$ whose Schur
      element is the Poincar\'e polynomial of $W$ 
      (see Definition \ref{localcyclo}, (\textsc{cs2'})).
      The degree of a character $\chi$ of $\CH_W$ is 
      (see Definition \ref{localcyclo}, (\textsc{sc3}))
      $$
        \Deg(\chi) = \dfrac{\Feg(R_1^\BG)}{S_\chi}
      $$
      and in particular
      $
        \Deg(1) = 1
        \,.
      $
  \end{itemize}
   
  The following theorem 
  (see \cite[Prop. 8.1]{maG}) has been proved by a case-by-case analysis.
  
\begin{theorem}\label{maintheorem}\hfill

  Assume Theorem--Conjecture \ref{proprHecke} holds.
  \begin{enumerate}
    \item
    The following assertions are equivalent.
    \begin{itemize}
      \item[(i)]
        For all $\chi\in\Irr(\CH_W)$, we have
        $\Deg(\chi)(v) \in K(x)$.
      \item[(ii)]
        $W$ is a product of some of the following reflection groups:
        \begin{itemize}
           \item
            $G(d,1,r)\,\,(d,r\ge 1)$, $G(e,e,r)\,\,(e,r \ge 2)$,
           \item         
            one of the well-generated exceptional groups 
            $G_i$ with $4\le i\le 37$ generated by true
            reflections,
         \item          
           $G_4$, $G_6$, $G_8$, $G_{14}$, $G_{25}$,
           $G_{26}$, $G_{32}$.
      \end{itemize}
    \end{itemize}
    \item
      If the preceding properties hold, then
      the specialization $S_1(x)$ of the generic Poincar\'e
      polynomial equals the ``ordinary'' Poincar\'e 
      polynomial of $W$:
      $$
        S_1(x) = \prod_{j=1}^{r} (1+x+\dots+x^{d_j-1})
        \,,
      $$
      where $r = \dim V$ and $d_1,\dots,d_r$ are the degrees of $W$
      (see \ref{generalizeddegrees}).
  \end{enumerate}
\end{theorem}

\begin{remark}\label{spetsialpara}\hfill

  \begin{itemize}
    \item
      A spetsial group of rank $r$ is well-generated,
      but not all well-generated reflection groups of rank $r$ are spetsial.
    \item
      From the classification of spetsial groups, it follows that all parabolic
      subgroups of a spetsial group are spetsial.
   \end{itemize}
\end{remark}
\smallskip

\subsubsection{Rouquier blocks of the spetsial algebra: special characters}\hfill
\smallskip

  Let $\BG = (V,W)$ be a spetsial split reflection coset on $K$.
  
  For $\th\in\Irr(W)$, we recall (see \ref{moduleandfake}) that the fake degree 
  $\Feg_\BG(R_\th)$ of the class function $R_\th$ on $W$ (which in the split case 
  coincides with $\th$) is the graded multiplicity of $\th$ in the graded regular
  representation $KW^\gr$.

  Let $\CH_W$ be the $1$-cyclotomic spetsial Hecke algebra
  (see Proposition \ref{characspetsial}).
  Let us choose an indeterminate $v$ such that $v^{|ZW|} = x$.
  We denote by
  $$
    \Irr(W) \iso \Irr(\CH_W) \,\,,\,\, \th \mapsto \chi_\th\,,
  $$
  the bijection defined by the specialization $v \mapsto 1$.

\begin{notation}\label{abandAB}\hfill

  For $\th\in\Irr(W)$ we define the following:
\begin{enumerate}
  \item 
    $a_\th$ and $A_\th$ :
  \begin{itemize}
    \item
      we denote by $a_\th$ the valuation of $\,\Deg(\chi_\th)$ 
      (\ie\ the largest integer such that $x^{-a_\th}\Deg(\chi_\th)$
      is a polynomial),
    \item
      and by $A_\th$ the degree of $\,\Deg(\chi_\th)$,
  \end{itemize}
    \item
      $b_\th$ and $B_\th$ :
  \begin{itemize}
    \item
      we denote by $b_\th$ the valuation of $\,\Feg_\BG(R_\th)$,
    \item
      and by $B_\th$ the degree of $\,\Feg_\BG(R_\th)$.
  \end{itemize}
\end{enumerate}
\end{notation}  
    
\begin{definition}\label{special}\hfill

    We say that $\th \in \Irr(W)$ is \emph{special} if $a_\th = b_\th$.
\end{definition}

  Let us recall (see Theorem \ref{aAconstant}) that $a_\th$ and $A_\th$ are 
  constant if $\chi_\th$
  runs over the set of characters in a given Rouquier block of $\Irr(\CH_W)$.
  Then if $\CB$ is a Rouquier block, we denote by $a_\CB$ and $A_\CB$
  the common value for $a_\th$ and $A_\th$ for $\chi_\th \in \CB$.

  The following result is proved in \cite[\S 5]{maro} (under certain assumptions
  for some of the exceptional spetsial groups), using essential tools from
  \cite[\S 8]{maG}.

\begin{theorem}\label{marouaA}\hfill

  Assume that $W$ is spetsial. Let $\CB$ be a Rouquier block (``family'')
  of the 1-spetsial algebra $\CH_W$.
  \begin{enumerate}
    \item
      $\CB$ contains a unique special character $\chi_{\th_0}$.
    \item
      For all $\th$ such that $\chi_\th \in \CB$, we have
      $$
        a_\CB \leq b_\th \quad\text{and}\quad B_\th \leq A_\CB
        \,.
      $$
  \end{enumerate}
\end{theorem}
\smallskip

%
\newpage

{\red \section{\red Axioms for spetses}\hfill
\smallskip
}

  Our goal is to attach to $\BG = (V,W\vp)$ (where $W$ is a spetsial reflection group)
  an abstract set of \emph{unipotent characters of $\BG$},  to each element of which we associate
  a \emph{degree} and
  an \emph{eigenvalue of Frobenius}. In the case where $\BG$
  is rational, these are the set of unipotent characters with their generic degrees and corresponding eigenvalues
  of Frobenius attached to the associated reductive groups. 
  
  These data have to satisfy certain axioms that we proceed to give below.

   We hope to give a general construction satisfying these axioms in a subsequent paper.
   For the time being,
   we only know how to attach that data to the particular case where $\BG$ has a split
   semi simple part (see below \S 6); the construction in that case is the
   object of \S 6.
\smallskip

\subsection{Axioms used in \S 6}\label{axiomsfirst}\hfill
\smallskip

\subsubsection{Compact and non compact types: Unipotent characters, degrees and eigenvalues}\hfill
\smallskip

  Given $\BG$ as above, we shall construct two finite sets:
  \begin{itemize}
    \item
      the set $\Uch^\oc(\BG)$\index{UchGc@$\Uch^\oc(\BG)$} of unipotent
      characters of compact type,
    \item
      the set $\Uch(\BG_\nc)$\index{UchGnc@$\Uch(\BG_\nc)$} of unipotent
      characters of noncompact type,
  \end{itemize}
  each of them (denoted $\Uch(\BG)$\index{UchBF@$\Uch(\BG)$} below), endowed with two maps
  \begin{itemize}
    \item 
       the map called \emph{degree} 
       $$
          \Deg : \Uch(\BG) \ra K[x]
          \,\,,\,\,
          \rho\mapsto\Deg(\rho) \,,
       $$ 
    \item
       the map called \emph{Frobenius eigenvalue} and denoted $\Fr$,
       which associates to each element $\rho\in\Uch(\BG)$ a monomial
       of the shape $\Fr(\rho) = \la_\rho x^{\nu_\rho}$ where
       \begin{itemize}
         \item
           $\la_\rho$ is a root of unity,
         \item
           $\nu_\rho$ is an element of $\BQ/\BZ$,
       \end{itemize}
  \end{itemize}
  with a bijection
  $$
    \Uch^\oc(\BG) \iso \Uch(\BG_\nc)
    \,\,,\,\,
    \rho \mapsto \rho^\nc\,,
  $$
  such that
\smallskip

  \begin{enumerate}\label{receipe}
    \item
      $\Deg({\rho^\nc}) = x^{\Nr_W} \Deg(\rho)^\vee\,,$
\smallskip
      
    \item
      $\Fr(\rho) \Fr({\rho^\nc}) = 1\,,$
  \end{enumerate}
  and subject to many further axioms to be given below.
\medskip
    
  In what follows, we shall construct the ``compact type case"  $\Uch^\oc(\BG)$ (which will be
  denoted simply $\Uch(\BG)$). 
  The noncompact type case
  can be obtained using the above properties of the bijection $\rho\mapsto\rho^\nc$.
\smallskip

\subsubsection{Basic axioms}\hfill
\smallskip

\begin{axioms}\label{prod}\hfill

  \begin{enumerate}
    \item
      If $\BG=\BG_1\times\BG_2$
      (with split semi-simple parts)
      then we have
      $\Uch(\BG)\simeq\Uch(\BG_1)\times\Uch(\BG_2)$. If we write
      $\rho=\rho_1\otimes\rho_2$ this product decomposition, then
      $\Fr(\rho)=\Fr(\rho_1)\Fr(\rho_2)$ and
      $\Deg(\rho)=\Deg({\rho_1})\Deg({\rho_2})$.
    \item
      A torus has a unique unipotent character $\Id$, with $\Deg(\Id) = \Fr(\Id)=1$.
  \end{enumerate}
\end{axioms}

  \begin{axiom}\label{degdividesG}\hfill
  
    For all $\rho \in  \Uch^\oc(\BG)$, $\Deg(\rho)$ divides $|\BG|_\oc$.
    
    For all $\rho \in  \Uch(\BG_\nc)$, $\Deg(\rho)$ divides $|\BG|_\nc$
   \end{axiom}
\smallskip
   
  \begin{axiom}\label{autom}\hfill
  
    There is an action of $N_{\GL(V)}(W\vp)/W$ on $\Uch(\BG)$ in a way
     which preserves $\Deg$ and $\Fr$.
  \end{axiom}

  The action mentioned above will be
  determined more precisely below by some further axioms (see
  \ref{1Harish-Chandra}(2)(a)).

\begin{remark}\hfill

  We recall that parabolic subgroups of spetsial groups are spetsial
  (see \ref{spetsialpara} above).
  For a Levi $\BL$ of $\BG$, we set
  $
    W_\BG(\BL) := N_W(\BL)/W_\BL
    \,.
  $  
  We have by  \ref{autom} a well-defined action of
  $W_\BG(\BL)$  on  $\Uch(\BL)$,  which  allows  us  to define for
  $\lambda\in\Uch(\BL)$ its stabilizer $W_\BG(\BL,\lambda)$.
\end{remark}
\smallskip

\subsubsection{Axioms for the principal $\zeta$-series}\label{axiomssecond}\hfill
\smallskip

\begin{definition}\label{principalzetaseriesdef}\hfill

  Let $\z \in \bmu$ and let $\P$ its minimal polynomial on $K$ (a $K$-cyclotomic polynomial).
  \begin{enumerate}
    \item
      The \emph{$\z$-principal series} is the set of unipotent characters of $\BG$ defined by
      $$
        \Uch(\BG,\z) := \{ \rho\in \Uch(\BG)\,\mid\, \Deg(\rho)(\z) \neq 0\,\}
         \,.
      $$
    \item
      An element $\rho\in\Uch(\BG)$ is said to be \emph{$\z$-cuspidal} or \emph{$\P$-cuspidal}  if
      $$
        \Deg(\rho)_\P = \dfrac{|\BG|_\P}{|Z\BG|_\P}
        \,.
      $$
  \end{enumerate}
\end{definition}

  Let us recall that, given a spetsial $\z$-cyclotomic Hecke
  algebra (either of compact or of noncompact type) $\CH_W(w\vp)$ associated with 
  a regular element $w\vp$,
  each irreducible character $\chi$ of $\CH_W(w\vp)$ comes equipped with a degree
  $\Deg(\chi)$ and a Frobenius eigenvalue $\Fr(\chi)$ (see \S 4 above).

\begin{axiom}\label{principalzetaseriesaxio}\hfill

  Let $w\vp \in W\vp$ be $\z$-regular.
      There is a spetsial $\z$-cyclotomic Hecke
      algebra of compact type $\CH_W(w\vp)$ associated with $w\vp$, a bijection
      $$
        \Irr ( \CH_W(w\vp) ) \iso \Uch^\oc(\BG,\z)
        \,\,,\,\,
        \chi \mapsto \rho_{\chi} \,,
      $$
      and a collection of signs $(\e_{\chi})_{\chi\in\Irr(\CH_W(w\vp))}$
      such that
      \begin{enumerate}
       \item
         that bijection is invariant under the action of $N_{\GL(V)}(W\vp)/W$\,, 
        \item
          $\Deg({\rho_{\chi}}) = \e_{\chi}\Deg({\chi})\,,$
        \item
          $\Fr({\rho_{\chi}}) \equiv \Fr({\chi}) \mod x^\BZ \,.$
      \end{enumerate}
\end{axiom}

\begin{remark}\label{degquotient}\hfill

  By \ref{localcyclo}(sc3), we see that condition (2) above is equivalent to
  $$
    \Deg(\rho_\chi) = \epsilon_\chi\dfrac{\Feg(R_{w\vp)}}{S_\chi}
    \,.
  $$
\end{remark}
\smallskip

\subsubsection{On Frobenius eigenvalues}\hfill
\smallskip

  For what follows, we use freely \S4, and in particular \S\ref{paragraphonFrob}.
  
  We denote by $v$ an indeterminate such that $\CH_W(w\vp)$ splits over
  $\BC(v)$ and such that $v^h = \z\inv x$ for some integer $h$. Then the
  specialization $v \mapsto 1$ induces a bijection
  $$
    \Irr(W(w\vp)) \iso \Irr(\CH_W(w\vp))\,\,,\,\, \th \mapsto \chi_\th
    \,.
  $$
  
  Assume $\z = \exp(2\pi ia/d)$ and let $\brh \in Z\bB_W(w\vp)$ be such
  that 
  $\brh^d=\bpi^{a\de}$. 
  Then we have the following equality modulo integral powers of $x$~:
\begin{equation}\label{frob}
  \begin{aligned}
    \Fr(\rho_{\chi_\th}) 
                 &= \Fr^\brh(\chi_\th)
                      = \zeta^{l(\brh)}\omega_{\chi_\th}(\brh) \\
                & = \zeta^{l(\brh)}\omega_{\th}(\rh)
                   (\zeta\inv x)^{l(\brh)-(a_{\rho_{\chi_\th}}+A_{\rho_{\chi_\th}})
                   \frac{l(\brh)}{l(\bpi)}} \\
                &= \omega_{\th}(\rh)
                 (\zeta\inv x)^{-(a_{\rho_{\chi_\th}}+A_{\rho_{\chi_\th}})\frac{l(\brh)}{l(\bpi)}}.
  \end{aligned}
\end{equation}

  The last formula should be interpreted as follows: for the power of $x$, one
  should take $-(a_{\rho_{\chi_\th}}+A_{\rho_{\chi_\th}})\frac{l(\brh)}{l(\bpi)}$ modulo
  $1$, and if $\chi_\th$ has an orbit under $\bpi$ of length $k$, then
  $\zeta^{(a_{\rho_{\chi_\th}}+A_{\rho_{\chi_\th}})\frac{l(\brh)}{l(\bpi)}}$ should be
  interpreted as attributing to the elements of the orbit of $\chi_\th$ the
  $k$-th roots of
  $\zeta^{k(a_{\rho_{\chi_\th}}+A_{\rho_{\chi_\th}})\frac{l(\brh)}{l(\bpi)}}$, a
  well-defined expression since the exponent is integral.
\smallskip

\subsubsection{Some consequences of the axioms}\hfill
\smallskip

  1.
  Let us recall (see \ref{degquotient} above) that
  $$
    \Deg({{\chi_\th}}) = \dfrac{\Feg(R_{w\vp})}{S_{\chi_\th}}
    \,.
  $$
  Since $\BC\CH_W(w\vp)$  specializes to $\BC W(w\vp)$ for $v\mapsto 1$,  
  we have 
  $$
    S_{\chi_\th}(\zeta) = \dfrac{|W(w\vp)|}{\th(1)}
    \,.
  $$ 
  We also have
  $\Feg(R_{w\vp})(\zeta)=|W(w\vp)|$ 
  (see Proposition \ref{fakeonregular}, (2)).
  Thus we get
  \begin{equation}\label{degatzeta}
    \Deg({\rho_{\chi_\th}})(\zeta)=\varepsilon_{\chi_\th} \th(1)
    \,.
  \end{equation}
\smallskip

  2.
  Let us denote by $\tau_W(w\vp)$ the canonical trace form of the algebra
  $\CH_W(w\vp)$.
  
  By definition of spetsial cyclotomic Hecke algebras (see 
  Definition \ref{localcyclo}, \textsc{(sc3)}),
  we have the following equality between linear forms 
  on $\BC(x)\CH_W(w\vp)$:
  $$
     \Feg(R_{w\vp})\tau_W(w\vp) = 
        \!\! \sum_{\chi\in\Irr(\CH_W(w\vp))}\!\!\ \Deg({\chi}) \chi
     \,,
  $$
  and taking the value at $1$ we get
  $$
    \begin{aligned}
     \Feg(R_{w\vp}) 
        &= \!\! \sum_{\chi\in\Irr(\CH_W(w\vp))}\!\!\ \chi(1) \Deg({\chi}) \\
        &= \!\! \sum_{\chi\in\Irr(\CH_W(w\vp))}\!\!\ \e_{\chi}\chi(1) \Deg(\rho_{\chi})
         \,.
     \end{aligned}
  $$
  
  Using the bijection
  $$
    \Irr(W(w\vp)) \iso \Irr(\CH_W(w\vp))
    \quad,\quad
    \th \mapsto \chi_\th
    \,,
  $$
  and setting $\e_\th := \e_{\chi_\th}$ for $\th\in\Irr(W(w\vp))$, we get
  \begin{equation}\label{alternatedsum}
     \Feg(R_{w\vp}) 
        = \!\! \sum_{\th\in\Irr(W(w\vp))}\!\!\ \e_\th\th(1) \Deg(\rho_{\chi_\th})
     \,.
  \end{equation}
\smallskip

  Finally, we note yet another numerical consequence of 
  Axiom \ref{principalzetaseriesaxio}.

\begin{notation}\label{aArho}\hfill

  For $\rho\in\Uch(\BG)$, let us denote (see above \ref{defaA}) by
  $a_\rho$ and $A_\rho$ respectively the valuation and the degree of 
  $\Deg(\rho)$ as a polynomial in $x$. We set
  $
    \de_\rho := a_\rho + A_\rho
    \,.
  $
\end{notation}

\begin{corollary}\label{invariances}\hfill
  
  Let $\z_1,\z_2\in\bmu$.
  If $\rho \in \Uch(\BG,\z_1) \cap  \Uch(\BG,\z_2)$, then 
  $
    \z_1^{\de_\rho} =  \z_2^{\de_\rho}
    \,.
  $
\end{corollary}

\begin{proof}\hfill

      By Lemma \ref{deltasigma}, (1), we see that whenever $\rho \in \Uch(\BG,\z)$, we have
      $
        \Deg(\rho)^\vee = (\z\inv x)^{-\de_\rho} \Deg(\rho)
        \,,
      $
      which implies the corollary.
\end{proof}
\smallskip
  
\subsubsection{Ennola transform}\hfill
\smallskip

\begin{axiom}\label{Ennola}\hfill

  For $\bz\in Z(\bB_W)$ with image $z\in Z(W)$, the algebra $\CH_W(zw\vp)$
  is the image of $\CH_W(w\vp)$ by the Ennola transform explained in 2.23.
  If $\xi\zeta$ and $\zeta$ are the corresponding regular eigenvalues,
  this defines a correspondence $E_\bz$ (a well-defined bijection except for
  irrational characters) between $\Uch(\BG,\zeta)$ and 
  $\Uch(\BG,\Id)$, such that 
  $$
     \Deg(E_\bz(\rho))(x)= \pm\Deg(\rho)(z\inv x)
  $$
  and $\Fr(E_\bz(\rho))/\Fr(\rho)$
  is given by \ref{ennolaandfrobeniusgeneral} taken modulo $x^\BZ$.
\end{axiom}
\smallskip

\subsubsection{Harish-Chandra series}\hfill
\smallskip

  Here we define a particular case of what will be called
  ``$\P$-Harish-Chandra series'' in the next section.

 \begin{definition}\label{defcuspidal}\hfill

  We call cuspidal pair for $\BG$ a pair $(\BL,\la)$ where
  \begin{itemize}
    \item
      $\BL$ is Levi subcoset of $\BG$ of type $\BL = (V,W_\BL\vp)$
      ($W_\BL$ is a parabolic subgroup of $W$), and
    \item
      $\la\in\Uch(\BL)$ is 1-cuspidal.
  \end{itemize}
\end{definition}

\begin{remark}\hfill
  
  \begin{itemize}
    \item By remark \ref{spetsialpara} a parabolic subgroup of a spetsial 
    group is spetsial thus it makes sense to consider $\Uch(\BL)$.
    \item
      A Levi subcoset $\BL$ has type $(V,W_\BL\vp)$ if and only if it
      is the centralizer of the $1$-Sylow subcoset of its center $Z\BL$.
    \item
      It can be checked case by case that whenever $(\BL,\la)$ is a cuspidal
      pair for $\BG$, then
      $W_\BG(\BL,\lambda)$ is a reflection group on the orthogonal of the
      intersection of the hyperplanes of $W_\BL$, which gives a meaning to (2)
      below.
  \end{itemize}
\end{remark}
  
\begin{axioms}[Harish--Chandra theory]\label{1Harish-Chandra}\hfill

\begin{enumerate}
    \item
      There is a partition
      $$
        \Uch(\BG) = \bigsqcup_{(\BL,\lambda)} \Uch_\BG(\BL,\lambda)
      $$
      where $(\BL,\lambda)$ runs over a complete set of representatives of the orbits
      of $W$ on cuspidal pairs for $\BG$.
    \item
      For each cuspidal pair $(\BL,\la)$, there is a $1$-cyclotomic Hecke algebra
      $\CH_\BG(\BL,\lambda)$ associated to $W_\BG(\BL,\lambda)$, 
      an associated bijection
      $$
        \Irr(\CH_\BG(\BL,\lambda)) \iso \Uch_\BG(\BL,\lambda)
         \,\,,\,\,
        \chi \mapsto \rho_{\chi} \,,
      $$
      with the following properties.
      \begin{enumerate}
        \item
          Those bijections are invariant under $N_{\GL(V)}(W\vp)/W$.
        \item
          If we denote by $S_{\chi}$ the Schur element of the
          character $\chi$ of $\CH_\BG(\BL,\lambda)$, we have
          $$
            \Deg({\rho_{\chi}})=  \dfrac{\Deg(\lambda)(|\BG|/|\BL|)_{x'}}{S_{\chi}}
            \,.
          $$
       \item
         If $\BG$ is assumed to have a split semisimple part
         (see below \S 6), for $\BT_{\vp} : =(V,\vp)$ the  corresponding maximal torus, 
          the algebra
          $\CH_\BG(\BT_{\vp},\Id)$ is the $1$-cyclotomic spetsial Hecke algebra
          $\CH_W$ and 
	  $\Uch_\BG(\BT_\vp,\Id)=\Uch(\BG,1)$. The bijection
          $\chi \mapsto \rho_{\chi}$ is the same as
	  that in \ref{principalzetaseriesaxio}, in particular
	  the signs $\e_\chi$ in \ref{principalzetaseriesaxio}
	  are $1$ when $\zeta=1$.
       \item
          For all $\chi\in \Irr(\CH_\BG(\BL,\lambda))$, we have 
          $\Fr(\rho_{\chi})=\Fr(\la)$.
       \end{enumerate}
    \item 
      What precedes is compatible with a product decomposition as
      in \ref{prod}(1).
  \end{enumerate}
\end{axioms}
\smallskip

\begin{remark}\hfill

  Since the canonical trace form $\tau$ of $\CH_\BG(\BL,\la)$ satisfies the formula
  $$
    \tau = \sum_\chi \dfrac{1}{S_\chi} \chi
    \,,
  $$
  it follows from formula (2)(b) above that
  \begin{equation}\label{substitutRLG}
    \Deg(\la)\dfrac{|\BG|_{x'}}{|\BL|_{x'}} = \sum_\chi \Deg(\rho_\chi) \chi(1)
    \,.
  \end{equation}
\end{remark}
\smallskip

\subsubsection{Reduction to the cyclic case}\hfill
\smallskip

  Assume that $\BL = (V,W_\BL \vp)$, and
  let $H$ be a reflecting hyperplane for $W_\BG(\BL,\lambda)$.
  We denote by $\BG_H$ the ``parabolic reflection subcoset'' of $\BG$
  defined by $\BG_H := (V,W_H\vp)$ where $W_H$ is the fixator (pointwise stabilizer)
  of $H$.
  Then $W_{\BG_H}(\BL,\lambda)$ is cyclic and contains a unique distinguished
  reflection (see \ref{notationforcrg}) of  $W_\BG(\BL,\lambda)$.

\begin{axiom}\label{U3}\hfill

  In the above situation the  parameters  of
  $\CH_{\BG_H}(\BL,\lambda)$  are the  same as the parameters
  corresponding to $H$ in $\CH_\BG(\BL,\lambda)$.
\end{axiom}

  This  allows us to reduce  the determination of the  parameters of
  $\CH_\BG(\BL,\lambda)$ to
  the case where $W_\BG(\BL,\lambda)$ is cyclic.
\smallskip

\subsubsection{Families of unipotent characters}\hfill
\smallskip

\begin{axioms}\label{U1}\hfill

  There is a partition
  $$
    \Uch(\BG) = \bigsqcup_{\CF\in\Fam(\BG)} \CF
  $$
  with the following properties.
  \begin{enumerate}
    \item
      Let $\CF\in\Fam(\BG)$. Whenever $\rho,\rho'\in \CF$, we have
      $$
        a_\rho = a_{\rho'}\,\,\text{ and }\,\,A_\rho = A_{\rho'}
        \,.
      $$
    \item
      Assume that $\BG$ has a split semisimple part.
      For $\CF\in\Fam(\BG)$, let us denote by
      $\CB_\CF$ the Rouquier block of the 1-cyclotomic spetsial Hecke algebra
      $\CH_W$ defined by $\CF \cap \Uch(\BG,1)$. Then
      $$
        \sum_{\rho\in\CF}\Deg(\rho)(x)\Deg(\rho^*)(y)=\!\!\!
        \sum_{\th\in\Irr(W)\,\mid\,\chi_\th \in\CB_\CF} \Feg_\BG(R_{\th})(x)\Feg_\BG(R_{\th})(y)
        \,.
       $$
    \item
      The partition of $\Uch(\BG)$ is globally stable by Ennola transforms. 
  \end{enumerate}
\end{axioms}
\smallskip

\begin{remark}\hfill
  
  Evaluating \ref{U1}(2) at the eigenvalue $\zeta$ of a regular element $w\vp$, we get 
  using \ref{principalzetaseriesdef} (1) and \ref{degatzeta}
  \begin{equation}\label{byfam}
    \sum_{\theta\in\Irr(W)\mid \chi_\theta\in\CB_\CF} |\Feg_\BG(R_\theta)(\zeta)|^2
     = \!\!
     \sum_{\{\th\in\Irr(W(w\vp))\,\mid\, \rho_{\chi_\th}\in\CF\}}
     \th(1)^2
     \,.
  \end{equation}
\end{remark}
\smallskip

\begin{remark}[When $W(w\vp)$ has only one class of hyperplanes]
\label{oneclass}\hfill
\smallskip

  If $W(w\vp)$ has only one class of hyperplanes,
  the algebra $\CH_W(w\vp)$ is defined by a family of parameters
  $\zeta_j v^{hm_j}$, which are
  in bijection with the linear characters of $W(w\vp)$.
  If $\th$ is a linear character of $W(w\vp)$, let $m_\th$ be the corresponding
  $m_j$. We have from Lemma \ref{mandsigma}, (2) that 
  $$
    \Nr_W+\Nh_W-a_{\rho_{\chi_\th}}-A_{\rho_{\chi_\th}} =e_{W(w\vp)}m_\th
    \,.
  $$ 
  If $\CF$ is the family of $\rho_\chi$, this can be written
  $$
    m_\th = (\Nr_W+\Nh_W-\delta_\CF)/e_{W(w\vp)}
    \,.
  $$
  Thus formula \ref{byfam}, since we know its left-hand side,
  gives a majoration (a precise value when $W(w\vp)$ is cyclic)
  of the number of $m_\th$ with a given value (equal to the
  number of $\theta$ such that $\rho_{\chi_\theta}\in\CF$
  with a given $\delta_\CF$).
\end{remark}
\smallskip

\subsection{Supplementary axioms for spetses}\label{supplaxioms}\hfill
\smallskip

  In this section, we state some supplementary axioms which should be true for the data
  (unipotent characters, degrees and Frobenius eigenvalues, families,
  Ennola transforms) that we hope to construct for any reflection coset
  $\BG = (V,W\vp)$ where $W$ is spetsial.
  
  On the data presented in \S 6 and appendix below we have checked \ref{galoisfrob}.

\subsubsection{General axioms}\hfill
\smallskip
  
  \begin{axiom}\label{galoisfrob}\hfill
  
    The Frobenius eigenvalues are globally invariant under the Galois group
    $\Gal(\overline\BQ/K)$.
  \end{axiom}
\smallskip

\begin{axiom}\label{nprincipalzetaseriesaxio}\hfill

  Let $w\vp \in W\vp$ be $\z$-regular.
      There is a bijection
      $$
        \Irr ( \CH_W^\nc(w\vp) ) \iso \Uch(\BG_\nc,\z)
        \,\,,\,\,
        \chi \mapsto \rho_{\chi}^\nc
      $$
      and a collection of signs $(\e_{\chi}^\nc)_{\chi\in\Irr(\CH_W(w\vp))}$
      such that
      \begin{enumerate}
        \item
         that bijection is invariant under the action of $N_{\GL(V)}(W\vp)/W$. 
        \item
          $\Deg({\rho_{\chi}^\nc}) = \e_{\chi}^\nc \z^{\Nh_W} \Deg({\chi^\nc})\,,$
        \item
          $\Fr({\rho_{\chi}^\nc}) \equiv \Fr({\chi^\nc})\pmod{x^\BZ}\,.$
      \end{enumerate}   
\end{axiom}

\begin{remark}[The real case]\hfill

      By \ref{discriminanttrivial}, if $\BG$ is defined over $\BR$, 
      then $\z^{\Nh_W} = \pm 1$, hence
      we have $\Deg({\rho_{\chi}^\nc})(x) = \pm  \Deg({\chi^\nc})(x)\,.$
\end{remark}
\smallskip

\subsubsection{Alvis--Curtis duality}\hfill
\smallskip

\begin{axiom}\label{alviscurtis}\hfill

\begin{enumerate}
  \item
    The map
    $$
      \Uch(\BG_\oc) \iso \Uch(\BG_\nc)
       \,\,,\,\,
       \rho \mapsto \rho^\nc\,,
    $$
    is stable under the action of $N_{\GL(V)}(W\vp)/W$.
  \item
    In the case where $\BG = (V,W)$ is split and $W$ is generated by true reflections,
    we have (by formulae \ref{polynomialorders}) $|\BG|_\nc = |\BG|_\oc$. In that case
    $$
     \Uch(\BG_\nc) = \Uch(\BG_\oc)
     \,,
    $$
    a set which we denote $\Uch(\BG)$, 
    and the map
    $$
      D_\BG :
      \left\{
      \begin{aligned}
        &\Uch(\BG) \iso \Uch(\BG) \\
        &\rho \mapsto \rho^\nc\,,
      \end{aligned}
      \right.
    $$
    is an involutive permutation such that
    \begin{enumerate}\label{nreceipe2}
      \item
        $\Deg({D_\BG(\rho)})(x) = x^{\Nr_W} \Deg(\rho)(x)^\vee\,,$
      \item
        $\Fr(\rho) \Fr({D_\BG(\rho)}) = 1\,,$
    \end{enumerate}
    called the Alvis--Curtis duality.
\end{enumerate}
\end{axiom}
\smallskip

  By definition of the $\z$-series, given the property connecting the degree of
  $\rho^\nc$ with the degree of $\rho$, it is clear that the map $\rho \mapsto \rho^\nc$
  induces a bijection
  $$
    \Uch(\BG_\oc,\z) \iso \Uch(\BG_\nc,\z)
    \,.
  $$
  In particular, if $\BG = (V,W)$ is split and $W$ is generated by true reflections,
  the Alvis--Curtis duality (see \ref{alviscurtis}, (2)) induces an involutive
  permutation of $\Uch(\BG,\z)$.
  
  This is expressed by a property of the corresponding spetsial $\z$-cyclotomic
  Hecke algebra.

\begin{axiom}\label{nalviscurtiszeta}\hfill

  Assume $\BG = (V,W)$ is split and $W$ is generated by true reflections.
  Let $w \in W$ be a $\z$-regular element, let $\CH_W(w)$ be the associated
  spetsial $\z$-cyclotomic Hecke algebra.
  \begin{enumerate}
    \item
      There is an involutive permutation
      $$
        D_W(w) : \Irr(\CH_W(w)) \iso \Irr(\CH_W(w))
      $$
      with the following properties, for all $\chi \in \Irr(\CH_W(w))$:
      \begin{enumerate}
        \item
          $\Deg({D_W(w)(\chi)}) = x^{\Nr_W}\Deg(\chi) \,,$
        \item
          $\Fr({D_W(w)(\chi)})\Fr(\chi) = 1\,.$
      \end{enumerate}
    \item
      This is a consequence of the following properties of the parameters of
      $\CH_W(w)$. Assume that $\CH_W(w)$ is defined by the family of polynomials
      $$
        \left(
        P_I(t,x) = \prod_{j=0}^{j=e_I-1} \left( t- \z_{e_I}^j(\z\inv x)^{m_{I,j}} \right)
        \right)_{I\in\CA_W(w)}
        \,.
      $$
      Then for all $I\in\CA_W(w)$,
      there is a unique $j_0$ ($0\leq j_0\leq e_I-1$) such that $m_{I,j_0} = 0$,
      and for all $j$ with $0\leq j \leq e_I-1$,
      we have
      $$
        m_{I,j} + m_{I,j_0-j} = m_I
         \,.
       $$
  \end{enumerate}
\end{axiom}
\smallskip

\subsubsection{$\P$-Harish-Chandra series}\hfill
\smallskip

  Let $\P$ be a $K$-cyclotomic polynomial.

 \begin{definition}\label{ndefcuspidal}\hfill

  We call $\P$-cuspidal pair for $\BG$ a pair $(\BL,\la)$ where
  \begin{itemize}
    \item
      $\BL$ is the centralizer of the $\P$-Sylow subdatum of its center $Z\BL$,
    \item
      $\la$ a $\P$-cuspidal unipotent character of $\BL$, \ie\ (see \ref{principalzetaseriesdef})
      $$
        \Deg(\la)_{\P} = \dfrac{|\BL|_\P}{|Z\BL|_\P}  
        \,.
      $$
  \end{itemize}
\end{definition}

\begin{axiom}\label{cusprefl}\hfill

  Whenever $(\BL,\la)$ is a $\P$-cuspidal
  pair for $\BG$, then
  $W_\BG(\BL,\lambda)$ is a reflection group on the orthogonal of the
  intersection of the hyperplanes of $W_\BL$.
\end{axiom}
  
\begin{axioms}[$\P$-Harish--Chandra theory]\label{nHarish-Chandra}\hfill

\begin{enumerate}
    \item
      There is a partition
      $$
        \Uch(\BG) = \bigsqcup_{(\BL,\lambda)} \Uch_\BG(\BL,\lambda)
      $$
      where $(\BL,\lambda)$ runs over a complete set of representatives of the orbits
      of $W$ on $\P$-cuspidal pairs of $\BG$.
    \item
      For each $\P$-cuspidal pair $(\BL,\la)$, there is a $\P$-cyclotomic Hecke algebra
      $\CH_\BG(\BL,\lambda)$, an associated bijection
      $$
        \Irr(\CH_\BG(\BL,\lambda)) \iso \Uch_\BG(\BL,\lambda)
         \,\,,\,\,
        \chi \mapsto \rho_{\chi} \,,
      $$
      and a collection $(\e_{\chi})_{\chi\in\Irr(\CH_\BG(\BL,\lambda))}$ of signs,
      with the following properties.
      \begin{enumerate}
        \item
          Those bijections are invariant under $N_{\GL(V)}(W\vp)/W$.
        \item
          If we denote by $S_{\chi}$ the Schur element of the
          character $\chi$ of $\CH_\BG(\BL,\lambda)$, we have
          $$
            \Deg({\rho_{\chi}})= \e_{\chi} \dfrac{\Deg(\lambda)(|\BG|/|\BL|)_{x'}}{S_{\chi}}
            \,.
          $$
       \item
          Assume that a root $\z$ of $P$ is regular for $W\vp$, and let $w\vp$ 
          be $\z$-regular.
          For $\BT_{w\vp} : =(V,w\vp)$ the  corresponding maximal torus, 
          the algebra
          $\CH_\BG(\BT_{w\vp},\Id)$ is a $\z$-cyclotomic spetsial Hecke algebra
          $\CH_W(w\vp)$.
       \item
          For all $\chi\in \Irr(\CH_\BG(\BL,\lambda))$, 
          $\Fr(\rho_{\chi})$ only depends on $\CH_\BG(\BL,\lambda)$
          and $\Fr(\la)$.
       \end{enumerate}
    \item 
      What precedes is compatible with a product decomposition as
      in \ref{prod}(1).
  \end{enumerate}
\end{axioms}
\smallskip

\subsubsection{Reduction to the cyclic case}\hfill
\smallskip

  Assume that $\BL = (V,W_\BL w\vp)$ is a $\P$-cuspidal pair, and
  let $H$ be a reflecting hyperplane for $W_\BG(\BL,\lambda)$.
  We denote by $\BG_H$ the ``parabolic reflection subdatum'' of $\BG$
  defined by $\BG_H := (V,W_Hw\vp)$ where $W_H$ is the fixator (pointwise stabilizer)
  of $H$.
  Then $W_{\BG_H}(\BL,\lambda)$ is cyclic and contains a unique distinguished
  reflection (see \ref{notationforcrg}) of  $W_\BG(\BL,\lambda)$.

\begin{axiom}\label{nU3}\hfill

  In the above situation the  parameters  of
  $\CH_{\BG_H}(\BL,\lambda)$  are the  same as the parameters
  corresponding to $H$ in $\CH_\BG(\BL,\lambda)$.
\end{axiom}
\smallskip

\subsubsection{Families, $\P$-Harish-Chandra series, Rouquier blocks}\hfill
\smallskip

 Her we refer the reader to \ref{U1} above.

\begin{axiom}\label{famHCblocks}\hfill

      For each $\P$-cuspidal pair $(\BL,\la)$ of $\BG$, the partition
      $$
        \Uch_\BG(\BL,\la) = \bigsqcup_{\CF\in\Fam(\BG)}
                                        \left( \CF \cap \Uch_\BG(\BL,\la) \right)
      $$
      composed with the bijection
      $$
        \Uch_\BG(\BL,\la) \iso \Irr \left( \CH_\BG(\BL,\la) \right)
      $$
      is the partition of $\Irr \left( \CH_\BG(\BL,\la) \right)$ into Rouquier blocks.
\end{axiom}
\smallskip

\subsubsection{Ennola transform}\hfill
\smallskip

  For $z\in \bmu(K)$, we define 
  $$
    \BG_z := (V,Wz\vp)
    \,.
  $$

\begin{axiom}\label{nEnnola}\hfill

  Let $\xi \in \bmu$ such that $z := \xi^{|ZW|} \in \bmu(K)$.
  There is a bijection
  $$
    E_\xi\,:\, \Uch(\BG) \iso \Uch(\BG_z)
  $$
  with the following properties.
  \begin{enumerate}
    \item
      $E_\xi$ is stable under the action of $N_{\GL(V)}(W\vp)/W$.
    \item
      For all $\rho\in\Uch(\BG)$, we have
      $$
         \Deg(E_\xi(\rho))(x)= \pm\Deg(\rho)(z\inv x)
         \,.
       $$
  \end{enumerate} 
\end{axiom}
\smallskip

\begin{axiom}\label{suppEnnola}\hfill

  \begin{enumerate}
    \item
      Let $\z\in\bmu(K)$, a root of the $K$-cyclotomic polynomial $\P(x)$.
      Let $(\BL,\la)$ be a $\P$-cuspidal pair. 
      \begin{enumerate}
        \item
           $(\BL_z,E_\xi(\la))$ is a $\P(z\inv x)$-cuspidal pair of $\BG_z$.
        \item
           $E_\xi$ induces a bijection 
           $$
             \Uch_\BG(\BL,\la) \iso \Uch_{\BG_z}(\BL_z,E_z(\la))
             \,.
           $$
        \item
          The parameters of the  $\P(z\inv x)$-cyclotomic Hecke algebra 
           $\CH_{\BG_z}(\BL_z,E_z(\la))$
           are obtained from those of $\CH_W(\BL,\la)$ by changing
           $x$ into $z\inv x$.
      \end{enumerate}
    \item
     The bijection $E_\xi$ induces a bijection $\Fam(\BG) \iso \Fam(\BG_z)\,.$   
      
  \end{enumerate}
\end{axiom}

\newpage

{\red \section{\red Determination of $\Uch(\BG)$: the algorithm}\hfill
\smallskip
}

 In this section, we consider reflection cosets $\BG=(V,W\vp)$  which have
  a \emph{split semi-simple part}, \ie\ $V$ has a $W\vp$-stable
  decomposition
  $$
    V=V_1\oplus V_2
    \,\text{ with }\,\,  W |_{V_2}=1 \,\text{ and }\, \vp |_{V_1}=1
    \,.
  $$
  In addition, we assume that $W$ is a spetsial group 
  (see \ref{characspetsial} above).
  
  We  show with the help of computations  done with the \CHEVIE\ package of
  \GAP3, that for all primitive special reflection groups there is a unique
  solution  which satisfies the axioms given in \S 5. Actually, a subset of
  the  axioms is sufficient to ensure unicity. More specifically, we finish
  the determination of unipotent degrees and Frobenius eigenvalues except for
  a  few cases in $G_{26}$  and $G_{32}$ in 6.5,  and at this stage we only
  use  \ref{1Harish-Chandra} for the  pair $(\BT,\Id)$. Also  we 
  only  use  \ref{U1}(3)  to  determine  the  families of characters.

  The tables in the appendix describe this solution.
\smallskip

 \subsection{Determination of $\Uch(\BG)$}\label{Compute dHC}\hfill
\smallskip

  The construction of $\Uch(\BG)$ proceeds as follows:
\smallskip

  \begin{enumerate}
    \item
     First stage.
    \begin{itemize}
      \item
        We start by constructing the principal series $\Uch(\BG,1)$
	using \ref{1Harish-Chandra}(2)(c) for the pair $(\BT,\Id)$.
      \item 
        We extend it by Ennola  transform using \ref{Ennola}
	to construct the union
        of the series $\Uch(\BG,\xi)$ for $\xi$ central in $W$.
      
        Let us denote by $U_1$ the subset of the set
        of unipotent characters that we have constructed at this stage.
    \end{itemize}
    \smallskip
    
    \item
      Second stage.
     \begin{itemize}
      \item
        Let $w_1\in W$ be a regular element of largest order in $W$, with regular
        eigenvalue $\z_1$. We have an algorithm allowing us to
        determine the parameters of the $\z_1$-spetsial cyclotomic Hecke
        algebra $\CH_W(w_1\vp)$, which in turn determines  $\Uch(\BG,\zeta_1)$.
      \item 
        We again use Ennola transform to determine $\Uch(\BG,\zeta_1\xi)$
        for $\xi$ central in $W$. Thus we know the series $\Uch(\BG,\xi\z_1)$,
        which can be added to our set $U_1$.
  
        Let us denote by $U_2$ the subset of the set
         of unipotent characters that we have constructed at this stage.
     \end{itemize}
     \smallskip
         
    \item
      Third stage.\\    
      We iterate the previous steps (proceeding in decreasing 
      orders of $w$, finding each time at least one reachable $\zeta$)
      until no $\Uch(\BG,\zeta)$ can be determined for any new $\zeta$. 
      At each iteration
  we can use  \ref{U1}(2) (whose right-hand side we know in
  advance) to check if we have finished the determination of $\Uch(\BG)$.
  \end{enumerate}
\smallskip
    
\smallskip

  This will succeed for every spetsial irreducible exceptional group,
  except for  $G_{26}, G_{32}$, where we will
  find \aposteriori\ that 1 (resp. 14) unipotent characters are missing at this point.
\smallskip

  For these remaining cases  we have to
  consider some series corresponding to 1-cuspidal pairs $(\BL,\lambda)$.
  
  Since we also want to label unipotent characters according to the
  1-Harish-Chandra series in which they lie, we shall actually determine
  (using a variation of the previous algorithm)
  the parameters of all these algebras $\CH_\BG(\BL,\lambda)$.
\smallskip
  We detail the steps (1)---(3) outlined above in sections 5.7 to 5.11.

\subsection{The principal series $\Uch(\BG,1)$}\hfill
\smallskip

  By \ref{1Harish-Chandra}(2)(c), the principal series $\Uch(\BG,1)$ is given 
  by the 1-spetsial algebra $\CH_W$.
  
  For $\chi\in\Irr(\CH_W)$ we have $\Fr(\chi)=1$ by 
  \ref{1Harish-Chandra}(2)(d) and \ref{prod}(2), and
  $\Deg(\rho_\chi)=\Deg(\chi)$ by \ref{principalzetaseriesaxio}(2).

\subsubsection{Example: the cyclic Spets.}\label{cyc1}\hfill
\smallskip

  Let $\BG_e := (\BC,\bmu_e)$ be the untwisted spets associated with the
  cyclic group $W=\bmu_e$ acting on $\BC$ by multiplication.
  
  We set $\BZ_e := \BZ[\bmu_e]$ and $\z := \exp({2\pi i/e})$.
\smallskip
  
  The spetsial Hecke algebra $\CH_W$ attached to $\BG_e$ is by
  \ref{characspetsial}
  the $\BZ_e[x^{\pm1}]$-algebra $\CH_e$ defined by
  $$
    \CH_e := \BZ_e[T]/(T-x)(T-\z)\cdots(T-\z^{e-1})
    \,.
  $$
  
  We denote by $\chi_0,\chi_1,\dots,\chi_{e-1} : \CH_e \ra \BZ_e[x,x\inv]$ the 
  irreducible characters of $\CH_e$, defined by
  $$
    \left\{
      \begin{aligned}
        &\chi_0 : T \mapsto x \,, \\
        &\chi_i : T \mapsto \z^i \quad\text{for }1\leq i \leq e-1\,.
      \end{aligned}
    \right.
  $$
  We denote by $S_0,S_1,\dots,S_{e-1}$ the corresponding family of Schur 
  elements, and by $\rho_0,\dots,\rho_{e-1}$ the corresponding
  (by \ref{principalzetaseriesaxio}(2) or \ref{1Harish-Chandra}(2)(c)) 
  unipotent characters in $\Uch(\BG_e,1)$.
  We have $\rho_0=\Id$.

  By \ref{1Harish-Chandra}(2)(b) we have $ \Deg({\rho_i})= \dfrac{S_0}{S_i}$ 
  since by \ref{brmazy} we have 
  $$
    S_0 = (|\BG|/|\BT|)_{x'}=\dfrac{x^e-1}{x-1}
    \,.
  $$
  So for $i\ne 0$ using \ref{brmazy} for the value of $S_i$ we get
  $$
    \Deg({\rho_i})= \dfrac{1-\z^i}{e} \, x \prod_{j\neq 0,i} (x-\z^j) 
    \,.
  $$

  We will see below that we need to add $\frac{(e-1)(e-2)}2$ other 1-cuspidal
  unipotent characters to the principal series before formula \ref{U1}(2)
  is satisfied.
\smallskip

\subsection{The series $\Uch(\BG,\xi)$ for $\xi\in ZW$}\hfill
\smallskip

From the principal series, we use (\ref{Ennola}) to determine the series $\Uch(\BG,\xi)$
for all $\xi \in ZW$.

Note  that for $\rho \in \Uch(\BG,\xi)$, this gives
$\Deg(\rho)$ only up to sign. 

In practice, we assign a sign arbitrarily and
go  on. However, for any character which is not 1-cuspidal the sign will be
determined  by the  sign chosen  for the  1-cuspidal character when we will
determine 1-Harish-Chandra series (see below).
\smallskip

We illustrate the process for 
the cyclic reflection coset, which in this case allows us to finish the determination
of $\Uch(\BG_e)$.

  We go on with the example \ref{cyc1} of $W=\bmu_e=\langle\zeta\rangle$
  where $\zeta=\exp(2i\pi/e)$. 

  We  have $Z(W)=W=\{\zeta^k\mid k= 0,1,\dots,e-1\}$.

  We determine $\Uch(\BG_e,\zeta^k)$ by Ennola transform. Let $\bz$ be
  the lift of $\zeta$ to $Z(\bB_W)$.
  The Ennola transform of $\CH_W$ by $\bz^k$ (for $k= 0,1,\dots,e-1$) is
  $$
    \CH_W(\z^k) = \BZ_e[T]/(T-\z^{-k}x)(T-\z)\cdots(T-\z^{e-1})
    \,,
  $$
  and the corresponding family of generic degrees is 
  $$
    \Deg(\rho_0)(\z^{-k}x),\Deg(\rho_1)(\z^{-k}x),\dots,
     \Deg(\rho_{e-1})(\z^{-k}x)
    \,.
  $$
  It is easy to check that, for all $i = 1,2,\dots,e-1$,
  $$
    \Deg(\rho_i)(\z^{-k}x)\,\,
    \left\{
      \begin{aligned}
        & = -\Deg(\rho_k)(x) &\ \text{if }i+k\equiv 0\mod e, \\
        &\notin  \{\pm\Deg(\rho)\mid\rho\in\Uch(\BG_e,1)\}
        &\ \text{if }i+k\not\equiv 0\mod e.
      \end{aligned}
    \right.
  $$

  Let us denote, for $0\le k<i\le e-1$,
  by $\rho_{i,k}$ the element of $\Uch(\BG_e,\zeta^k)$ 
  corresponding to the $(i-k)$-th character of $\CH_W(\z^k)$.
  Thus
  $$
    \Deg(\rho_{i,k})(x) := \Deg(\rho_{i-k})(\z^{-k}x) = 
      \dfrac{\z^k-\z^i}{e}\,x\dfrac{x^e-1}{(x-\z^k)(x-\z^i)}
    \,,
  $$
  and by \ref{Ennola},
  $$  \Fr( \rho_{i,k}) = \z^{ik}.$$
  With this notation, our original characters $\rho_i$
  become $\rho_{i,0}$, except for $\rho_0=\Id$ which is a special case.
  We see that if we extend the notation $\rho_{i,k}$ to be 
  $-\rho_{k,i}$ when $i<k$ the above formulae for the degree and eigenvalue
  remain consistent. It can be checked that with this notation
  $E_{\bz^j}(\rho_{i,k})=\rho_{i+j,k+j}$, where the indices are taken
  $\pmod e$; thus we have taken in account all characters obtained
  by Ennola from the principal series. We claim we have obtained
  all the unipotent characters.
  
\begin{theorem}\label{cyclicspets}\hfill

  $\Uch(\BG_e)$ consists of the $1+{e\choose 2}$ elements
  $\{\Id\} \bigcup \{ \rho_{i,k}\}_{0\leq k<i \leq e-1}$
  with degrees and eigenvalues as given above.
\end{theorem}

\begin{proof}\hfill

  Using that $\Feg_{\BG_e}(\chi_i)=x^i$, the reader can check (a non-trivial
  exercise) that formula \ref{U1}(2) is satisfied.
\end{proof}
\smallskip

\subsection{An algorithm to determine some $\Uch(\BG,\zeta)$ for $\zeta$
regular}
\label{an algo}
\hfill
\smallskip

  Assume that $U$ is one of the sets $U_1, U_2, \ldots$
  of unipotents characters of $\BG$
  mentioned in the introduction of section \ref{Compute dHC}. In particular, for
  all $\rho\in U$, we know $\Fr(\rho)$ and $\pm\Deg(\rho)$.
\smallskip
 
  We outline in this section an algorithm which allows us to determine
  the parameters of $\CH_W(w\vp)$ for some well-chosen $\zeta$-regular elements  $w\vp$
  (we call such well-chosen elements \emph{reachable from $U$}). Knowing the
  $\z$-spetsial algebra $\CH_W(w\vp)$, we then construct the series $\Uch(\BG,\z)$.
\smallskip

\subsubsection*{
    First step: determine the complete list of degrees of parameters $m_{I,j}$
    for the algebra $\CH_W(w\vp)$.}\hfill
\smallskip

  \begin{itemize}
    \item
      If  $W(w\vp)$ is cyclic this is easy since in this case (as explained in remark \ref{oneclass})
      formula \ref{byfam} determines the list of parameters $m_{I,j}$.
    \smallskip
    \item
      If  $W(w\vp)$ has only one conjugacy  class of hyperplanes, 
      for each $\rho_{\chi_\th} \in U$ such that  $\Deg(\rho_{\chi_\th})(\zeta)=\pm 1$
      we get (as explained in Remark \ref{oneclass}) the number $m_{\chi_\th}$ . 
    
      For  the  other  $m_\chi$, we  can  restrict  the  possibilities by using
      \ref{byfam}: they are equal to some
      $(\Nr+\Nh-\delta_\CB)/e_{W(w\vp)}$ where $\CB$ runs over the set of
      Rouquier blocks of $\CH_W$.
     
      Finally \ref{localcyclo}(\textsc{cs3})
      is  a good test  to weed out possibilities.
    \smallskip

    \item
      When $W(w\vp)$ has more than one class of hyperplanes, one can do the same with
      the equations of \ref{mandsigma} to restrict the possibilities;  we
      are helped by the fact that, in this case, the $e_I$'s are rather small.

      By the above, in all cases we are able to start with at most a few dozens of
      possibilities  for the list of $m_{I,j}$, and we proceed with the following
      steps with each one of this lists.
  \end{itemize}
\smallskip
  
\subsubsection*{
   Second step: for a given $I$, assign a specific $j$ to each element
  of our collection of $m_{I,j}$.}\hfill
\smallskip

  It turns out that when $U$ has a ``large enough''
  intersection with $\Uch(\BG,\zeta)$, 
  there is only one assignment such that $m_{I,1}$ is the largest of the
  $m_{I,j}$ and the resulting $\Uch(\BG,\z)$ contains $U$ as a subset.
\smallskip

  However, trying all possible assignments for the above test is not feasible
  in general, since $W(w\vp)$ can be for example the cyclic group of order 42
  (and $42!$ is too big). It happens that the product of the $e_I!$ is small
  enough when there is more than one of them; so we can concentrate on the case
  where there is only one $e_I$, that we will denote $e$;
  the linear characters of $W(w\vp)$ are the $\det^i$ for $i=0,\ldots,e-1$,
  and we will denote $u_i$ for $u_{I,\chi_\th}$ and $\rho_i$ for $\rho_{\chi_\th}$
  when $\th=\det^i$.
  
  If $\rho_i\in U$ we can reduce some of the arbitrariness for the assignment
  of $m$ to $j$ since \ref{frob} implies that the root of unity part of
  $\Fr(\rho_i)$ is given by $\zeta^{i+(a_{\rho_i}+A_{\rho_i})a/d}\,,$ 
  which gives $i\mod d$. 
  
  We can then reduce somewhat the remaining
  permutations by using the ``rationality type property'' $P_I(t,x)\in K(x)[t]$.

  We say that $\zeta$ is \emph{reachable from $U$} when we  can determine a unique
  algebra  $\CH_W(w\vp)$ in a reasonable time. Given  $U$, we then define
  $U'$  as the union of $U$ with all the $\Uch(\BG,\zeta\xi)$ where $\zeta$ is
  reachable from $U$, and $\xi\in ZW$ (see section \ref{Compute dHC}).
\smallskip

\subsection{An example of computational problems}\hfill
\smallskip

  The worst computation encountered during the process described above
  occurs in $G_{27}$, at the first step, \ie\ when starting from the initial 
  $
    U_1 = \bigcup_{\xi\in ZW} \Uch(\BG,\xi)\,,
  $
  we try to determine $\CH_W(w\vp)$ for $w$ of maximal order.
  In that case, $W(w\vp)$ is a cyclic  group  of  order $h:=30$.
  We  know  $18$  out  of $30$ parameters as
  corresponding  to elements of $U$. For the 12 remaining, we know the list
  of  $m_{I,j}$, which is 
  $$
    [1,1,1,1,3/2,3/2,3/2,3/2,2,2,2,2]
  $$
  and there are
  34650  arrangements of  that list  in the  remaining 12  slots; in  a few
  minutes  of CPU we find  that 420 of them  provide $P_I(t,x)\in K(x)[t]$ and
  after  a  few  more  minutes  that  just one of them gives a $\Uch(\BG,\zeta_h)$
  containing the $U\cap\Uch(\BG,\zeta_h)$ we started with.

  As mentioned in section \ref{an algo}, at the end of this process, 
  we discover by \ref{U1}(2) that we
  have found all unipotent degrees,  except in the cases 
  of $G_{26}$ and $G_{32}$, where we will find
  \aposteriori\ that 1 (resp. 14) unipotent characters are missing.
\smallskip

\subsection{Determination of 1-Harish-Chandra series}\hfill
\smallskip

  We next finish determining $\Uch(\BG)$  for  $G_{26}$  and $G_{32}$, by
  considering  some series corresponding to 1-cuspidal pairs $(\BL,\lambda)$.
  We  will find that in $G_{26}$ the missing character  occurs in the series
  $\Uch(\BG_3,\rho_{2,1})$  (where $\rho_{2,1}$, as seen in \ref{cyclicspets},
  is the only 1-cuspidal character for
  $\BG_3$, and where $W_\BG(\BL,\lambda)$ is $G(6,1,2)$), 
  while in $G_{32}$ the missing characters occur in
  $\Uch(\BG_3,\rho_{2,1})$ and in
  $\Uch(\BG_3\times\BG_3,\rho_{2,1}\otimes\rho_{2,1})$ where 
  $W_\BG(\BL,\lambda)$ is respectively $G_{26}$ and $G(6,1,2)$. 
  In these cases, $W_\BG(\BL,\lambda)$ is not cyclic, so by the reduction to
  the cyclic case \ref{U3}, the computation of the parameters of
  $\CH_\BG(\BL,\lambda)$ is reduced to the case of sub-Spets
  where all unipotent characters are known.

  Following the practice of Lusztig and Carter for reductive groups,
  we will name unipotent characters by their 1-Harish-Chandra data, that
  is each character will be indexed by a 1-cuspidal pair $(\BL,\lambda)$
  and a character of $\CH_\BG(\BL,\lambda)$; thus to do this indexing
  we want anyway to determine all the 1-series.
  
  Let us examine now the computations involved (in a Spets where all unipotent
  characters are known).
\smallskip

\subsubsection*{Step1: determine cuspidal pairs $(\BL,\lambda)$.}\hfill
\smallskip

  First we must find the $W$-orbits of cuspidal pairs
  $(\BL,\lambda)$  and  the corresponding groups
  $W_\BG(\BL,\lambda)$, and for that we must know the action of an
  automorphism of $\BL$ given by an element of $W_\BG(\BL)$) on 
  $\lambda\in\Uch(\BL)$.
  
  Cases when this is determined by our axioms \ref{autom} and
  \ref{1Harish-Chandra}(2)(a) are

\begin{itemize}
  \item 
    when $\pm\Deg(\lambda)$ is unique,
  \item 
    when the pair $(\pm\Deg(\lambda),\Fr(\lambda))$ is unique,
  \item 
    when $\lambda$ is in the principal $1$-series;
    the automorphism then acts on $\lambda$ as on $\Irr(W_\BL)$.
\end{itemize}
  In the first two cases $W_\BG(\BL,\lambda)=W_\BG(\BL)$.

  The above conditions are sufficient in the cases we need for $G_{26}$ and
  $G_{32}$. They are not sufficient in some other cases.

  Now, for 1-series where $W_\BG(\BL,\lambda)$ is
  \emph{not} cyclic, we know (by induction using \ref{U3}) the parameters of
  $\CH_\BG(\BL,\lambda)$. Since in the cases we need for $G_{26}$ and $G_{32}$,
  the group $W_\BG(\BL,\lambda)$ is not cyclic, we may assume from now on that
  we know all of $\Uch(\BG)$.
\smallskip

\subsubsection*{Step 2: find the elements of $\Uch_\BG(\BL,\lambda)$.}\hfill
\smallskip

  A candidate element $\rho\in\Uch_\BG(\BL,\lambda)$ must satisfy the
  following properties.
\begin{itemize}
  \item 
    $\Deg(\rho)$ must be divisible by $\Deg(\lambda)$ by
    \ref{1Harish-Chandra}(2)(b).
  \item 
    By specializing \ref{1Harish-Chandra}(2)(b) to $x=1$, we get,
    if $\rho=\rho_\chi$,
    $$
     |W_\BG(\BL,\lambda)|
     \left(\dfrac{\Deg(\rho)|\BL|_{x'}}{\Deg(\lambda)|\BG|_{x'}} \right)(1)=\chi(1)
     \,.
     $$ 
    Then $\chi(1)$ must be the degree of a character of $|W_\BG(\BL,\lambda)|$.
  \item 
    Formula \ref{1Harish-Chandra}(2)(b) yields a Schur element
    $S_{\chi_\rho}$ which must be a Laurent polynomial
    (indeed, we assume that the relative Hecke algebras are $1$-cyclotomic
    in the sense of \cite[6E]{sp1}, thus their Schur elements have these
    rationality properties).
\end{itemize}

  If there are exactly
  $|\Irr(\CH_\BG(\BL,\lambda))|$  candidates left  at this  stage, we are done.
  If there
  are  too  many  candidates  left,  a  useful  test  is  to filter candidate
  $|\Irr(\CH_\BG(\BL,\lambda))|$-tuples by the condition \ref{substitutRLG}.
  In  practice  this
  always yields only one acceptable tuple%
\footnote{This is not always the case if one
  tries  to determine $\zeta$-Harish-Chandra series by similar techniques for
  $\zeta\ne 1$.}.
\smallskip

\subsubsection*{Step 3: Parametrize  elements  of  $\Uch_\BG(\BL,\lambda)$  by  characters  of
  $\CH_\BG(\BL,\lambda)$.}\hfill
\smallskip

  This problem is equivalent to determining the Schur
  elements of $\CH_\BG(\BL,\lambda)$, which in turn is equivalent to determining
  the parameters of this algebra (up to a common scalar). Thanks to \ref{U3}, it is
  sufficient to consider the case when $W_\BG(\BL,\lambda)$ is cyclic.
  Then one can use techniques analogous to that of section \ref{Compute dHC}.
  
  For example, let us consider the case where $\BG$ is the split reflection
  coset  associated with the exceptional  reflection group $G_{26}$. Let us
  recall  that  $\BG_e$  denotes  the  reflection coset associated with the
  cyclic group $\bmu_e$ (see \ref{cyc1}).

  We  have $W_\BG(\BG_3,\rho_{2,1})=G(6,2,2)$. To determine the parameters,
  for each hyperplane $I$ of that group, we have to look at the same series
  $\Uch(\BG_3,\rho_{2,1})$   in  the  group  $W_I$  which  is  respectively
  $G(3,1,2)$,  $G_4$ and $\bmu_3\times  \bmu_2$; in each  case the relative
  group is $\BG_{e_I}$ where $e_I$ is respectively $3,2,2$. 
  
\begin{remark}\label{parameters}
  In each case we can determine
  the parameters up to a constant, which must be a power of $x$ times a root of
  unity for the algebra to be $1$-cyclotomic; we have chosen them so that
  the lowest power of $x$ is $0$, in which case they are determined up to a
  cyclic permutation. We have chosen this permutation so that the polynomial
  $P_I(t,x)$ is as rational as possible, and amongst the remaining ones we chose
  the list $m_{I,0},\ldots,m_{I,e_I-1}$ to be the lexicographically biggest
  possible. 
\end{remark}
  
  In our case we get that the relative Hecke algebra 
  $\CH_\BG(\BL,\la)$ is
  $$\CH_{G(6,2,2)}(1,\zeta_3 x^2,\zeta_3^2 x^2;x^3,-1;x,-1).$$
  Here we put the group $W_\BG(\BL,\lambda)$ as an index and
  each list separated by a semicolon is the list of parameters for one of
  the 3 orbits of hyperplanes. We list the parameters in each list
  $u_{I,0},\ldots,u_{I,e_I-1}$ in an order such that $u_{I,j}$ specializes
  to $\zeta_{e_I}^j$.

  Similarly, in $G_{32}$, for $\Uch(\BG_3,\rho_{2,1})$ we get the Hecke
  algebra 
  $$
    \CH_{G_{26}}(x^3,\zeta_3,\zeta_3^2;x,-1)
  $$
  and for
  $\Uch(\BG_3\times\BG_3,\rho_{2,1}\otimes\rho_{2,1})$ we get
  $$
    \CH_{G(6,1,2)}(x^3,-\zeta_3^2x^3,\zeta_3x^2,-1,\zeta_3^2,-\zeta_3x^2;
      x^3,-1)
    \,.
  $$
\smallskip

\subsection{Determination of families}\hfill
\smallskip

  The families can be completely determined from their intersection with
  the  principal series, which are the Rouquier blocks which were
  determined in \cite{maro}, and \ref{U1}(3).

  Indeed, the only cases in the tables of Malle and Rouquier where two blocks
  share the same pair $(a,A)$, and none of them is a one-element block,
  are the pairs of blocks $(12,13)$, $(14,15)$, $(21,27)$ in $G_{34}$
  (the numbers refer to the order in which the families appear in the \CHEVIE\
  data; see the tables in the appendix to this paper).

  For each of these pairs, we have a list $L$ of unipotent characters that
  we must split into two families $\CF_1$ and $\CF_2$. To do this, we can
  use the axiom of stability of families by Ennola, since in each case
  all degrees in $L$ are Ennola-transforms of those in the intersection of
  $L$ with the principal series, and there are no degrees in common
  between the intersections of $\CF_1$ and $\CF_2$ with the principal
  series.
\smallskip

\subsection{The main theorem}\hfill
\smallskip

  We can now summarize the main result of this paper as follows:

\begin{theorem}\label{sowhatagain}\hfill

  Given a primitive irreducible spetsial reflection group $W$, and the
  associated split coset $\BG$, there is a unique set $\Uch(\BG)$, with
  a unique function $\Fr$ and a unique (up to an arbitrary choice of signs
  for 1-cuspidal characters) function $\Deg$
  which satisfy the axioms \ref{prod}, \ref{autom},
  \ref{principalzetaseriesaxio}, \ref{Ennola}, \ref{1Harish-Chandra},
  \ref{U3} and \ref{U1}.
\end{theorem}
\newpage

\appendix\section{\red Tables}
In  this appendix, we give tables of unipotent degrees for split Spetses of
primitive  finite complex reflexion groups, as  well as for the imprimitive
groups  $\BZ/3$, $\BZ/4$,  $G(3,1,2)$, $G(3,3,3)$  and $G(4,4,3)$ which are
involved  in  their  construction  or  in  the  labeling of their unipotent
characters.

The  characters are  named by  the pair  of the  1-cuspidal unipotent above
which  they lie and the corresponding character of the relative Weyl group.
As  pointed out in \ref{parameters}, this  last character is for the moment
somewhat  arbitrary since it depends on an ordering (defined up to a cyclic
permutation) of the parameters of the relative Hecke algebra, thus may have
to  be changed in the  future if the theory  comes to prescribe a different
ordering   than  the  one   we  have  chosen.   Also,  as  pointed  out  in
\ref{sowhatagain},  the sign of  the degree of  1-cuspidal characters (and
consequently  of the  corresponding 1-Harish-Chandra  series) is arbitrary,
though  we fixed it such that the leading term is positive when it is real.
For  the imprimitive groups  we give the  correspondence between our labels
and symbols as in \cite{maIG}.

For primitive groups, the labeling of characters of $W$ is as in \cite{maG}.

The unipotent characters are listed family by family. In a Rouquier family,
there  is a  unique character  $\theta$ such  that $a_\theta=b_\theta$, the
special  character (see  \ref{marouaA}), and  a unique  character such that
$A_\theta=B_\theta$, called the {\em cospecial character}, which may or may
not  coincide with the special character. The special character in a family
is indicated by a $*$ sign in the first column. If it is different from the
special  character, the cospecial character is  indicated by a $\#$ sign in
the first column.

In  the third column we give $\Fr$, as a root of unity times a power
of $x$ in $\BQ/\BZ$ (this power is most of the time equal to $0$).

We  denote the cuspidal  unipotent characters by  the name of  the group if
there is only one, otherwise the name of the group followed by the $\Fr$, with
an  additional  exponent  if  needed  to  resolve ambiguities. For instance
$G_6[\zeta_8^3]$   is  the  cuspidal  unipotent  character  of  $G_6$  with
$\Fr=\zeta_8^3$, while $G_6^2[-1]$ is the second one with $\Fr=-1$.

For each group we list the Hecke algebras used in the construction: we give
the parameters of the spetsial $\zeta$-cyclotomic Hecke algebras of compact
type for representatives of the regular $\zeta$ under the action of the centre;
we  omit the central $\zeta$  for which the parameters  are always given by
\ref{HW} and its Ennola transforms.

We also list the parameters of the 1-cyclotomic Hecke algebras attached to
cuspidal pairs, chosen as in \ref{parameters}.

For each of these Hecke algebras, the parameters are displayed as explained
in remark \ref{parameters}.

\vfill\eject
\subsection{Irreducible $K$-cyclotomic polynomials}
\subsubsection*{$\BQ(i)$-cyclotomic polynomials}
${\Phi'_4}=x-i$, ${\Phi''_4}=x+i$, ${\Phi'_8}=x^2-i$, ${\Phi''_8}=x^2+i$, ${
\Phi'_{12}}=x^2-ix-1$, ${\Phi''_{12}}=x^2+ix-1$, ${\Phi'''_{20}}=x^4+ix^3-x^2-
ix+1$, ${\Phi''''_{20}}=x^4-ix^3-x^2+ix+1$
\subsubsection*{$\BQ(\zeta_3)$-cyclotomic polynomials}
${\Phi'_3}=x-\zeta_3$, ${\Phi''_3}=x-\zeta_3^2$, ${\Phi'_6}=x+\zeta_3^2$, ${
\Phi''_6}=x+\zeta_3$, ${\Phi'_9}=x^3-\zeta_3$, ${\Phi''_9}=x^3-\zeta_3^2$, ${
\Phi'''_{12}}=x^2+\zeta_3^2$, ${\Phi''''_{12}}=x^2+\zeta_3$, ${\Phi'''_{15}}=
x^4+\zeta_3^2x^3+\zeta_3x^2+x+\zeta_3^2$, ${\Phi''''_{15}}=x^4+\zeta_3x^3+
\zeta_3^2x^2+x+\zeta_3$, ${\Phi'_{18}}=x^3+\zeta_3^2$, ${\Phi''_{18}}=x^3+
\zeta_3$, ${\Phi'_{21}}=x^6+\zeta_3x^5+\zeta_3^2x^4+x^3+\zeta_3x^2+\zeta_3^2x+
1$, ${\Phi''_{21}}=x^6+\zeta_3^2x^5+\zeta_3x^4+x^3+\zeta_3^2x^2+\zeta_3x+
1$, ${\Phi'_{24}}=x^4+\zeta_3^2$, ${\Phi''_{24}}=x^4+\zeta_3$, ${\Phi'''_{30}}
=x^4-\zeta_3x^3+\zeta_3^2x^2-x+\zeta_3$, ${\Phi''''_{30}}=x^4-\zeta_3^2x^3+
\zeta_3x^2-x+\zeta_3^2$, ${\Phi'_{42}}=x^6-\zeta_3^2x^5+\zeta_3x^4-x^3+\zeta_
3^2x^2-\zeta_3x+1$, ${\Phi''_{42}}=x^6-\zeta_3x^5+\zeta_3^2x^4-x^3+\zeta_3x^2-
\zeta_3^2x+1$
\subsubsection*{$\BQ(\sqrt 3)$-cyclotomic polynomials}
${\Phi^{(5)}_{12}}=x^2-\sqrt 3x+1$, ${\Phi^{(6)}_{12}}=x^2+\sqrt 3x+1$
\subsubsection*{$\BQ(\sqrt 5)$-cyclotomic polynomials}
${\Phi'_5}=x^2+\frac{1-\sqrt 5}2x+1$, ${\Phi''_5}=x^2+\frac{1+\sqrt 5}2x+
1$, ${\Phi'_{10}}=x^2+\frac{-1-\sqrt 5}2x+1$, ${\Phi''_{10}}=x^2+\frac{-1+
\sqrt 5}2x+1$, ${\Phi'_{15}}=x^4+\frac{-1-\sqrt 5}2x^3+\frac{1+\sqrt 5}2x^2+
\frac{-1-\sqrt 5}2x+1$, ${\Phi''_{15}}=x^4+\frac{-1+\sqrt 5}2x^3+\frac{1-
\sqrt 5}2x^2+\frac{-1+\sqrt 5}2x+1$, ${\Phi'_{30}}=x^4+\frac{1-\sqrt 5}2x^3+
\frac{1-\sqrt 5}2x^2+\frac{1-\sqrt 5}2x+1$, ${\Phi''_{30}}=x^4+\frac{1+
\sqrt 5}2x^3+\frac{1+\sqrt 5}2x^2+\frac{1+\sqrt 5}2x+1$
\subsubsection*{$\BQ(\sqrt{-2})$-cyclotomic polynomials}
${\Phi^{(5)}_8}=x^2-\sqrt {-2}x-1$, ${\Phi^{(6)}_8}=x^2+\sqrt {-2}x-1$, ${
\Phi^{(7)}_{24}}=x^4+\sqrt {-2}x^3-x^2-\sqrt {-2}x+1$, ${\Phi^{(8)}_{24}}=x^4-
\sqrt {-2}x^3-x^2+\sqrt {-2}x+1$
\subsubsection*{$\BQ(\sqrt{-7})$-cyclotomic polynomials}
${\Phi'_7}=x^3+\frac{1-\sqrt {-7}}2x^2+\frac{-1-\sqrt {-7}}2x-1$, ${\Phi''_7}=
x^3+\frac{1+\sqrt {-7}}2x^2+\frac{-1+\sqrt {-7}}2x-1$, ${\Phi'_{14}}=x^3+
\frac{-1+\sqrt {-7}}2x^2+\frac{-1-\sqrt {-7}}2x+1$, ${\Phi''_{14}}=x^3+\frac{-
1-\sqrt {-7}}2x^2+\frac{-1+\sqrt {-7}}2x+1$
\subsubsection*{$\BQ(\sqrt 6)$-cyclotomic polynomials}
${\Phi^{(5)}_{24}}=x^4-\sqrt 6x^3+3x^2-\sqrt 6x+1$, ${\Phi^{(6)}_{24}}=x^4+
\sqrt 6x^3+3x^2+\sqrt 6x+1$
\subsubsection*{$\BQ(\zeta_{12})$-cyclotomic polynomials}
${\Phi^{(7)}_{12}}=x+\zeta_{12}^7$, ${\Phi^{(8)}_{12}}=x+\zeta_{12}^{11}$, ${
\Phi^{(9)}_{12}}=x+\zeta_{12}$, ${\Phi^{(10)}_{12}}=x+\zeta_{12}^5$
\subsubsection*{$\BQ(\sqrt 5,\zeta_3)$-cyclotomic polynomials}
${\Phi^{(5)}_{15}}=x^2+\frac{(1+\sqrt 5)\zeta_3^2}2x+\zeta_3$, ${\Phi^{(6)}_{
15}}=x^2+\frac{(1-\sqrt 5)\zeta_3^2}2x+\zeta_3$, ${\Phi^{(7)}_{15}}=x^2+\frac{
(1+\sqrt 5)\zeta_3}2x+\zeta_3^2$, ${\Phi^{(8)}_{15}}=x^2+\frac{(1-
\sqrt 5)\zeta_3}2x+\zeta_3^2$, ${\Phi^{(5)}_{30}}=x^2+\frac{(-1+\sqrt 5)\zeta_
3^2}2x+\zeta_3$, ${\Phi^{(6)}_{30}}=x^2+\frac{(-1-\sqrt 5)\zeta_3^2}2x+\zeta_
3$, ${\Phi^{(7)}_{30}}=x^2+\frac{(-1+\sqrt 5)\zeta_3}2x+\zeta_3^2$, ${\Phi^{
(8)}_{30}}=x^2+\frac{(-1-\sqrt 5)\zeta_3}2x+\zeta_3^2$
\subsubsection*{$\BQ(\sqrt{-2},\zeta_3)$-cyclotomic polynomials}
${\Phi^{(9)}_{24}}=x^2+\sqrt {-2}\zeta_3^2x-\zeta_3$, ${\Phi^{(10)}_{24}}=x^2-
\sqrt {-2}\zeta_3^2x-\zeta_3$, ${\Phi^{(11)}_{24}}=x^2+\sqrt {-2}\zeta_3x-
\zeta_3^2$, ${\Phi^{(12)}_{24}}=x^2-\sqrt {-2}\zeta_3x-\zeta_3^2$
\vfill\eject
\subsection{Unipotent characters for $Z_{3}$}
\begin{center}
\tablehead{\hline \gamma&\mbox{Deg($\gamma$)}&\mbox{Fr($\gamma$)}&\mbox{Symbol\
}\\\hline}
\tabletail{\hline}
\begin{supertabular}{|R|RRR|}
\shrinkheight{30pt}
*\hfill 1&1&1&(1,,)\\
\hline
Z_3&\frac{\sqrt {-3}}3q\Phi_1&\zeta_3^2&(,01,01)\\
\#\hfill \zeta_3^2&\frac{3+\sqrt {-3}}6q{\Phi'_3}&1&(01,0,1)\\
*\hfill \zeta_3&\frac{3-\sqrt {-3}}6q{\Phi''_3}&1&(01,1,0)\\
\end{supertabular}
\end{center}
To  simplify,  we  used an  obvious notation for the
characters of the principal series, that is $\zeta_3$ denotes the reflection
character, and denoted by $Z_3$ the unique cuspidal
unipotent character. The corresponding symbols are given in the last column.
\subsection{Unipotent characters for $Z_{4}$}
\begin{center}
\tablehead{\hline \gamma&\mbox{Deg($\gamma$)}&\mbox{Fr($\gamma$)}&\mbox{Symbol\
}\\\hline}
\tabletail{\hline}
\begin{supertabular}{|R|RRR|}
\shrinkheight{30pt}
*\hfill 1&1&1&(1,,,)\\
\hline
-1&\frac12q\Phi_4&1&(01,0,1,0)\\
*\hfill i&\frac{-i+1}4q\Phi_2{\Phi''_4}&1&(01,1,0,0)\\
Z_{4}^{1022}&\frac{-i+1}4q\Phi_1{\Phi'_4}&-1&(0,,01,01)\\
Z_{4}^{0212}&\frac i2q\Phi_1\Phi_2&-i&(,01,0,01)\\
\#\hfill -i&\frac{i+1}4q\Phi_2{\Phi'_4}&1&(01,0,0,1)\\
Z_{4}^{1220}&\frac{-i-1}4q\Phi_1{\Phi''_4}&-1&(0,01,01,)\\
\end{supertabular}
\end{center}
We  used an obvious notation for the characters of the principal series, and
the  shape of  the symbols  for the  cuspidal characters. 
\vfill\eject
\subsection{Unipotent characters for $G_{4}$}
\begin{center}Some principal $\zeta$-series\end{center}
$\zeta_{4}$ : ${\mathcal H}_{Z_{4}}
(ix^3,\allowbreak i,\allowbreak ix,\allowbreak -i)$\hfill\break
$\zeta_{3}^{2}$ : ${\mathcal H}_{Z_{6}}(\zeta_3^2x^2,\allowbreak -\zeta_
3^2,\allowbreak \zeta_3,\allowbreak -\zeta_3^4x,\allowbreak \zeta_
3^2,\allowbreak -\zeta_3^2x)$\hfill\break
$\zeta_{3}$ : ${\mathcal H}_{Z_{6}}(\zeta_3x^2,\allowbreak -\zeta_
3^4x,\allowbreak \zeta_3,\allowbreak -\zeta_3^2x,\allowbreak \zeta_
3^2,\allowbreak -\zeta_3^4)$\hfill\break
\begin{center}Non-principal $1$-Harish-Chandra series\end{center}
${\mathcal H}_{G_{4}}(Z_3)={\mathcal H}_{A_1}(x^3,-1)$\hfill\break
\par\tablehead{\hline \gamma&\mbox{Deg($\gamma$)}&\mbox{Fr($\gamma$)}\\\hline}
\tabletail{\hline}
\begin{supertabular}{|R|RR|}
\shrinkheight{30pt}
*\hfill \phi_{1,0}&1&1\\
\hline
*\hfill \phi_{2,1}&\frac{3-\sqrt {-3}}6x{\Phi'_3}\Phi_4{\Phi''_6}&1\\
\#\hfill \phi_{2,3}&\frac{3+\sqrt {-3}}6x{\Phi''_3}\Phi_4{\Phi'_6}&1\\
Z_3:2&\frac{\sqrt {-3}}3x\Phi_1\Phi_2\Phi_4&\zeta_3^2\\
\hline
*\hfill \phi_{3,2}&x^2\Phi_3\Phi_6&1\\
\hline
*\hfill \phi_{1,4}&\frac{-\sqrt {-3}}6x^4{\Phi''_3}\Phi_4{\Phi''_6}&1\\
\phi_{2,5}&\frac12x^4\Phi_2^2\Phi_6&1\\
G_4&\frac12x^4\Phi_1^2\Phi_3&-1\\
Z_3:11&\frac{\sqrt {-3}}3x^4\Phi_1\Phi_2\Phi_4&\zeta_3^2\\
\#\hfill \phi_{1,8}&\frac{\sqrt {-3}}6x^4{\Phi'_3}\Phi_4{\Phi'_6}&1\\
\end{supertabular}
\vfill\eject
\subsection{Unipotent characters for $G_{6}$}
\begin{center}Some principal $\zeta$-series\end{center}
$\zeta_{3}$ : ${\mathcal H}_{Z_{12}}(\zeta_3x^2,\allowbreak -ix,\allowbreak -
\zeta_3^4x,\allowbreak \zeta_{12}^7x^{1/2},\allowbreak x,\allowbreak \zeta_{
12}x,\allowbreak -1,\allowbreak ix,\allowbreak \zeta_3x,\allowbreak \zeta_{12}
x^{1/2},\allowbreak -\zeta_3^4,\allowbreak \zeta_{12}^7x)$\hfill\break
$\zeta_{3}^{2}$ : ${\mathcal H}_{Z_{12}}(\zeta_3^2x^2,\allowbreak \zeta_{12}
^5x,\allowbreak -\zeta_3^2,\allowbreak \zeta_{12}^{11}x^{1/2}
,\allowbreak \zeta_3^2x,\allowbreak -ix,\allowbreak -1,\allowbreak \zeta_{12}
^{11}x,\allowbreak x,\allowbreak \zeta_{12}^5x^{1/2},\allowbreak -\zeta_
3^2x,\allowbreak ix)$\hfill\break
\begin{center}Non-principal $1$-Harish-Chandra series\end{center}
${\mathcal H}_{G_{6}}(Z_3)={\mathcal H}_{Z_{4}}
(x^3,\allowbreak ix^3,\allowbreak -1,\allowbreak -ix^3)$\hfill\break
{\small
\par\tablehead{\hline \gamma&\mbox{Deg($\gamma$)}&\mbox{Fr($\gamma$)}\\\hline}
\tabletail{\hline}
\begin{supertabular}{|R|RR|}
\shrinkheight{30pt}
*\hfill \phi_{1,0}&1&1\\
\hline
*\hfill \phi_{2,1}&\frac{(3-\sqrt 3)(i+1)}{24}x\Phi_2^2{\Phi'_3}{\Phi'_4}
^2\Phi_6{\Phi''_{12}}{\Phi^{(8)}_{12}}&1\\
\phi_{3,2}&\frac14x\Phi_3\Phi_4^2\Phi_{12}&1\\
\phi_{2,3}'&\frac{(3+\sqrt 3)(-i+1)}{24}x\Phi_2^2{\Phi'_3}{\Phi''_4}^2\Phi_6{
\Phi'_{12}}{\Phi^{(10)}_{12}}&1\\
\phi_{1,4}&\frac{-\sqrt {-3}}{12}x{\Phi''_3}\Phi_4^2{\Phi''_6}\Phi_{12}&1\\
G_6^2[-i]&\frac{\sqrt 3}{12}x\Phi_1^2\Phi_2^2\Phi_3\Phi_6{\Phi^{(5)}_{12}}&
-i\\
G_6^2[-1]&\frac{(-3-\sqrt 3)(-i+1)}{24}x\Phi_1^2\Phi_3{\Phi'_4}^2{\Phi''_6}{
\Phi''_{12}}{\Phi^{(8)}_{12}}&-1\\
G_6[i]&\frac{-i}4x\Phi_1^2\Phi_2^2\Phi_3\Phi_6{\Phi'_{12}}&i\\
G_6[-\zeta_3^2]&\frac{-\sqrt {-3}(i+1)}{12}x\Phi_1^2\Phi_2\Phi_3\Phi_4{\Phi'_
4}{\Phi''_{12}}&-\zeta_3^2\\
Z_3:1&\frac{\sqrt {-3}}6x\Phi_1\Phi_2\Phi_4^2\Phi_{12}&\zeta_3^2\\
G_6^2[-\zeta_3^2]&\frac{-\sqrt 3(i+1)}{12}x\Phi_1^2\Phi_2\Phi_3\Phi_4{\Phi''_
4}{\Phi'_{12}}&-\zeta_3^2\\
\#\hfill \phi_{2,3}''&\frac{(3-\sqrt 3)(-i+1)}{24}x\Phi_2^2{\Phi''_3}{\Phi''_
4}^2\Phi_6{\Phi'_{12}}{\Phi^{(9)}_{12}}&1\\
G_6^2[i]&\frac{-i}4x\Phi_1^2\Phi_2^2\Phi_3\Phi_6{\Phi''_{12}}&i\\
G_6^3[-1]&\frac{(-3-\sqrt 3)(i+1)}{24}x\Phi_1^2\Phi_3{\Phi''_4}^2{\Phi'_6}{
\Phi'_{12}}{\Phi^{(9)}_{12}}&-1\\
\phi_{1,8}&\frac{\sqrt {-3}}{12}x{\Phi'_3}\Phi_4^2{\Phi'_6}\Phi_{12}&1\\
\phi_{2,5}'&\frac{(3+\sqrt 3)(i+1)}{24}x\Phi_2^2{\Phi''_3}{\Phi'_4}^2\Phi_6{
\Phi''_{12}}{\Phi^{(7)}_{12}}&1\\
Z_3:i&\frac{\sqrt 3(i+1)}{12}x\Phi_1\Phi_2^2\Phi_4{\Phi'_4}\Phi_6{\Phi''_{12}}
&\zeta_3^2\\
G_6[-\zeta_{12}^{-1}]&\frac{-\sqrt 3}6x\Phi_1^2\Phi_2^2\Phi_3\Phi_4\Phi_6&
\zeta_{12}^5\\
G_6[-1]&\frac{(-3+\sqrt 3)(i+1)}{24}x\Phi_1^2\Phi_3{\Phi''_4}^2{\Phi''_6}{
\Phi'_{12}}{\Phi^{(10)}_{12}}&-1\\
\phi_{1,6}&\frac14x\Phi_4^2\Phi_6\Phi_{12}&1\\
G_6[-i]&\frac{-\sqrt 3}{12}x\Phi_1^2\Phi_2^2\Phi_3\Phi_6{\Phi^{(6)}_{12}}&
-i\\
Z_3:-i&\frac{\sqrt {-3}(i+1)}{12}x\Phi_1\Phi_2^2\Phi_4{\Phi''_4}\Phi_6{\Phi'_
{12}}&\zeta_3^2\\
G_6^4[-1]&\frac{(-3+\sqrt 3)(-i+1)}{24}x\Phi_1^2\Phi_3{\Phi'_4}^2{\Phi'_6}{
\Phi''_{12}}{\Phi^{(7)}_{12}}&-1\\
\hline
*\hfill \phi_{3,4}&x^4\Phi_3\Phi_6\Phi_{12}&1\\
\hline
G_6[\zeta_8^3]&\frac i2x^5\Phi_1^2\Phi_2^2\Phi_3\Phi_6&\zeta_8^3x^{1/2}\\
*\hfill \phi_{2,5}''&\frac12x^5\Phi_4^2\Phi_{12}&1\\
G_6[-\zeta_8^3]&\frac i2x^5\Phi_1^2\Phi_2^2\Phi_3\Phi_6&\zeta_8^7x^{1/2}\\
\phi_{2,7}&\frac12x^5\Phi_4^2\Phi_{12}&1\\
\hline
*\hfill \phi_{1,10}&\frac{3-\sqrt {-3}}6x^{10}{\Phi''_3}{\Phi'_6}{\Phi''''_{
12}}&1\\
\#\hfill \phi_{1,14}&\frac{3+\sqrt {-3}}6x^{10}{\Phi'_3}{\Phi''_6}{\Phi'''_{
12}}&1\\
Z_3:-1&\frac{\sqrt {-3}}3x^{10}\Phi_1\Phi_2\Phi_4&\zeta_3^2\\
\end{supertabular}}
\vfill\eject
\subsection{Unipotent characters for $G_{8}$}
\begin{center}Some principal $\zeta$-series\end{center}
$\zeta_{8}$ : ${\mathcal H}_{Z_{8}}(\zeta_8^5x^3,\allowbreak \zeta_
8,\allowbreak \zeta_8x,\allowbreak \zeta_8^3,\allowbreak \zeta_
8^3x,\allowbreak \zeta_8^5,\allowbreak \zeta_8^5x,\allowbreak \zeta_
8^7)$\hfill\break
$\zeta_{3}$ : ${\mathcal H}_{Z_{12}}(\zeta_3x^2,\allowbreak \zeta_{12}
,\allowbreak -\zeta_3^2,\allowbreak \zeta_{12}^7x^{1/2},\allowbreak \zeta_
3,\allowbreak \zeta_{12}x,\allowbreak -1,\allowbreak \zeta_{12}
^7,\allowbreak \zeta_3x,\allowbreak \zeta_{12}x^{1/2},\allowbreak -\zeta_
3^4,\allowbreak \zeta_{12}^7x)$\hfill\break
$\zeta_{3}^{2}$ : ${\mathcal H}_{Z_{12}}(\zeta_3^2x^2,\allowbreak \zeta_{12}
^5x,\allowbreak -\zeta_3^2,\allowbreak \zeta_{12}^{11}x^{1/2}
,\allowbreak \zeta_3^2x,\allowbreak \zeta_{12}^5,\allowbreak -
1,\allowbreak \zeta_{12}^{11}x,\allowbreak \zeta_3^2,\allowbreak \zeta_{12}
^5x^{1/2},\allowbreak -\zeta_3^4,\allowbreak \zeta_{12}^{11})$\hfill\break
\begin{center}Non-principal $1$-Harish-Chandra series\end{center}
${\mathcal H}_{G_{8}}(Z_{4}^{1220})={\mathcal H}_{Z_{4}}
(x^3,\allowbreak i,\allowbreak -1,\allowbreak -ix^2)$\hfill\break
${\mathcal H}_{G_{8}}(Z_{4}^{0212})={\mathcal H}_{Z_{4}}
(x^3,\allowbreak i,\allowbreak -x^2,\allowbreak -i)$\hfill\break
${\mathcal H}_{G_{8}}(Z_{4}^{1022})={\mathcal H}_{Z_{4}}
(x^3,\allowbreak ix^2,\allowbreak -1,\allowbreak -i)$\hfill\break
\par\tablehead{\hline \gamma&\mbox{Deg($\gamma$)}&\mbox{Fr($\gamma$)}\\\hline}
\tabletail{\hline}
\begin{supertabular}{|R|RR|}
\shrinkheight{30pt}
*\hfill \phi_{1,0}&1&1\\
\hline
*\hfill \phi_{2,1}&\frac{-i+1}4x\Phi_2{\Phi'_4}\Phi_6\Phi_8{\Phi''_{12}}&1\\
\phi_{2,4}&\frac12x\Phi_4\Phi_8\Phi_{12}&1\\
\#\hfill \phi_{2,7}'&\frac{i+1}4x\Phi_2{\Phi''_4}\Phi_6\Phi_8{\Phi'_{12}}&1\\
Z_{4}^{1220}:1&\frac{-i-1}4x\Phi_1\Phi_3{\Phi'_4}\Phi_8{\Phi''_{12}}&-1\\
Z_{4}^{0212}:1&\frac i2x\Phi_1\Phi_2\Phi_3\Phi_6\Phi_8&-i\\
Z_{4}^{1022}:1&\frac{-i+1}4x\Phi_1\Phi_3{\Phi''_4}\Phi_8{\Phi'_{12}}&-1\\
\hline
*\hfill \phi_{3,2}&\frac{-i+1}4x^2\Phi_3\Phi_4\Phi_6{\Phi''_8}\Phi_{12}&1\\
\phi_{3,4}&\frac12x^2\Phi_3\Phi_6\Phi_8\Phi_{12}&1\\
\#\hfill \phi_{3,6}&\frac{i+1}4x^2\Phi_3\Phi_4\Phi_6{\Phi'_8}\Phi_{12}&1\\
Z_{4}^{1022}:i&\frac{-i-1}4x^2\Phi_1\Phi_2\Phi_3\Phi_6{\Phi''_8}\Phi_{12}&-1\\
Z_{4}^{0212}:-1&\frac i2x^2\Phi_1\Phi_2\Phi_3\Phi_4\Phi_6\Phi_{12}&-i\\
Z_{4}^{1220}:-i&\frac{-i+1}4x^2\Phi_1\Phi_2\Phi_3\Phi_6{\Phi'_8}\Phi_{12}&-
1\\
\hline
G_8[\zeta_8^3]&\frac i2x^3\Phi_1^2\Phi_2^2\Phi_3\Phi_6\Phi_8&\zeta_8^3x^{1/2}
\\
*\hfill \phi_{4,3}&\frac12x^3\Phi_4^2\Phi_8\Phi_{12}&1\\
G_8[-\zeta_8^3]&\frac i2x^3\Phi_1^2\Phi_2^2\Phi_3\Phi_6\Phi_8&\zeta_8^7x^{1/2}
\\
\phi_{4,5}&\frac12x^3\Phi_4^2\Phi_8\Phi_{12}&1\\
\hline
*\hfill \phi_{1,6}&\frac{-1}{12}x^6\Phi_3{\Phi''_4}^2\Phi_6\Phi_8{\Phi''_{12}}
&1\\
\#\hfill \phi_{1,18}&\frac{-1}{12}x^6\Phi_3{\Phi'_4}^2\Phi_6\Phi_8{\Phi'_{12}}
&1\\
\phi_{1,12}&\frac14x^6\Phi_4\Phi_6\Phi_8\Phi_{12}&1\\
\phi_{2,7}''&\frac{-i}4x^6\Phi_2\Phi_4{\Phi''_4}\Phi_6{\Phi''_8}\Phi_{12}&1\\
\phi_{2,13}&\frac i4x^6\Phi_2\Phi_4{\Phi'_4}\Phi_6{\Phi'_8}\Phi_{12}&1\\
\phi_{2,10}&\frac1{12}x^6\Phi_2^2\Phi_3\Phi_8\Phi_{12}&1\\
\phi_{3,8}&\frac14x^6\Phi_3\Phi_4\Phi_8\Phi_{12}&1\\
Z_{4}^{1220}:-1&\frac{-i}4x^6\Phi_1\Phi_3\Phi_4{\Phi''_4}{\Phi'_8}\Phi_{12}&-
1\\
Z_{4}^{1022}:-1&\frac{-i}4x^6\Phi_1\Phi_3\Phi_4{\Phi'_4}{\Phi''_8}\Phi_{12}&-
1\\
Z_{4}^{1220}:i&\frac{-1}4x^6\Phi_1\Phi_2\Phi_3\Phi_6\Phi_8{\Phi''_{12}}&-1\\
Z_{4}^{1022}:-i&\frac14x^6\Phi_1\Phi_2\Phi_3\Phi_6\Phi_8{\Phi'_{12}}&-1\\
Z_{4}^{0212}:i&\frac i4x^6\Phi_1\Phi_2^2\Phi_3{\Phi'_4}\Phi_6{\Phi'_8}{\Phi''_
{12}}&-i\\
Z_{4}^{0212}:-i&\frac i4x^6\Phi_1\Phi_2^2\Phi_3{\Phi''_4}\Phi_6{\Phi''_8}{
\Phi'_{12}}&-i\\
G_8[1]&\frac1{12}x^6\Phi_1^2\Phi_6\Phi_8\Phi_{12}&1\\
G_8[i]&\frac{-i}4x^6\Phi_1^2\Phi_2\Phi_3{\Phi'_4}\Phi_6{\Phi''_8}{\Phi''_{12}}
&i\\
G_8^2[i]&\frac{-i}4x^6\Phi_1^2\Phi_2\Phi_3{\Phi''_4}\Phi_6{\Phi'_8}{\Phi'_{12}
}&i\\
G_8[\zeta_3]&\frac13x^6\Phi_1^2\Phi_2^2\Phi_4^2\Phi_8&\zeta_3\\
G_8[\zeta_3^2]&\frac13x^6\Phi_1^2\Phi_2^2\Phi_4^2\Phi_8&\zeta_3^2\\
\end{supertabular}
\vfill\eject
\subsection{Unipotent characters for $G_{14}$}
\begin{center}Some principal $\zeta$-series\end{center}
$\zeta_{4}$ : ${\mathcal H}_{Z_{24}}(-x^2,\allowbreak \zeta_{24}^{19}
x,\allowbreak -\zeta_3^4x,\allowbreak x^{1/2},\allowbreak \zeta_3^2x^{2/3}
,\allowbreak \zeta_{24}^{23}x,\allowbreak x,\allowbreak \zeta_{24}
x,\allowbreak \zeta_3,\allowbreak ix^{1/2},\allowbreak -\zeta_
3^2x,\allowbreak \zeta_{24}^5x,\allowbreak x^{2/3},\allowbreak \zeta_{24}
^7x,\allowbreak \zeta_3x,\allowbreak -x^{1/2},\allowbreak \zeta_
3^2,\allowbreak \zeta_{24}^{11}x,\allowbreak -x,\allowbreak \zeta_{24}^{13}
x,\allowbreak \zeta_3x^{2/3},\allowbreak -ix^{1/2},\allowbreak \zeta_
3^2x,\allowbreak \zeta_{24}^{17}x)$\hfill\break
$\zeta_{8}$ : ${\mathcal H}_{Z_{24}}(-ix^2,\allowbreak \zeta_{12}^{11}
x,\allowbreak \zeta_{12},\allowbreak \zeta_{16}x^{1/2},\allowbreak \zeta_{12}
^5x^{2/3},\allowbreak \zeta_{12}x,\allowbreak \zeta_{16}^3x^{1/2}
,\allowbreak -\zeta_3^2x,\allowbreak \zeta_{24}
^5x,\allowbreak ix,\allowbreak \zeta_{12}^5,\allowbreak \zeta_
3x,\allowbreak -ix^{2/3},\allowbreak \zeta_{12}^5x,\allowbreak \zeta_{24}^{
11}x,\allowbreak \zeta_{16}^9x^{1/2},\allowbreak \zeta_{24}^{13}
x,\allowbreak \zeta_{12}^7x,\allowbreak \zeta_{16}^{11}x^{1/2}
,\allowbreak \zeta_3^2x,\allowbreak \zeta_{12}x^{2/3}
,\allowbreak -ix,\allowbreak \zeta_{24}^{19}x,\allowbreak -\zeta_
3^4x)$\hfill\break
$\zeta_{8}^{3}$ : ${\mathcal H}_{Z_{24}}(ix^2,\allowbreak \zeta_
3^2x,\allowbreak \zeta_{24}^{17}x,\allowbreak -ix,\allowbreak \zeta_{12}^{11}
x^{2/3},\allowbreak -\zeta_3^4x,\allowbreak \zeta_{16}x^{1/2}
,\allowbreak \zeta_{12}^{11}x,\allowbreak \zeta_{24}^{23}x,\allowbreak \zeta_{
16}^3x^{1/2},\allowbreak \zeta_{24}x,\allowbreak \zeta_{12}x,\allowbreak ix^{
2/3},\allowbreak -\zeta_3^2x,\allowbreak \zeta_{12}
^7,\allowbreak ix,\allowbreak \zeta_{24}^7x,\allowbreak \zeta_
3x,\allowbreak \zeta_{16}^9x^{1/2},\allowbreak \zeta_{12}
^5x,\allowbreak \zeta_{12}^7x^{2/3},\allowbreak \zeta_{16}^{11}x^{1/2}
,\allowbreak \zeta_{12}^{11},\allowbreak \zeta_{12}^7x)$\hfill\break
\begin{center}Non-principal $1$-Harish-Chandra series\end{center}
${\mathcal H}_{G_{14}}(Z_3)={\mathcal H}_{Z_{6}}(x^3,\allowbreak -\zeta_
3^2x^4,\allowbreak \zeta_3x^4,\allowbreak -1,\allowbreak \zeta_
3^2x^4,\allowbreak -\zeta_3^4x^4)$\hfill\break
\par\tablehead{\hline \gamma&\mbox{Deg($\gamma$)}&\mbox{Fr($\gamma$)}\\\hline}
\tabletail{\hline}
\begin{supertabular}{|R|RR|}
\shrinkheight{30pt}
*\hfill \phi_{1,0}&1&1\\
\hline
Z_3:-\zeta_3^2&\frac{\zeta_3}6x\Phi_1\Phi_2^2{\Phi''_3}\Phi_4{\Phi'_6}^2\Phi_
8{\Phi''''_{12}}\Phi_{24}&\zeta_3^2\\
G_{14}^2[\zeta_3^2]&\frac{\zeta_3}6x\Phi_1^2\Phi_2{\Phi''_3}^2\Phi_4{\Phi'_6}
\Phi_8{\Phi''''_{12}}\Phi_{24}&\zeta_3^2\\
Z_3:\zeta_3^2&\frac{\zeta_3}6x\Phi_1\Phi_2^2{\Phi''_3}\Phi_4\Phi_6^2\Phi_
8\Phi_{12}{\Phi'_{24}}&\zeta_3^2\\
G_{14}^2[-\zeta_3^2]&\frac{\zeta_3}6x\Phi_1^2\Phi_2\Phi_3^2\Phi_4{\Phi'_6}
\Phi_8\Phi_{12}{\Phi'_{24}}&-\zeta_3^2\\
\phi_{2,4}&\frac{-\zeta_3}{12}x\Phi_2^2{\Phi'_3}^2{\Phi''_6}^2\Phi_8\Phi_{12}
\Phi_{24}&1\\
\phi_{1,16}&\frac{\zeta_3}{12}x{\Phi'_3}^2\Phi_4\Phi_6^2\Phi_8{\Phi'''_{12}}
\Phi_{24}&1\\
G_{14}[1]&\frac{-\zeta_3}{12}x\Phi_1^2{\Phi'_3}^2{\Phi''_6}^2\Phi_8\Phi_{12}
\Phi_{24}&1\\
\phi_{3,4}&\frac{-\zeta_3}{12}x\Phi_3^2\Phi_4{\Phi''_6}^2\Phi_8{\Phi'''_{12}}
\Phi_{24}&1\\
G_{14}[\zeta_8]&\frac{\sqrt 6\zeta_3}{12}x\Phi_1^2\Phi_2^2\Phi_3^2\Phi_4\Phi_
6^2\Phi_{12}{\Phi''_{24}}&\zeta_8\\
G_{14}[\zeta_8^3]&\frac{\sqrt 6\zeta_3}{12}x\Phi_1^2\Phi_2^2\Phi_3^2\Phi_
4\Phi_6^2\Phi_{12}{\Phi''_{24}}&\zeta_8^3\\
\phi_{2,7}&\frac{(-2-\sqrt 6)\zeta_3}{24}x\Phi_2^2{\Phi'_3}^2\Phi_4\Phi_6^2{
\Phi^{(6)}_8}\Phi_{12}{\Phi''_{24}}{\Phi^{(10)}_{24}}&1\\
*\hfill \phi_{2,1}&\frac{(-2+\sqrt 6)\zeta_3}{24}x\Phi_2^2{\Phi'_3}^2\Phi_
4\Phi_6^2{\Phi^{(5)}_8}\Phi_{12}{\Phi''_{24}}{\Phi^{(9)}_{24}}&1\\
G_{14}^2[-1]&\frac{(2+\sqrt 6)\zeta_3}{24}x\Phi_1^2\Phi_3^2\Phi_4{\Phi''_6}^2{
\Phi^{(5)}_8}\Phi_{12}{\Phi''_{24}}{\Phi^{(9)}_{24}}&-1\\
G_{14}[-1]&\frac{(2-\sqrt 6)\zeta_3}{24}x\Phi_1^2\Phi_3^2\Phi_4{\Phi''_6}^2{
\Phi^{(6)}_8}\Phi_{12}{\Phi''_{24}}{\Phi^{(10)}_{24}}&-1\\
G_{14}^2[i]&\frac{\sqrt 6\zeta_3}{24}x\Phi_1^2\Phi_2^2\Phi_3^2\Phi_6^2{\Phi^{
(6)}_8}{\Phi'''_{12}}{\Phi''_{24}}{\Phi^{(9)}_{24}}&i\\
G_{14}^3[-i]&\frac{\sqrt 6\zeta_3}{24}x\Phi_1^2\Phi_2^2\Phi_3^2\Phi_6^2{\Phi^{
(5)}_8}{\Phi'''_{12}}{\Phi''_{24}}{\Phi^{(10)}_{24}}&-i\\
G_{14}^2[-i]&\frac{\sqrt 6\zeta_3}{24}x\Phi_1^2\Phi_2^2\Phi_3^2\Phi_6^2{\Phi^{
(6)}_8}{\Phi'''_{12}}{\Phi''_{24}}{\Phi^{(9)}_{24}}&-i\\
G_{14}^3[i]&\frac{\sqrt 6\zeta_3}{24}x\Phi_1^2\Phi_2^2\Phi_3^2\Phi_6^2{\Phi^{
(5)}_8}{\Phi'''_{12}}{\Phi''_{24}}{\Phi^{(10)}_{24}}&i\\
Z_3:-\zeta_3^4&\frac{-\zeta_3^2}6x\Phi_1\Phi_2^2{\Phi'_3}\Phi_4{\Phi''_6}
^2\Phi_8{\Phi'''_{12}}\Phi_{24}&\zeta_3^2\\
G_{14}[\zeta_3^2]&\frac{\zeta_3^2}6x\Phi_1^2\Phi_2{\Phi'_3}^2\Phi_4{\Phi''_6}
\Phi_8{\Phi'''_{12}}\Phi_{24}&\zeta_3^2\\
Z_3:\zeta_3&\frac{-\zeta_3^2}6x\Phi_1\Phi_2^2{\Phi'_3}\Phi_4\Phi_6^2\Phi_
8\Phi_{12}{\Phi''_{24}}&\zeta_3^2\\
G_{14}[-\zeta_3^2]&\frac{\zeta_3^2}6x\Phi_1^2\Phi_2\Phi_3^2\Phi_4{\Phi''_6}
\Phi_8\Phi_{12}{\Phi''_{24}}&-\zeta_3^2\\
\phi_{1,8}&\frac{\zeta_3^2}{12}x{\Phi''_3}^2\Phi_4\Phi_6^2\Phi_8{\Phi''''_{12}
}\Phi_{24}&1\\
\phi_{2,8}&\frac{-\zeta_3^2}{12}x\Phi_2^2{\Phi''_3}^2{\Phi'_6}^2\Phi_8\Phi_{
12}\Phi_{24}&1\\
\phi_{3,2}&\frac{-\zeta_3^2}{12}x\Phi_3^2\Phi_4{\Phi'_6}^2\Phi_8{\Phi''''_{12}
}\Phi_{24}&1\\
G_{14}^2[1]&\frac{-\zeta_3^2}{12}x\Phi_1^2{\Phi''_3}^2{\Phi'_6}^2\Phi_8\Phi_{
12}\Phi_{24}&1\\
G_{14}^2[\zeta_8]&\frac{\sqrt 6\zeta_3^2}{12}x\Phi_1^2\Phi_2^2\Phi_3^2\Phi_
4\Phi_6^2\Phi_{12}{\Phi'_{24}}&\zeta_8\\
G_{14}^2[\zeta_8^3]&\frac{\sqrt 6\zeta_3^2}{12}x\Phi_1^2\Phi_2^2\Phi_3^2\Phi_
4\Phi_6^2\Phi_{12}{\Phi'_{24}}&\zeta_8^3\\
\phi_{2,11}&\frac{(-2-\sqrt 6)\zeta_3^2}{24}x\Phi_2^2{\Phi''_3}^2\Phi_4\Phi_
6^2{\Phi^{(5)}_8}\Phi_{12}{\Phi'_{24}}{\Phi^{(11)}_{24}}&1\\
\#\hfill \phi_{2,5}&\frac{(-2+\sqrt 6)\zeta_3^2}{24}x\Phi_2^2{\Phi''_3}^2\Phi_
4\Phi_6^2{\Phi^{(6)}_8}\Phi_{12}{\Phi'_{24}}{\Phi^{(12)}_{24}}&1\\
G_{14}^4[-1]&\frac{(2+\sqrt 6)\zeta_3^2}{24}x\Phi_1^2\Phi_3^2\Phi_4{\Phi'_6}
^2{\Phi^{(6)}_8}\Phi_{12}{\Phi'_{24}}{\Phi^{(12)}_{24}}&-1\\
G_{14}^3[-1]&\frac{(2-\sqrt 6)\zeta_3^2}{24}x\Phi_1^2\Phi_3^2\Phi_4{\Phi'_6}
^2{\Phi^{(5)}_8}\Phi_{12}{\Phi'_{24}}{\Phi^{(11)}_{24}}&-1\\
G_{14}[i]&\frac{\sqrt 6\zeta_3^2}{24}x\Phi_1^2\Phi_2^2\Phi_3^2\Phi_6^2{\Phi^{
(6)}_8}{\Phi''''_{12}}{\Phi'_{24}}{\Phi^{(11)}_{24}}&i\\
G_{14}^4[-i]&\frac{\sqrt 6\zeta_3^2}{24}x\Phi_1^2\Phi_2^2\Phi_3^2\Phi_6^2{
\Phi^{(5)}_8}{\Phi''''_{12}}{\Phi'_{24}}{\Phi^{(12)}_{24}}&-i\\
G_{14}[-i]&\frac{\sqrt 6\zeta_3^2}{24}x\Phi_1^2\Phi_2^2\Phi_3^2\Phi_6^2{\Phi^{
(6)}_8}{\Phi''''_{12}}{\Phi'_{24}}{\Phi^{(11)}_{24}}&-i\\
G_{14}^4[i]&\frac{\sqrt 6\zeta_3^2}{24}x\Phi_1^2\Phi_2^2\Phi_3^2\Phi_6^2{
\Phi^{(5)}_8}{\Phi''''_{12}}{\Phi'_{24}}{\Phi^{(12)}_{24}}&i\\
\phi_{1,12}&\frac16x\Phi_3\Phi_6^2\Phi_8\Phi_{12}\Phi_{24}&1\\
\phi_{3,6}'&\frac16x\Phi_3^2\Phi_6\Phi_8\Phi_{12}\Phi_{24}&1\\
\phi_{4,3}&\frac16x\Phi_2^2\Phi_3\Phi_4\Phi_6^2\Phi_{12}\Phi_{24}&1\\
G_{14}^5[-1]&\frac16x\Phi_1^2\Phi_3^2\Phi_4\Phi_6\Phi_{12}\Phi_{24}&-1\\
G_{14}^2[\zeta_3]&\frac1{12}x\Phi_1^2\Phi_2^2\Phi_4{\Phi'_6}^2\Phi_8{\Phi'''_{
12}}\Phi_{24}&\zeta_3\\
G_{14}[\zeta_3]&\frac1{12}x\Phi_1^2\Phi_2^2\Phi_4{\Phi''_6}^2\Phi_8{\Phi''''_{
12}}\Phi_{24}&\zeta_3\\
G_{14}^4[\zeta_3]&\frac{-1}{12}x\Phi_1^2\Phi_2^2{\Phi''_3}^2\Phi_4\Phi_8{
\Phi'''_{12}}\Phi_{24}&\zeta_3\\
G_{14}^3[\zeta_3]&\frac{-1}{12}x\Phi_1^2\Phi_2^2{\Phi'_3}^2\Phi_4\Phi_8{
\Phi''''_{12}}\Phi_{24}&\zeta_3\\
G_{14}[\zeta_{24}^{11}]&\frac{\sqrt 6}{12}x\Phi_1^2\Phi_2^2\Phi_3^2\Phi_4\Phi_
6^2\Phi_8\Phi_{12}&\zeta_{24}^{11}\\
G_{14}[\zeta_{24}^{17}]&\frac{\sqrt 6}{12}x\Phi_1^2\Phi_2^2\Phi_3^2\Phi_4\Phi_
6^2\Phi_8\Phi_{12}&\zeta_{24}^{17}\\
G_{14}^6[\zeta_3]&\frac{2+\sqrt 6}{24}x\Phi_1^2\Phi_2^2\Phi_4\Phi_6^2\Phi_
8\Phi_{12}{\Phi^{(6)}_{24}}&\zeta_3\\
G_{14}^5[\zeta_3]&\frac{2-\sqrt 6}{24}x\Phi_1^2\Phi_2^2\Phi_4\Phi_6^2\Phi_
8\Phi_{12}{\Phi^{(5)}_{24}}&\zeta_3\\
G_{14}^2[-\zeta_3]&\frac{2+\sqrt 6}{24}x\Phi_1^2\Phi_2^2\Phi_3^2\Phi_4\Phi_
8\Phi_{12}{\Phi^{(5)}_{24}}&-\zeta_3^4\\
G_{14}[-\zeta_3]&\frac{2-\sqrt 6}{24}x\Phi_1^2\Phi_2^2\Phi_3^2\Phi_4\Phi_
8\Phi_{12}{\Phi^{(6)}_{24}}&-\zeta_3^4\\
G_{14}[\zeta_{12}^7]&\frac{\sqrt 6}{24}x\Phi_1^2\Phi_2^2\Phi_3^2\Phi_4\Phi_
6^2\Phi_8{\Phi^{(8)}_{24}}&\zeta_{12}^7\\
G_{14}^2[-\zeta_{12}^7]&\frac{\sqrt 6}{24}x\Phi_1^2\Phi_2^2\Phi_3^2\Phi_4\Phi_
6^2\Phi_8{\Phi^{(7)}_{24}}&\zeta_{12}\\
G_{14}[-\zeta_{12}^7]&\frac{\sqrt 6}{24}x\Phi_1^2\Phi_2^2\Phi_3^2\Phi_4\Phi_
6^2\Phi_8{\Phi^{(8)}_{24}}&\zeta_{12}\\
G_{14}^2[\zeta_{12}^7]&\frac{\sqrt 6}{24}x\Phi_1^2\Phi_2^2\Phi_3^2\Phi_4\Phi_
6^2\Phi_8{\Phi^{(7)}_{24}}&\zeta_{12}^7\\
\hline
*\hfill \phi_{4,5}&\frac{3-\sqrt {-3}}6x^5{\Phi'_3}\Phi_4{\Phi''_6}\Phi_8\Phi_
{12}\Phi_{24}&1\\
\#\hfill \phi_{4,7}&\frac{3+\sqrt {-3}}6x^5{\Phi''_3}\Phi_4{\Phi'_6}\Phi_
8\Phi_{12}\Phi_{24}&1\\
Z_3:1&\frac{\sqrt {-3}}3x^5\Phi_1\Phi_2\Phi_4\Phi_8\Phi_{12}\Phi_{24}&\zeta_
3^2\\
\hline
*\hfill \phi_{3,6}''&\frac13x^6\Phi_3^2\Phi_6^2\Phi_{12}\Phi_{24}&1\\
G_{14}^2[\zeta_9^5]&\frac{\zeta_3^2}3x^6\Phi_1^2\Phi_2^2{\Phi'_3}^2\Phi_4{
\Phi''_6}^2\Phi_8{\Phi'''_{12}}{\Phi''_{24}}&\zeta_9^5x^{2/3}\\
G_{14}[\zeta_9^2]&\frac{\zeta_3}3x^6\Phi_1^2\Phi_2^2{\Phi''_3}^2\Phi_4{\Phi'_
6}^2\Phi_8{\Phi''''_{12}}{\Phi'_{24}}&\zeta_9^2x^{1/3}\\
\#\hfill \phi_{3,8}&\frac13x^6\Phi_3^2\Phi_6^2\Phi_{12}\Phi_{24}&1\\
G_{14}^2[\zeta_9^8]&\frac{\zeta_3^2}3x^6\Phi_1^2\Phi_2^2{\Phi'_3}^2\Phi_4{
\Phi''_6}^2\Phi_8{\Phi'''_{12}}{\Phi''_{24}}&\zeta_9^8x^{2/3}\\
G_{14}[\zeta_9^8]&\frac{\zeta_3}3x^6\Phi_1^2\Phi_2^2{\Phi''_3}^2\Phi_4{\Phi'_
6}^2\Phi_8{\Phi''''_{12}}{\Phi'_{24}}&\zeta_9^8x^{1/3}\\
\phi_{3,10}&\frac13x^6\Phi_3^2\Phi_6^2\Phi_{12}\Phi_{24}&1\\
G_{14}^2[\zeta_9^2]&\frac{\zeta_3^2}3x^6\Phi_1^2\Phi_2^2{\Phi'_3}^2\Phi_4{
\Phi''_6}^2\Phi_8{\Phi'''_{12}}{\Phi''_{24}}&\zeta_9^2x^{2/3}\\
G_{14}[\zeta_9^5]&\frac{\zeta_3}3x^6\Phi_1^2\Phi_2^2{\Phi''_3}^2\Phi_4{\Phi'_
6}^2\Phi_8{\Phi''''_{12}}{\Phi'_{24}}&\zeta_9^5x^{1/3}\\
\hline
\phi_{2,15}&\frac14x^9\Phi_2^2\Phi_6^2\Phi_8\Phi_{24}&1\\
*\hfill \phi_{2,9}&\frac14x^9\Phi_2^2\Phi_6^2\Phi_8\Phi_{24}&1\\
\phi_{2,12}&\frac12x^9\Phi_4\Phi_8\Phi_{12}\Phi_{24}&1\\
G_{14}^6[-1]&\frac14x^9\Phi_1^2\Phi_3^2\Phi_8\Phi_{24}&-1\\
G_{14}^7[-1]&\frac14x^9\Phi_1^2\Phi_3^2\Phi_8\Phi_{24}&-1\\
G_{14}[\zeta_{16}^5]&\frac{-\sqrt {-2}}4x^9\Phi_1^2\Phi_2^2\Phi_3^2\Phi_4\Phi_
6^2\Phi_{12}&\zeta_{16}^5x^{1/2}\\
G_{14}[\zeta_{16}^7]&\frac{-\sqrt {-2}}4x^9\Phi_1^2\Phi_2^2\Phi_3^2\Phi_4\Phi_
6^2\Phi_{12}&\zeta_{16}^7x^{1/2}\\
G_{14}[-\zeta_{16}^5]&\frac{-\sqrt {-2}}4x^9\Phi_1^2\Phi_2^2\Phi_3^2\Phi_
4\Phi_6^2\Phi_{12}&\zeta_{16}^{13}x^{1/2}\\
G_{14}[-\zeta_{16}^7]&\frac{-\sqrt {-2}}4x^9\Phi_1^2\Phi_2^2\Phi_3^2\Phi_
4\Phi_6^2\Phi_{12}&\zeta_{16}^{15}x^{1/2}\\
\hline
*\hfill \phi_{1,20}&\frac{3-\sqrt {-3}}6x^{20}{\Phi'_3}{\Phi''_6}{\Phi'''_{12}
}{\Phi''_{24}}&1\\
\#\hfill \phi_{1,28}&\frac{3+\sqrt {-3}}6x^{20}{\Phi''_3}{\Phi'_6}{\Phi''''_{
12}}{\Phi'_{24}}&1\\
Z_3:-1&\frac{\sqrt {-3}}3x^{20}\Phi_1\Phi_2\Phi_4\Phi_8&\zeta_3^2\\
\end{supertabular}
\vfill\eject
\subsection{Unipotent characters for $G_{3,1,2}$}
\begin{center}Some principal $\zeta$-series\end{center}
$-1$ : ${\mathcal H}_{Z_{6}}(x^2,\allowbreak \zeta_3^2x,\allowbreak \zeta_
3,\allowbreak x,\allowbreak \zeta_3^2,\allowbreak \zeta_3x)$\hfill\break
\begin{center}Non-principal $1$-Harish-Chandra series\end{center}
${\mathcal H}_{G_{3,1,2}}(Z_3)={\mathcal H}_{Z_{3}}(1,\allowbreak \zeta_
3x^2,\allowbreak \zeta_3^2x^2)$\hfill\break\par
\tablehead{\hline \gamma&\mbox{Deg($\gamma$)}&\mbox{Fr($\gamma$)}&\mbox{Symbol\
}\\\hline}
\tabletail{\hline}
\begin{supertabular}{|R|RRR|}
\shrinkheight{30pt}
11..&\frac13q\Phi_3\Phi_6&1&(12,0,0)\\
Z_3:\zeta_3^2&\frac{-\zeta_3^2}3q\Phi_1\Phi_2{\Phi'_3}{\Phi''_6}&\zeta_3^2&(,0\
1,02)\\
Z_3:\zeta_3&\frac{\zeta_3}3q\Phi_1\Phi_2{\Phi''_3}{\Phi'_6}&\zeta_3^2&(,02,01)\
\\
..2&\frac{-\zeta_3^2}3q{\Phi'_3}^2\Phi_6&1&(01,0,2)\\
G_{3,1,2}^{130}&\frac{-\zeta_3}3q\Phi_1^2\Phi_2{\Phi''_6}&\zeta_3&(0,012,)\\
\#\hfill 1..1&\frac13q\Phi_2{\Phi''_3}^2{\Phi'_6}&1&(02,0,1)\\
.2.&\frac{-\zeta_3}3q{\Phi''_3}^2\Phi_6&1&(01,2,0)\\
*\hfill 1.1.&\frac13q\Phi_2{\Phi'_3}^2{\Phi''_6}&1&(02,1,0)\\
G_{3,1,2}^{103}&\frac{-\zeta_3^2}3q\Phi_1^2\Phi_2{\Phi'_6}&\zeta_3&(0,,012)\\
\hline
*\hfill .1.1&q^3\Phi_2\Phi_6&1&(01,1,1)\\
\hline
*\hfill 2..&1&1&(2,,)\\
\hline
Z_3:1&\frac{\sqrt {-3}}3q^5\Phi_1\Phi_2&\zeta_3^2&(1,012,012)\\
\#\hfill ..11&\frac{3+\sqrt {-3}}6q^5{\Phi''_3}{\Phi'_6}&1&(012,01,12)\\
*\hfill .11.&\frac{3-\sqrt {-3}}6q^5{\Phi'_3}{\Phi''_6}&1&(012,12,01)\\
\end{supertabular}
\par
We  used partition tuples for the principal series, the shape of the symbol
for  cuspidals  and  notation  coming  the  relative  group $Z_3$ for the
characters Harish-Chandra induced from $Z_3$.
\vfill\eject
\subsection{Unipotent characters for $G_{24}$}
\begin{center}Some principal $\zeta$-series\end{center}
$\zeta_{3}^{2}$ : ${\mathcal H}_{Z_{6}}(\zeta_3x^7,\allowbreak -\zeta_3^4x^{
7/2},\allowbreak \zeta_3x^3,\allowbreak -\zeta_3^4x^4,\allowbreak \zeta_3x^{
7/2},\allowbreak -\zeta_3^4)$\hfill\break
$\zeta_{3}$ : ${\mathcal H}_{Z_{6}}(\zeta_3^2x^7,\allowbreak -\zeta_
3^2,\allowbreak \zeta_3^2x^{7/2},\allowbreak -\zeta_3^2x^4,\allowbreak \zeta_
3^2x^3,\allowbreak -\zeta_3^2x^{7/2})$\hfill\break
$\zeta_{7}^{4}$ : ${\mathcal H}_{Z_{14}}(\zeta_7^2x^3,\allowbreak -\zeta_7^{
10}x^2,\allowbreak \zeta_7^2x^{3/2},\allowbreak -\zeta_
7^4x^2,\allowbreak \zeta_7^5x,\allowbreak -\zeta_7^9x,\allowbreak \zeta_
7^2x^2,\allowbreak -\zeta_7^6x^2,\allowbreak x,\allowbreak -\zeta_7^9x^{3/2}
,\allowbreak \zeta_7x,\allowbreak -\zeta_7^9,\allowbreak \zeta_
7^2x,\allowbreak -\zeta_7^9x^2)$\hfill\break
$\zeta_{7}$ : ${\mathcal H}_{Z_{14}}(\zeta_7^4x^3,\allowbreak -\zeta_
7^4,\allowbreak x,\allowbreak -\zeta_7^4x,\allowbreak \zeta_7^4x^{3/2}
,\allowbreak -\zeta_7^4x^2,\allowbreak \zeta_7^2x,\allowbreak -\zeta_
7^5x^2,\allowbreak \zeta_7^3x,\allowbreak -\zeta_7^6x^2,\allowbreak \zeta_
7^4x,\allowbreak -\zeta_7^4x^{3/2},\allowbreak \zeta_7^4x^2,\allowbreak -
\zeta_7^8x^2)$\hfill\break
\begin{center}Non-principal $1$-Harish-Chandra series\end{center}
${\mathcal H}_{G_{24}}(B_2)={\mathcal H}_{A_1}(x^7,-1)$\hfill\break
\par\tablehead{\hline \gamma&\mbox{Deg($\gamma$)}&\mbox{Fr($\gamma$)}\\\hline}
\tabletail{\hline}
\begin{supertabular}{|R|RR|}
\shrinkheight{30pt}
*\hfill \phi_{1,0}&1&1\\
\hline
*\hfill \phi_{3,1}&\frac{\sqrt {-7}}{14}x\Phi_3\Phi_4\Phi_6{\Phi'_7}{\Phi''_{
14}}&1\\
\#\hfill \phi_{3,3}&\frac{-\sqrt {-7}}{14}x\Phi_3\Phi_4\Phi_6{\Phi''_7}{\Phi'_
{14}}&1\\
\phi_{6,2}&\frac12x\Phi_2^2\Phi_3\Phi_6\Phi_{14}&1\\
B_2:2&\frac12x\Phi_1^2\Phi_3\Phi_6\Phi_7&-1\\
G_{24}[\zeta_7^4]&\frac{\sqrt {-7}}7x\Phi_1^3\Phi_2^3\Phi_3\Phi_4\Phi_6&\zeta_
7^4\\
G_{24}[\zeta_7^2]&\frac{\sqrt {-7}}7x\Phi_1^3\Phi_2^3\Phi_3\Phi_4\Phi_6&\zeta_
7^2\\
G_{24}[\zeta_7]&\frac{\sqrt {-7}}7x\Phi_1^3\Phi_2^3\Phi_3\Phi_4\Phi_6&\zeta_
7\\
\hline
*\hfill \phi_{7,3}&x^3\Phi_7\Phi_{14}&1\\
\hline
*\hfill \phi_{8,4}&\frac12x^4\Phi_2^3\Phi_4\Phi_6\Phi_{14}&1\\
\#\hfill \phi_{8,5}&\frac12x^4\Phi_2^3\Phi_4\Phi_6\Phi_{14}&1\\
G_{24}[i]&\frac{-1}2x^4\Phi_1^3\Phi_3\Phi_4\Phi_7&ix^{1/2}\\
G_{24}[-i]&\frac{-1}2x^4\Phi_1^3\Phi_3\Phi_4\Phi_7&-ix^{1/2}\\
\hline
*\hfill \phi_{7,6}&x^6\Phi_7\Phi_{14}&1\\
\hline
\#\hfill \phi_{3,10}&\frac{\sqrt {-7}}{14}x^8\Phi_3\Phi_4\Phi_6{\Phi'_7}{
\Phi''_{14}}&1\\
*\hfill \phi_{3,8}&\frac{-\sqrt {-7}}{14}x^8\Phi_3\Phi_4\Phi_6{\Phi''_7}{
\Phi'_{14}}&1\\
\phi_{6,9}&\frac12x^8\Phi_2^2\Phi_3\Phi_6\Phi_{14}&1\\
B_2:11&\frac12x^8\Phi_1^2\Phi_3\Phi_6\Phi_7&-1\\
G_{24}[\zeta_7^3]&\frac{\sqrt {-7}}7x^8\Phi_1^3\Phi_2^3\Phi_3\Phi_4\Phi_6&
\zeta_7^3\\
G_{24}[\zeta_7^5]&\frac{\sqrt {-7}}7x^8\Phi_1^3\Phi_2^3\Phi_3\Phi_4\Phi_6&
\zeta_7^5\\
G_{24}[\zeta_7^6]&\frac{\sqrt {-7}}7x^8\Phi_1^3\Phi_2^3\Phi_3\Phi_4\Phi_6&
\zeta_7^6\\
\hline
*\hfill \phi_{1,21}&x^{21}&1\\
\end{supertabular}
\vfill\eject
\subsection{Unipotent characters for $G_{25}$}
\begin{center}Some principal $\zeta$-series\end{center}
$\zeta_{9}$ : ${\mathcal H}_{Z_{9}}(\zeta_9^5x^4,\allowbreak \zeta_
9^8x^2,\allowbreak \zeta_9^2,\allowbreak \zeta_9^2x,\allowbreak \zeta_
9^2x^2,\allowbreak \zeta_9^5,\allowbreak \zeta_9^5x,\allowbreak \zeta_
9^5x^2,\allowbreak \zeta_9^8)$\hfill\break
$\zeta_{4}^{3}$ : ${\mathcal H}_{Z_{12}}(-ix^3,\allowbreak \zeta_{12}
,\allowbreak \zeta_{12}^5x,\allowbreak i,\allowbreak \zeta_{12}
^7x,\allowbreak \zeta_{12}^5,\allowbreak -ix,\allowbreak \zeta_{12}
x^2,\allowbreak \zeta_{12}^{11}x,\allowbreak -i,\allowbreak \zeta_{12}
x,\allowbreak \zeta_{12}^5x^2)$\hfill\break
$\zeta_{4}$ : ${\mathcal H}_{Z_{12}}(ix^3,\allowbreak \zeta_{12}
^7x^2,\allowbreak \zeta_{12}^{11}x,\allowbreak i,\allowbreak \zeta_{12}
x,\allowbreak \zeta_{12}^{11}x^2,\allowbreak ix,\allowbreak \zeta_{12}
^7,\allowbreak \zeta_{12}^5x,\allowbreak -i,\allowbreak \zeta_{12}
^7x,\allowbreak \zeta_{12}^{11})$\hfill\break
$-1$ : ${\mathcal H}_{G_{5}}(x^2,\zeta_3,\zeta_3^2;-x,\zeta_3,\zeta_3^2)$\hfill\break
\begin{center}Non-principal $1$-Harish-Chandra series\end{center}
${\mathcal H}_{G_{25}}(Z_3)={\mathcal H}_{G_{3,1,2}}
(x,\zeta_3,\zeta_3^2;x^3,-1)$\hfill\break
${\mathcal H}_{G_{25}}(Z_3\otimes Z_3)={\mathcal H}_{Z_{6}}(x^3,\allowbreak -
\zeta_3^2x^3,\allowbreak \zeta_3x^2,\allowbreak -1,\allowbreak \zeta_
3^2,\allowbreak -\zeta_3^4x^2)$\hfill\break
${\mathcal H}_{G_{25}}(G_4)={\mathcal H}_{Z_{3}}(1,\allowbreak \zeta_
3x^4,\allowbreak \zeta_3^2x^4)$\hfill\break
\par\tablehead{\hline \gamma&\mbox{Deg($\gamma$)}&\mbox{Fr($\gamma$)}\\\hline}
\tabletail{\hline}
\begin{supertabular}{|R|RR|}
\shrinkheight{30pt}
*\hfill \phi_{1,0}&1&1\\
\hline
*\hfill \phi_{3,1}&\frac{3-\sqrt {-3}}6x{\Phi''_3}{\Phi'_6}\Phi_9{\Phi''''_{
12}}&1\\
\#\hfill \phi_{3,5}'&\frac{3+\sqrt {-3}}6x{\Phi'_3}{\Phi''_6}\Phi_9{\Phi'''_{
12}}&1\\
Z_3:2..&\frac{\sqrt {-3}}3x\Phi_1\Phi_2\Phi_4\Phi_9&\zeta_3^2\\
\hline
\phi_{8,3}&\frac13x^2\Phi_2^2\Phi_4\Phi_9\Phi_{12}&1\\
*\hfill \phi_{6,2}&\frac13x^2\Phi_3{\Phi'_3}\Phi_4\Phi_6{\Phi''_9}\Phi_{12}&
1\\
\#\hfill \phi_{6,4}'&\frac13x^2\Phi_3{\Phi''_3}\Phi_4\Phi_6{\Phi'_9}\Phi_{12}&
1\\
\phi_{2,9}&\frac{\zeta_3}3x^2\Phi_4{\Phi'_6}^2\Phi_9\Phi_{12}&1\\
Z_3\otimes Z_3:-\zeta_3^2&\frac{\zeta_3}3x^2\Phi_1^2\Phi_2{\Phi'_3}\Phi_4{
\Phi''_6}{\Phi''_9}\Phi_{12}&\zeta_3\\
Z_3:..2&\frac{\zeta_3}3x^2\Phi_1\Phi_2{\Phi'_3}^2\Phi_4{\Phi''_6}{\Phi'_9}
\Phi_{12}&\zeta_3^2\\
\phi_{2,3}&\frac{\zeta_3^2}3x^2\Phi_4{\Phi''_6}^2\Phi_9\Phi_{12}&1\\
Z_3:.2.&\frac{-\zeta_3^2}3x^2\Phi_1\Phi_2{\Phi''_3}^2\Phi_4{\Phi'_6}{\Phi''_9}
\Phi_{12}&\zeta_3^2\\
Z_3\otimes Z_3:1&\frac{\zeta_3^2}3x^2\Phi_1^2\Phi_2{\Phi''_3}\Phi_4{\Phi'_6}{
\Phi'_9}\Phi_{12}&\zeta_3\\
\hline
*\hfill \phi_{6,4}''&\frac{-\zeta_3}6x^4\Phi_2^2{\Phi'_3}^2{\Phi''_6}^2\Phi_
9\Phi_{12}&1\\
\phi_{9,5}&\frac{3-\sqrt {-3}}{12}x^4{\Phi''_3}^3\Phi_4{\Phi'_6}\Phi_9\Phi_{
12}&1\\
\phi_{3,6}&\frac13x^4\Phi_3\Phi_6^2\Phi_9\Phi_{12}&1\\
\phi_{9,7}&\frac{3+\sqrt {-3}}{12}x^4{\Phi'_3}^3\Phi_4{\Phi''_6}\Phi_9\Phi_{
12}&1\\
\#\hfill \phi_{6,8}''&\frac{-\zeta_3^2}6x^4\Phi_2^2{\Phi''_3}^2{\Phi'_6}
^2\Phi_9\Phi_{12}&1\\
Z_3\otimes Z_3:\zeta_3&\frac16x^4\Phi_1^2\Phi_2^2\Phi_4{\Phi''_6}^2\Phi_9{
\Phi''''_{12}}&\zeta_3\\
G_4:\zeta_3&\frac{3+\sqrt {-3}}{12}x^4\Phi_1^2\Phi_3^2{\Phi'_3}{\Phi''_6}\Phi_
9{\Phi'''_{12}}&-1\\
Z_3:1..1&\frac{-\zeta_3^2}3x^4\Phi_1\Phi_2^2{\Phi'_3}\Phi_4{\Phi''_6}^2\Phi_9{
\Phi'''_{12}}&\zeta_3^2\\
G_{25}[-\zeta_3]&\frac{\sqrt {-3}}6x^4\Phi_1^3\Phi_2\Phi_3^2\Phi_4\Phi_9&-
\zeta_3^4\\
\phi_{3,5}''&\frac{\zeta_3^2}6x^4{\Phi''_3}^2\Phi_4\Phi_6^2\Phi_9{\Phi''''_{
12}}&1\\
G_{25}[\zeta_3]&\frac{\sqrt {-3}}6x^4\Phi_1^3\Phi_2\Phi_4\Phi_9\Phi_{12}&
\zeta_3\\
Z_3:1.1.&\frac{\zeta_3}3x^4\Phi_1\Phi_2^2{\Phi''_3}\Phi_4{\Phi'_6}^2\Phi_9{
\Phi''''_{12}}&\zeta_3^2\\
\phi_{3,13}''&\frac{\zeta_3}6x^4{\Phi'_3}^2\Phi_4\Phi_6^2\Phi_9{\Phi'''_{12}}&
1\\
G_4:\zeta_3^2&\frac{3-\sqrt {-3}}{12}x^4\Phi_1^2\Phi_3^2{\Phi''_3}{\Phi'_6}
\Phi_9{\Phi''''_{12}}&-1\\
Z_3\otimes Z_3:-\zeta_3^4&\frac16x^4\Phi_1^2\Phi_2^2\Phi_4{\Phi'_6}^2\Phi_9{
\Phi'''_{12}}&\zeta_3\\
\hline
*\hfill \phi_{8,6}&\frac{3-\sqrt {-3}}6x^6\Phi_2^2\Phi_4\Phi_6^2{\Phi''_9}
\Phi_{12}&1\\
\#\hfill \phi_{8,9}&\frac{3+\sqrt {-3}}6x^6\Phi_2^2\Phi_4\Phi_6^2{\Phi'_9}
\Phi_{12}&1\\
Z_3:.1.1&\frac{\sqrt {-3}}3x^6\Phi_1\Phi_2^2\Phi_3\Phi_4\Phi_6^2\Phi_{12}&
\zeta_3^2\\
\hline
*\hfill \phi_{6,8}'&\frac{3-\sqrt {-3}}6x^8{\Phi'_3}\Phi_4{\Phi''_6}\Phi_
9\Phi_{12}&1\\
\#\hfill \phi_{6,10}&\frac{3+\sqrt {-3}}6x^8{\Phi''_3}\Phi_4{\Phi'_6}\Phi_
9\Phi_{12}&1\\
Z_3:11..&\frac{\sqrt {-3}}3x^8\Phi_1\Phi_2\Phi_4\Phi_9\Phi_{12}&\zeta_3^2\\
\hline
*\hfill \phi_{1,12}&\frac{-1}6x^{12}\Phi_4{\Phi''_6}^2\Phi_9{\Phi''''_{12}}&
1\\
\phi_{3,13}'&\frac{-\zeta_3}3x^{12}\Phi_3{\Phi''_3}\Phi_6{\Phi''_9}\Phi_{12}&
1\\
Z_3:..11&\frac{\zeta_3}3x^{12}\Phi_1\Phi_2{\Phi'_3}^2\Phi_4{\Phi''_6}{\Phi''_
9}{\Phi'''_{12}}&\zeta_3^2\\
\phi_{3,17}&\frac{-\zeta_3^2}3x^{12}\Phi_3{\Phi'_3}\Phi_6{\Phi'_9}\Phi_{12}&
1\\
Z_3:.11.&\frac{-\zeta_3^2}3x^{12}\Phi_1\Phi_2{\Phi''_3}^2\Phi_4{\Phi'_6}{
\Phi'_9}{\Phi''''_{12}}&\zeta_3^2\\
\#\hfill \phi_{1,24}&\frac{-1}6x^{12}\Phi_4{\Phi'_6}^2\Phi_9{\Phi'''_{12}}&1\\
Z_3\otimes Z_3:\zeta_3^2&\frac{\zeta_3^2}3x^{12}\Phi_1^2\Phi_2{\Phi'_3}\Phi_4{
\Phi''_6}{\Phi'_9}{\Phi'''_{12}}&\zeta_3\\
G_4:1&\frac12x^{12}\Phi_1^2\Phi_3^2\Phi_9&-1\\
\phi_{2,15}&\frac16x^{12}\Phi_2^2\Phi_9\Phi_{12}&1\\
Z_3\otimes Z_3:-1&\frac{\zeta_3}3x^{12}\Phi_1^2\Phi_2{\Phi''_3}\Phi_4{\Phi'_6}
{\Phi''_9}{\Phi''''_{12}}&\zeta_3\\
\end{supertabular}
\vfill\eject
\subsection{Unipotent characters for $G_{26}$}
\begin{center}Some principal $\zeta$-series\end{center}
$\zeta_{9}^{2}$ : ${\mathcal H}_{Z_{18}}(\zeta_3x^3,\allowbreak -\zeta_
3^4x,\allowbreak \zeta_9^8x,\allowbreak -\zeta_3^4x^{3/2}
,\allowbreak x,\allowbreak -\zeta_3^4x^2,\allowbreak \zeta_3,\allowbreak -
\zeta_3^2x,\allowbreak \zeta_9^2x,\allowbreak -1,\allowbreak \zeta_
3x,\allowbreak -\zeta_3^2x^2,\allowbreak \zeta_3x^{3/2},\allowbreak -
x,\allowbreak \zeta_9^5x,\allowbreak -\zeta_3^4,\allowbreak \zeta_
3^2x,\allowbreak -x^2)$\hfill\break
$\zeta_{9}$ : ${\mathcal H}_{Z_{18}}(\zeta_3^2x^3,\allowbreak -\zeta_
3^4x^2,\allowbreak x,\allowbreak -\zeta_3^2,\allowbreak \zeta_9x,\allowbreak -
\zeta_3^2x,\allowbreak \zeta_3^2x^{3/2},\allowbreak -\zeta_
3^2x^2,\allowbreak \zeta_3x,\allowbreak -1,\allowbreak \zeta_
9^4x,\allowbreak -x,\allowbreak \zeta_3^2,\allowbreak -x^2,\allowbreak \zeta_
3^2x,\allowbreak -\zeta_3^2x^{3/2},\allowbreak \zeta_9^7x,\allowbreak -\zeta_
3^4x)$\hfill\break
\begin{center}Non-principal $1$-Harish-Chandra series\end{center}
${\mathcal H}_{G_{26}}(Z_3)={\mathcal H}_{G_{6,2,2}}(1,\allowbreak \zeta_
3x^2,\allowbreak \zeta_3^2x^2;x^3,-1;x,-1)$\hfill\break
${\mathcal H}_{G_{26}}(G_4)={\mathcal H}_{Z_{6}}(x^3,\allowbreak -\zeta_
3^2x^4,\allowbreak \zeta_3x^3,\allowbreak -1,\allowbreak \zeta_
3^2x^3,\allowbreak -\zeta_3^4x^4)$\hfill\break
${\mathcal H}_{G_{26}}(G_{3,1,2}^{103})={\mathcal H}_{Z_{6}}(x^4,\allowbreak -
\zeta_3^2x^3,\allowbreak \zeta_3x,\allowbreak -x,\allowbreak \zeta_
3^2,\allowbreak -\zeta_3^4x)$\hfill\break
${\mathcal H}_{G_{26}}(G_{3,1,2}^{130})={\mathcal H}_{Z_{6}}(x^4,\allowbreak -
\zeta_3^2x,\allowbreak \zeta_3,\allowbreak -x,\allowbreak \zeta_
3^2x,\allowbreak -\zeta_3^4x^3)$\hfill\break
\par\tablehead{\hline \gamma&\mbox{Deg($\gamma$)}&\mbox{Fr($\gamma$)}\\\hline}
\tabletail{\hline}
\begin{supertabular}{|R|RR|}
\shrinkheight{30pt}
*\hfill \phi_{1,0}&1&1\\
\hline
\phi_{1,9}&\frac13x\Phi_9\Phi_{12}\Phi_{18}&1\\
\#\hfill \phi_{3,5}'&\frac13x\Phi_3{\Phi'_3}\Phi_6{\Phi''_6}{\Phi''_9}\Phi_{
12}{\Phi'_{18}}&1\\
*\hfill \phi_{3,1}&\frac13x\Phi_3{\Phi''_3}\Phi_6{\Phi'_6}{\Phi'_9}\Phi_{12}{
\Phi''_{18}}&1\\
\phi_{2,9}&\frac{-\zeta_3^2}3x\Phi_4\Phi_9{\Phi'''_{12}}\Phi_{18}&1\\
G_{3,1,2}^{103}:1&\frac{-\zeta_3^2}3x\Phi_1^2\Phi_2^2{\Phi'_3}\Phi_4{\Phi''_6}
{\Phi''_9}{\Phi'''_{12}}{\Phi'_{18}}&\zeta_3\\
Z_3:.....2&\frac{-\zeta_3^2}3x\Phi_1\Phi_2{\Phi'_3}^2\Phi_4{\Phi''_6}^2{\Phi'_
9}{\Phi'''_{12}}{\Phi''_{18}}&\zeta_3^2\\
\phi_{2,3}&\frac{-\zeta_3}3x\Phi_4\Phi_9{\Phi''''_{12}}\Phi_{18}&1\\
Z_3:....2.&\frac{\zeta_3}3x\Phi_1\Phi_2{\Phi''_3}^2\Phi_4{\Phi'_6}^2{\Phi''_9}
{\Phi''''_{12}}{\Phi'_{18}}&\zeta_3^2\\
G_{3,1,2}^{130}:1&\frac{-\zeta_3}3x\Phi_1^2\Phi_2^2{\Phi''_3}\Phi_4{\Phi'_6}{
\Phi'_9}{\Phi''''_{12}}{\Phi''_{18}}&\zeta_3\\
\hline
\phi_{3,6}&\frac13x^2\Phi_3\Phi_6\Phi_9\Phi_{12}\Phi_{18}&1\\
\phi_{3,8}'&\frac{-\zeta_3^2}3x^2{\Phi''_3}^2{\Phi'_6}^2\Phi_9\Phi_{12}\Phi_{
18}&1\\
\phi_{3,4}&\frac{-\zeta_3}3x^2{\Phi'_3}^2{\Phi''_6}^2\Phi_9\Phi_{12}\Phi_{18}&
1\\
\#\hfill \phi_{6,4}'&\frac13x^2{\Phi'_3}^2\Phi_4{\Phi''_6}^2\Phi_9{\Phi'''_{
12}}\Phi_{18}&1\\
G_{3,1,2}^{130}:-\zeta_3^4&\frac{-\zeta_3^2}3x^2\Phi_1^2\Phi_2^2\Phi_4\Phi_9{
\Phi'''_{12}}\Phi_{18}&\zeta_3\\
Z_3:.1..-&\frac{\zeta_3}3x^2\Phi_1\Phi_2{\Phi'_3}\Phi_4{\Phi''_6}\Phi_9{
\Phi'''_{12}}\Phi_{18}&\zeta_3^2\\
*\hfill \phi_{6,2}&\frac13x^2{\Phi''_3}^2\Phi_4{\Phi'_6}^2\Phi_9{\Phi''''_{12}
}\Phi_{18}&1\\
Z_3:..1.-&\frac{-\zeta_3^2}3x^2\Phi_1\Phi_2{\Phi''_3}\Phi_4{\Phi'_6}\Phi_9{
\Phi''''_{12}}\Phi_{18}&\zeta_3^2\\
G_{3,1,2}^{103}:-\zeta_3^2&\frac{-\zeta_3}3x^2\Phi_1^2\Phi_2^2\Phi_4\Phi_9{
\Phi''''_{12}}\Phi_{18}&\zeta_3\\
\hline
*\hfill \phi_{8,3}&\frac12x^3\Phi_2^3\Phi_4\Phi_6^3\Phi_{12}\Phi_{18}&1\\
\#\hfill \phi_{8,6}'&\frac12x^3\Phi_2^3\Phi_4\Phi_6^3\Phi_{12}\Phi_{18}&1\\
G_{26}[i]&\frac{-1}2x^3\Phi_1^3\Phi_3^3\Phi_4\Phi_9\Phi_{12}&ix^{1/2}\\
G_{26}[-i]&\frac{-1}2x^3\Phi_1^3\Phi_3^3\Phi_4\Phi_9\Phi_{12}&-ix^{1/2}\\
\hline
*\hfill \phi_{6,4}''&\frac{3-\sqrt {-3}}{12}x^4\Phi_2^2{\Phi''_3}\Phi_6^2{
\Phi'_6}\Phi_9{\Phi''''_{12}}\Phi_{18}&1\\
\#\hfill \phi_{6,8}''&\frac{3+\sqrt {-3}}{12}x^4\Phi_2^2{\Phi'_3}\Phi_6^2{
\Phi''_6}\Phi_9{\Phi'''_{12}}\Phi_{18}&1\\
Z_3:....1.1&\frac{\sqrt {-3}}6x^4\Phi_1\Phi_2^3\Phi_4\Phi_6^2\Phi_9\Phi_{18}&
\zeta_3^2\\
\phi_{3,13}''&\frac{-3+\sqrt {-3}}{12}x^4{\Phi''_3}\Phi_4{\Phi'_6}^3\Phi_
9\Phi_{12}\Phi_{18}&1\\
\phi_{3,5}''&\frac{-3-\sqrt {-3}}{12}x^4{\Phi'_3}\Phi_4{\Phi''_6}^3\Phi_9\Phi_
{12}\Phi_{18}&1\\
Z_3:..1..1.&\frac{\sqrt {-3}}6x^4\Phi_1\Phi_2^3\Phi_4\Phi_9\Phi_{12}\Phi_{18}&
\zeta_3^2\\
\phi_{9,7}&\frac{3-\sqrt {-3}}{12}x^4{\Phi''_3}^3\Phi_4{\Phi'_6}\Phi_9\Phi_{
12}\Phi_{18}&1\\
\phi_{9,5}&\frac{3+\sqrt {-3}}{12}x^4{\Phi'_3}^3\Phi_4{\Phi''_6}\Phi_9\Phi_{
12}\Phi_{18}&1\\
G_{26}[\zeta_{3}^2]&\frac{\sqrt {-3}}6x^4\Phi_1^3\Phi_2\Phi_4\Phi_9\Phi_{12}
\Phi_{18}&\zeta_3^2\\
G_4:-\zeta_3^2&\frac{3-\sqrt {-3}}{12}x^4\Phi_1^2\Phi_3^2{\Phi''_3}{\Phi'_6}
\Phi_9{\Phi''''_{12}}\Phi_{18}&-1\\
G_4:-\zeta_3^4&\frac{3+\sqrt {-3}}{12}x^4\Phi_1^2\Phi_3^2{\Phi'_3}{\Phi''_6}
\Phi_9{\Phi'''_{12}}\Phi_{18}&-1\\
G_{26}[-\zeta_{3}^2]&\frac{\sqrt {-3}}6x^4\Phi_1^3\Phi_2\Phi_3^2\Phi_4\Phi_
9\Phi_{18}&-\zeta_3^2\\
\hline
*\hfill \phi_{6,5}&\frac{3-\sqrt {-3}}6x^5{\Phi'_3}\Phi_4{\Phi''_6}\Phi_9\Phi_
{12}\Phi_{18}&1\\
\#\hfill \phi_{6,7}'&\frac{3+\sqrt {-3}}6x^5{\Phi''_3}\Phi_4{\Phi'_6}\Phi_
9\Phi_{12}\Phi_{18}&1\\
Z_3:...2..&\frac{\sqrt {-3}}3x^5\Phi_1\Phi_2\Phi_4\Phi_9\Phi_{12}\Phi_{18}&
\zeta_3^2\\
\hline
*\hfill \phi_{8,6}''&\frac{-\sqrt {-3}}{54}x^6\Phi_2^3{\Phi''_3}^3\Phi_4\Phi_
9\Phi_{12}\Phi_{18}&1\\
\#\hfill \phi_{8,12}&\frac{\sqrt {-3}}{54}x^6\Phi_2^3{\Phi'_3}^3\Phi_4\Phi_
9\Phi_{12}\Phi_{18}&1\\
\phi_{6,11}''&\frac{-\zeta_3^2}6x^6\Phi_2^2\Phi_3{\Phi'_3}\Phi_6^3{\Phi'_9}
\Phi_{12}\Phi_{18}&1\\
\phi_{6,7}''&\frac{-\zeta_3}6x^6\Phi_2^2\Phi_3{\Phi''_3}\Phi_6^3{\Phi''_9}
\Phi_{12}\Phi_{18}&1\\
\phi_{2,15}&\frac16x^6\Phi_2^2\Phi_6^3\Phi_9\Phi_{12}\Phi_{18}&1\\
\phi_{8,9}''&\frac16x^6\Phi_2^3\Phi_4\Phi_6^2{\Phi''_6}\Phi_9{\Phi'''_{12}}
\Phi_{18}&1\\
\phi_{8,9}'&\frac16x^6\Phi_2^3\Phi_4\Phi_6^2{\Phi'_6}\Phi_9{\Phi''''_{12}}
\Phi_{18}&1\\
\phi_{6,8}'&\frac{3-\sqrt {-3}}{18}x^6\Phi_3{\Phi'_3}\Phi_4\Phi_6{\Phi''_6}
\Phi_9\Phi_{12}\Phi_{18}&1\\
\phi_{6,10}&\frac{3+\sqrt {-3}}{18}x^6\Phi_3{\Phi''_3}\Phi_4\Phi_6{\Phi'_6}
\Phi_9\Phi_{12}\Phi_{18}&1\\
\phi_{3,8}''&\frac{3-\sqrt {-3}}{36}x^6\Phi_3{\Phi'_3}\Phi_4\Phi_6^3\Phi_
9\Phi_{12}{\Phi''_{18}}&1\\
\phi_{3,16}''&\frac{3+\sqrt {-3}}{36}x^6\Phi_3{\Phi''_3}\Phi_4\Phi_6^3\Phi_
9\Phi_{12}{\Phi'_{18}}&1\\
\phi_{9,8}&\frac{3-\sqrt {-3}}{36}x^6\Phi_3^3\Phi_4\Phi_6{\Phi''_6}{\Phi'_9}
\Phi_{12}\Phi_{18}&1\\
\phi_{9,10}&\frac{3+\sqrt {-3}}{36}x^6\Phi_3^3\Phi_4\Phi_6{\Phi'_6}{\Phi''_9}
\Phi_{12}\Phi_{18}&1\\
\phi_{1,12}&\frac{\sqrt {-3}}{54}x^6{\Phi''_3}^3\Phi_4{\Phi''_6}^3\Phi_9\Phi_{
12}\Phi_{18}&1\\
\phi_{1,24}&\frac{-\sqrt {-3}}{54}x^6{\Phi'_3}^3\Phi_4{\Phi'_6}^3\Phi_9\Phi_{
12}\Phi_{18}&1\\
\phi_{2,18}&\frac{\sqrt {-3}}{27}x^6{\Phi''_3}^3\Phi_4{\Phi'_6}^3\Phi_9\Phi_{
12}\Phi_{18}&1\\
\phi_{2,12}&\frac{-\sqrt {-3}}{27}x^6{\Phi'_3}^3\Phi_4{\Phi''_6}^3\Phi_9\Phi_{
12}\Phi_{18}&1\\
Z_3:.1..+&\frac{-3+\sqrt {-3}}{18}x^6\Phi_1\Phi_2{\Phi''_3}^2\Phi_4{\Phi'_6}
^2\Phi_9\Phi_{12}\Phi_{18}&\zeta_3^2\\
Z_3:..1.+&\frac{3+\sqrt {-3}}{18}x^6\Phi_1\Phi_2{\Phi'_3}^2\Phi_4{\Phi''_6}
^2\Phi_9\Phi_{12}\Phi_{18}&\zeta_3^2\\
Z_3:1...-&\frac{\sqrt {-3}}9x^6\Phi_1\Phi_2\Phi_3\Phi_4\Phi_6\Phi_9\Phi_{12}
\Phi_{18}&\zeta_3^2\\
Z_3:...1.1.&\frac{3+\sqrt {-3}}{36}x^6\Phi_1\Phi_2^3{\Phi'_3}^2\Phi_4{\Phi''_
6}^3\Phi_9\Phi_{12}{\Phi'_{18}}&\zeta_3^2\\
Z_3:..1.1..&\frac{-3+\sqrt {-3}}{36}x^6\Phi_1\Phi_2^3{\Phi''_3}^2\Phi_4{\Phi'_
6}^3\Phi_9\Phi_{12}{\Phi''_{18}}&\zeta_3^2\\
Z_3:.1..1..&\frac{\zeta_3}6x^6\Phi_1\Phi_2^3{\Phi'_3}^2\Phi_4\Phi_6^2{\Phi''_
6}{\Phi''_9}{\Phi'''_{12}}\Phi_{18}&\zeta_3^2\\
Z_3:...1..1&\frac{-\zeta_3^2}6x^6\Phi_1\Phi_2^3{\Phi''_3}^2\Phi_4\Phi_6^2{
\Phi'_6}{\Phi'_9}{\Phi''''_{12}}\Phi_{18}&\zeta_3^2\\
G_4:1&\frac16x^6\Phi_1^2\Phi_3^3\Phi_9\Phi_{12}\Phi_{18}&-1\\
G_4:\zeta_3&\frac{\zeta_3^2}6x^6\Phi_1^2\Phi_3^3\Phi_6{\Phi''_6}\Phi_9\Phi_{
12}{\Phi''_{18}}&-1\\
G_4:\zeta_3^2&\frac{\zeta_3}6x^6\Phi_1^2\Phi_3^3\Phi_6{\Phi'_6}\Phi_9\Phi_{12}
{\Phi'_{18}}&-1\\
G_{3,1,2}^{103}:\zeta_3&\frac{-\zeta_3}6x^6\Phi_1^2\Phi_2^3{\Phi''_3}\Phi_
4\Phi_6^2{\Phi'_6}{\Phi''_9}{\Phi''''_{12}}\Phi_{18}&\zeta_3\\
G_{3,1,2}^{130}:\zeta_3^2&\frac{-\zeta_3^2}6x^6\Phi_1^2\Phi_2^3{\Phi'_3}\Phi_
4\Phi_6^2{\Phi''_6}{\Phi'_9}{\Phi'''_{12}}\Phi_{18}&\zeta_3\\
G_{3,1,2}^{103}:-\zeta_3^4&\frac{-3-\sqrt {-3}}{36}x^6\Phi_1^2\Phi_2^3{\Phi''_
3}\Phi_4{\Phi'_6}^3\Phi_9\Phi_{12}{\Phi'_{18}}&\zeta_3\\
G_{3,1,2}^{130}:-\zeta_3^2&\frac{-3+\sqrt {-3}}{36}x^6\Phi_1^2\Phi_2^3{\Phi'_
3}\Phi_4{\Phi''_6}^3\Phi_9\Phi_{12}{\Phi''_{18}}&\zeta_3\\
G_{3,1,2}^{130}:-1&\frac{3+\sqrt {-3}}{18}x^6\Phi_1^2\Phi_2^2{\Phi''_3}\Phi_4{
\Phi'_6}\Phi_9\Phi_{12}\Phi_{18}&\zeta_3\\
G_{3,1,2}^{103}:-1&\frac{3-\sqrt {-3}}{18}x^6\Phi_1^2\Phi_2^2{\Phi'_3}\Phi_4{
\Phi''_6}\Phi_9\Phi_{12}\Phi_{18}&\zeta_3\\
G_{26}[-1]&\frac{-1}6x^6\Phi_1^3\Phi_3^2{\Phi'_3}\Phi_4\Phi_9{\Phi'''_{12}}
\Phi_{18}&-1\\
G_{26}^2[-1]&\frac{-1}6x^6\Phi_1^3\Phi_3^2{\Phi''_3}\Phi_4\Phi_9{\Phi''''_{12}
}\Phi_{18}&-1\\
G_{26}^2[\zeta_3^2]&\frac{3+\sqrt {-3}}{36}x^6\Phi_1^3\Phi_2{\Phi'_3}^3\Phi_4{
\Phi''_6}^2{\Phi''_9}\Phi_{12}\Phi_{18}&\zeta_3^2\\
G_{26}^3[\zeta_3^2]&\frac{-3+\sqrt {-3}}{36}x^6\Phi_1^3\Phi_2{\Phi''_3}^3\Phi_
4{\Phi'_6}^2{\Phi'_9}\Phi_{12}\Phi_{18}&\zeta_3^2\\
G_{26}^2[-\zeta_3^2]&\frac{\zeta_3}6x^6\Phi_1^3\Phi_2\Phi_3^2{\Phi'_3}\Phi_4{
\Phi''_6}^2\Phi_9{\Phi'''_{12}}{\Phi'_{18}}&-\zeta_3^2\\
G_{26}^3[-\zeta_3^2]&\frac{-\zeta_3^2}6x^6\Phi_1^3\Phi_2\Phi_3^2{\Phi''_3}
\Phi_4{\Phi'_6}^2\Phi_9{\Phi''''_{12}}{\Phi''_{18}}&-\zeta_3^2\\
G_{26}[\zeta_3]&\frac{3+\sqrt {-3}}{36}x^6\Phi_1^3\Phi_2^2{\Phi''_3}^3\Phi_4{
\Phi'_6}{\Phi''_9}\Phi_{12}\Phi_{18}&\zeta_3\\
G_{26}^2[\zeta_3]&\frac{-3+\sqrt {-3}}{36}x^6\Phi_1^3\Phi_2^2{\Phi'_3}^3\Phi_
4{\Phi''_6}{\Phi'_9}\Phi_{12}\Phi_{18}&\zeta_3\\
G_{26}[-\zeta_3]&\frac{\zeta_3}6x^6\Phi_1^3\Phi_2^2\Phi_3^2{\Phi''_3}\Phi_4{
\Phi'_6}\Phi_9{\Phi''''_{12}}{\Phi'_{18}}&-\zeta_3^4\\
G_{26}^2[-\zeta_3]&\frac{\zeta_3^2}6x^6\Phi_1^3\Phi_2^2\Phi_3^2{\Phi'_3}\Phi_
4{\Phi''_6}\Phi_9{\Phi'''_{12}}{\Phi''_{18}}&-\zeta_3^4\\
G_{26}[1]&\frac{\sqrt {-3}}{27}x^6\Phi_1^3\Phi_2^3\Phi_4\Phi_9\Phi_{12}\Phi_{
18}&1\\
G_{26}^2[1]&\frac{\sqrt {-3}}{54}x^6\Phi_1^3\Phi_4{\Phi''_6}^3\Phi_9\Phi_{12}
\Phi_{18}&1\\
G_{26}^3[1]&\frac{\sqrt {-3}}{54}x^6\Phi_1^3\Phi_4{\Phi'_6}^3\Phi_9\Phi_{12}
\Phi_{18}&1\\
G_{26}[\zeta_9^8]&\frac{\sqrt {-3}}9x^6\Phi_1^3\Phi_2^3\Phi_3^3\Phi_4\Phi_
6^3\Phi_{12}&\zeta_9^8\\
G_{26}[\zeta_9^5]&\frac{\sqrt {-3}}9x^6\Phi_1^3\Phi_2^3\Phi_3^3\Phi_4\Phi_
6^3\Phi_{12}&\zeta_9^5\\
G_{26}[\zeta_9^2]&\frac{\sqrt {-3}}9x^6\Phi_1^3\Phi_2^3\Phi_3^3\Phi_4\Phi_
6^3\Phi_{12}&\zeta_9^2\\
\hline
\phi_{3,15}&\frac13x^{11}\Phi_3\Phi_6\Phi_9\Phi_{12}\Phi_{18}&1\\
\#\hfill \phi_{6,13}&\frac13x^{11}{\Phi'_3}^2\Phi_4{\Phi''_6}^2\Phi_9{\Phi'''_
{12}}\Phi_{18}&1\\
*\hfill \phi_{6,11}'&\frac13x^{11}{\Phi''_3}^2\Phi_4{\Phi'_6}^2\Phi_9{
\Phi''''_{12}}\Phi_{18}&1\\
\phi_{3,17}&\frac{-\zeta_3^2}3x^{11}{\Phi''_3}^2{\Phi'_6}^2\Phi_9\Phi_{12}
\Phi_{18}&1\\
G_{3,1,2}^{103}:\zeta_3^2&\frac{-\zeta_3^2}3x^{11}\Phi_1^2\Phi_2^2\Phi_4\Phi_
9{\Phi'''_{12}}\Phi_{18}&\zeta_3\\
Z_3:.....11&\frac{-\zeta_3^2}3x^{11}\Phi_1\Phi_2{\Phi''_3}\Phi_4{\Phi'_6}\Phi_
9{\Phi''''_{12}}\Phi_{18}&\zeta_3^2\\
\phi_{3,13}'&\frac{-\zeta_3}3x^{11}{\Phi'_3}^2{\Phi''_6}^2\Phi_9\Phi_{12}\Phi_
{18}&1\\
Z_3:....11.&\frac{\zeta_3}3x^{11}\Phi_1\Phi_2{\Phi'_3}\Phi_4{\Phi''_6}\Phi_9{
\Phi'''_{12}}\Phi_{18}&\zeta_3^2\\
G_{3,1,2}^{130}:\zeta_3&\frac{-\zeta_3}3x^{11}\Phi_1^2\Phi_2^2\Phi_4\Phi_9{
\Phi''''_{12}}\Phi_{18}&\zeta_3\\
\hline
*\hfill \phi_{3,16}'&\frac{3-\sqrt {-3}}6x^{16}{\Phi''_3}{\Phi'_6}\Phi_9{
\Phi''''_{12}}\Phi_{18}&1\\
\#\hfill \phi_{3,20}&\frac{3+\sqrt {-3}}6x^{16}{\Phi'_3}{\Phi''_6}\Phi_9{
\Phi'''_{12}}\Phi_{18}&1\\
Z_3:1...+&\frac{\sqrt {-3}}3x^{16}\Phi_1\Phi_2\Phi_4\Phi_9\Phi_{18}&\zeta_
3^2\\
\hline
*\hfill \phi_{1,21}&\frac{-\sqrt {-3}}6x^{21}\Phi_4{\Phi''_9}\Phi_{12}{\Phi''_
{18}}&1\\
\phi_{2,24}&\frac12x^{21}\Phi_2^2\Phi_6^2\Phi_{18}&1\\
G_4:-1&\frac12x^{21}\Phi_1^2\Phi_3^2\Phi_9&-1\\
Z_3:...11..&\frac{\sqrt {-3}}3x^{21}\Phi_1\Phi_2\Phi_3\Phi_4\Phi_6\Phi_{12}&
\zeta_3^2\\
\#\hfill \phi_{1,33}&\frac{\sqrt {-3}}6x^{21}\Phi_4{\Phi'_9}\Phi_{12}{\Phi'_{
18}}&1\\
\end{supertabular}
\vfill\eject
\subsection{Unipotent characters for $G_{27}$}
\begin{center}Some principal $\zeta$-series\end{center}
$\zeta_{5}^{3}$ : ${\mathcal H}_{Z_{30}}(\zeta_5x^3,\allowbreak -\zeta_
3^4x^2,\allowbreak \zeta_{15}^{13}x^2,\allowbreak -\zeta_5^6x^{3/2}
,\allowbreak \zeta_{15}^8x,\allowbreak -\zeta_{15}^{13}x^{4/3}
,\allowbreak \zeta_5x^{5/3},\allowbreak -\zeta_{15}^8x^2,\allowbreak \zeta_
3^2x,\allowbreak -\zeta_5^7x^{3/2},\allowbreak \zeta_{15}^{11}x,\allowbreak -
\zeta_3^2x^2,\allowbreak x^{3/2},\allowbreak -\zeta_{15}^{11}
x^2,\allowbreak \zeta_{15}^{13}x,\allowbreak -\zeta_5^6x^{4/3}
,\allowbreak \zeta_{15}^8x^{5/3},\allowbreak -\zeta_{15}^{13}
x^2,\allowbreak \zeta_5x^{3/2},\allowbreak -\zeta_{15}^8x,\allowbreak \zeta_{
15}x,\allowbreak -\zeta_5^6,\allowbreak \zeta_{15}^8x^2,\allowbreak -\zeta_{
15}^{16}x^2,\allowbreak \zeta_5^2x^{3/2},\allowbreak -\zeta_{15}^8x^{4/3}
,\allowbreak \zeta_{15}^{13}x^{5/3},\allowbreak -x^{3/2},\allowbreak \zeta_
3x,\allowbreak -\zeta_{15}^{13}x)$\hfill\break
$\zeta_{5}$ : ${\mathcal H}_{Z_{30}}(\zeta_5^2x^3,\allowbreak -\zeta_{15}^{17}
x^2,\allowbreak \zeta_{15}x^{5/3},\allowbreak -\zeta_5^4x^{3/2}
,\allowbreak \zeta_{15}^{11}x^2,\allowbreak -\zeta_{15}^{16}x^{4/3}
,\allowbreak \zeta_5^2x^{3/2},\allowbreak -\zeta_3^4x^2,\allowbreak \zeta_{15}
x,\allowbreak -x^{3/2},\allowbreak \zeta_{15}^2x,\allowbreak -\zeta_{15}^{22}
x^2,\allowbreak \zeta_5^2x^{5/3},\allowbreak -\zeta_{15}^{11}
x,\allowbreak \zeta_{15}x^2,\allowbreak -\zeta_5^7x^{4/3},\allowbreak \zeta_
3x,\allowbreak -\zeta_3^2x^2,\allowbreak \zeta_5^4x^{3/2},\allowbreak -\zeta_{
15}^{11}x^2,\allowbreak \zeta_{15}^7x,\allowbreak -\zeta_5^7x^{3/2}
,\allowbreak \zeta_{15}^{11}x^{5/3},\allowbreak -\zeta_{15}^{16}
x,\allowbreak x^{3/2},\allowbreak -\zeta_{15}^{11}x^{4/3},\allowbreak \zeta_
3^2x,\allowbreak -\zeta_5^7,\allowbreak \zeta_{15}^{11}x,\allowbreak -\zeta_{
15}^{16}x^2)$\hfill\break
$\zeta_{5}^{2}$ : ${\mathcal H}_{Z_{30}}(\zeta_5^4x^3,\allowbreak -\zeta_{15}
^{17}x,\allowbreak \zeta_3^2x,\allowbreak -x^{3/2},\allowbreak \zeta_{15}^2x^{
5/3},\allowbreak -\zeta_{15}^{22}x^{4/3},\allowbreak \zeta_5^3x^{3/2}
,\allowbreak -\zeta_{15}^{14}x^2,\allowbreak \zeta_{15}^7x^2,\allowbreak -
\zeta_5^4,\allowbreak \zeta_{15}^{14}x,\allowbreak -\zeta_{15}^{22}
x,\allowbreak \zeta_5^4x^{3/2},\allowbreak -\zeta_{15}^{17}
x^2,\allowbreak \zeta_{15}^7x^{5/3},\allowbreak -\zeta_5^4x^{4/3}
,\allowbreak \zeta_{15}^2x,\allowbreak -\zeta_{15}^{19}x^2,\allowbreak x^{3/2}
,\allowbreak -\zeta_3^4x^2,\allowbreak \zeta_{15}^4x,\allowbreak -\zeta_5^3x^{
3/2},\allowbreak \zeta_3x,\allowbreak -\zeta_{15}^{22}x^2,\allowbreak \zeta_
5^4x^{5/3},\allowbreak -\zeta_{15}^{17}x^{4/3},\allowbreak \zeta_{15}
^7x,\allowbreak -\zeta_5^4x^{3/2},\allowbreak \zeta_{15}^2x^2,\allowbreak -
\zeta_3^2x^2)$\hfill\break
\begin{center}Non-principal $1$-Harish-Chandra series\end{center}
${\mathcal H}_{G_{27}}(I_2(5)[1,3])={\mathcal H}_{Z_{6}}(x^{5/2},\allowbreak -
\zeta_3^2x^5,\allowbreak \zeta_3,\allowbreak -x^{5/2},\allowbreak \zeta_
3^2,\allowbreak -\zeta_3^4x^5)$\hfill\break
${\mathcal H}_{G_{27}}(I_2(5)[1,2])={\mathcal H}_{Z_{6}}(x^{5/2},\allowbreak -
\zeta_3^2x^5,\allowbreak \zeta_3,\allowbreak -x^{5/2},\allowbreak \zeta_
3^2,\allowbreak -\zeta_3^4x^5)$\hfill\break
${\mathcal H}_{G_{27}}(B_2)={\mathcal H}_{Z_{6}}(x^4,\allowbreak -\zeta_
3^2x^5,\allowbreak \zeta_3,\allowbreak -x,\allowbreak \zeta_3^2,\allowbreak -
\zeta_3^4x^5)$\hfill\break
\par\tablehead{\hline \gamma&\mbox{Deg($\gamma$)}&\mbox{Fr($\gamma$)}\\\hline}
\tabletail{\hline}
\begin{supertabular}{|R|RR|}
\shrinkheight{30pt}
*\hfill \phi_{1,0}&1&1\\
\hline
*\hfill \phi_{3,1}&\frac{-\sqrt {-15}\zeta_3^2}{30}x{\Phi''_3}^3\Phi_4{\Phi'_
5}{\Phi'_6}^3{\Phi'_{10}}\Phi_{12}{\Phi''_{15}}{\Phi^{(5)}_{15}}{\Phi''_{30}}{
\Phi^{(5)}_{30}}&1\\
\phi_{3,7}&\frac{\sqrt {-15}\zeta_3^2}{30}x{\Phi''_3}^3\Phi_4{\Phi''_5}{\Phi'_
6}^3{\Phi''_{10}}\Phi_{12}{\Phi'_{15}}{\Phi^{(6)}_{15}}{\Phi'_{30}}{\Phi^{(6)}
_{30}}&1\\
I_2(5)[1,2]:-\zeta_3^2&\frac{\sqrt {-15}\zeta_3^2}{15}x\Phi_1^2\Phi_2^2\Phi_
3^2{\Phi''_3}\Phi_4\Phi_6^2{\Phi'_6}\Phi_{12}{\Phi'''_{15}}{\Phi''''_{30}}&
\zeta_5^3\\
I_2(5)[1,3]:-\zeta_3^2&\frac{\sqrt {-15}\zeta_3^2}{15}x\Phi_1^2\Phi_2^2\Phi_
3^2{\Phi''_3}\Phi_4\Phi_6^2{\Phi'_6}\Phi_{12}{\Phi'''_{15}}{\Phi''''_{30}}&
\zeta_5^2\\
B_2:-\zeta_3^2&\frac{3-\sqrt {-3}}{12}x\Phi_1^2\Phi_3^2{\Phi''_3}\Phi_5{\Phi'_
6}^3{\Phi''''_{12}}\Phi_{15}{\Phi''''_{30}}&-1\\
\phi_{6,4}&\frac{3-\sqrt {-3}}{12}x\Phi_2^2{\Phi''_3}^3\Phi_6^2{\Phi'_6}\Phi_{
10}{\Phi''''_{12}}{\Phi'''_{15}}\Phi_{30}&1\\
\#\hfill \phi_{3,5}''&\frac{\sqrt {-15}\zeta_3}{30}x{\Phi'_3}^3\Phi_4{\Phi'_5}
{\Phi''_6}^3{\Phi'_{10}}\Phi_{12}{\Phi''_{15}}{\Phi^{(7)}_{15}}{\Phi''_{30}}{
\Phi^{(7)}_{30}}&1\\
\phi_{3,5}'&\frac{-\sqrt {-15}\zeta_3}{30}x{\Phi'_3}^3\Phi_4{\Phi''_5}{\Phi''_
6}^3{\Phi''_{10}}\Phi_{12}{\Phi'_{15}}{\Phi^{(8)}_{15}}{\Phi'_{30}}{\Phi^{(8)}
_{30}}&1\\
I_2(5)[1,2]:-\zeta_3^4&\frac{-\sqrt {-15}\zeta_3}{15}x\Phi_1^2\Phi_2^2\Phi_
3^2{\Phi'_3}\Phi_4\Phi_6^2{\Phi''_6}\Phi_{12}{\Phi''''_{15}}{\Phi'''_{30}}&
\zeta_5^3\\
I_2(5)[1,3]:-\zeta_3^4&\frac{-\sqrt {-15}\zeta_3}{15}x\Phi_1^2\Phi_2^2\Phi_
3^2{\Phi'_3}\Phi_4\Phi_6^2{\Phi''_6}\Phi_{12}{\Phi''''_{15}}{\Phi'''_{30}}&
\zeta_5^2\\
B_2:-\zeta_3^4&\frac{3+\sqrt {-3}}{12}x\Phi_1^2\Phi_3^2{\Phi'_3}\Phi_5{\Phi''_
6}^3{\Phi'''_{12}}\Phi_{15}{\Phi'''_{30}}&-1\\
\phi_{6,2}&\frac{3+\sqrt {-3}}{12}x\Phi_2^2{\Phi'_3}^3\Phi_6^2{\Phi''_6}\Phi_{
10}{\Phi'''_{12}}{\Phi''''_{15}}\Phi_{30}&1\\
G_{27}^3[\zeta_3^2]&\frac{-\sqrt {-15}}{30}x\Phi_1^3\Phi_2^3\Phi_4\Phi_5\Phi_{
10}\Phi_{12}{\Phi''_{15}}{\Phi''_{30}}&\zeta_3^2\\
G_{27}^2[\zeta_3^2]&\frac{-\sqrt {-15}}{30}x\Phi_1^3\Phi_2^3\Phi_4\Phi_5\Phi_{
10}\Phi_{12}{\Phi'_{15}}{\Phi'_{30}}&\zeta_3^2\\
G_{27}[\zeta_{15}^4]&\frac{-\sqrt {-15}}{15}x\Phi_1^3\Phi_2^3\Phi_3^2\Phi_
4\Phi_5\Phi_6^2\Phi_{10}\Phi_{12}&\zeta_{15}^4\\
G_{27}[\zeta_{15}]&\frac{-\sqrt {-15}}{15}x\Phi_1^3\Phi_2^3\Phi_3^2\Phi_4\Phi_
5\Phi_6^2\Phi_{10}\Phi_{12}&\zeta_{15}\\
G_{27}^2[-\zeta_3]&\frac{\sqrt {-3}}6x\Phi_1^3\Phi_2^3\Phi_3^2\Phi_4\Phi_
5\Phi_{10}\Phi_{15}&-\zeta_3^2\\
G_{27}^4[\zeta_3^2]&\frac{\sqrt {-3}}6x\Phi_1^3\Phi_2^3\Phi_4\Phi_5\Phi_
6^2\Phi_{10}\Phi_{30}&\zeta_3^2\\
\hline
*\hfill \phi_{10,3}&\frac12x^3\Phi_2^2\Phi_5\Phi_6^2\Phi_{10}\Phi_{15}\Phi_{
30}&1\\
\phi_{5,6}''&\frac12x^3\Phi_4\Phi_5\Phi_{10}\Phi_{12}\Phi_{15}\Phi_{30}&1\\
\phi_{5,6}'&\frac12x^3\Phi_4\Phi_5\Phi_{10}\Phi_{12}\Phi_{15}\Phi_{30}&1\\
B_2:1&\frac12x^3\Phi_1^2\Phi_3^2\Phi_5\Phi_{10}\Phi_{15}\Phi_{30}&-1\\
\hline
\#\hfill \phi_{9,6}&\frac13x^4\Phi_3^3\Phi_6^3\Phi_{12}\Phi_{15}\Phi_{30}&1\\
G_{27}^2[\zeta_9^4]&\frac13x^4\Phi_1^3\Phi_2^3{\Phi''_3}^3\Phi_4\Phi_5{\Phi'_
6}^3\Phi_{10}{\Phi''''_{12}}{\Phi'''_{15}}{\Phi''''_{30}}&\zeta_9^4x^{2/3}\\
G_{27}^2[\zeta_9^7]&\frac13x^4\Phi_1^3\Phi_2^3{\Phi'_3}^3\Phi_4\Phi_5{\Phi''_
6}^3\Phi_{10}{\Phi'''_{12}}{\Phi''''_{15}}{\Phi'''_{30}}&\zeta_9^7x^{1/3}\\
\phi_{9,8}&\frac13x^4\Phi_3^3\Phi_6^3\Phi_{12}\Phi_{15}\Phi_{30}&1\\
G_{27}^2[\zeta_9]&\frac13x^4\Phi_1^3\Phi_2^3{\Phi''_3}^3\Phi_4\Phi_5{\Phi'_6}
^3\Phi_{10}{\Phi''''_{12}}{\Phi'''_{15}}{\Phi''''_{30}}&\zeta_9x^{2/3}\\
G_{27}[\zeta_9]&\frac13x^4\Phi_1^3\Phi_2^3{\Phi'_3}^3\Phi_4\Phi_5{\Phi''_6}
^3\Phi_{10}{\Phi'''_{12}}{\Phi''''_{15}}{\Phi'''_{30}}&\zeta_9x^{1/3}\\
*\hfill \phi_{9,4}&\frac13x^4\Phi_3^3\Phi_6^3\Phi_{12}\Phi_{15}\Phi_{30}&1\\
G_{27}[\zeta_9^7]&\frac13x^4\Phi_1^3\Phi_2^3{\Phi''_3}^3\Phi_4\Phi_5{\Phi'_6}
^3\Phi_{10}{\Phi''''_{12}}{\Phi'''_{15}}{\Phi''''_{30}}&\zeta_9^7x^{2/3}\\
G_{27}[\zeta_9^4]&\frac13x^4\Phi_1^3\Phi_2^3{\Phi'_3}^3\Phi_4\Phi_5{\Phi''_6}
^3\Phi_{10}{\Phi'''_{12}}{\Phi''''_{15}}{\Phi'''_{30}}&\zeta_9^4x^{1/3}\\
\hline
*\hfill \phi_{15,5}&\frac{3+\sqrt {-3}}6x^5{\Phi''_3}^3\Phi_5{\Phi'_6}^3\Phi_{
10}{\Phi''''_{12}}\Phi_{15}\Phi_{30}&1\\
\#\hfill \phi_{15,7}&\frac{3-\sqrt {-3}}6x^5{\Phi'_3}^3\Phi_5{\Phi''_6}^3\Phi_
{10}{\Phi'''_{12}}\Phi_{15}\Phi_{30}&1\\
G_{27}^4[\zeta_3]&\frac{-\sqrt {-3}}3x^5\Phi_1^3\Phi_2^3\Phi_4\Phi_5\Phi_{10}
\Phi_{15}\Phi_{30}&\zeta_3\\
\hline
*\hfill \phi_{8,6}&\frac{5-\sqrt 5}{20}x^6\Phi_2^3\Phi_4{\Phi'_5}\Phi_6^3\Phi_
{10}\Phi_{12}{\Phi''_{15}}\Phi_{30}&1\\
\phi_{8,12}&\frac{5+\sqrt 5}{20}x^6\Phi_2^3\Phi_4{\Phi''_5}\Phi_6^3\Phi_{10}
\Phi_{12}{\Phi'_{15}}\Phi_{30}&1\\
I_2(5)[1,2]:1&\frac{\sqrt 5}{10}x^6\Phi_1^2\Phi_2^3\Phi_3^2\Phi_4\Phi_6^3\Phi_
{10}\Phi_{12}\Phi_{30}&\zeta_5^3\\
I_2(5)[1,3]:1&\frac{\sqrt 5}{10}x^6\Phi_1^2\Phi_2^3\Phi_3^2\Phi_4\Phi_6^3\Phi_
{10}\Phi_{12}\Phi_{30}&\zeta_5^2\\
\#\hfill \phi_{8,9}'&\frac{5-\sqrt 5}{20}x^6\Phi_2^3\Phi_4{\Phi'_5}\Phi_
6^3\Phi_{10}\Phi_{12}{\Phi''_{15}}\Phi_{30}&1\\
\phi_{8,9}''&\frac{5+\sqrt 5}{20}x^6\Phi_2^3\Phi_4{\Phi''_5}\Phi_6^3\Phi_{10}
\Phi_{12}{\Phi'_{15}}\Phi_{30}&1\\
I_2(5)[1,2]:-1&\frac{\sqrt 5}{10}x^6\Phi_1^2\Phi_2^3\Phi_3^2\Phi_4\Phi_
6^3\Phi_{10}\Phi_{12}\Phi_{30}&\zeta_5^3\\
I_2(5)[1,3]:-1&\frac{\sqrt 5}{10}x^6\Phi_1^2\Phi_2^3\Phi_3^2\Phi_4\Phi_
6^3\Phi_{10}\Phi_{12}\Phi_{30}&\zeta_5^2\\
G_{27}[i]&\frac{-5+\sqrt 5}{20}x^6\Phi_1^3\Phi_3^3\Phi_4\Phi_5{\Phi''_{10}}
\Phi_{12}\Phi_{15}{\Phi'_{30}}&ix^{1/2}\\
G_{27}^2[i]&\frac{-5-\sqrt 5}{20}x^6\Phi_1^3\Phi_3^3\Phi_4\Phi_5{\Phi'_{10}}
\Phi_{12}\Phi_{15}{\Phi''_{30}}&ix^{1/2}\\
G_{27}[\zeta_{20}^{17}]&\frac{-\sqrt 5}{10}x^6\Phi_1^3\Phi_2^2\Phi_3^3\Phi_
4\Phi_5\Phi_6^2\Phi_{12}\Phi_{15}&\zeta_{20}^{17}x^{1/2}\\
G_{27}[\zeta_{20}^{13}]&\frac{-\sqrt 5}{10}x^6\Phi_1^3\Phi_2^2\Phi_3^3\Phi_
4\Phi_5\Phi_6^2\Phi_{12}\Phi_{15}&\zeta_{20}^{13}x^{1/2}\\
G_{27}^2[-i]&\frac{-5+\sqrt 5}{20}x^6\Phi_1^3\Phi_3^3\Phi_4\Phi_5{\Phi''_{10}}
\Phi_{12}\Phi_{15}{\Phi'_{30}}&-ix^{1/2}\\
G_{27}[-1]&\frac{-5-\sqrt 5}{20}x^6\Phi_1^3\Phi_3^3\Phi_4\Phi_5{\Phi'_{10}}
\Phi_{12}\Phi_{15}{\Phi''_{30}}&-ix^{1/2}\\
G_{27}[-\zeta_{20}^{17}]&\frac{-\sqrt 5}{10}x^6\Phi_1^3\Phi_2^2\Phi_3^3\Phi_
4\Phi_5\Phi_6^2\Phi_{12}\Phi_{15}&\zeta_{20}^7x^{1/2}\\
G_{27}[-\zeta_{20}^{13}]&\frac{-\sqrt 5}{10}x^6\Phi_1^3\Phi_2^2\Phi_3^3\Phi_
4\Phi_5\Phi_6^2\Phi_{12}\Phi_{15}&\zeta_{20}^3x^{1/2}\\
\hline
*\hfill \phi_{15,8}&\frac{3-\sqrt {-3}}6x^8{\Phi'_3}^3\Phi_5{\Phi''_6}^3\Phi_{
10}{\Phi'''_{12}}\Phi_{15}\Phi_{30}&1\\
\#\hfill \phi_{15,10}&\frac{3+\sqrt {-3}}6x^8{\Phi''_3}^3\Phi_5{\Phi'_6}
^3\Phi_{10}{\Phi''''_{12}}\Phi_{15}\Phi_{30}&1\\
G_{27}[\zeta_3^2]&\frac{\sqrt {-3}}3x^8\Phi_1^3\Phi_2^3\Phi_4\Phi_5\Phi_{10}
\Phi_{15}\Phi_{30}&\zeta_3^2\\
\hline
*\hfill \phi_{9,9}&\frac13x^9\Phi_3^3\Phi_6^3\Phi_{12}\Phi_{15}\Phi_{30}&1\\
G_{27}^2[\zeta_9^5]&\frac{\zeta_3^2}3x^9\Phi_1^3\Phi_2^3{\Phi'_3}^3\Phi_4\Phi_
5{\Phi''_6}^3\Phi_{10}{\Phi'''_{12}}{\Phi''''_{15}}{\Phi'''_{30}}&\zeta_9^5x^{
2/3}\\
G_{27}^5[1]&\frac{\zeta_3}3x^9\Phi_1^3\Phi_2^3{\Phi''_3}^3\Phi_4\Phi_5{\Phi'_
6}^3\Phi_{10}{\Phi''''_{12}}{\Phi'''_{15}}{\Phi''''_{30}}&\zeta_9^2x^{1/3}\\
\#\hfill \phi_{9,11}&\frac13x^9\Phi_3^3\Phi_6^3\Phi_{12}\Phi_{15}\Phi_{30}&1\\
G_{27}[\zeta_9^8]&\frac{\zeta_3^2}3x^9\Phi_1^3\Phi_2^3{\Phi'_3}^3\Phi_4\Phi_5{
\Phi''_6}^3\Phi_{10}{\Phi'''_{12}}{\Phi''''_{15}}{\Phi'''_{30}}&\zeta_9^8x^{
2/3}\\
G_{27}^2[\zeta_9^8]&\frac{\zeta_3}3x^9\Phi_1^3\Phi_2^3{\Phi''_3}^3\Phi_4\Phi_
5{\Phi'_6}^3\Phi_{10}{\Phi''''_{12}}{\Phi'''_{15}}{\Phi''''_{30}}&\zeta_9^8x^{
1/3}\\
\phi_{9,13}&\frac13x^9\Phi_3^3\Phi_6^3\Phi_{12}\Phi_{15}\Phi_{30}&1\\
G_{27}^7[\zeta_3]&\frac{\zeta_3^2}3x^9\Phi_1^3\Phi_2^3{\Phi'_3}^3\Phi_4\Phi_5{
\Phi''_6}^3\Phi_{10}{\Phi'''_{12}}{\Phi''''_{15}}{\Phi'''_{30}}&\zeta_9^2x^{
2/3}\\
G_{27}[\zeta_9^5]&\frac{\zeta_3}3x^9\Phi_1^3\Phi_2^3{\Phi''_3}^3\Phi_4\Phi_5{
\Phi'_6}^3\Phi_{10}{\Phi''''_{12}}{\Phi'''_{15}}{\Phi''''_{30}}&\zeta_9^5x^{
1/3}\\
\hline
*\hfill \phi_{10,12}&\frac12x^{12}\Phi_2^2\Phi_5\Phi_6^2\Phi_{10}\Phi_{15}
\Phi_{30}&1\\
\phi_{5,15}''&\frac12x^{12}\Phi_4\Phi_5\Phi_{10}\Phi_{12}\Phi_{15}\Phi_{30}&
1\\
\phi_{5,15}'&\frac12x^{12}\Phi_4\Phi_5\Phi_{10}\Phi_{12}\Phi_{15}\Phi_{30}&1\\
B_2:-1&\frac12x^{12}\Phi_1^2\Phi_3^2\Phi_5\Phi_{10}\Phi_{15}\Phi_{30}&-1\\
\hline
*\hfill \phi_{3,16}&\frac{\sqrt {-15}\zeta_3}{30}x^{16}{\Phi'_3}^3\Phi_4{
\Phi'_5}{\Phi''_6}^3{\Phi'_{10}}\Phi_{12}{\Phi''_{15}}{\Phi^{(7)}_{15}}{
\Phi''_{30}}{\Phi^{(7)}_{30}}&1\\
\phi_{3,22}&\frac{-\sqrt {-15}\zeta_3}{30}x^{16}{\Phi'_3}^3\Phi_4{\Phi''_5}{
\Phi''_6}^3{\Phi''_{10}}\Phi_{12}{\Phi'_{15}}{\Phi^{(8)}_{15}}{\Phi'_{30}}{
\Phi^{(8)}_{30}}&1\\
I_2(5)[1,3]:\zeta_3^2&\frac{-\sqrt {-15}\zeta_3}{15}x^{16}\Phi_1^2\Phi_
2^2\Phi_3^2{\Phi'_3}\Phi_4\Phi_6^2{\Phi''_6}\Phi_{12}{\Phi''''_{15}}{\Phi'''_{
30}}&\zeta_5^2\\
I_2(5)[1,2]:\zeta_3^2&\frac{-\sqrt {-15}\zeta_3}{15}x^{16}\Phi_1^2\Phi_
2^2\Phi_3^2{\Phi'_3}\Phi_4\Phi_6^2{\Phi''_6}\Phi_{12}{\Phi''''_{15}}{\Phi'''_{
30}}&\zeta_5^3\\
B_2:\zeta_3^2&\frac{3+\sqrt {-3}}{12}x^{16}\Phi_1^2\Phi_3^2{\Phi'_3}\Phi_5{
\Phi''_6}^3{\Phi'''_{12}}\Phi_{15}{\Phi'''_{30}}&-1\\
\phi_{6,19}&\frac{3+\sqrt {-3}}{12}x^{16}\Phi_2^2{\Phi'_3}^3\Phi_6^2{\Phi''_6}
\Phi_{10}{\Phi'''_{12}}{\Phi''''_{15}}\Phi_{30}&1\\
\#\hfill \phi_{3,20}''&\frac{-\sqrt {-15}\zeta_3^2}{30}x^{16}{\Phi''_3}^3\Phi_
4{\Phi'_5}{\Phi'_6}^3{\Phi'_{10}}\Phi_{12}{\Phi''_{15}}{\Phi^{(5)}_{15}}{
\Phi''_{30}}{\Phi^{(5)}_{30}}&1\\
\phi_{3,20}'&\frac{\sqrt {-15}\zeta_3^2}{30}x^{16}{\Phi''_3}^3\Phi_4{\Phi''_5}
{\Phi'_6}^3{\Phi''_{10}}\Phi_{12}{\Phi'_{15}}{\Phi^{(6)}_{15}}{\Phi'_{30}}{
\Phi^{(6)}_{30}}&1\\
I_2(5)[1,3]:\zeta_3&\frac{\sqrt {-15}\zeta_3^2}{15}x^{16}\Phi_1^2\Phi_2^2\Phi_
3^2{\Phi''_3}\Phi_4\Phi_6^2{\Phi'_6}\Phi_{12}{\Phi'''_{15}}{\Phi''''_{30}}&
\zeta_5^2\\
I_2(5)[1,2]:\zeta_3&\frac{\sqrt {-15}\zeta_3^2}{15}x^{16}\Phi_1^2\Phi_2^2\Phi_
3^2{\Phi''_3}\Phi_4\Phi_6^2{\Phi'_6}\Phi_{12}{\Phi'''_{15}}{\Phi''''_{30}}&
\zeta_5^3\\
B_2:\zeta_3&\frac{3-\sqrt {-3}}{12}x^{16}\Phi_1^2\Phi_3^2{\Phi''_3}\Phi_5{
\Phi'_6}^3{\Phi''''_{12}}\Phi_{15}{\Phi''''_{30}}&-1\\
\phi_{6,17}&\frac{3-\sqrt {-3}}{12}x^{16}\Phi_2^2{\Phi''_3}^3\Phi_6^2{\Phi'_6}
\Phi_{10}{\Phi''''_{12}}{\Phi'''_{15}}\Phi_{30}&1\\
G_{27}^2[\zeta_3]&\frac{\sqrt {-15}}{30}x^{16}\Phi_1^3\Phi_2^3\Phi_4\Phi_
5\Phi_{10}\Phi_{12}{\Phi''_{15}}{\Phi''_{30}}&\zeta_3\\
G_{27}[\zeta_3]&\frac{\sqrt {-15}}{30}x^{16}\Phi_1^3\Phi_2^3\Phi_4\Phi_5\Phi_{
10}\Phi_{12}{\Phi'_{15}}{\Phi'_{30}}&\zeta_3\\
G_{27}[\zeta_{15}^{11}]&\frac{\sqrt {-15}}{15}x^{16}\Phi_1^3\Phi_2^3\Phi_
3^2\Phi_4\Phi_5\Phi_6^2\Phi_{10}\Phi_{12}&\zeta_{15}^{11}\\
G_{27}[\zeta_{15}^{14}]&\frac{\sqrt {-15}}{15}x^{16}\Phi_1^3\Phi_2^3\Phi_
3^2\Phi_4\Phi_5\Phi_6^2\Phi_{10}\Phi_{12}&\zeta_{15}^{14}\\
G_{27}[-\zeta_3]&\frac{-\sqrt {-3}}6x^{16}\Phi_1^3\Phi_2^3\Phi_3^2\Phi_4\Phi_
5\Phi_{10}\Phi_{15}&-\zeta_3^4\\
G_{27}^3[\zeta_3]&\frac{-\sqrt {-3}}6x^{16}\Phi_1^3\Phi_2^3\Phi_4\Phi_5\Phi_
6^2\Phi_{10}\Phi_{30}&\zeta_3\\
\hline
*\hfill \phi_{1,45}&x^{45}&1\\
\end{supertabular}
\vfill\eject
\subsection{Unipotent characters for $G_{3,3,3}$}
\begin{center}Some principal $\zeta$-series\end{center}
$-1$ : ${\mathcal H}_{Z_{6}}(-x^3,\allowbreak -\zeta_3^2x^2,\allowbreak -
\zeta_3^4x,\allowbreak -1,\allowbreak -\zeta_3^2x,\allowbreak -\zeta_
3^4x^2)$\hfill\break\par
\tablehead{\hline \gamma&\mbox{Deg($\gamma$)}&\mbox{Fr($\gamma$)}&\mbox{Symbol\
}\\\hline}
\tabletail{\hline}
\begin{supertabular}{|R|RRR|}
\shrinkheight{30pt}
*\hfill 1.+&q^3\Phi_2\Phi_6&1&(1.+)\\
\hline
*\hfill 1.\zeta_3&q^3\Phi_2\Phi_6&1&(1.\zeta_3)\\
\hline
*\hfill 1.\zeta_3^2&q^3\Phi_2\Phi_6&1&(1.\zeta_3^2)\\
\hline
*\hfill .11.1&\frac{3-\sqrt {-3}}6q^4{\Phi''_3}^3{\Phi'_6}&1&(01,12,02)\\
\#\hfill .1.11&\frac{3+\sqrt {-3}}6q^4{\Phi'_3}^3{\Phi''_6}&1&(01,02,12)\\
G_{3,3,3}[\zeta_3^2]&\frac{\sqrt {-3}}3q^4\Phi_1^3\Phi_2&\zeta_3^2&(012,012,)\\
\hline
*\hfill ..111&q^9&1&(012,012,123)\\
\hline
*\hfill .2.1&\frac{3+\sqrt {-3}}6q{\Phi'_3}^3{\Phi''_6}&1&(0,2,1)\\
\#\hfill .1.2&\frac{3-\sqrt {-3}}6q{\Phi''_3}^3{\Phi'_6}&1&(0,1,2)\\
G_{3,3,3}[\zeta_3]&\frac{-\sqrt {-3}}3q\Phi_1^3\Phi_2&\zeta_3&(012,,)\\
\hline
*\hfill ..21&q^3\Phi_2\Phi_6&1&(01,01,13)\\
\hline
*\hfill ..3&1&1&(0,0,3)\\
\end{supertabular}
\par
We  used  partition  tuples  for  the  principal series. The partition with
repeated  parts  $1.1.1$  gives  rise  to  3  characters  denoted by $1.+$,
$1.\zeta_3$  and  $1.\zeta_3^2$.  The  cuspidals  are labeled by $\Fr$.
\vfill\eject
\subsection{Unipotent characters for $G_{4,4,3}$}
\begin{center}Some principal $\zeta$-series\end{center}
$\zeta_{3}$ : ${\mathcal H}_{Z_{3}}(\zeta_3x^8,\allowbreak \zeta_
3,\allowbreak \zeta_3x^4)$\hfill\break
$\zeta_{8}$ : ${\mathcal H}_{Z_{8}}(\zeta_8^5x^3,\allowbreak \zeta_
8^7x^2,\allowbreak \zeta_8x,\allowbreak \zeta_8x^2,\allowbreak \zeta_
8^3x,\allowbreak \zeta_8^5,\allowbreak \zeta_8^5x,\allowbreak \zeta_
8^5x^2)$\hfill\break
$-1$ : ${\mathcal H}_{G_{4,1,2}({-i-1\over2})}
(x^2,\allowbreak -ix,\allowbreak -1,\allowbreak ix;-x,-1)$\hfill\break
\begin{center}Non-principal $1$-Harish-Chandra series\end{center}
${\mathcal H}_{G_{4,4,3}}(B_2)={\mathcal H}_{Z_{4}}
(x^2,\allowbreak ix^2,\allowbreak -1,\allowbreak -i)$\hfill\break\par
\tablehead{\hline \gamma&\mbox{Deg($\gamma$)}&\mbox{Fr($\gamma$)}&\mbox{Symbol\
}\\\hline}
\tabletail{\hline}
\begin{supertabular}{|R|RRR|}
\shrinkheight{30pt}
*\hfill .1.1.1&q^3\Phi_3\Phi_8&1&(0,1,1,1)\\
\hline
*\hfill ..11.1&\frac{-i+1}4q^5\Phi_3{\Phi''_4}^2{\Phi'_8}&1&(01,01,12,02)\\
.1..11&\frac12q^5\Phi_3\Phi_8&1&(01,02,01,12)\\
\#\hfill ..1.11&\frac{i+1}4q^5\Phi_3{\Phi'_4}^2{\Phi''_8}&1&(01,01,02,12)\\
B_2:-i&\frac{i+1}4q^5\Phi_1^2\Phi_3{\Phi'_8}&-1&(012,012,0,1)\\
G_{4,4,3}[-i]&\frac{-i}2q^5\Phi_1^3\Phi_2\Phi_3&-i&(012,01,012,)\\
B_2:-1&\frac{-i+1}4q^5\Phi_1^2\Phi_3{\Phi''_8}&-1&(012,012,1,0)\\
\hline
*\hfill ...111&q^{12}&1&(012,012,012,123)\\
\hline
*\hfill ..2.1&\frac{i+1}4q\Phi_3{\Phi'_4}^2{\Phi''_8}&1&(0,0,2,1)\\
.1..2&\frac12q\Phi_3\Phi_8&1&(0,1,0,2)\\
\#\hfill ..1.2&\frac{-i+1}4q\Phi_3{\Phi''_4}^2{\Phi'_8}&1&(0,0,1,2)\\
B_2:i&\frac{-i+1}4q\Phi_1^2\Phi_3{\Phi''_8}&-1&(01,02,,)\\
G_{4,4,3}[i]&\frac i2q\Phi_1^3\Phi_2\Phi_3&i&(012,,0,)\\
B_2:1&\frac{i+1}4q\Phi_1^2\Phi_3{\Phi'_8}&-1&(02,01,,)\\
\hline
*\hfill ...21&q^4\Phi_8&1&(01,01,01,13)\\
\hline
*\hfill ...3&1&1&(0,0,0,3)\\
\end{supertabular}
\par
We  used  partition  tuples  for  the  principal  series. The cuspidals are
labeled by $\Fr$, and the characters 1-Harish-Chandra induced from $B_2$ by
the corresponding labels.
\vfill\eject
\subsection{Unipotent characters for $G_{29}$}
\begin{center}Some principal $\zeta$-series\end{center}
$\zeta_{5}^{4}$ : ${\mathcal H}_{Z_{20}}(\zeta_5^4x^4,\allowbreak \zeta_{20}
x^{5/2},\allowbreak -x^2,\allowbreak \zeta_{20}^{11}x^2,\allowbreak \zeta_
5^4x^3,\allowbreak \zeta_{20}x^{3/2},\allowbreak -\zeta_
5^6x^2,\allowbreak \zeta_{20}^{11}x,\allowbreak \zeta_
5^4x^2,\allowbreak \zeta_{20}x^3,\allowbreak -\zeta_5^7x^2,\allowbreak \zeta_{
20}^{11}x^{5/2},\allowbreak \zeta_5^4x,\allowbreak \zeta_{20}x^2,\allowbreak -
\zeta_5^3x^2,\allowbreak \zeta_{20}^{11}x^{3/2},\allowbreak \zeta_
5^4,\allowbreak \zeta_{20}x,\allowbreak -\zeta_5^4x^2,\allowbreak \zeta_{20}^{
11}x^3)$\hfill\break
$\zeta_{5}^{3}$ : ${\mathcal H}_{Z_{20}}(\zeta_5^3x^4,\allowbreak \zeta_{20}^{
17}x^2,\allowbreak -\zeta_5^7x^2,\allowbreak \zeta_{20}
^7x^3,\allowbreak \zeta_5^3x,\allowbreak \zeta_{20}^{17}x^{3/2},\allowbreak -
\zeta_5^3x^2,\allowbreak \zeta_{20}^7x^{5/2},\allowbreak \zeta_
5^3x^3,\allowbreak \zeta_{20}^{17}x,\allowbreak -\zeta_
5^4x^2,\allowbreak \zeta_{20}^7x^2,\allowbreak \zeta_5^3,\allowbreak \zeta_{
20}^{17}x^3,\allowbreak -x^2,\allowbreak \zeta_{20}^7x^{3/2}
,\allowbreak \zeta_5^3x^2,\allowbreak \zeta_{20}^{17}x^{5/2},\allowbreak -
\zeta_5^6x^2,\allowbreak \zeta_{20}^7x)$\hfill\break
$\zeta_{5}$ : ${\mathcal H}_{Z_{20}}(\zeta_5x^4,\allowbreak \zeta_{20}
^9x^3,\allowbreak -\zeta_5^6x^2,\allowbreak \zeta_{20}^{19}
x,\allowbreak \zeta_5,\allowbreak \zeta_{20}^9x^{3/2},\allowbreak -\zeta_
5^7x^2,\allowbreak \zeta_{20}^{19}x^2,\allowbreak \zeta_5x,\allowbreak \zeta_{
20}^9x^{5/2},\allowbreak -\zeta_5^3x^2,\allowbreak \zeta_{20}^{19}
x^3,\allowbreak \zeta_5x^2,\allowbreak \zeta_{20}^9x,\allowbreak -\zeta_
5^4x^2,\allowbreak \zeta_{20}^{19}x^{3/2},\allowbreak \zeta_
5x^3,\allowbreak \zeta_{20}^9x^2,\allowbreak -x^2,\allowbreak \zeta_{20}^{19}
x^{5/2})$\hfill\break
\begin{center}Non-principal $1$-Harish-Chandra series\end{center}
${\mathcal H}_{G_{29}}(B_2)={\mathcal H}_{G_{4,1,2}}
(x^2,\allowbreak ix^2,\allowbreak -1,\allowbreak -i;x^3,-1)$\hfill\break
${\mathcal H}_{G_{29}}(G_{4,4,3}[i])={\mathcal H}_{Z_{4}}
(x^6,\allowbreak ix,\allowbreak -1,\allowbreak -ix)$\hfill\break
${\mathcal H}_{G_{29}}(G_{4,4,3}[-i])={\mathcal H}_{Z_{4}}
(x^6,\allowbreak ix^5,\allowbreak -1,\allowbreak -ix^5)$\hfill\break
{\small
\par\tablehead{\hline \gamma&\mbox{Deg($\gamma$)}&\mbox{Fr($\gamma$)}\\\hline}
\tabletail{\hline}
\begin{supertabular}{|R|RR|}
\shrinkheight{30pt}
*\hfill \phi_{1,0}&1&1\\
\hline
*\hfill \phi_{4,1}&\frac{i+1}4x\Phi_2^2{\Phi''_4}^2\Phi_6{\Phi'_8}\Phi_{10}{
\Phi'_{12}}{\Phi''''_{20}}&1\\
\phi_{4,4}&\frac12x\Phi_4^3\Phi_{12}\Phi_{20}&1\\
\#\hfill \phi_{4,3}&\frac{-i+1}4x\Phi_2^2{\Phi'_4}^2\Phi_6{\Phi''_8}\Phi_{10}{
\Phi''_{12}}{\Phi'''_{20}}&1\\
B_2:.2..&\frac{-i+1}4x\Phi_1^2\Phi_3{\Phi''_4}^2\Phi_5{\Phi''_8}{\Phi'_{12}}{
\Phi''''_{20}}&-1\\
G_{4,4,3}[i]:1&\frac i2x\Phi_1^3\Phi_2^3\Phi_3\Phi_5\Phi_6\Phi_{10}&i\\
B_2:2...&\frac{i+1}4x\Phi_1^2\Phi_3{\Phi'_4}^2\Phi_5{\Phi'_8}{\Phi''_{12}}{
\Phi'''_{20}}&-1\\
\hline
*\hfill \phi_{10,2}&x^2\Phi_5\Phi_8\Phi_{10}\Phi_{20}&1\\
\hline
G_{29}[-\zeta_8]&\frac{-i}2x^3\Phi_1^4\Phi_2^4\Phi_3\Phi_5\Phi_6\Phi_8\Phi_{
10}&\zeta_8^5x^{1/2}\\
*\hfill \phi_{16,3}&\frac12x^3\Phi_4^4\Phi_8\Phi_{12}\Phi_{20}&1\\
G_{29}[\zeta_8]&\frac{-i}2x^3\Phi_1^4\Phi_2^4\Phi_3\Phi_5\Phi_6\Phi_8\Phi_{10}
&\zeta_8x^{1/2}\\
\phi_{16,5}&\frac12x^3\Phi_4^4\Phi_8\Phi_{12}\Phi_{20}&1\\
\hline
*\hfill \phi_{15,4}''&\frac12x^4\Phi_3\Phi_5\Phi_6\Phi_8\Phi_{10}\Phi_{20}&1\\
\phi_{5,8}&\frac12x^4\Phi_5\Phi_8\Phi_{10}\Phi_{12}\Phi_{20}&1\\
\phi_{10,6}&\frac12x^4\Phi_4^2\Phi_5\Phi_{10}\Phi_{12}\Phi_{20}&1\\
B_2:1.1..&\frac12x^4\Phi_1^2\Phi_2^2\Phi_3\Phi_5\Phi_6\Phi_{10}\Phi_{20}&-1\\
\hline
*\hfill \phi_{15,4}'&x^4\Phi_3\Phi_5\Phi_6\Phi_{10}\Phi_{12}\Phi_{20}&1\\
\hline
*\hfill \phi_{20,5}&\frac{-i+1}4x^5\Phi_2^2{\Phi''_4}^2\Phi_5\Phi_6{\Phi''_8}
\Phi_{10}{\Phi'_{12}}\Phi_{20}&1\\
\phi_{20,6}&\frac12x^5\Phi_4^3\Phi_5\Phi_{10}\Phi_{12}\Phi_{20}&1\\
\#\hfill \phi_{20,7}&\frac{i+1}4x^5\Phi_2^2{\Phi'_4}^2\Phi_5\Phi_6{\Phi'_8}
\Phi_{10}{\Phi''_{12}}\Phi_{20}&1\\
B_2:...2&\frac{i+1}4x^5\Phi_1^2\Phi_3{\Phi''_4}^2\Phi_5{\Phi'_8}\Phi_{10}{
\Phi'_{12}}\Phi_{20}&-1\\
G_{4,4,3}[-i]:1&\frac{-i}2x^5\Phi_1^3\Phi_2^3\Phi_3\Phi_5\Phi_6\Phi_{10}\Phi_{
20}&-i\\
B_2:..2.&\frac{-i+1}4x^5\Phi_1^2\Phi_3{\Phi'_4}^2\Phi_5{\Phi''_8}\Phi_{10}{
\Phi''_{12}}\Phi_{20}&-1\\
\hline
*\hfill \phi_{24,6}&\frac1{20}x^6\Phi_2^4\Phi_3\Phi_5\Phi_6\Phi_8\Phi_{12}
\Phi_{20}&1\\
G_{29}[1]&\frac1{20}x^6\Phi_1^4\Phi_3\Phi_6\Phi_8\Phi_{10}\Phi_{12}\Phi_{20}&
1\\
\phi_{6,10}''''&\frac{-1}{20}x^6\Phi_3{\Phi''_4}^4\Phi_5\Phi_6\Phi_8\Phi_{10}
\Phi_{12}{\Phi'''_{20}}&1\\
\phi_{6,10}'''&\frac{-1}{20}x^6\Phi_3{\Phi'_4}^4\Phi_5\Phi_6\Phi_8\Phi_{10}
\Phi_{12}{\Phi''''_{20}}&1\\
\phi_{6,10}''&\frac15x^6\Phi_3\Phi_5\Phi_6\Phi_8\Phi_{10}\Phi_{12}\Phi_{20}&
1\\
\phi_{24,9}&\frac14x^6\Phi_2^2\Phi_3\Phi_4^3\Phi_6\Phi_{10}\Phi_{12}\Phi_{20}&
1\\
B_2:.1..1&\frac14x^6\Phi_1^2\Phi_3\Phi_4^3\Phi_5\Phi_6\Phi_{12}\Phi_{20}&-1\\
G_{4,4,3}[i]:-i&\frac{-1}4x^6\Phi_1^3\Phi_2^3\Phi_3{\Phi'_4}^2\Phi_5\Phi_
6\Phi_{10}\Phi_{12}{\Phi'''_{20}}&i\\
G_{4,4,3}[-i]:i&\frac{-1}4x^6\Phi_1^3\Phi_2^3\Phi_3{\Phi''_4}^2\Phi_5\Phi_
6\Phi_{10}\Phi_{12}{\Phi''''_{20}}&-i\\
\phi_{24,7}&\frac14x^6\Phi_2^2\Phi_3\Phi_4^3\Phi_6\Phi_{10}\Phi_{12}\Phi_{20}&
1\\
B_2:1..1.&\frac14x^6\Phi_1^2\Phi_3\Phi_4^3\Phi_5\Phi_6\Phi_{12}\Phi_{20}&-1\\
G_{4,4,3}[-i]:-i&\frac14x^6\Phi_1^3\Phi_2^3\Phi_3{\Phi'_4}^2\Phi_5\Phi_6\Phi_
{10}\Phi_{12}{\Phi'''_{20}}&-i\\
G_{4,4,3}[i]:i&\frac14x^6\Phi_1^3\Phi_2^3\Phi_3{\Phi''_4}^2\Phi_5\Phi_6\Phi_{
10}\Phi_{12}{\Phi''''_{20}}&i\\
\phi_{30,8}&\frac14x^6\Phi_3\Phi_4^2\Phi_5\Phi_6\Phi_8\Phi_{12}\Phi_{20}&1\\
\phi_{6,12}&\frac14x^6\Phi_3\Phi_4^2\Phi_6\Phi_8\Phi_{10}\Phi_{12}\Phi_{20}&
1\\
B_2:.1.1.&\frac14x^6\Phi_1^2\Phi_2^2\Phi_3\Phi_5\Phi_6\Phi_8\Phi_{10}\Phi_{12}
{\Phi''''_{20}}&-1\\
B_2:1...1&\frac14x^6\Phi_1^2\Phi_2^2\Phi_3\Phi_5\Phi_6\Phi_8\Phi_{10}\Phi_{12}
{\Phi'''_{20}}&-1\\
\phi_{6,10}'&\frac15x^6\Phi_3\Phi_5\Phi_6\Phi_8\Phi_{10}\Phi_{12}\Phi_{20}&1\\
G_{29}[\zeta_5]&\frac15x^6\Phi_1^4\Phi_2^4\Phi_3\Phi_4^4\Phi_6\Phi_8\Phi_{12}&
\zeta_5\\
G_{29}[\zeta_5^2]&\frac15x^6\Phi_1^4\Phi_2^4\Phi_3\Phi_4^4\Phi_6\Phi_8\Phi_{
12}&\zeta_5^2\\
G_{29}[\zeta_5^3]&\frac15x^6\Phi_1^4\Phi_2^4\Phi_3\Phi_4^4\Phi_6\Phi_8\Phi_{
12}&\zeta_5^3\\
G_{29}[\zeta_5^4]&\frac15x^6\Phi_1^4\Phi_2^4\Phi_3\Phi_4^4\Phi_6\Phi_8\Phi_{
12}&\zeta_5^4\\
\hline
*\hfill \phi_{20,9}&\frac{i+1}4x^9\Phi_2^2{\Phi'_4}^2\Phi_5\Phi_6{\Phi'_8}
\Phi_{10}{\Phi''_{12}}\Phi_{20}&1\\
\phi_{20,10}&\frac12x^9\Phi_4^3\Phi_5\Phi_{10}\Phi_{12}\Phi_{20}&1\\
\#\hfill \phi_{20,11}&\frac{-i+1}4x^9\Phi_2^2{\Phi''_4}^2\Phi_5\Phi_6{\Phi''_
8}\Phi_{10}{\Phi'_{12}}\Phi_{20}&1\\
B_2:.11..&\frac{-i+1}4x^9\Phi_1^2\Phi_3{\Phi'_4}^2\Phi_5{\Phi''_8}\Phi_{10}{
\Phi''_{12}}\Phi_{20}&-1\\
G_{4,4,3}[i]:-1&\frac i2x^9\Phi_1^3\Phi_2^3\Phi_3\Phi_5\Phi_6\Phi_{10}\Phi_{
20}&i\\
B_2:11...&\frac{i+1}4x^9\Phi_1^2\Phi_3{\Phi''_4}^2\Phi_5{\Phi'_8}\Phi_{10}{
\Phi'_{12}}\Phi_{20}&-1\\
\hline
*\hfill \phi_{15,12}'&x^{12}\Phi_3\Phi_5\Phi_6\Phi_{10}\Phi_{12}\Phi_{20}&1\\
\hline
*\hfill \phi_{15,12}''&\frac12x^{12}\Phi_3\Phi_5\Phi_6\Phi_8\Phi_{10}\Phi_{20}
&1\\
\phi_{5,16}&\frac12x^{12}\Phi_5\Phi_8\Phi_{10}\Phi_{12}\Phi_{20}&1\\
\phi_{10,14}&\frac12x^{12}\Phi_4^2\Phi_5\Phi_{10}\Phi_{12}\Phi_{20}&1\\
B_2:..1.1&\frac12x^{12}\Phi_1^2\Phi_2^2\Phi_3\Phi_5\Phi_6\Phi_{10}\Phi_{20}&-
1\\
\hline
G_{29}[\zeta_8^3]&\frac i2x^{13}\Phi_1^4\Phi_2^4\Phi_3\Phi_5\Phi_6\Phi_8\Phi_{
10}&\zeta_8^3x^{1/2}\\
\phi_{16,13}&\frac12x^{13}\Phi_4^4\Phi_8\Phi_{12}\Phi_{20}&1\\
G_{29}[-\zeta_8^3]&\frac i2x^{13}\Phi_1^4\Phi_2^4\Phi_3\Phi_5\Phi_6\Phi_8\Phi_
{10}&\zeta_8^7x^{1/2}\\
*\hfill \phi_{16,15}&\frac12x^{13}\Phi_4^4\Phi_8\Phi_{12}\Phi_{20}&1\\
\hline
*\hfill \phi_{10,18}&x^{18}\Phi_5\Phi_8\Phi_{10}\Phi_{20}&1\\
\hline
*\hfill \phi_{4,21}&\frac{-i+1}4x^{21}\Phi_2^2{\Phi'_4}^2\Phi_6{\Phi''_8}\Phi_
{10}{\Phi''_{12}}{\Phi'''_{20}}&1\\
\phi_{4,24}&\frac12x^{21}\Phi_4^3\Phi_{12}\Phi_{20}&1\\
\#\hfill \phi_{4,23}&\frac{i+1}4x^{21}\Phi_2^2{\Phi''_4}^2\Phi_6{\Phi'_8}\Phi_
{10}{\Phi'_{12}}{\Phi''''_{20}}&1\\
B_2:...11&\frac{i+1}4x^{21}\Phi_1^2\Phi_3{\Phi'_4}^2\Phi_5{\Phi'_8}{\Phi''_{
12}}{\Phi'''_{20}}&-1\\
G_{4,4,3}[-i]:-1&\frac{-i}2x^{21}\Phi_1^3\Phi_2^3\Phi_3\Phi_5\Phi_6\Phi_{10}&
-i\\
B_2:..11.&\frac{-i+1}4x^{21}\Phi_1^2\Phi_3{\Phi''_4}^2\Phi_5{\Phi''_8}{\Phi'_{
12}}{\Phi''''_{20}}&-1\\
\hline
*\hfill \phi_{1,40}&x^{40}&1\\
\end{supertabular}}
\vfill\eject
\subsection{Unipotent characters for $G_{32}$}
\begin{center}Some principal $\zeta$-series\end{center}
$\zeta_{8}^{3}$ : ${\mathcal H}_{Z_{24}}(\zeta_8x^5,\allowbreak \zeta_{24}
^7x^2,\allowbreak \zeta_{24}^{11}x^{5/3},\allowbreak \zeta_
8^3x^2,\allowbreak \zeta_{24}^{19}x,\allowbreak \zeta_{24}^{11}
x^2,\allowbreak \zeta_8x^3,\allowbreak \zeta_{24}^7,\allowbreak \zeta_{24}^{
23}x,\allowbreak \zeta_8^5x^2,\allowbreak \zeta_{24}^{19}x^{5/3}
,\allowbreak \zeta_{24}^{11},\allowbreak \zeta_8x,\allowbreak \zeta_{24}^{19}
x^2,\allowbreak \zeta_{24}^{11}x^3,\allowbreak \zeta_8^7x^2,\allowbreak \zeta_
{24}^7x,\allowbreak \zeta_{24}^{23}x^2,\allowbreak \zeta_8x^{5/3}
,\allowbreak \zeta_{24}^{19},\allowbreak \zeta_{24}^{11}x,\allowbreak \zeta_
8x^2,\allowbreak \zeta_{24}^{19}x^3,\allowbreak \zeta_{24}^{23})$\hfill\break
$\zeta_{8}$ : ${\mathcal H}_{Z_{24}}(\zeta_8^3x^5,\allowbreak \zeta_{24}
,\allowbreak \zeta_{24}^{17}x^3,\allowbreak \zeta_8^7x^2,\allowbreak \zeta_{
24}x,\allowbreak \zeta_{24}^5,\allowbreak \zeta_8^3x^{5/3},\allowbreak \zeta_{
24}x^2,\allowbreak \zeta_{24}^5x,\allowbreak \zeta_8x^2,\allowbreak \zeta_{24}
x^3,\allowbreak \zeta_{24}^5x^2,\allowbreak \zeta_8^3x,\allowbreak \zeta_{24}
^{13},\allowbreak \zeta_{24}^{17}x^{5/3},\allowbreak \zeta_
8^3x^2,\allowbreak \zeta_{24}^{13}x,\allowbreak \zeta_{24}^{17}
,\allowbreak \zeta_8^3x^3,\allowbreak \zeta_{24}^{13}x^2,\allowbreak \zeta_{
24}^{17}x,\allowbreak \zeta_8^5x^2,\allowbreak \zeta_{24}x^{5/3}
,\allowbreak \zeta_{24}^{17}x^2)$\hfill\break
$\zeta_{5}^{3}$ : ${\mathcal H}_{Z_{30}}(\zeta_5^3x^4,\allowbreak -\zeta_{15}
^{14}x,\allowbreak \zeta_{15}^4x^{4/3},\allowbreak -x,\allowbreak \zeta_{15}^{
14}x^2,\allowbreak -\zeta_{15}^{19}x^{3/2},\allowbreak \zeta_
5^3x,\allowbreak -\zeta_{15}^{14}x^3,\allowbreak \zeta_{15}^4,\allowbreak -
\zeta_5^6x,\allowbreak \zeta_{15}^{14}x^{3/2},\allowbreak -\zeta_{15}^{19}
x,\allowbreak \zeta_5^3x^{4/3},\allowbreak -\zeta_{15}^{14},\allowbreak \zeta_
{15}^4x^2,\allowbreak -\zeta_5^7x,\allowbreak \zeta_{15}^{14}x,\allowbreak -
\zeta_{15}^{19}x^3,\allowbreak \zeta_5^3,\allowbreak -\zeta_{15}^{14}
x^2,\allowbreak \zeta_{15}^4x^{3/2},\allowbreak -\zeta_5^3x,\allowbreak \zeta_
{15}^{14}x^{4/3},\allowbreak -\zeta_{15}^{19},\allowbreak \zeta_
5^3x^2,\allowbreak -\zeta_{15}^{14}x^{3/2},\allowbreak \zeta_{15}
^4x,\allowbreak -\zeta_5^4x,\allowbreak \zeta_{15}^{14},\allowbreak -\zeta_{
15}^{19}x^2)$\hfill\break
$\zeta_{5}^{2}$ : ${\mathcal H}_{Z_{30}}(\zeta_5^2x^4,\allowbreak -\zeta_{15}
^{11}x^2,\allowbreak \zeta_{15},\allowbreak -\zeta_5^6x,\allowbreak \zeta_{15}
^{11}x,\allowbreak -\zeta_{15}^{16}x^{3/2},\allowbreak \zeta_
5^2x^2,\allowbreak -\zeta_{15}^{11},\allowbreak \zeta_{15}x^{4/3}
,\allowbreak -\zeta_5^7x,\allowbreak \zeta_{15}^{11}x^{3/2},\allowbreak -
\zeta_{15}^{16}x^2,\allowbreak \zeta_5^2,\allowbreak -\zeta_{15}^{11}
x^3,\allowbreak \zeta_{15}x,\allowbreak -\zeta_5^3x,\allowbreak \zeta_{15}^{
11}x^2,\allowbreak -\zeta_{15}^{16},\allowbreak \zeta_5^2x^{4/3},\allowbreak -
\zeta_{15}^{11}x,\allowbreak \zeta_{15}x^{3/2},\allowbreak -\zeta_
5^4x,\allowbreak \zeta_{15}^{11},\allowbreak -\zeta_{15}^{16}
x^3,\allowbreak \zeta_5^2x,\allowbreak -\zeta_{15}^{11}x^{3/2}
,\allowbreak \zeta_{15}x^2,\allowbreak -x,\allowbreak \zeta_{15}^{11}x^{4/3}
,\allowbreak -\zeta_{15}^{16}x)$\hfill\break
$\zeta_{5}$ : ${\mathcal H}_{Z_{30}}(\zeta_5x^4,\allowbreak -\zeta_{15}
^8,\allowbreak \zeta_{15}^{13}x,\allowbreak -\zeta_5^7x,\allowbreak \zeta_{15}
^8x^{4/3},\allowbreak -\zeta_{15}^{13}x^{3/2},\allowbreak \zeta_
5,\allowbreak -\zeta_{15}^8x,\allowbreak \zeta_{15}^{13}x^2,\allowbreak -
\zeta_5^3x,\allowbreak \zeta_{15}^8x^{3/2},\allowbreak -\zeta_{15}^{13}
,\allowbreak \zeta_5x,\allowbreak -\zeta_{15}^8x^2,\allowbreak \zeta_{15}^{13}
x^{4/3},\allowbreak -\zeta_5^4x,\allowbreak \zeta_{15}^8,\allowbreak -\zeta_{
15}^{13}x,\allowbreak \zeta_5x^2,\allowbreak -\zeta_{15}
^8x^3,\allowbreak \zeta_{15}^{13}x^{3/2},\allowbreak -x,\allowbreak \zeta_{15}
^8x,\allowbreak -\zeta_{15}^{13}x^2,\allowbreak \zeta_5x^{4/3},\allowbreak -
\zeta_{15}^8x^{3/2},\allowbreak \zeta_{15}^{13},\allowbreak -\zeta_
5^6x,\allowbreak \zeta_{15}^8x^2,\allowbreak -\zeta_{15}^{13}x^3)$\hfill\break
$\zeta_{4}$ : ${\mathcal H}_{G_{10}((-3-\sqrt 3)\zeta_3^2)}(-
x^2,\zeta_3,\zeta_3^2;ix^3,\allowbreak i,\allowbreak ix,\allowbreak -i)$\hfill\break
\begin{center}Non-principal $1$-Harish-Chandra series\end{center}
${\mathcal H}_{G_{32}}(Z_3)={\mathcal H}_{G_{26}}
(x^3,-1; x,\zeta_3,\zeta_3^2)$\hfill\break
${\mathcal H}_{G_{32}}(G_4)={\mathcal H}_{G_{5}}(x,\zeta_3,\zeta_3^2;1,\allowbreak \zeta_
3x^4,\allowbreak \zeta_3^2x^4)$\hfill\break
${\mathcal H}_{G_{32}}(Z_3\otimes Z_3)={\mathcal H}_{G_{6,1,2}}
(x^3,\allowbreak -\zeta_3^2x^3,\allowbreak \zeta_3x^2,\allowbreak -
1,\allowbreak \zeta_3^2,\allowbreak -\zeta_3^4x^2;x^3,-1)$\hfill\break
${\mathcal H}_{G_{32}}(G_{25}[\zeta_{3}])={\mathcal H}_{Z_{6}}
(x^6,\allowbreak -\zeta_3^2x^4,\allowbreak \zeta_3x,\allowbreak -
1,\allowbreak \zeta_3^2x,\allowbreak -\zeta_3^4x^4)$\hfill\break
${\mathcal H}_{G_{32}}(G_{25}[-\zeta_{3}])={\mathcal H}_{Z_{6}}
(x^6,\allowbreak -\zeta_3^2x,\allowbreak \zeta_3x^4,\allowbreak -
1,\allowbreak \zeta_3^2x^4,\allowbreak -\zeta_3^4x)$\hfill\break
${\mathcal H}_{G_{32}}(G_4\otimes Z_3)={\mathcal H}_{Z_{6}}(x^9,\allowbreak -
\zeta_3^2x^8,\allowbreak \zeta_3x^5,\allowbreak -1,\allowbreak \zeta_
3^2x^5,\allowbreak -\zeta_3^4x^8)$\hfill\break
{\small
\par\tablehead{\hline \gamma&\mbox{Deg($\gamma$)}&\mbox{Fr($\gamma$)}\\\hline}
\tabletail{\hline}
\begin{supertabular}{|R|RR|}
\shrinkheight{30pt}
*\hfill \phi_{1,0}&1&1\\
\hline
*\hfill \phi_{4,1}&\frac{3-\sqrt {-3}}6x{\Phi''_3}\Phi_4{\Phi'_6}\Phi_8\Phi_{
12}{\Phi'''_{15}}\Phi_{24}{\Phi''''_{30}}&1\\
\#\hfill \phi_{4,11}&\frac{3+\sqrt {-3}}6x{\Phi'_3}\Phi_4{\Phi''_6}\Phi_8\Phi_
{12}{\Phi''''_{15}}\Phi_{24}{\Phi'''_{30}}&1\\
Z_3:\phi_{1,0}&\frac{\sqrt {-3}}3x\Phi_1\Phi_2\Phi_4\Phi_5\Phi_8\Phi_{10}\Phi_
{12}\Phi_{24}&\zeta_3^2\\
\hline
\phi_{15,6}&\frac13x^2\Phi_3^2\Phi_5\Phi_6^2\Phi_{10}\Phi_{12}\Phi_{15}\Phi_{
24}\Phi_{30}&1\\
*\hfill \phi_{10,2}&\frac13x^2{\Phi'_3}\Phi_5{\Phi''_6}\Phi_8{\Phi'_9}\Phi_{
10}{\Phi'''_{12}}^2\Phi_{15}{\Phi''_{18}}{\Phi''_{24}}\Phi_{30}&1\\
\#\hfill \phi_{10,10}&\frac13x^2{\Phi''_3}\Phi_5{\Phi'_6}\Phi_8{\Phi''_9}\Phi_
{10}{\Phi''''_{12}}^2\Phi_{15}{\Phi'_{18}}{\Phi'_{24}}\Phi_{30}&1\\
\phi_{5,20}&\frac{\zeta_3}3x^2{\Phi'_3}\Phi_5{\Phi''_6}{\Phi''_9}\Phi_{10}{
\Phi'''_{12}}^2\Phi_{15}{\Phi'_{18}}\Phi_{24}\Phi_{30}&1\\
\hline
\end{supertabular}
}
\vfill\eject
\subsection{Unipotent characters for $G_{33}$}
\begin{center}Some principal $\zeta$-series\end{center}
$\zeta_{5}$ : ${\mathcal H}_{Z_{10}}(\zeta_5x^9,\allowbreak -\zeta_5^6x^{9/2}
,\allowbreak \zeta_5x^5,\allowbreak -\zeta_5^6x^3,\allowbreak \zeta_
5x^6,\allowbreak -\zeta_5^6x^4,\allowbreak \zeta_5x^{9/2},\allowbreak -\zeta_
5^6,\allowbreak \zeta_5x^3,\allowbreak -\zeta_5^6x^6)$\hfill\break
$\zeta_{5}^{3}$ : ${\mathcal H}_{Z_{10}}(\zeta_5^3x^9,\allowbreak -\zeta_
5^3,\allowbreak \zeta_5^3x^6,\allowbreak -\zeta_5^3x^{9/2},\allowbreak \zeta_
5^3x^3,\allowbreak -\zeta_5^3x^4,\allowbreak \zeta_5^3x^5,\allowbreak -\zeta_
5^3x^6,\allowbreak \zeta_5^3x^{9/2},\allowbreak -\zeta_5^3x^3)$\hfill\break
$\zeta_{9}^{5}$ : ${\mathcal H}_{Z_{18}}(\zeta_9^2x^5,\allowbreak -\zeta_
9^8x^3,\allowbreak \zeta_9^2x^{5/2},\allowbreak -\zeta_
9^5x^2,\allowbreak \zeta_9^5x^3,\allowbreak -\zeta_9^5x^4,\allowbreak \zeta_
9^2x^2,\allowbreak -\zeta_9^{11}x^3,\allowbreak \zeta_9^8x,\allowbreak -\zeta_
9^8x^2,\allowbreak \zeta_9^8x^3,\allowbreak -\zeta_9^{11}x^{5/2}
,\allowbreak \zeta_9^5x^2,\allowbreak -\zeta_9^{11},\allowbreak \zeta_
9^2x,\allowbreak -\zeta_9^{11}x^2,\allowbreak \zeta_9^2x^3,\allowbreak -\zeta_
9^{11}x^4)$\hfill\break
$\zeta_{9}$ : ${\mathcal H}_{Z_{18}}(\zeta_9^4x^5,\allowbreak -\zeta_
9x^4,\allowbreak \zeta_9^7x^3,\allowbreak -\zeta_9^{13}x^2,\allowbreak \zeta_
9^4x^{5/2},\allowbreak -\zeta_9^{13}x^3,\allowbreak \zeta_9x^2,\allowbreak -
\zeta_9^{13}x^4,\allowbreak \zeta_9x^3,\allowbreak -\zeta_
9^7x^2,\allowbreak \zeta_9^4x,\allowbreak -\zeta_9^7x^3,\allowbreak \zeta_
9^4x^2,\allowbreak -\zeta_9^{13}x^{5/2},\allowbreak \zeta_9^4x^3,\allowbreak -
\zeta_9x^2,\allowbreak \zeta_9^7x,\allowbreak -\zeta_9^{13})$\hfill\break
$\zeta_{6}$ : ${\mathcal H}_{G_{26}}(-\zeta_3x,-1;\zeta_
3^2x^2,\allowbreak \zeta_3,\allowbreak -x)$\hfill\break
$\zeta_{3}$ : ${\mathcal H}_{G_{26}}(\zeta_3^2x,-1;\zeta_
3x^2,\allowbreak x,\allowbreak \zeta_3^2)$\hfill\break
\begin{center}Non-principal $1$-Harish-Chandra series\end{center}
${\mathcal H}_{G_{33}}(G_{3,3,3}[\zeta_3])={\mathcal H}_{G_{4}}
(x^3,\zeta_3,\zeta_3^2)$\hfill\break
${\mathcal H}_{G_{33}}(G_{3,3,3}[\zeta_3^2])={\mathcal H}_{G_{4}}
(1,\allowbreak \zeta_3x^3,\allowbreak \zeta_3^2x^3)$\hfill\break
${\mathcal H}_{G_{33}}(D_4)={\mathcal H}_{Z_{6}}(x^5,\allowbreak -\zeta_
3^2x^4,\allowbreak \zeta_3x,\allowbreak -1,\allowbreak \zeta_
3^2x,\allowbreak -\zeta_3^4x^4)$\hfill\break
{\small
\par\tablehead{\hline \gamma&\mbox{Deg($\gamma$)}&\mbox{Fr($\gamma$)}\\\hline}
\tabletail{\hline}
\begin{supertabular}{|R|RR|}
\shrinkheight{30pt}
*\hfill \phi_{1,0}&1&1\\
\hline
*\hfill \phi_{5,1}&\frac{3+\sqrt {-3}}6x\Phi_5{\Phi'_9}\Phi_{10}{\Phi'''_{12}}
{\Phi''_{18}}&1\\
\#\hfill \phi_{5,3}&\frac{3-\sqrt {-3}}6x\Phi_5{\Phi''_9}\Phi_{10}{\Phi''''_{
12}}{\Phi'_{18}}&1\\
G_{3,3,3}[\zeta_3]:\phi_{1,0}&\frac{-\sqrt {-3}}3x\Phi_1^3\Phi_2^3\Phi_4\Phi_
5\Phi_{10}&\zeta_3\\
\hline
*\hfill \phi_{15,2}&x^2\Phi_5\Phi_9\Phi_{10}\Phi_{18}&1\\
\hline
*\hfill \phi_{30,3}&\frac12x^3\Phi_4^2\Phi_5\Phi_9\Phi_{12}\Phi_{18}&1\\
\phi_{6,5}&\frac12x^3\Phi_4^2\Phi_9\Phi_{10}\Phi_{12}\Phi_{18}&1\\
\phi_{24,4}&\frac12x^3\Phi_2^4\Phi_6^2\Phi_9\Phi_{10}\Phi_{18}&1\\
D_4:1&\frac12x^3\Phi_1^4\Phi_3^2\Phi_5\Phi_9\Phi_{18}&-1\\
\hline
*\hfill \phi_{30,4}&\frac{-\zeta_3^2}6x^4{\Phi'_3}^3\Phi_4^2\Phi_5{\Phi''_9}
\Phi_{10}\Phi_{12}\Phi_{18}&1\\
\phi_{40,5}''&\frac{3+\sqrt {-3}}{12}x^4\Phi_2^4\Phi_5\Phi_6^2{\Phi'_6}{\Phi'_
9}\Phi_{10}{\Phi''''_{12}}\Phi_{18}&1\\
\phi_{20,6}&\frac13x^4\Phi_4^2\Phi_5\Phi_9\Phi_{10}\Phi_{12}\Phi_{18}&1\\
\phi_{40,5}'&\frac{3-\sqrt {-3}}{12}x^4\Phi_2^4\Phi_5\Phi_6^2{\Phi''_6}{
\Phi''_9}\Phi_{10}{\Phi'''_{12}}\Phi_{18}&1\\
\#\hfill \phi_{30,6}&\frac{-\zeta_3}6x^4{\Phi''_3}^3\Phi_4^2\Phi_5{\Phi'_9}
\Phi_{10}\Phi_{12}\Phi_{18}&1\\
G_{3,3,3}[\zeta_3^2]:\phi_{1,8}&\frac16x^4\Phi_1^3\Phi_2^3\Phi_4^2\Phi_5{
\Phi''_9}\Phi_{10}\Phi_{12}{\Phi''_{18}}&\zeta_3^2\\
D_4:-\zeta_3^4&\frac{-3+\sqrt {-3}}{12}x^4\Phi_1^4\Phi_3^2{\Phi'_3}\Phi_5\Phi_
9\Phi_{10}{\Phi'''_{12}}{\Phi'_{18}}&-1\\
G_{3,3,3}[\zeta_3]:\phi_{2,3}&\frac{-\zeta_3}3x^4\Phi_1^3\Phi_2^3\Phi_4^2\Phi_
5{\Phi''_9}\Phi_{10}\Phi_{12}{\Phi'_{18}}&\zeta_3\\
G_{33}[-\zeta_3^2]&\frac{-\sqrt {-3}}6x^4\Phi_1^5\Phi_2^3\Phi_3^2\Phi_4\Phi_
5\Phi_9\Phi_{10}&-\zeta_3^2\\
\phi_{10,8}''&\frac{\zeta_3}6x^4\Phi_4^2\Phi_5{\Phi'_6}^3\Phi_9\Phi_{10}\Phi_{
12}{\Phi''_{18}}&1\\
G_{3,3,3}[\zeta_3^2]:\phi_{2,5}&\frac{\sqrt {-3}}6x^4\Phi_1^3\Phi_2^5\Phi_
4\Phi_5\Phi_6^2\Phi_{10}\Phi_{18}&\zeta_3^2\\
G_{3,3,3}[\zeta_3]:\phi_{2,1}&\frac{\zeta_3^2}3x^4\Phi_1^3\Phi_2^3\Phi_
4^2\Phi_5{\Phi'_9}\Phi_{10}\Phi_{12}{\Phi''_{18}}&\zeta_3\\
\phi_{10,8}'&\frac{\zeta_3^2}6x^4\Phi_4^2\Phi_5{\Phi''_6}^3\Phi_9\Phi_{10}
\Phi_{12}{\Phi'_{18}}&1\\
D_4:-\zeta_3^2&\frac{-3-\sqrt {-3}}{12}x^4\Phi_1^4\Phi_3^2{\Phi''_3}\Phi_
5\Phi_9\Phi_{10}{\Phi''''_{12}}{\Phi''_{18}}&-1\\
G_{3,3,3}[\zeta_3^2]:\phi_{1,4}&\frac{-1}6x^4\Phi_1^3\Phi_2^3\Phi_4^2\Phi_5{
\Phi'_9}\Phi_{10}\Phi_{12}{\Phi'_{18}}&\zeta_3^2\\
\hline
*\hfill \phi_{81,6}&x^6\Phi_3^3\Phi_6^3\Phi_9\Phi_{12}\Phi_{18}&1\\
\hline
*\hfill \phi_{60,7}&x^7\Phi_4^2\Phi_5\Phi_9\Phi_{10}\Phi_{12}\Phi_{18}&1\\
\hline
*\hfill \phi_{45,7}&\frac{3+\sqrt {-3}}6x^7{\Phi''_3}^3\Phi_5{\Phi'_6}^3\Phi_
9\Phi_{10}{\Phi''''_{12}}\Phi_{18}&1\\
\#\hfill \phi_{45,9}&\frac{3-\sqrt {-3}}6x^7{\Phi'_3}^3\Phi_5{\Phi''_6}^3\Phi_
9\Phi_{10}{\Phi'''_{12}}\Phi_{18}&1\\
G_{3,3,3}[\zeta_3]:\phi_{3,2}&\frac{-\sqrt {-3}}3x^7\Phi_1^3\Phi_2^3\Phi_
4\Phi_5\Phi_9\Phi_{10}\Phi_{18}&\zeta_3\\
\hline
*\hfill \phi_{64,8}&\frac12x^8\Phi_2^5\Phi_4^2\Phi_6^3\Phi_{10}\Phi_{12}\Phi_{
18}&1\\
\#\hfill \phi_{64,9}&\frac12x^8\Phi_2^5\Phi_4^2\Phi_6^3\Phi_{10}\Phi_{12}\Phi_
{18}&1\\
G_{33}[i]&\frac12x^8\Phi_1^5\Phi_3^3\Phi_4^2\Phi_5\Phi_9\Phi_{12}&ix^{1/2}\\
G_{33}[-i]&\frac12x^8\Phi_1^5\Phi_3^3\Phi_4^2\Phi_5\Phi_9\Phi_{12}&-ix^{1/2}
\\
\hline
*\hfill \phi_{15,9}&x^9\Phi_5\Phi_9\Phi_{10}\Phi_{12}\Phi_{18}&1\\
\hline
*\hfill \phi_{45,10}&\frac{3-\sqrt {-3}}6x^{10}{\Phi'_3}^3\Phi_5{\Phi''_6}
^3\Phi_9\Phi_{10}{\Phi'''_{12}}\Phi_{18}&1\\
\#\hfill \phi_{45,12}&\frac{3+\sqrt {-3}}6x^{10}{\Phi''_3}^3\Phi_5{\Phi'_6}
^3\Phi_9\Phi_{10}{\Phi''''_{12}}\Phi_{18}&1\\
G_{3,3,3}[\zeta_3^2]:\phi_{3,2}&\frac{\sqrt {-3}}3x^{10}\Phi_1^3\Phi_2^3\Phi_
4\Phi_5\Phi_9\Phi_{10}\Phi_{18}&\zeta_3^2\\
\hline
*\hfill \phi_{81,11}&x^{11}\Phi_3^3\Phi_6^3\Phi_9\Phi_{12}\Phi_{18}&1\\
\hline
*\hfill \phi_{60,10}&x^{10}\Phi_4^2\Phi_5\Phi_9\Phi_{10}\Phi_{12}\Phi_{18}&1\\
\hline
*\hfill \phi_{15,12}&x^{12}\Phi_5\Phi_9\Phi_{10}\Phi_{12}\Phi_{18}&1\\
\hline
*\hfill \phi_{30,13}&\frac{-\zeta_3}6x^{13}{\Phi''_3}^3\Phi_4^2\Phi_5{\Phi'_9}
\Phi_{10}\Phi_{12}\Phi_{18}&1\\
\phi_{40,14}''&\frac{3-\sqrt {-3}}{12}x^{13}\Phi_2^4\Phi_5\Phi_6^2{\Phi''_6}{
\Phi''_9}\Phi_{10}{\Phi'''_{12}}\Phi_{18}&1\\
\phi_{20,15}&\frac13x^{13}\Phi_4^2\Phi_5\Phi_9\Phi_{10}\Phi_{12}\Phi_{18}&1\\
\phi_{40,14}'&\frac{3+\sqrt {-3}}{12}x^{13}\Phi_2^4\Phi_5\Phi_6^2{\Phi'_6}{
\Phi'_9}\Phi_{10}{\Phi''''_{12}}\Phi_{18}&1\\
\#\hfill \phi_{30,15}&\frac{-\zeta_3^2}6x^{13}{\Phi'_3}^3\Phi_4^2\Phi_5{
\Phi''_9}\Phi_{10}\Phi_{12}\Phi_{18}&1\\
G_{3,3,3}[\zeta_3]:\phi_{1,8}&\frac16x^{13}\Phi_1^3\Phi_2^3\Phi_4^2\Phi_5{
\Phi'_9}\Phi_{10}\Phi_{12}{\Phi'_{18}}&\zeta_3\\
D_4:\zeta_3&\frac{-3-\sqrt {-3}}{12}x^{13}\Phi_1^4\Phi_3^2{\Phi''_3}\Phi_
5\Phi_9\Phi_{10}{\Phi''''_{12}}{\Phi''_{18}}&-1\\
G_{3,3,3}[\zeta_3^2]:\phi_{2,3}&\frac{-\zeta_3^2}3x^{13}\Phi_1^3\Phi_2^3\Phi_
4^2\Phi_5{\Phi'_9}\Phi_{10}\Phi_{12}{\Phi''_{18}}&\zeta_3^2\\
G_{33}[-\zeta_3]&\frac{\sqrt {-3}}6x^{13}\Phi_1^5\Phi_2^3\Phi_3^2\Phi_4\Phi_
5\Phi_9\Phi_{10}&-\zeta_3^4\\
\phi_{10,17}''&\frac{\zeta_3^2}6x^{13}\Phi_4^2\Phi_5{\Phi''_6}^3\Phi_9\Phi_{
10}\Phi_{12}{\Phi'_{18}}&1\\
G_{3,3,3}[\zeta_3]:\phi_{2,5}&\frac{-\sqrt {-3}}6x^{13}\Phi_1^3\Phi_2^5\Phi_
4\Phi_5\Phi_6^2\Phi_{10}\Phi_{18}&\zeta_3\\
G_{3,3,3}[\zeta_3^2]:\phi_{2,1}&\frac{\zeta_3}3x^{13}\Phi_1^3\Phi_2^3\Phi_
4^2\Phi_5{\Phi''_9}\Phi_{10}\Phi_{12}{\Phi'_{18}}&\zeta_3^2\\
\phi_{10,17}'&\frac{\zeta_3}6x^{13}\Phi_4^2\Phi_5{\Phi'_6}^3\Phi_9\Phi_{10}
\Phi_{12}{\Phi''_{18}}&1\\
D_4:\zeta_3^2&\frac{-3+\sqrt {-3}}{12}x^{13}\Phi_1^4\Phi_3^2{\Phi'_3}\Phi_
5\Phi_9\Phi_{10}{\Phi'''_{12}}{\Phi'_{18}}&-1\\
G_{3,3,3}[\zeta_3]:\phi_{1,4}&\frac{-1}6x^{13}\Phi_1^3\Phi_2^3\Phi_4^2\Phi_5{
\Phi''_9}\Phi_{10}\Phi_{12}{\Phi''_{18}}&\zeta_3\\
\hline
*\hfill \phi_{30,18}&\frac12x^{18}\Phi_4^2\Phi_5\Phi_9\Phi_{12}\Phi_{18}&1\\
\phi_{6,20}&\frac12x^{18}\Phi_4^2\Phi_9\Phi_{10}\Phi_{12}\Phi_{18}&1\\
\phi_{24,19}&\frac12x^{18}\Phi_2^4\Phi_6^2\Phi_9\Phi_{10}\Phi_{18}&1\\
D_4:-1&\frac12x^{18}\Phi_1^4\Phi_3^2\Phi_5\Phi_9\Phi_{18}&-1\\
\hline
*\hfill \phi_{15,23}&x^{23}\Phi_5\Phi_9\Phi_{10}\Phi_{18}&1\\
\hline
*\hfill \phi_{5,28}&\frac{3-\sqrt {-3}}6x^{28}\Phi_5{\Phi''_9}\Phi_{10}{
\Phi''''_{12}}{\Phi'_{18}}&1\\
\#\hfill \phi_{5,30}&\frac{3+\sqrt {-3}}6x^{28}\Phi_5{\Phi'_9}\Phi_{10}{
\Phi'''_{12}}{\Phi''_{18}}&1\\
G_{3,3,3}[\zeta_3^2]:\phi_{1,0}&\frac{\sqrt {-3}}3x^{28}\Phi_1^3\Phi_2^3\Phi_
4\Phi_5\Phi_{10}&\zeta_3^2\\
\hline
*\hfill \phi_{1,45}&x^{45}&1\\
\end{supertabular}}
\vfill\eject
\subsection{Unipotent characters for $G_{34}$}
\begin{center}Some principal $\zeta$-series\end{center}
$\zeta_{7}^{6}$ : ${\mathcal H}_{Z_{42}}(\zeta_7^6x^6,\allowbreak -\zeta_{21}
^{11}x^{7/2},\allowbreak \zeta_{21}^4x^{10/3},\allowbreak -
x^3,\allowbreak \zeta_{21}^{11}x^3,\allowbreak  -\zeta_{21}^{25}
x^4,\allowbreak \zeta_7^6x^{8/3},\allowbreak -\zeta_{21}^{11}x^{5/2}
,\allowbreak \zeta_{21}^4x^{7/2},\allowbreak -\zeta_7^8x^3,\allowbreak \zeta_{
21}^{11}x^2,\allowbreak -\zeta_{21}^{25}x^3,\allowbreak \zeta_
7^6x^4,\allowbreak -\zeta_{21}^{11}x^5,\allowbreak \zeta_{21}^4x^{5/2}
,\allowbreak -\zeta_7^9x^3,\allowbreak \zeta_{21}^{11}x^{10/3},\allowbreak -
\zeta_{21}^{25}x^2,\allowbreak \zeta_7^6x^3,\allowbreak -\zeta_{21}^{11}
x^4,\allowbreak \zeta_{21}^4x^{8/3},\allowbreak -\zeta_7^{10}
x^3,\allowbreak \zeta_{21}^{11}x^{7/2},\allowbreak -\zeta_{21}^{25}
x,\allowbreak \zeta_7^6x^2,\allowbreak -\zeta_{21}^{11}x^3,\allowbreak \zeta_{
21}^4x^4,\allowbreak -\zeta_7^4x^3,\allowbreak \zeta_{21}^{11}x^{5/2}
,\allowbreak -\zeta_{21}^{25}x^{7/2},\allowbreak \zeta_7^6x^{10/3}
,\allowbreak -\zeta_{21}^{11}x^2,\allowbreak \zeta_{21}^4x^3,\allowbreak -
\zeta_7^5x^3,\allowbreak \zeta_{21}^{11}x^{8/3},\allowbreak -\zeta_{21}^{25}
x^{5/2},\allowbreak \zeta_7^6,\allowbreak -\zeta_{21}^{11}x,\allowbreak \zeta_
{21}^4x^2,\allowbreak -\zeta_7^6x^3,\allowbreak \zeta_{21}^{11}
x^4,\allowbreak -\zeta_{21}^{25}x^5)$\hfill\break
$\zeta_{7}^{5}$ : ${\mathcal H}_{Z_{42}}(\zeta_7^5x^6,\allowbreak -\zeta_{21}
^{29}x^3,\allowbreak \zeta_{21}x^{7/2},\allowbreak -\zeta_7^{10}
x^3,\allowbreak \zeta_{21}^8x^{10/3},\allowbreak -\zeta_{21}^{22}
x^5,\allowbreak \zeta_7^5x^2,\allowbreak -\zeta_{21}^{29}x^{5/2}
,\allowbreak \zeta_{21}x^3,\allowbreak -\zeta_7^4x^3,\allowbreak \zeta_{21}
^8x^4,\allowbreak -\zeta_{21}^{22}x,\allowbreak \zeta_7^5x^{8/3},\allowbreak -
\zeta_{21}^{29}x^2,\allowbreak \zeta_{21}x^{5/2},\allowbreak -\zeta_
7^5x^3,\allowbreak \zeta_{21}^8x^{7/2},\allowbreak -\zeta_{21}^{22}
x^4,\allowbreak \zeta_7^5x^{10/3},\allowbreak -\zeta_{21}^{29}
x^5,\allowbreak \zeta_{21}x^2,\allowbreak -\zeta_7^6x^3,\allowbreak \zeta_{21}
^8x^3,\allowbreak -\zeta_{21}^{22}x^{7/2},\allowbreak \zeta_
7^5x^4,\allowbreak -\zeta_{21}^{29}x,\allowbreak \zeta_{21}x^{8/3}
,\allowbreak -x^3,\allowbreak \zeta_{21}^8x^{5/2},\allowbreak -\zeta_{21}^{22}
x^3,\allowbreak \zeta_7^5,\allowbreak -\zeta_{21}^{29}x^4,\allowbreak \zeta_{
21}x^{10/3},\allowbreak -\zeta_7^8x^3,\allowbreak \zeta_{21}
^8x^2,\allowbreak -\zeta_{21}^{22}x^{5/2},\allowbreak \zeta_
7^5x^3,\allowbreak -\zeta_{21}^{29}x^{7/2},\allowbreak \zeta_{21}
x^4,\allowbreak -\zeta_7^9x^3,\allowbreak \zeta_{21}^8x^{8/3},\allowbreak -
\zeta_{21}^{22}x^2)$\hfill\break
$\zeta_{7}^{4}$ : ${\mathcal H}_{Z_{42}}(\zeta_7^4x^6,\allowbreak -\zeta_{21}
^{26}x^4,\allowbreak \zeta_{21}^{19}x^2,\allowbreak -\zeta_
7^6x^3,\allowbreak \zeta_{21}^5x^{8/3},\allowbreak -\zeta_{21}^{19}
x^3,\allowbreak \zeta_7^4x^{10/3},\allowbreak -\zeta_{21}^{26}x^{5/2}
,\allowbreak \zeta_{21}^{19}x^4,\allowbreak -x^3,\allowbreak \zeta_{21}^5x^{
7/2},\allowbreak -\zeta_{21}^{19}x^5,\allowbreak \zeta_7^4x^3,\allowbreak -
\zeta_{21}^{26}x,\allowbreak \zeta_{21}^{19}x^{5/2},\allowbreak -\zeta_
7^8x^3,\allowbreak \zeta_{21}^5x^2,\allowbreak -\zeta_{21}^{19}x^{7/2}
,\allowbreak \zeta_7^4x^{8/3},\allowbreak -\zeta_{21}^{26}
x^3,\allowbreak \zeta_{21}^{19}x^{10/3},\allowbreak -\zeta_
7^9x^3,\allowbreak \zeta_{21}^5x^4,\allowbreak -\zeta_{21}^{19}
x^2,\allowbreak \zeta_7^4,\allowbreak -\zeta_{21}^{26}x^5,\allowbreak \zeta_{
21}^{19}x^3,\allowbreak -\zeta_7^{10}x^3,\allowbreak \zeta_{21}^5x^{5/2}
,\allowbreak -\zeta_{21}^{19}x^4,\allowbreak \zeta_7^4x^2,\allowbreak -\zeta_{
21}^{26}x^{7/2},\allowbreak \zeta_{21}^{19}x^{8/3},\allowbreak -\zeta_
7^4x^3,\allowbreak \zeta_{21}^5x^{10/3},\allowbreak -\zeta_{21}^{19}x^{5/2}
,\allowbreak \zeta_7^4x^4,\allowbreak -\zeta_{21}^{26}x^2,\allowbreak \zeta_{
21}^{19}x^{7/2},\allowbreak -\zeta_7^5x^3,\allowbreak \zeta_{21}
^5x^3,\allowbreak -\zeta_{21}^{19}x)$\hfill\break
$\zeta_{7}$ : ${\mathcal H}_{Z_{42}}(\zeta_7x^6,\allowbreak -\zeta_{21}^{17}
x^5,\allowbreak \zeta_{21}^{10}x^4,\allowbreak -\zeta_
7^8x^3,\allowbreak \zeta_{21}^{17}x^2,\allowbreak -\zeta_{21}^{31}
x,\allowbreak \zeta_7,\allowbreak -\zeta_{21}^{17}x^{5/2},\allowbreak \zeta_{
21}^{10}x^{8/3},\allowbreak -\zeta_7^9x^3,\allowbreak \zeta_{21}^{17}
x^3,\allowbreak -\zeta_{21}^{31}x^2,\allowbreak \zeta_7x^{10/3},\allowbreak -
\zeta_{21}^{17}x^{7/2},\allowbreak \zeta_{21}^{10}x^{5/2},\allowbreak -\zeta_
7^{10}x^3,\allowbreak \zeta_{21}^{17}x^4,\allowbreak -\zeta_{21}^{31}
x^3,\allowbreak \zeta_7x^2,\allowbreak -\zeta_{21}^{17}x,\allowbreak \zeta_{
21}^{10}x^{7/2},\allowbreak -\zeta_7^4x^3,\allowbreak \zeta_{21}^{17}x^{8/3}
,\allowbreak -\zeta_{21}^{31}x^4,\allowbreak \zeta_7x^3,\allowbreak -\zeta_{
21}^{17}x^2,\allowbreak \zeta_{21}^{10}x^{10/3},\allowbreak -\zeta_
7^5x^3,\allowbreak \zeta_{21}^{17}x^{5/2},\allowbreak -\zeta_{21}^{31}
x^5,\allowbreak \zeta_7x^4,\allowbreak -\zeta_{21}^{17}x^3,\allowbreak \zeta_{
21}^{10}x^2,\allowbreak -\zeta_7^6x^3,\allowbreak \zeta_{21}^{17}x^{7/2}
,\allowbreak -\zeta_{21}^{31}x^{5/2},\allowbreak \zeta_7x^{8/3},\allowbreak -
\zeta_{21}^{17}x^4,\allowbreak \zeta_{21}^{10}x^3,\allowbreak -
x^3,\allowbreak \zeta_{21}^{17}x^{10/3},\allowbreak -\zeta_{21}^{31}x^{7/2}
)$\hfill\break
\begin{center}Non-principal $1$-Harish-Chandra series\end{center}
${\mathcal H}_{G_{34}}(G_{3,3,3}[\zeta_{3}])={\mathcal H}_{G_{26}}
(x,-1;x^3,\zeta_3,\zeta_3^2)$\hfill\break
${\mathcal H}_{G_{34}}(G_{3,3,3}[\zeta_{3}^2])={\mathcal H}_{G_{26}}
(x,-1;1,\allowbreak \zeta_3x^3,\allowbreak \zeta_3^2x^3)$\hfill\break
${\mathcal H}_{G_{34}}(D_4)={\mathcal H}_{G_{6,1,2}}(x^5,\allowbreak -\zeta_
3^2x^4,\allowbreak \zeta_3x,\allowbreak -1,\allowbreak \zeta_
3^2x,\allowbreak -\zeta_3^4x^4;x^4,-1)$\hfill\break
${\mathcal H}_{G_{34}}(G_{33}[i])={\mathcal H}_{Z_{6}}(x^5,\allowbreak -\zeta_
3^2,\allowbreak \zeta_3x^7,\allowbreak -x^2,\allowbreak \zeta_
3^2x^7,\allowbreak -\zeta_3^4)$\hfill\break
${\mathcal H}_{G_{34}}(G_{33}[-i])={\mathcal H}_{Z_{6}}(x^5,\allowbreak -
\zeta_3^2,\allowbreak \zeta_3x^7,\allowbreak -x^2,\allowbreak \zeta_
3^2x^7,\allowbreak -\zeta_3^4)$\hfill\break
${\mathcal H}_{G_{34}}(G_{33}[-\zeta_3])={\mathcal H}_{Z_{6}}
(x^3,\allowbreak -\zeta_3^2x^8,\allowbreak \zeta_3x^7,\allowbreak -
1,\allowbreak \zeta_3^2x^7,\allowbreak -\zeta_3^4x^8)$\hfill\break
${\mathcal H}_{G_{34}}(G_{33}[-\zeta_3^2])={\mathcal H}_{Z_{6}}
(x^8,\allowbreak -\zeta_3^2x,\allowbreak \zeta_3,\allowbreak -
x^5,\allowbreak \zeta_3^2,\allowbreak -\zeta_3^4x)$\hfill\break
{\footnotesize
\par\tablehead{\hline \gamma&\mbox{Deg($\gamma$)}&\mbox{Fr($\gamma$)}\\\hline}
\tabletail{\hline}
\begin{supertabular}{|R|RR|}
\shrinkheight{30pt}
*\hfill \phi_{1,0}&1&1\\
\hline
*\hfill \phi_{6,1}&\frac{3+\sqrt {-3}}6x{\Phi'_3}^3{\Phi''_6}^3\Phi_8{\Phi'''_
{12}}{\Phi''''_{15}}{\Phi'_{21}}\Phi_{24}{\Phi'''_{30}}{\Phi''_{42}}&1\\
\#\hfill \phi_{6,5}&\frac{3-\sqrt {-3}}6x{\Phi''_3}^3{\Phi'_6}^3\Phi_8{
\Phi''''_{12}}{\Phi'''_{15}}{\Phi''_{21}}\Phi_{24}{\Phi''''_{30}}{\Phi'_{42}}&
1\\
G_{3,3,3}[\zeta_{3}]:\phi_{1,0}&\frac{-\sqrt {-3}}3x\Phi_1^3\Phi_2^3\Phi_
4\Phi_5\Phi_7\Phi_8\Phi_{10}\Phi_{14}\Phi_{24}&\zeta_3\\
\hline
*\hfill \phi_{21,2}&\frac{3+\sqrt {-3}}6x^2{\Phi''_3}^3{\Phi'_6}^3\Phi_7{
\Phi''''_{12}}\Phi_{14}{\Phi'''_{15}}\Phi_{21}{\Phi'_{24}}{\Phi''''_{30}}\Phi_
{42}&1\\
\#\hfill \phi_{21,4}&\frac{3-\sqrt {-3}}6x^2{\Phi'_3}^3{\Phi''_6}^3\Phi_7{
\Phi'''_{12}}\Phi_{14}{\Phi''''_{15}}\Phi_{21}{\Phi''_{24}}{\Phi'''_{30}}\Phi_
{42}&1\\
G_{3,3,3}[\zeta_{3}]:\phi_{1,9}&\frac{-\sqrt {-3}}3x^2\Phi_1^3\Phi_2^3\Phi_
4\Phi_5\Phi_7\Phi_8\Phi_{10}\Phi_{14}\Phi_{21}\Phi_{42}&\zeta_3\\
\hline
*\hfill \phi_{56,3}&\frac12x^3\Phi_2^4\Phi_6^4\Phi_7\Phi_{10}\Phi_{14}\Phi_{
18}\Phi_{21}\Phi_{30}\Phi_{42}&1\\
\hline
\end{supertabular}
}
\newpage
\section{Errata for \cite{sp1}.}

\begin{itemize}
\item
Proof of 1.17: this forgets the case of $G(2e,e,2)$ with 3 classes of
hyperplanes. This case is still open.

\item
Page 184, generalized sign: let $\Delta_\BG$ be the eigenvalue of 
$\phi$  on  the discriminant $\Delta$. Then change the definition of
$\varepsilon_\BG$ to
$\varepsilon_\BG=(-1)^r\zeta_1\ldots\zeta_r\Delta_\BG^*$ where $r=\dim V$.

Most of the subsequent errata come from this, and are superceded by results in
the current paper, see in particular \ref{correctspets1}.
\item
3 lines below: $\{\zeta_1,\ldots,\zeta_r\}$ is the spectrum of $w\phi$ (in its
action on $V'$), and $\Delta_\BG=1$. In particular we have then
$\varepsilon_\BG=(-1)^r\det_V(w\phi)$. In general, if
$w\phi$   is  $\zeta$-regular  then the spectrum of $w\phi$  is
$\{\zeta_i  \zeta^{-d_i+1}\}$  so $\det_V(w\phi)  =  \zeta_1\ldots\zeta_r
\zeta^{-N^\vee}$; and
$\varepsilon_\BG = (-1)^r\zeta^{-N}\det_V(w\phi)$.

\item
Second line of 3.3:  ``Moreover, if there exists 
$v\phi\in W\phi$  such that $v\phi$ admits a fixed point in
$V-\bigcup_{H\in\CA}H$, then 
$\varepsilon_\BT=\varepsilon_\BG \det_V(vw^{-1})^*$.''

\item
3.5: $|\BG|=\varepsilon_\BG
x^N\prod_i(1-\zeta_i^*x^{d_j})=x^N\Delta_\BG^* \prod_i(x^{d_i}-\zeta_i)$.

\item
3.6 (i): $|\BG|=x^N\Delta_\BG^*\prod_\Phi\Phi^{a(\Phi)}$.

\item
3.6  (ii): $|\BG|(1/x)=\Delta_\BG^*\varepsilon_\BG
x^{-(2N+N^\vee+r)}|\BG|(x)^*$.

\item
3.7 (ls.2): $\phi^{(a)}$   is  the product of
$(1,\ldots,1,\phi)$ acting on $V\times\ldots\times V$ by the  $a$-cycle which
permutes cyclically the factors $V$ of $V^{(a)}$. 
The $\zeta_i$  for $\BG^{(a)}$
are    $\{\zeta_a^j\zeta_i^{1/a}\}_{j=0..a-1,   i=1..r}$ and
$\Delta_{\BG^{(a)}}=\Delta_\BG$  so $|\BG^{(a)}|(x)=|\BG|(x^a)$.

\item
3.8:  The   $\zeta_i$  for  $\BG^\zeta$  are   $\zeta^{d_i}\zeta_i$  and
$\Delta_{\BG^\zeta}=\Delta_\BG\zeta^{N+N^\vee}$  and thus  we  get
$|\BG^\zeta|(x)=\zeta^r|\BG|(\zeta^{-1}x)$.

\item
4.9: $\Deg(R^\BG_{w\phi})=\tr_{R\BG}(w\phi)^*$.

\item
4.25, second equality: $\Deg(\alpha\det_{V^*})=\Delta_\BG(-1)^r\varepsilon_\BG^*
x^{N^\vee} \Deg_\BG(\alpha^*)(1/x)^*$.

\item
4.26, first equality: $\Deg(\det_V^*)=(-1)^r\Delta_\BG\varepsilon_\BG^* x^{N^\vee}$.

\item
bottom of page 198: suppress the first $|\BG|S_\BG(\alpha)$.

\item
last equality in proof of 5.3: suppress $TG$.

\item
5.4:   In particular we have:
$$
\begin{aligned}
\varepsilon_\BG x^{N(\BG)}&\equiv \varepsilon_\BL x^{N(\BL)}\pmod\Phi,\\
\Delta_\BG\varepsilon_\BG^* x^{N^\vee(\BG)}&\equiv 
\Delta_\BL\varepsilon_\BL^* x^{N^\vee(\BL)}\pmod\Phi,\\
\Delta_\BG  x^{N(\BG)+N^\vee(\BG)}&\equiv \Delta_\BL x^{N(\BL)+N^\vee(\BL)}
\pmod\Phi
\end{aligned}
$$
\item
6.1(b):  the polynomial in $t$
$$\prod_{j=0}^{j=e_\CC-1}(t-\zeta^j_{e_\CC}y^{n_{\CC,j}|\mu(K)|})$$
\end{itemize}
\newpage
\printindex

\begin{thebibliography}{BrMaRo}

\bibitem[Ben]{benard}
 \textsc{M. Benard,}
 {\em Schur indices and splitting fields of the unitary reflection groups,}
 J. Algebra {\bf 38} (1976), 318--342.

\bibitem[Bes1]{field}
 \textsc{D.~Bessis,}
 {\em Sur le corps de d\'efinition d'un groupe de r\'eflexions complexe,}
 Comm. in Algebra {\bf  25} (8) (1997), 2703--2716.


\bibitem[Bes3]{bessiszariski} 
 \textsc{D.~Bessis,}
 {\em Zariski theorems and diagrams for braid groups,}
  Invent. Math. {\bf 145} (2001), 487--507.

\bibitem[Bes4]{bessisarxiv} 
 \textsc{D.~Bessis,}
 {\em Complex reflection arrangements are $K(\pi,1)$,}
 ArXiv~: math/0610777.


\bibitem[BoRo]{boro}
 \textsc{C.~Bonaf\'e and R.~Rouquier,}
 {\em On the irreducibility of Deligne-Lusztig varieties,}
 CRAS {\bf 343} (2006), 37--39.

\bibitem[Bou]{bou}
 \textsc{N.~Bourbaki,}
 {\em Groupes et Alg\`ebres de Lie, Chap. IV, V et VI,}
 Hermann,  Paris, 1968.

\bibitem[BreMa]{brema}
  \textsc{K.~Bremke and G.~Malle,}
   {\em Reduced words and a length function for $G(e,1,n)$,}
  Indag. Mathem. {\bf 8} (1997), 453--469.

\bibitem[BrSa]{brsa}
 \textsc{E.~Brieskorn and K.~Saito,}
 {\em Artin-Gruppen und Coxeter-Gruppen,}
 Invent. Math. {\bf 17} (1972), 245--271.

\bibitem[Bro0]{abdefconj}
 \textsc{M.~Brou\'e,}
 {\em Isom\'etries parfaites, types de blocs, cat\'egories d\'eriv\'ees,}
 Ast\'erisque {\bf 181--182} (1990), 61--92.

\bibitem[Bro1]{boston}
 \textsc{M.~Brou\'e,}
 {\em  Reflection Groups, Braid Groups, Hecke Algebras,
      Finite Reductive Groups,}
 Current Developments in Mathematics, 2000, 
 Harvard Univ. 
 (B.~Mazur, W.~Schmidt, S.~T.~Yau) and M.~I.~T. (J.~de Jong, D.~Jerison, G.~Lusztig)
 International Press, Boston (2001), 1--107.

\bibitem[Bro2]{berkeley} 
 \textsc{M.~Brou\'e,}
 {\em  Introduction to complex reflection groups and their braid groups,}
 Springer LNM 1988 (2010).

\bibitem[BrKi]{brki}
 \textsc{M.~Brou\'e and S.~Kim,}
 {\em Familles de caract\`eres des alg\`ebres de Hecke cyclotomiques,}
 Adv. Math. {\bf 172} (2003), 53--136.

\bibitem[BrMa1]{brma1}
 \textsc{M.~Brou\'e and G.~Malle,}
 {\em Th\'eor\`emes de Sylow g\'en\'eriques pour les
 groupes r\'eductifs sur les corps finis,}
 Math. Ann. {\bf 292} (1992), 241--262.

\bibitem[BrMa2]{brma2}
 \textsc{M.~Brou\'e and G.~Malle,}
 {\em Zyklotomische Heckealgebren,}
 Ast\'erisque {\bf 212} (1993), 119--189.

\bibitem[BrMaMi]{bmm}
 \textsc{M. Brou\'e, G. Malle and J. Michel,}
 {\em Generic blocks of finite reductive groups,} 
 Ast\'erisque {\bf 212} (1993), 7--92.

\bibitem[Spets1]{sp1}
 \textsc{M.~Brou\'e, G.~Malle and J.~Michel,} 
 {\em Towards Spetses I,}
 Transform. groups {\bf 4} (1999), 157--218.

\bibitem[BrMaRo]{bmr} 
 \textsc{M.~Brou\'e, G.~Malle and R.~Rouquier,}
 {\em Complex reflection groups, braid groups, Hecke algebras,}
 J. reine angew. Math. {\bf 500} (1998), 127--190. 

\bibitem[BrMi]{brmi} 
 \textsc{M.~Brou\'e and J.~Michel, }
 {\em Sur certains \'el\'ements r\'eguliers des
 groupes de Weyl et les vari\'et\'es de Deligne-Lusztig associ\'ees,}
 Progress in Math. {\bf 141} (1997), 73--140.

\bibitem[Chl1]{maria} 
 \textsc{M.~Chlouveraki,}
 {\em Blocks and Families for Cyclotomic Hecke Algebras,}
 Springer LNM {\bf 1981} (2009).

\bibitem[Chl2]{maria2} 
 \textsc{M.~Chlouveraki,}
 {\em Degree and valuation of the Schur elements of cyclotomic Hecke algebras,}
 J. Algebra {\bf 320} (2008), 3935--3949.

\bibitem[Del1]{deligne1}
 \textsc{P.~Deligne,}
 {\em Les immeubles des groupes de tresses g\'en\'eralis\'es,}
 Invent. Math. {\bf 17} (1972), 273--302.

\bibitem[Del2]{deligne2}
 \textsc{P.~Deligne,}
 {\em Dualit\'e,}
 Cohomologie \'etale (SGA $4\frac 12$),
 Springer LNM {\bf 569} (1977), Berlin, 154--167.

\bibitem[Del3]{deligne3}
 \textsc{P.~Deligne,}
 {\em Action du groupe des tresses sur une cat\'egorie,}
 Invent. Math. {\bf 128} (1997), 159--175.

\bibitem[DeLu]{delu}
 \textsc{P.~Deligne and G.~Lusztig,}
 {\em Reductive groups over finite fields,}
Annals  of Math. {\bf 103} (1976),
103--161.

\bibitem[DiMaMi]{dmm} 
 \textsc{F.~Digne, I.~Marin and J.~Michel,}
 {\em The Center of Pure Complex Braid Groups,}
 J. of Alg. {\bf 347} (2011), 206--213.

\bibitem[DiMi1]{dm0} 
 \textsc{F.~Digne and J.~Michel,}
 {\em Fonctions $L$ des vari\'et\'es de Deligne-Lusztig et descente de Shintani,}
 M\'emoires SMF {\bf 20} (1985), 1--144.

\bibitem[DiMi2]{dm} 
 \textsc{F.~Digne and J.~Michel,}
 {\em Endomorphisms of Deligne-Lusztig varieties,}
 Nagoya Math. J. {\bf 183} (2006), 35--103.

\bibitem[DiMi3]{dm2} 
 \textsc{F.~Digne and J.~Michel,}
 {\em Parabolic Deligne-Lusztig varieties,}
 arXiv:1110.4863.

\bibitem[DiMiRo]{dmr}  
 \textsc{F.~Digne, J.~Michel and R.~Rouquier, }
 {\em Cohomologie des vari\'et\'es de Deligne-Lusztig,}
 Adv. Math. {\bf 209} (2007), 749--822.

\bibitem[Ge1]{geck}
 \textsc{M.~Geck,}
 {\em Beitr\"age zur Darstellungstheorie von Iwahori--Hecke--Algebren,}
 RWTH Aachen, Habilitations\-schrift (1993).

\bibitem[Ge2]{geck2}
\textsc{M.~Geck,}
{\em The Schur indices of the cuspidal unipotent characters of the Chevalley
groups $E_7(q)$}, 
Osaka J. Math. {\bf 42} (2005), 201--215.

\bibitem[GIM]{gim}
 \textsc{M.~Geck, L.~Iancu and G.~Malle,} 
 {\em Weights of Markov traces and generic degrees,}
  Indag. Mathem. {\bf 11} (2000), 379--397.



\bibitem[Lu]{cox}
 \textsc{G.~Lusztig,}
 {\em Coxeter orbits and eigenvalues of Frobenius,}
 Invent. Math. {\bf 38} (1976), 101--159.
 
 \bibitem[Lu1]{lu1}
 \textsc{G.~Lusztig,}
 {\em A class of irreducible representations of a Weyl group,}
 Indag. Math. {\bf 41} (1979), 323--335.
 
  \bibitem[Lu2]{lu2}
 \textsc{G.~Lusztig,}
 {\em A class of irreducible representations of a Weyl group II,}
 Proc. Ned. Acad. {\bf 85} (1982), 219--2269.

\bibitem[Lu3]{app}
 \textsc{G.~Lusztig,}
 {\em Coxeter groups and unipotent representations,}
 Ast\'erisque {\bf 212} (1993), 191--203.

\bibitem[Lu4]{exo}
 \textsc{G.~Lusztig,}
 {\em Exotic Fourier transform,}
 Duke J. Math. {\bf 73} (1994), 243--248.

\bibitem[Ma0]{maFou}
 \textsc{G.~Malle,}
 {\em Appendix: An exotic Fourier transform for $H_4$}
  Duke J. Math. {\bf 73} (1994), 243--248.

\bibitem[Ma1]{maIG}
 \textsc{G.~Malle,}
 {\em  Unipotente Grade imprimitiver komplexer Spiegelungsgruppen,}
 J. Algebra {\bf 177} (1995), 768--826.

\bibitem[Ma2]{maS} 
 \textsc{G.~Malle,}
 {\em Spetses,}
 Doc. Math. {ICM II} (1998), 87--96.

\bibitem[Ma3]{maR} 
 \textsc{G. Malle,}
 {\em On the rationality and fake degrees of characters of
 cyclotomic algebras,}
 J. Math. Sci. Univ. Tokyo {\bf 6} (1999), 647--677.

\bibitem[Ma4]{maG} 
 \textsc{G.~Malle,}
 {\em On the generic degrees of cyclotomic algebras,}
 Represent. Theory {\bf 4} (2000), 342--369.

\bibitem[Ma5]{maEg}
 \textsc{G.~Malle,}
 {\em Splitting fields for extended complex reflection groups and Hecke
 algebras,}
 Transform. Groups {\bf 11} (2) (2006), 195--216.

\bibitem[MaMi]{mami}
 \textsc{G.~Malle and J.~Michel,}
 {\em Constructing representations of Hecke algebras for complex reflection
  groups,}
 LMS J. Comput. Math. {\bf 13} (2010), 426--450.

\bibitem[MaRo]{maro}
 \textsc{G.~Malle and R.~Rouquier,}
 {\em Familles de caract\`eres des groupes de r\'eflexion complexes,}
 Represent. Theory {\bf 7} (2003), 610--640.

\bibitem[Rou]{rou}  
 \textsc{R. Rouquier,}
 {\em Familles et blocs d'alg{\`e}bres de Hecke,}
 C. R. Acad. Sciences {\bf 329} (1999), 1037--1042.

\bibitem[Sp]{springer}
 \textsc{T.~A.~Springer,}
 {\em Regular elements of finite reflection groups,}
 Invent. Math. {\bf 25} (1974), 159--198.

\bibitem[St]{steinberg}
 \textsc{R.~Steinberg,}
 {\em Endomorphisms of linear algebraic groups,}
 Memoirs of the AMS no. 80 (1968).


\end{thebibliography}
\end{document}